\documentclass[10pt]{article}
\usepackage[T1]{fontenc}
\usepackage[utf8]{inputenc}
\usepackage[english]{babel}
\usepackage[autostyle, english=american]{csquotes}
\MakeOuterQuote{"}

\usepackage{xcolor}
\definecolor{reference}{rgb}{0.20,0.36,0.74}
\definecolor{citation}{rgb}{0,.40,.80}
\usepackage{url}
\usepackage{breakurl}
\usepackage[breaklinks,colorlinks,linktocpage,urlcolor=blue,linkcolor=reference,citecolor=citation]{hyperref}

\usepackage{verbatim}
\usepackage{mathrsfs}
\usepackage{mathtools}
\usepackage{latexsym}
\usepackage{graphicx}
\usepackage{amscd,amssymb,amsmath,amsbsy,amsfonts,amsthm}
\usepackage[mathscr]{eucal}
\usepackage[all, 2cell]{xy}
\UseTwocells
\xyoption{2cell}{\UseTwocells}

\usepackage{enumerate}
\usepackage{tikz}
\usepackage[left=2cm, top=2cm, right=2cm, bottom=2cm]{geometry}
\usepackage{enumitem} 

\usepackage[backend=biber,style=alphabetic,sorting=nty,firstinits=true,maxbibnames=99,maxalphanames=99]{biblatex}

\usepackage[capitalise]{cleveref}
\crefformat{equation}{(#2#1#3)}

\DeclareFontFamily{OT1}{pzc}{}
\DeclareFontShape{OT1}{pzc}{m}{it}{<-> s * [1.200] pzcmi7t}{}
\DeclareMathAlphabet{\mathpzc}{OT1}{pzc}{m}{it}

\usepackage{relsize}
\usepackage[bbgreekl]{mathbbol}
\DeclareSymbolFontAlphabet{\mathbb}{AMSb} 
\DeclareSymbolFontAlphabet{\mathbbl}{bbold}

\newcommand{\Prism}{{\mathlarger{\mathbbl{\Delta}}}}

\DeclareTextFontCommand{\emdef}{\it}

\newtheorem{thm}{Theorem}
\newtheorem*{thm*}{Theorem}
\newtheorem{lem}[thm]{Lemma}
\newtheorem{cor}[thm]{Corollary}
\newtheorem*{cor*}{Corollary}
\newtheorem{prop}[thm]{Proposition}
\newtheorem*{prop*}{Proposition}

\theoremstyle{definition}
\newtheorem{defn}[thm]{Definition}
\newtheorem{construction}[thm]{Construction}
\newtheorem{notation}[thm]{Notation}

\newtheorem{ex}[thm]{Example}

\newtheorem{rem}[thm]{Remark}

\newtheorem{conj}[thm]{Conjecture}
\newtheorem*{conj*}{Conjecture}

\newtheorem{assumption}[thm]{Assumption}

\definecolor{note_color}{rgb}{0.0,0.7,0.0}

\newcommand{\eset}{\varnothing}
\newcommand{\fcat}{\mathscr}
\newcommand{\inj}{\hookrightarrow}
\newcommand{\surj}{\twoheadrightarrow}
\newcommand{\areq}{\mathbin{{\xrightarrow{\,\sim\,}}}}

\renewcommand{\phi}{\varphi}
\renewcommand{\epsilon}{\varepsilon}
\newcommand{\lra}{\longrightarrow}

\DeclareMathOperator{\rank}{rank}
\DeclareMathOperator{\Sym}{Sym}
\newcommand{\Mod}{\mathrm{Mod}}

\newcommand{\GL}{{\mathrm{GL}}}

\DeclareMathOperator{\gr}{gr}
\DeclareMathOperator{\rk}{rk}
\DeclareMathOperator{\tr}{tr}

\DeclareMathOperator{\coMod}{coMod}
\DeclareMathOperator{\Frac}{Frac}

\DeclareMathOperator{\Ad}{Ad}

\DeclareMathOperator{\Rep}{Rep}

\renewcommand{\Im}{\mathrm{Im}}

\newcommand{\Top}{ {\mathrm{Top}} }

\DeclareMathOperator{\LocSys}{LocSys}

\newcommand{\Sch}{ {\mathcal S\mathrm{ch}} }

\DeclareMathOperator{\Spec}{Spec}
\DeclareMathOperator{\Spf}{Spf}

\DeclareMathOperator{\Pic}{Pic}

\DeclareMathOperator{\Coh}{Coh}

\DeclareMathOperator{\QCoh}{QCoh}

\newcommand{\mstack}{\mathpzc}
\DeclareMathOperator{\PStk}{{\mathscr{PS}tk}}
\DeclareMathOperator{\Stk}{{\mathscr{S}tk}}
\newcommand{\Hdg}{{\mathrm H}}
\newcommand{\dR}{{\mathrm{dR}}}
\newcommand{\crys}{{\mathrm{crys}}}
\newcommand{\perf}{{\mathrm{perf}}}
\newcommand{\fp}{{\mathrm{fp}}}

\DeclareMathOperator{\Aff}{Aff}
\newcommand{\sm}{{\mathrm{sm}}}

\DeclareMathOperator{\rad}{rad}
\DeclareMathOperator{\Spa}{Spa}

\newcommand{\et}{{\mathrm{\acute et}}}
\newcommand{\pd}{{\mathrm{pd}}}

\newcommand{\Type}{ {\fcat{S}} }
\newcommand{\Set}{ {\fcat{S}\mathrm{et}} }

\newcommand{\Cat}{ {\fcat C\mathrm{at}} }

\newcommand{\Prs}{ {\mathrm{Pr}} }

\DeclareMathOperator{\Grp}{Grp}
\DeclareMathOperator{\Hom}{Hom}

\DeclareMathOperator{\Tor}{Tor}

\DeclareMathOperator{\End}{End}

\DeclareMathOperator{\Aut}{Aut}
\DeclareMathOperator{\Map}{Map}
\DeclareMathOperator{\Fun}{Fun}
\DeclareMathOperator{\Exc}{Exc}

\DeclareMathOperator{\Sp}{Sp}

\DeclareMathOperator{\creff}{cr}

\DeclareMathOperator*{\colim}{colim}
\DeclareMathOperator*{\fib}{fib}

\DeclareMathOperator*{\cofib}{cofib}

\DeclareMathOperator{\Alg}{Alg}
\DeclareMathOperator{\CAlg}{CAlg}

\DeclareMathOperator{\Lie}{Lie}
\DeclareMathOperator{\Lan}{Lan}
\DeclareMathOperator{\Ran}{Ran}

\newcommand{\prolim}[1][]{\lim\limits_{\xleftarrow[#1]{}}}

\DeclareMathOperator{\Tot}{Tot}

\DeclareMathOperator{\Fil}{Fil}
\DeclareMathOperator{\Id}{Id}
\newcommand{\PShv}{{\mathcal{PS}\mathrm{hv}}}
\newcommand{\Shv}{{\mathcal S\mathrm{hv}}}
\newcommand{\op}{{\mathrm{op}}}

\mathchardef\mdef="2D

\usepackage{enumerate}
\usepackage{tocloft}
\usepackage{textcomp}
\DeclareFontFamily{OT1}{pzc}{}
\DeclareFontShape{OT1}{pzc}{m}{it}{<-> s * [1.200] pzcmi7t}{}
\DeclareMathAlphabet{\mathpzc}{OT1}{pzc}{m}{it}

\DeclareMathOperator*{\holim}{holim}

\newcommand{\cotimes}{\mathbin{\widehat\otimes}}

\newcommand{\ol}{\overline}
\newcommand{\Ker}{\mr{Ker}}

\newcommand{\Ainf}{{A_{\mathrm{inf}}}}

\newcommand{\mf}{\mathfrak}
\newcommand{\mc}{\mathcal}
\newcommand{\mbb}{\mathbb}
\newcommand{\mr}{\mathrm}
\newcommand{\mathcalr}{\mathscr}

\newcommand{\ul}{\underline}
\newcommand{\xra}{\xrightarrow}

\newcommand{\RG}{R\Gamma}

\newcommand{\ra}{\rightarrow}
\newcommand{\ft}{{\mathrm{ft}}}
\newcommand{\lft}{{\mathrm{lft}}}
\newcommand{\lfp}{{\mathrm{lfp}}}
\newcommand{\fl}{{\mathrm{flat}}}
\newcommand{\fg}{{\mathrm{fg}}}
\newcommand{\free}{{\mathrm{free}}}

\newcommand{\id}{{\mathrm{id}}}
\newcommand{\fr}{^{(1)}}

\newcommand{\Fl}{{\mathrm{Fl}}}
\newcommand{\fppf}{{\mathrm{fppf}}}
\newcommand{\constr}{{\mathrm{c}}}
\newcommand{\ism}{{\xymatrix{\ar[r]^\sim &}}}

\DeclarePairedDelimiter\floor{\lfloor}{\rfloor}

\newcommand{\DMod}[1]{D(\Mod_{#1})}
\newcommand{\UMod}[1]{\Mod_{#1}}

\newcommand{\PrL}{\Prs^{\mathrm L}}
\newcommand{\PrR}{\Prs^{\mathrm R}}
\newcommand{\Afd}{{\mathrm{Afd}}}
\newcommand{\BK}{{\mathrm{BK}}}
\newcommand{\iinf}{{\mathrm{inf}}}
\newcommand{\an}{{\mathrm{an}}}
\newcommand{\ET}{{\mathrm{\acute et}}}
\newcommand{\Art}{{\mathrm{Art}}}
\newcommand{\length}{{\mathrm{length}}}
\newcommand{\HT}{{\mathrm{HT}}}

\usepackage{tikz}
\usetikzlibrary{matrix}

\newcommand{\Poly}{{\mathrm{Poly}}}
\newcommand{\sing}{{\mathrm{sing}}}
\newcommand{\fpqc}{{\mathrm{fpqc}}}

\newcommand{\tto}{{\xymatrix{\ar[r]&}}}

\numberwithin{thm}{subsection}
\numberwithin{equation}{subsection}

\title{$p$-adic Hodge theory for Artin stacks}
\author{Dmitry Kubrak and Artem Prikhodko}
\date{}

\begin{document}
\maketitle

\begin{abstract}
This work is devoted to the study of integral $p$-adic Hodge theory in the context of Artin stacks. For a Hodge-proper stack, using the formalism of prismatic cohomology, we establish a version of $p$-adic Hodge theory with the \'etale cohomology of the Raynaud generic fiber as an input. In particular, we show that the corresponding Galois representation is crystalline and that the associated Breuil-Kisin module is given by the prismatic cohomology. An interesting new feature of the stacky setting is that the natural map between \'etale cohomology of the algebraic and the Raynaud generic fibers is often an equivalence even outside of the proper case. In particular, we show that this holds for global quotients $[X/G]$ where $X$ is a smooth proper scheme and $G$ is a reductive group. As applications we deduce Totaro's conjectural inequality and also set up a theory of $A_{\mathrm{inf}}$-characteristic classes.
\end{abstract}

\tableofcontents

\section{Introduction}
During the last few years new results in integral $p$-adic Hodge theory that were achieved in the series of papers \cite{BMS1}, \cite{BMS2}, \cite{BS_prisms} have quite revolutionized the field. Meanwhile the demand for analogous results in the context of Artin stacks has also been growing. Motivated by his computations in \cite{Totaro_deRhamBG} Totaro conjectured an inequality on dimensions of de Rham and singular cohomology of the classifying stack of a reductive group and a similar conjecture was made independently for the case of a conical resolution in \cite{Kubrak_Travkin} by Travkin and the first author. At the same time the idea of using classifying stacks to construct examples of schemes with a certain pathological behavior of cohomology (which goes back at least to \cite{serre1958topologie}) has bore some fresh fruit with the construction of counterexamples to degeneration of the HKR spectral sequence in char $p$ and Hodge-de Rham degeneration over ramified mixed characteristic local rings in \cite{antieau2019counterexamples} and \cite{Shizhang} respectfully. Besides that in \cite{Mondal} the classical notion of a Dieudonn\'e module (as well as the more recent prismatic Dieudonn\'e module of \cite{AnschutzLeBras}) of a $p$-divisible group was given a natural interpretation in terms of crystalline (resp. prismatic) cohomology of the corresponding classifying stack. 

In this work we perform a systematic study of prismatic cohomology of smooth Artin stacks over $p$-adic bases. As in the schematic case, we show that it carries all the information about the de Rham, crystalline, and \'etale cohomology. If the stack is Hodge-proper it, in fact, "interpolates" between them as in the case of a smooth proper schemes. We then also establish $p$-adic Hodge theory in the Hodge-proper stacks setting: namely we show that the \'etale cohomology groups are crystalline Galois representations and that the associated Breuil-Kisin modules are given by the prismatic cohomology. The \'etale cohomology one sees, however, is \'etale cohomology of \emph{Raynaud's generic fiber}, while for the applications like Totaro's conjecture or \cite[Conjecture 5.2.3]{Kubrak_Travkin} one rather needs to use the \emph{algebraic generic fiber}. This problem doesn't show up in the case of proper schemes, essentially because then the natural map from Raynaud's generic fiber to the analytification of the algebraic generic fiber is an isomorphism. More or less the same argument applies to proper stacks. This, however, is not satisfactory, since main examples of our interest, like classifying stacks $BG$ for $G$ reductive, are almost never proper.

The bulk of the paper is then devoted to showing that the two $\mbb F_p$-\'etale cohomology theories agree in the case of the quotient stack $[X/G]$ of a smooth proper scheme $X$ by an action of a reductive group $G$. The main difficulty here is that the statement only holds globally and it doesn't reduce to any local one in the case of schemes. Nevertheless, using some structural theory of reductive groups, we are able to reduce to some cases which are accessible "by hand". To make these reductions we had to develop some rudiments of \'etale sheaf theory both on algebraic and rigid analytic stacks, which then allowed us to extend some tools available for schemes (base change, local systems, nice behavior of pushforwards, et.c.) to the stacky setting. This can be considered as the technical core of the paper, and might be of independent interest.
For the completeness of the picture we also discuss the case of $\mathbb F_\ell$-coefficients for $\ell\neq p$ in some detail.


As one of the main practical applications we then get a generalized form of Totaro's conjecture. As another consequence of the above results one gets a natural $G$-equivariant version of the (integral) $p$-adic Hodge theory for smooth proper schemes. We also propose more general classes of stacks for which $p$-adic Hodge theory should behave the same way as it does for smooth proper schemes. More precisely, we expect it to hold rationally for \textit{all} Hodge-proper stacks, and integrally for the so called "formally proper" ones. Besides that, we elaborate on Totaro's computations of de Rham cohomology of $BG$ and  build a variant for the theory of $\Ainf$-characteristic classes. As one of the applications of the latter we show that the mod $p$ de Rham characteristic classes give obstructions to the topological triviality of a vector bundle under some $\mu$-torsion free assumptions on the $\Ainf$-cohomology.

Even though our applications mostly use the cases of Breuil-Kisin prisms or $\Ainf$, when discussing prismatic cohomology and the corresponding comparisons, we make the setting as general as possible, keeping track of which properties of prisms/stacks are actually needed for each statement. This makes the exposition a bit lengthy, but more complete.

We also note that another proof of Totaro's conjectural inequality was recently found by Bhatt and Li in \cite{BhattLi}.

\subsection{$p$-adic Hodge theory for schemes}
$p$-adic Hodge theory is a very rich subject to which many brilliant people contributed their work over the years. We will not attempt to mention every single contribution and instead will focus on the work that we mainly use in our paper, namely \cite{BS_prisms}, mentioning the related references when it seems appropriate. The reader familiar with $p$-adic Hodge theory and \cite{BS_prisms} can freely skip to \Cref{introsec:Totaro's conjecture}.

Let $K$ be a finite extension of $\mbb Q_p$ and let $\mc O_K\subset K$ be the ring of integers. Also let $k$ be the residue field and $\ol K\simeq \ol{\mbb Q}_p$ be the algebraic closure of $K$. Let $X$ be a scheme over $\mathcal O_K$, which we assume to be smooth and proper. $p$-adic Hodge theory aims to establish a precise relation between various $p$-adic cohomology theories that one can associate to $X$. Most notably:
\begin{itemize}[wide]
	\item The \'etale cohomology of generic fiber $H^i_{\et}(X_{\ol K}, \mathbb Z_p)$; this is a $\mbb Z_p$-module with a continuous action of the absolute Galois group $G_K$.
	
	\item The de Rham cohomology $H^i_{\dR}(X/\mathcal O_K)$, an $\mathcal O_K$-module equipped with the Hodge filtration.
	
	\item The crystalline cohomology of the special fiber $H^i_{\crys}(X_k/W(k))$, a $W(k)$-module equipped with a Frobenius-linear endomorphism $\phi$.
\end{itemize}
In \cite{fontaine1982} Fontaine made a conjecture on how these theories are related rationally, i.e. after inverting $p$. First of all, by \cite{berthelot1983f} one knows that $H^i_{\crys}(X_k/W(k))[\tfrac{1}{p}]\otimes_{W(k)[\tfrac{1}{p}]} K\simeq H^i_\dR(X/K)$. Then, in \cite{fontaine1982} Fontaine has defined the period ring $B_\crys$ with an action of $G_K$ and a Frobenius $\phi$, and proposed the existence of a natural $(\phi,G_K)$-equivariant comparison isomorphism:
$$
H^i_\et(X_{\ol K}, \mathbb Q_p)\otimes_{\mbb Q_p} B_\crys \simeq H^i_{\crys}(X_k/W(k))[\tfrac{1}{p}]\otimes_{W(k)[\tfrac{1}{p}]} B_\crys.
$$
Given such an isomorphism one reconstructs $H^i_{\crys}(X_k/W(k))[\tfrac{1}{p}]$ as the Galois invariants $(H^i_\et(X_{\ol K}, \mathbb Q_p)\otimes_{\mbb Q_p} B_\crys)^{G_K}$. In the same work Fontaine also showed that assuming the comparison isomorphism above one can reconstruct $H^i_\et(X_{\ol K}, \mathbb Q_p)$ (as a $G_K$-module) from $H^i_{\crys}(X_k/W(k))[\tfrac{1}{p}]$ and the Hodge filtration on $H^i_\dR(X/K)$.
  Fontaine's conjecture was eventually proved by Tsuji \cite{Tsuji_Cst} and generalized to the formal setting more recently in \cite{ColmezNiziol} and \cite{BMS1}. As a corollary \footnote{We note that chronologically it was proven much earlier by Faltings in \cite{Faltings_p-adic_Hodge}.} one also gets the Hodge-Tate decomposition 
  $$
  H^n_{\et}(X_{\mbb C_p},\mbb Q_p)\otimes_{\mbb Q_p}\mbb C_p \simeq \bigoplus_{i+j=n} H^j(X_{K},\Omega^i_{X/K})\otimes_{K}\mbb C_p(-i).
  $$
Here $\mbb C_p$ is the $p$-adic completion of $\ol{\mbb Q}_p$ and $(-i)$ denotes the $-i$-th Tate twist.

Without inverting $p$ the situation becomes more subtle. In \cite{kisin2006crystalline}, building on the previous work by Breuil \cite{Breuil2000}, Kisin constructed a fully faithful functor $\BK$ (called $\mf M$ in \cite{kisin2006crystalline}) from the category of lattices in crystalline $G_K$-representations to the category of so-called Breuil-Kisin modules. Let's briefly recall the definition of the latter. Let $\mf S=W(k)[[u]]$ and let's pick a uniformizer $\pi\in \mc O_K$. There is a natural $W(k)$-linear map $\theta\colon \mf S \ra \mc O_K$ sending $u$ to $\pi$ and its kernel is generated by an element $E(u)$ which is an Eisenstein polynomial for $\pi$. One considers the Frobenius operator $\phi_{\mf S}$ on $\mf S$ lifting the natural Frobenius on $W(k)$ and sending $u$ to $u^p$. A Breuil-Kisin module is the data of a finitely generated $\mf S$-module $M$ endowed with an isomorphism
$$
\phi^*_{\mf S}M[\tfrac{1}{E(u)}] \xymatrix{\ar[r]^\sim &} M[\tfrac{1}{E(u)}].
$$
Kisin's work suggested that if one plugs $L=H^i_{\et}(X_{\ol K}, \mathbb Z_p)$ in the Kisin's functor $\BK$, the resulting $\mf S$-module $\BK(L)$ should specialize to the de Rham and crystalline cohomology when restricted to certain specific subschemes of $\Spec \mf S$. However, it wasn't clear how to establish such identifications without providing a geometric construction for $\BK(L)$ starting from a scheme $X$.

Such a construction\footnote{It is worth to mention that a different construction for small $n$ was first given in \cite{cais2019breuil} using the crystalline cohomology. It is not, however, as well suited for extension to stacks as the construction of \cite{BS_prisms}.} was given in terms of topological cyclic homology in \cite{BMS2} by Bhatt, Morrow and Scholze, and then was reinterpreted in terms of prismatic site in \cite{BS_prisms} by Bhatt and Scholze. Below let $\widehat{\Sch}_{\mathcal O_K}^\sm$ denote the category of smooth $p$-adic formal schemes over $\mc O_K$ and let $\DMod{\mathfrak S}$ be the derived category of $\mf S$-modules (considered as an $\infty$-category). We also denote by $\mbb C_p^\flat$ the tilt of $\mbb C_p$ (see \Cref{sec: Ainf} for details). There is a natural map $\mf S\ra W(\mbb C_p^\flat)$ sending $u$ to $[\pi^\flat]$ (this depends on a choice of $\pi^\flat=(\pi,\pi^{1/p},\ldots)\in \mc O_{\mbb C_p^\flat}$). In \cite{BS_prisms} Bhatt and Scholze show the following:

\begin{thm}[{\cite[Theorem 1.8]{BS_prisms}}]\label{intro: prismatic for schemes}
In the above notations:
	\begin{enumerate}
		\item 
For any choice of a uniformizer $\pi\in \mc O_K$ there exists a functor
	$$R\Gamma_\Prism(-/\mathfrak S) \colon \widehat{\Sch}_{\mathcal O_K}^{\sm,\op} \xymatrix{\ar[r] &} \DMod{\mathfrak S},$$
	with the following properties:
	\begin{itemize}
		\item $R\Gamma_\Prism(-/\mathfrak S)$ is derived $(p,E(u))$-complete.
		\item There is a natural Frobenius morphism $\phi_\Prism\colon \phi_{\mf S}^*R\Gamma_\Prism(-/\mathfrak S) \ra R\Gamma_\Prism(-/\mf S)$ that induces an equivalence after inverting $E(u)$:
		$$
		\phi_\Prism\colon \phi_{\mf S}^*R\Gamma_\Prism(-/\mathfrak S)[\tfrac{1}{E}] \xymatrix{\ar[r]^\sim &} R\Gamma_\Prism(-/\mf S)[\tfrac{1}{E}].
		$$
		\item(De Rham comparison) There is a canonical equivalence
		$$\phi_{\mathfrak S}^*R\Gamma_\Prism(\mathfrak X/\mathfrak S)\otimes_{\mathfrak S} \mathcal O_K \simeq R\Gamma_\dR(\mathfrak X / \mathcal O_K).$$
		
		\item(Crystalline comparison) There is a canonical $\phi$-equivariant equivalence
		$$\phi_{\mathfrak S}^*R\Gamma_\Prism(\mathfrak X/\mathfrak S)\otimes_{\mathfrak S} W(k) \simeq R\Gamma_{\crys}(\mathfrak X_k / W(k)).$$
		
		\item(\'Etale comparison) There are canonical equivalences
		$$R\Gamma_{\et}({\mathfrak X}_{\mbb C_p}, \mathbb Z/p^n) \simeq \left(R\Gamma_\Prism(\mathfrak X/ \mf S)\otimes_{\mf S} W_n(\mbb C_p^\flat)\right)^{\phi_\Prism = 1} \  \text{and}\quad R\Gamma_{\et}({\mathfrak X}_{\mbb C_p}, \mathbb Z_p) \simeq \left(R\Gamma_\Prism(\mathfrak X/ \mathfrak S)\widehat{\otimes}_{\mathfrak S} W(\mbb C_p^\flat) \right)^{\phi_\Prism = 1},$$
		where ${\mathfrak X}_{\mbb C_p}$ denotes the Raynaud's (geometric) generic fiber of ${\mathfrak X}$. 
	\end{itemize}
		\item If the formal scheme $\mf X\in  \widehat{\Sch}_{\mathcal O_K}^\sm$ is proper then 
		
		\begin{itemize}
			\item The complex $R\Gamma_\Prism(\mf X/\mathfrak S)\in \DMod{\mf S}$ is perfect.
			\item The \'etale comparison induces an equivalence
			$$
			R\Gamma_{\et}({\mathfrak X}_{\mbb C_p}, \mathbb Z_p)\otimes_{\mbb Z_p}W(\mbb C_p^\flat)\ism R\Gamma_\Prism(\mathfrak X/ \mathfrak S){\otimes}_{\mathfrak S} W(\mbb C_p^\flat). 
			$$
			\item If $H^i_\crys(\mf X_k/W(k))$ is $p$-torsion free, then so is $H^i_\et(\mf X_{\mbb C_p},\mbb Z_p)$ and 
			$$
			\BK(H^i_\et(\mf X_{\mbb C_p},\mbb Z_p))\simeq H^i_\Prism(\mf X/\mf S).
			$$	
		\end{itemize}
		
	\end{enumerate}
\end{thm}

In particular, one can put $\mf X=\widehat X$ to be the $p$-adic completion of a smooth proper $\mc O_K$-scheme $X$. Quite importantly, in this case one has $R\Gamma_{\et}({\mathfrak X}_{\mbb C_p}, \mathbb Z_p)\simeq R\Gamma_{\et}({X}_{\mbb C_p}, \mathbb Z_p)$ and (under some further torsion-free assumption) $H^i_\Prism(\mf X/\mf S)$ gives a description for the Breuil-Kisin module associated to $H^i_{\et}({X}_{\mbb C_p}, \mathbb Z_p)\simeq H^i_{\et}({X}_{\ol K}, \mathbb Z_p)$. The complex $\RG_\Prism(X/\mf S)\coloneqq \RG_\Prism(\widehat X/\mf S)$ gives a "universal" $p$-adic cohomology theory which interpolates between the de Rham, crystalline, and \'etale cohomology of $X$. We point out that the relation between these cohomology theories is now more intricate than in the rational case: they are merely fibers of a certain perfect complex of sheaves on $\Spec \mf S$ over some specific subschemes (or schemes mapping to it). 
In particular, in many examples it is not true that $\dim_{k} H^i_\dR( X_k/k)$ is equal to $\dim_{\mbb F_p} H^i_\et(X_{\mbb C_p},\mbb F_p) $; the difference is explicitly controlled by the stalks of $C\coloneqq \RG_{\Prism}(X/\mf S)\otimes_{\mbb Z}\mbb F_p$. 

Let us be a bit more precise here. The complex $C$ is a perfect complex of $k[[u]]$-modules and from \Cref{intro: prismatic for schemes}  it follows that $\RG_\dR(X_k/k)\simeq C\otimes_{k[[u]]} k$ and $\dim_{k} H^i_\dR( X_k/k)= \dim_k  H^i(C\otimes_{k[[u]]} k)$. On the other hand $\RG_\et(X_{\mbb C_p},\mbb F_p)\otimes_{\mbb F_p}\mbb C_p^\flat\simeq C\otimes_{k[[u]]}\mbb C_p^\flat$; the map $k[[u]]\ra \mbb C_p^\flat$ factors through $k((u))$ and $\dim_{\mbb F_p} H^i_\et(X_{\mbb C_p},\mbb F_p)=  \dim_{k((u))}  H^i(C\otimes_{k[[u]]}k((u))$. Thus the dimensions of de Rham and \'etale cohomology are described as the dimensions of (the cohomology of) stalks of $C$ at the special and the generic point of $\Spec k[[u]]$ accordingly (and in the examples these can be different). Nevertheless, by semicontinuity for dimensions of stalks of a perfect complex at least one always has an inequality
$$
 \dim_{\mbb F_p} H^i_\et(X_{\mbb C_p},\mbb F_p)\le \dim_{k} H^i_\dR( X_k/k).
$$
It can also be fruitful to interpret $H^i_\et(X_{\mbb C_p},\mbb F_p)$ in terms of singular cohomology. Namely after fixing an identification $\iota\colon \mbb C \xra{\sim} \mbb C_p$ we have Artin's isomorphism $H^i_\et(X_{\mbb C_p},\mbb F_p)\simeq H^i_\sing(X(\mbb C),\mbb F_p)$ and thus also get an inequality 
\begin{equation}\label{inequality of dimensions for schemes}
\dim_{\mbb F_p} H^i_\sing(X(\mbb C),\mbb F_p)\le \dim_{k} H^i_\dR( X_k/k),
\end{equation}
which gives a quite surprising relation between the mod $p$ topological data of $X(\mbb C)$ and the the algebraic de Rham data of the mod $p$ reduction $X_k$.

\subsection{Totaro's conjecture}\label{introsec:Totaro's conjecture}
Let us now discuss Totaro's conjecture. Let $G$ be a split reductive group scheme over $\mbb Z$. Let $G(\mbb C)$ be the topological group of $\mbb C$-points of $G$. By a theorem of Grothendieck, one can make a canonical identification 
$$
H^*_\sing(G(\mbb C),\mbb C)\simeq H^*_\dR(G/\mbb C).$$
Similarly, using smooth descent one can then identify the singular cohomology of the classifying space $BG(\mathbb C)$ with the de Rham cohomology of the classifying stack $BG$:
$$H^*_\sing(BG(\mbb C),\mbb C)\simeq H^*_\dR(BG/\mbb C).$$  Replacing  the coefficients $\mbb C$ with $\mbb F_p$ one could hope that $H^*_\sing(G(\mbb C),\mbb F_p)$ or $H^*_\sing(BG(\mbb C), \mathbb F_p)$ have analogous algebro-geometric descriptions. Note, however, that it is definitely not true that $H^*_\sing(G(\mbb C),\mbb F_p)$ is isomorphic to $H^*_\dR(G/\mbb F_p)$: indeed $G(\mbb C)$ is a finite CW-complex, and thus $H^*_\sing(G(\mbb C),\mbb F_p)$ is a finite-dimensional $\mbb F_p$-vector space, while $H^*_\dR(G/\mbb F_p)\simeq \bigoplus_i \Omega^i_{G/\mbb F_p}$ by the Cartier isomorphism, and thus is infinite-dimensional if $\dim G>0$. Nevertheless, Totaro proved in \cite[Theorem 0.2]{Totaro_deRhamBG} that if $p$ is a non-torsion prime\footnote{For reminder of the notion of a torsion prime see \Cref{sec: de Rham and prismatic cohomology of BG} and \Cref{defn: torsion primes} in particular.} for $G$, there is an isomorphism
$$H^*_\sing(BG(\mathbb C), \mathbb F_p) \simeq H^*_\dR(BG/\mathbb F_p).$$
Let us point out that the isomorphism Totaro constructed is not canonical, in fact he just showed that under the above assumption on $p$ both algebras are polynomial with generators in same (even) degrees. The naive supposition that both sides could be isomorphic for torsion primes as well turns out to be wrong, namely, as Totaro showed by an explicit computation in the same paper, one has a strict inequality:
$$
 \dim_{\mbb F_2} H^{32}_\sing(B\mr{Spin}_{11}(\mbb C),\mbb F_2)<\dim_{\mbb F_2} H^{32}_\dR(B\mr{Spin}_{11}/\mbb F_2).
$$
Though not explicitly stated in \cite{Totaro_deRhamBG}, the following became known as Totaro's conjecture:

\begin{conj}[Totaro]\label{intro_Totaros_conjecture}
	Let $G$ be a split reductive group scheme over $\mbb Z$. Then for all primes $p$ and all $i\in\mathbb Z_{\ge 0}$ the inequality
	$$\dim_{\mathbb F_p} H_\sing^i(BG(\mathbb C), \mathbb F_p) \le \dim_{\mbb F_p} H^i_\dR(BG/\mbb F_p)$$
	holds true.
\end{conj}
The conjecture was of course motivated by Inequality \Cref{inequality of dimensions for schemes} which became known for schemes right at the time. Totaro suggested that one should be able to deduce it from an appropriate generalization of $p$-adic Hodge theory in the setting of stacks. 
However, as he points out in \cite{Totaro_deRhamBG}, the main obstruction to directly apply the prismatic cohomology to \Cref{intro_Totaros_conjecture} is that $BG$ is not a proper stack. 

 Before introducing the concept of Hodge-proper stacks, which at least partially redeems the situation, let us comment on why \Cref{intro_Totaros_conjecture} is interesting. First, when $p$ is a torsion prime it can be very difficult to compute both sides explicitly. For example (at least to our knowledge) there is no description of the $\mbb F_p$-cohomology of $\mr{PGL}_n(\mbb C)$ in the case when $p$ divides $n$, $n>p$ even as a graded $\mathbb F_p$-vector space. And even though many explicit computations on the topological side were successfully made over the years, not much is known about the de Rham side, where for now the only few computations that were made are for $\mr{SO}_n$ and $\mr{O}_n$ in \cite{Totaro_deRhamBG}, and for $G_2$ and $\mr{Spin}_{n}$, $n\le 11$ in \cite{Primozic}, all for $p=2$. 
Second, as shown by Totaro in \cite{Totaro_deRhamBG} the Hodge cohomology of $BG$ have a natural representation-theoretic interpretation as the cohomology of $G$ in modules like $\Sym^i \mf g^*$. For small primes these things are often very hard to compute, while \Cref{intro_Totaros_conjecture} would at least give some lower bound in (much better understood) topological terms. Finally, the inequality itself might not be as interesting as understanding the precise relation between the two parts: further we show that it is exactly controlled by the prismatic cohomology (as in the smooth proper scheme case). This then also gives a very efficient way to actually compute the de Rham cohomology of $BG$ for small primes in some cases (see \Cref{rem: bhatt-li-applications} for another comment on that).

\subsection{$p$-adic Hodge theory in the context of Hodge-proper stacks}\label{secintro:p-adic Hodge theory for Hodge-proper stacks}
Even though the classifying stack $BG$ is usually not proper, if $G$ is reductive, $BG$ still satisfies a weaker condition which we call Hodge-properness. We recall its definition:
\begin{defn}[{\cite{KubrakPrikhodko_HdR}}]\label{intro_defn_of_hodge_properness}
	A smooth quasi-compact quasi-separated Artin stack $\mstack X$ over a Noetherian ring $R$ is called \emdef{Hodge-proper} if for any $i,j\in \mathbb Z_{\ge 0}$ the cohomology $H^j(\mstack X, \wedge^i \mathbb L_{\mstack X/R})\in \UMod{R}$ is a finitely generated $R$-module.
\end{defn}
Here $\mbb L_{\mstack X/R}\in \QCoh(\mstack X)$ is the cotangent complex of $\mstack X$ and $\wedge^i \mathbb L_{\mstack X/R}\in \QCoh(\mstack X)$ is its $i$-th wedge power, and $\mstack X$ is allowed to be a higher geometric stack in the sense of \cite{TV_HAGII}. Under the smoothness assumption, the condition on cohomology of $\wedge^i \mathbb L_{\mstack X/R}$'s in \Cref{intro_defn_of_hodge_properness} is in fact equivalent to the total Hodge cohomology
$$H^n_\Hdg(\mstack X/R)\coloneqq \bigoplus_{i=0}^\infty H^n(\mstack X,\wedge^i \mathbb L_{\mstack X/R}[-i])\simeq \bigoplus_{i=0}^\infty H^{n-i}(\mstack X,\wedge^i \mathbb L_{\mstack X/R})$$
being finitely generated over $R$ for any $n\in \mbb Z_{\ge 0}$. Note that in the case $\mstack X=X$ is a smooth scheme, we have $\mbb L_{X/R}\simeq \Omega^1_{X/R}$ and the formula for $H^n_\Hdg(\mstack X/R)$ gives the classical Hodge cohomology. Thus, a Hodge-proper stack $\mstack X$ looks "proper" from the point of view of its Hodge cohomology. See \Cref{section:Hodge_proper_stacks} for more details.

In \cite{KubrakPrikhodko_HdR} we used the notion of Hodge-properness to prove a version of the Hodge-to-de Rham degeneration in the context of Artin stacks in characteristic 0 via reduction to a large enough characteristic $p$ (analogous to the work of Deligne-Illusie \cite{DeligneIllusie}). While that can be seen as an extension of a result related to the classical Hodge theory over $\mbb C$, in this paper we attempt to extend the results of $p$-adic Hodge theory to smooth Hodge-proper stacks over $\mc O_K$. 
 
In fact, up to some point the theory works rather well for any smooth Artin stack $\mstack X$ without any extra restrictions. Namely, let $\mstack X \in \Stk_{\mc O_K}^\sm$ be a smooth $n$-Artin stack over $\mc O_K$. Given a functor $F\colon \Aff_{\mathcal O_K}^{\sm, \op} \ra \mathscr C$ to a complete $\infty$-category $\mathscr C$ one can right Kan extend it along the embedding $\Aff_{\mathcal O_K}^\sm\subset \Stk_{\mc O_K}^\sm$ obtaining a functor $F\colon  \Stk_{\mc O_K}^\sm \ra \mathscr C$. More explicitly, given $\mstack X \in \Stk_{\mc O_K}^\sm$ one has 
$$
F(\mstack X) \simeq \holim_{(T\ra \mstack X)\in (\Aff^{\sm}_{/\mstack X})^{\op}} F(T).$$  This way one can define the following cohomology theories associated to $\mstack X$ (for details see \Cref{section:p_adic_cohomology_for_stacks}):
\begin{itemize}
 	\item The de Rham cohomology $\RG_{\dR}(\mstack X/\mc O_K)\in \DMod{\mc O_K}$. In the $1$-Artin setting this agrees with the definition of de Rham cohomology considered by Totaro in \cite{Totaro_deRhamBG}.

	\item The crystalline cohomology of the special fiber $\RG_{\crys}(\mstack X_k/W(k))\in \DMod{W(k)}$. In the $1$-Artin setting this agrees with the crystalline cohomology of stacks previously considered by Olsson in \cite{Olson_CrysCoh}.
	
	\item The \'etale cohomology of the Raynaud (geometric) generic fiber $\RG_\et(\widehat{\mstack X}_{\mathbb C_p},\mathbb Z_p)$. Here one takes the right Kan extension of $T\mapsto \RG_\et(\widehat{T}_{\mathbb C_p},\mathbb Z_p)$.

	\item The prismatic cohomology $\RG_\Prism(\mstack X/\mf S)\in \DMod{\mf S}$ (corresponding to a choice of uniformizer $\pi\in \mc O_K$). Here one views $\RG_\Prism(-/\mf S)$ as a functor $\Aff^{\sm, \op}_{/\mstack X}\ra \DMod{\mf S}$ by putting $\RG_\Prism(T/\mf S)\coloneqq \RG_\Prism(\widehat T/\mf S)$, where $\widehat T\in \widehat{\Sch}^\sm_{\mc O_K}$ is the $p$-adic completion of $T$.
\end{itemize}

More or less formally one then gets the following extension of part 1 of \Cref{intro: prismatic for schemes} to Artin stacks (for details see \Cref{section:p_adic_cohomology_for_stacks}):

\begin{thm}\label{intro: comparisons for smooth stacks} Let $\mstack X$ be a smooth Artin stack over $\mc O_K$. Them
		\begin{itemize}
			\item The complex $R\Gamma_\Prism(\mstack X/\mathfrak S)$ is derived $(p,E)$-complete and there is a natural Frobenius morphism 
			$$
			\phi_\Prism\colon \phi_{\mf S}^*R\Gamma_\Prism(-/\mathfrak S) \xymatrix{\ar[r]&} R\Gamma_\Prism(-/\mf S)
			$$
		 that induces an equivalence after inverting $E(u)$. See \Cref{cor: Frobenius is a twist if tor-dimension is finite}.
			
			\item (De Rham comparison) $$\phi_{\mf S}^*\RG_\Prism(\mstack X/\mf S)\otimes_{\mf S} \mc O_K \simeq \RG_\dR(\mstack X/\mc O_K)^\wedge_p,$$ where $(-)^\wedge_p$ denotes the derived $p$-completion. See \Cref{prop: de Rham comparison} and \Cref{rem: drop qcqs assumption in the perfect case}.
			\item (Crystalline comparison) $$\phi_{\mf S}^*\RG_\Prism(\mstack X/\mf S)\otimes_{\mf S} W(k)\simeq \RG_\crys(\mstack X_k/W(k)).$$ Here see \Cref{prop: cristalline comparison} and Remarks \ref{rem: drop qcqs assumption in the perfect case}, \ref{rem: prismatic base change in the perfect case}.
			
			\item(\'Etale comparison) Assume $\mstack X$ is also quasi-compact and quasi-separated. Then 	
		$$
		R\Gamma_{\et}({\widehat{\mstack X}}_{\mbb C_p}, \mathbb Z/p^n) \simeq \left(R\Gamma_\Prism({\mstack X}/ \mf S)\otimes_{\mf S} W_n(\mbb C_p^\flat)\right)^{\phi_\Prism = 1} \  \text{and}\quad R\Gamma_{\et}({\widehat{\mstack X}}_{\mbb C_p}, \mathbb Z_p) \simeq \left(R\Gamma_\Prism({\mstack X}/ \mathfrak S)\widehat{\otimes}_{\mathfrak S} W(\mbb C_p^\flat) \right)^{\phi_\Prism = 1},
		$$
		where the tensor product on the right is $p$-completed. See \Cref{cor: adic_etale_comparison2}.
		\end{itemize}

\end{thm}

Here the appearance of the $p$-adic completion in the de Rham comparison is explained by the fact that the prismatic cohomology for classical schemes compares with the de Rham cohomology of the formal completion (and not the scheme itself).

 \begin{rem}
 	For a more general stack $\mstack X$ or to work in a relative situation it is probably better to use the prismatic stack $\mstack X_\Prism$ developed by Bhatt-Lurie and Drinfeld (see \cite{Drinfeld_prismatization}), namely $\RG_\Prism(\mstack X/\mf S)$ can be defined as $\RG(\mstack X_{\Prism/\mf S}, \mc O_{\mstack X_{\Prism/\mf S}})$, where $\mstack X_{\Prism/\mf S}\coloneqq \mstack X_{\Prism}\times_{(\Spf \mc O_K)_\Prism} \Spf \mf S$.
 \end{rem}
\begin{rem}
To establish the crystalline and de Rham comparisons over a more general base prism one needs to impose more conditions both on the stack $\mstack X$ and the prism $A$. First of all, one needs to slightly correct the formula, namely the pull-back $\phi_A^*$ needs to be considered in the category of $(p,I)$-complete $A$-modules. Then we also need to assume that the stack $\mstack X$ is quasi-compact and quasi-separated and, finally, that $\phi_{A*}A$ is an $A$-module of finite $(p,I)$-complete Tor-amplitude. See \Cref{section:p_adic_cohomology_for_stacks} for more details.
\end{rem}	 

 The framework of Hodge-proper stacks is exactly where one has a natural analogue of part 2 of \Cref{intro: prismatic for schemes}. Below let $\Coh^+(R)\subset \DMod{R}$ denote the full subcategory of bounded below modules with coherent cohomology (see \Cref{defn: bounded below coherent sheaves}).

\begin{prop}\label{intro: cor for Hodge-proper}
Let $\mstack X$ be a smooth Hodge-proper stack over $\mc O_K$. Then:\begin{itemize}
	\item The complex $\RG_\Prism(\mstack X/\mf S)$ lies in $\Coh^+(\mf S)$ (\Cref{cor: prismatic cohomology are bounded below coherent}).
	\item One has natural $\phi$-equivariant equivalences
\begin{align*}
R\Gamma_{\et}({\widehat{\mstack X}}_{\mbb C_p}, \mathbb Z/p^n)\otimes_{\mbb Z/p^n} W_n(\mbb C_p^\flat) &\simeq R\Gamma_\Prism({\mstack X}/ \mf S)\otimes_{\mf S} W_n(\mbb C_p^\flat), \\ 
 R\Gamma_{\et}({\widehat{\mstack X}}_{\mbb C_p}, \mathbb Z_p)\otimes_{\mbb Z_p} W(\mbb C_p^\flat)  &\simeq R\Gamma_\Prism({\mstack X}/ \mathfrak S){\otimes}_{\mathfrak S} W(\mbb C_p^\flat). 
\end{align*}
Consequently, $R\Gamma_{\et}({\widehat{\mstack X}}_{\mbb C_p}, \mathbb Z_p)\in \Coh^+(\mbb Z_p)$ and $H^i_\et({\widehat{\mstack X}}_{\mbb C_p}, \mathbb F_p)$ is finite-dimensional over $\mbb F_p$ for any $i$ (\Cref{etale_coh_finiteness}).
\end{itemize}
 	
\end{prop}

 We note that under the Hodge-properness assumption the de Rham complex $\RG_\dR(\mstack X/\mc O_K)$ is derived $p$-complete (\Cref{cor: de Rham coh of Hodge-proper are complete}) and so one also has $\phi_{\mf S}^*\RG_\Prism(\mstack X/\mf S)\otimes_{\mf S} \mc O_K \simeq \RG_\dR(\mstack X/\mc O_K)$. Thus $\RG_\Prism(\mstack X/\mf S)$ plays a very similar role to the one it played in the case of smooth proper schemes: it interpolates between the \'etale, de Rham, and crystalline cohomology and its cohomology are finitely generated $\mf S$-modules. However we stress that the \'etale cohomology  we get are the \'etale cohomology of Raynaud generic fiber $R\Gamma_{\et}({\widehat{\mstack X}}_{\mbb C_p}, \mathbb Z_p)$,  and not the \'etale cohomology of the algebraic generic fiber $R\Gamma_{\et}({\mstack X}_{\mbb C_p}, \mathbb Z_p)$ (which we could compare with the singular cohomology of the underlying homotopy type of $\mbb C$-points via Artin comparison afterwards). Nevertheless, for any smooth Hodge-proper stack $\mstack X$ over $\mc O_K$ we get an analogue of Totaro's inequality (\Cref{inequality}):
\begin{equation}\label{intro: ineq for rigid fiber}
	\dim_{\mathbb F_p} H_\et^i(\widehat{\mstack X}_{\mbb C_p}, \mathbb F_p) \le \dim_{k} H^i_\dR(\mstack X_k/k).
\end{equation}
This applies in particular to the classifying stack $BG$ of a reductive group $G$. To deduce Totaro's conjecture it remains somehow to relate the left hand side with $H^i_\sing(BG(\mbb C),\mbb F_p)$.

\begin{rem}
	It is crucial for \Cref{intro: cor for Hodge-proper} that the base prism $\mf S$ is Noetherian. It is also convenient that $\mf S$ is regular and so any finitely generated $\mf S$-module is perfect: for example this allows to drop the qcqs assumption in the de Rham and crystalline comparisons (\Cref{intro: comparisons for smooth stacks}). We remark as well that the results above should generalize to smooth \emph{formal} stacks over $\mc O_K$ in a  more or less straightforward way, however we do not really touch this topic in our work. It would also be interesting to find an appropriate generalization of both the notion of Hodge-properness and \Cref{intro: cor for Hodge-proper} to a more general non-Noetherian context, for example of smooth formal stacks over $\mc O_{\mbb C_p}$ and $\Ainf$-cohomology, similar to \cite{BMS1}. 
\end{rem}

We defer the discussion of the proof of Totaro's conjecture until \Cref{introsec: local acyclicity and the proof of Totaro's} and instead discuss how to extend some of the results of rational $p$-adic Hodge theory to the general context of Hodge-proper stacks. We refer the reader to \Cref{sssection:C_cris_for_stacks} for more details.

Let $\mstack X$ be a Hodge-proper stack over $\mc O_K$ and consider $\mbb Q_p$-\'etale cohomology $H^i_\et(\widehat{ \mstack X}_{\mbb C_p},\mbb Q_p)\coloneqq H^i_\et(\widehat{ \mstack X}_{\mbb C_p},\mbb Z_p)[\frac{1}{p}]$ of its Raynaud generic fiber. Consider also rational versions $H^i_\dR(\mstack X/K)$ and $H^i_\crys(\mstack X_k/W(k))[\frac{1}{p}]$ of the de Rham and crystalline cohomology correspondingly. For convenience put $K_0\coloneqq W(k)[\frac{1}{p}]$. 

By \Cref{intro: cor for Hodge-proper}, $H^i_{\Prism}(\mstack X/\mf S)$ is finitely generated, and so $\phi_\Prism$ endows it with a structure of Breuil-Kisin module. Consequently (by a structural result for Breuil-Kisin modules, see e.g. \Cref{ex:stuff on BK-modules}), its rationalization $H^i_{\Prism}(\mstack X/\mf S)[\frac{1}{p}]$ is a free $\mf S[\frac{1}{p}]$-module of finite rank. From the comparisons it then follows (\Cref{prop: rationally all cohomology has the same dimension}) that
$$
\dim_{\mbb Q_p} H^i_\et(\widehat{ \mstack X}_{\mbb C_p},\mbb Q_p) = \dim_{K} H^i_\dR(\mstack X/K) =\footnote{With all being equal to $\rk_{\mf S[\frac{1}{p}]}H^i_{\Prism}(\mstack X/\mf S)[\frac{1}{p}]$.}  \dim_{K_0} H^i_\crys(\mstack X_k/W(k))[\tfrac{1}{p}],
$$ exactly as in the case of smooth proper schemes. The similarity doesn't end here. Note that as usual, there is a natural $G_K$-action on  $H^i_\et(\widehat{ \mstack X}_{\mbb C_p},\mbb Q_p)$; moreover one can deduce  from \Cref{intro: cor for Hodge-proper} that this action is continuous. If $\mstack X$ is a smooth proper scheme, $H^i_\et(\widehat{ \mstack X}_{\mbb C_p},\mbb Q_p)$ is a crystalline $G_K$-representation: this is a part of (proven) Fontaine's  $C_{\mr{cris}}$-conjecture. An analogous result holds for Hodge-proper stacks:

\begin{thm*}[\ref{thm: Fontaine's conjecture}]
	Let $\mstack X$ be a smooth Hodge-proper stack over $\mc O_K$. Then for any $i$ there is a natural $(\phi,G_K)$-equivariant isomorphism 
$$
H^i_\et(\widehat{\mstack X}_{\mbb C_p}, \mathbb Q_p)\otimes_{\mbb Q_p} B_\crys \simeq H^i_{\crys}(\mstack X_k/W(k))[\tfrac{1}{p}]\otimes_{K_0} B_\crys.
$$
In particular, $H^i_\et(\widehat{ \mstack X}_{\mbb C_p},\mbb Q_p)$ is a crystalline $G_K$-representation.
\end{thm*}

From the classical results of Fontaine on admissible representations \cite{Fointane_admissible_reps} it then follows that $H^i_\et(\widehat{\mstack X}_{\mbb C_p}, \mathbb Q_p)$ is also de Rham and Hodge-Tate. In particular one also gets a natural $G_K$-equivariant isomorphism 

$$H^i_\et(\widehat{\mstack X}_{\mbb C_p}, \mathbb Q_p)\otimes_{\mbb Q_p} B_\dR \xymatrix{\ar[r]^\sim &} H^i_\dR(\mstack X/K)\otimes_K B_\dR,$$
where $B_\dR$ is the de Rham period ring. However, from the proof of \Cref{thm: Fontaine's conjecture} it is not a priori clear that this map is an isomorphism of \emph{filtered} spaces, where filtrations in question come from the one on $B_\dR$, trivial filtration on $H^i_\et(\widehat{\mstack X}_{\mbb C_p}, \mathbb Q_p)$, and the Hodge filtration on $H^i_\dR(\mstack X/K)$. This can be resolved using the Nygaard filtration on $\RG_{\Prism^{(1)}}(\mstack X/\mf S)$ and the degeneration of the Hodge-Tate spectral sequence (see \Cref{sssection: Hodge-Tate decomposition}). As a corollary, by passing to the associated graded one also gets the Hodge-Tate decomposition: 
\begin{thm*}[\ref{thm: hodge-tate decomposition}]
	Let $\mstack X$ be a smooth Hodge-proper stack over $\mc O_K$. Then there is a natural $G_K$-equivariant decomposition
		$$
	H^n_{\et}(\widehat{\mstack X}_{\mbb C_p},\mbb Q_p)\otimes_{\mbb Q_p}\mbb C_p \simeq \bigoplus_{i+j=n} H^j(\mstack X_{K},\wedge^i\mbb L_{\mstack X_K/K})\otimes_{K}\mbb C_p(-i),
	$$
	where $\mbb C_p(-i)$ denotes $(-i)$-th Tate twist.
\end{thm*}

One can also show that if $\mstack X$ is Hodge-proper over $\mc O_K$ the Hodge-to-de Rham spectral sequence for $\mstack X_K$ degenerates (\Cref{thm: HdR degeneration for stacks}).  This gives a partial strengthening of results of \cite{KubrakPrikhodko_HdR} on the degeneration of Hodge-to-de Rham spectral sequence: namely, to show that it degenerates for a given smooth Artin stack over $K$ it is enough to find a Hodge-proper model over $\mc O_K$. Comparing with the results of \cite{KubrakPrikhodko_HdR}, it shows that it's enough to "Hodge-properly spread out" a Hodge-proper stack to a single prime, rather than all but finite set of them.

Finally, it is also possible to describe the Breuil-Kisin module associated to $H^i_\et(\widehat{\mstack X}_{\mbb C_p}, \mbb Z_p)$ in terms of prismatic cohomology, namely, provided $H^i_\crys(\mstack X_k/W(k))$ is $p$-torsion free one has 
$$
\BK(H^i_\et(\widehat{\mstack X}_{\mbb C_p}, \mbb Z_p))\simeq H^i_{\Prism}(\mstack X/\mf S).
$$
There is also a more general version of this isomorphism without any $p$-torsion free assumptions, see \Cref{rem: associated BK-module/refined}.
\begin{rem}
	Note that the statements of Theorems \ref{thm: Fontaine's conjecture} and \ref{thm: hodge-tate decomposition} depend only on the Raynaud generic fiber $\mstack X_K$ (and not $\mstack X$). However, we stress that one needs Hodge-properness of $\mstack X$ over $\mc O_K$ for them to hold: the condition that  $\mstack X_K$ is  Hodge-proper over $K$ is not sufficient. This is not so surprising for Fontaine's conjecture, where in the schematic case one is supposed to have a smooth proper model over $\mc O_K$. Quite differently, the Hodge-Tate decomposition usually doesn't require the existence of any nice model: so here the analogy is not complete.
\end{rem}

\subsection{Local acyclicity and the proof of Totaro's conjecture}\label{introsec: local acyclicity and the proof of Totaro's}
The key to the proof of Totaro's conjecture is in a careful study of the map $\Upsilon_{\mstack X}\colon \RG_\et({\mstack X}_{\mbb C_p}, \mbb F_p) \ra  \RG_\et(\widehat{\mstack X}_{\mbb C_p}, \mbb F_p)$ (construction of which we will remind below) in the case $\mstack X=BG$. For more details see Sections \ref{sect:etale_sheaves}, \ref{sec_etale_cohomology_of_stacks}, and \ref{sec: diff_etale}. We stress that while (most of) the results of \Cref{secintro:p-adic Hodge theory for Hodge-proper stacks} are more or less formal consequences of the analogous results in \cite{BS_prisms}, the comparison of the two versions of the \'etale cohomology via $\Upsilon_{BG}$ is not a formality and can be seen as the essential core of the proof. 

Let $X$ be a scheme over $\mc O_{\mbb C_p}$ of finite type. There are two natural ways to associate a rigid analytic space to $X$: one can either consider the formal completion $\widehat X$ and then the Raynaud generic fiber $\widehat X_{\mbb C_p}$ or, first take the algebraic generic fiber $X_{\mbb C_p}$ and consider its analytification $X_{\mbb C_p}^\an$. There is always a map $\widehat X_{\mbb C_p} \ra X_{\mbb C_p}^\an$; in fact there are maps 
$$
\widehat X_{\mbb C_p} \xymatrix{\ar[r]^{\psi_X}&} X_{\mbb C_p}^\an \xymatrix{\ar[r]^{\phi_X}&} X_{\mbb C_p}
$$
of locally ringed spaces, which induce maps of the corresponding \'etale sites, thus providing pull-back maps:
$$
\RG_\et(X_{\mbb C_p},\mbb F_p) \xymatrix{\ar[r]^{\phi_X^{-1}}&} \RG_\et( X_{\mbb C_p}^\an,\mbb F_p) \xymatrix{\ar[r]^{\psi_X^{-1}}&} \RG_\et(\widehat X_{\mbb C_p},\mbb F_p).
$$ 
It is known (\cite[Theorem 3.8.1]{Huber_adicSpaces}) that $\phi_X^{-1}$ is always an equivalence, however $\psi_X^{-1}$ is usually far from being one (see \Cref{ex:F-p etale cohomology of rigid affine line} for $X=\mbb A^1$). One assumption which makes $\psi_X^{-1}$ an equivalence is properness of $X$: then in fact $X_{\mbb C_p}^\an \simeq \widehat X_{\mbb C_p}$. 


Right Kan extending to Artin stacks one obtains maps 
$$
\RG_\et(\mstack X_{\mbb C_p},\mbb F_p) \xymatrix{\ar[r]^{\phi_X^{-1}}&} \RG_\et( \mstack X_{\mbb C_p}^\an,\mbb F_p) \xymatrix{\ar[r]^{\psi_X^{-1}}&} \RG_\et(\widehat{\mstack X}_{\mbb C_p},\mbb F_p),
$$ 
and we define $\Upsilon_{\mstack X}\coloneqq \psi_X^{-1} \circ \phi_X^{-1}$. It is still true that $\phi_X^{-1}$ is an equivalence for all $\mstack X$ of locally finite type (see \Cref{commutation with psi}) and that $\psi_X^{-1}$ is an equivalence for $\mstack X$ proper (\Cref{adic_proper_base_change}). Consequently, $\Upsilon_{\mstack X}$ is an equivalence if $\mstack X$ is proper. However, since $BG$ is not proper (at least in most of the cases), one doesn't see any formal reason for $\Upsilon_{BG}$ to be an equivalence. 

\begin{rem}\label{rem: bhatt-li-applications} One way to proceed is to "approximate" $BG$ by smooth proper schemes, similarly to what Ekedahl does in \cite{Ekedahl}. This strategy was successfully realized by Bhatt and Li in their short note \cite{BhattLi}; using this method they are able to show that the map $H^i_\et(BG_{\mbb C_p},\mbb F_p)\ra H^i_\et(\widehat{BG}_{\mbb C_p},\mbb F_p)$ is always injective. Their result then already suffices to prove the inequality in \Cref{intro_Totaros_conjecture}. In this work we prove a stronger statement that $\Upsilon_{BG}$ is an equivalence (see \Cref{sec: diff_etale}). This then implies not only the inequality, but also that the exact difference between the de Rham and singular cohomology is controlled by the prismatic cohomology of $BG$.  
	
	The latter then allows to fully compute the de Rham cohomology in many cases by using the existing computations on the topological side. As an example, in his UROP project \cite{Chua} Anlong Chua was able to compute the mod 2 de Rham cohomology of $B\mr{Spin}_n$ up to $n\le 13$, using the classical Quillen's topological computation in \cite{quillen1971mod} as an input. This computation assumed the existence of the Hochschild-Serre-type spectral sequence in the prismatic context, which Scavia and the first author are going to set up in \cite{ScaviaKubrak}. Using it, they plan to make explicit computations of both the de Rham and Hodge cohomology in other cases, that seem not yet accessible in any other way.
\end{rem}

In this paper we use a different approach, exploiting the idea of "approximation by proper schemes" only in the simplest non-proper case $\mstack X= B\mbb G_m$.  Starting from there we are able to show the following result below. Let's say that an Artin stack $\mstack X$ over $\mc O_{\mbb C_p}$ is \emdef{$\mbb F_p$-locally acyclic} if $\Upsilon_X$ is an equivalence. 
\begin{thm*}[\ref{thm:main_result}]
	Let $X$ be a smooth proper scheme over $\mathcal O_K$ equipped with an action of a reductive group scheme $G$. Then the quotient stack $[X/G]_{\mc O_{\mbb C_p}}$ is $\mbb F_p$-locally acyclic. 
\end{thm*}

\begin{rem}
The term "locally acyclic" comes form the interpretation of $\RG_\et(\widehat{\mstack X}_{\mbb C_p}, \mbb F_\ell)$ in terms of the cohomology of the nearby cycles sheaf (see \Cref{sec: Algebraic nearby cycles and local } and \Cref{prop: F-l-acyclicity in terms of the neraby cycles}) in the case when $\ell\neq p$. One can give a similar interpretation in the $\ell=p$ case as well using the recent concept of $\mr{arc_p}$-descent of \cite{BhattMathew_Arc} (see \Cref{rem: interpretation of Upsilon using arc-descent}). To clarify, by "locally" here we mean locally at $p$ (since our base is usually a local $p$-complete ring) but not locally on the stack $\mstack X$.
\end{rem}

Let us outline the proof of \Cref{thm:main_result}. Let $T\subset B \subset G$ be a maximal torus and a Borel subgroup of $G$. The proof consists of the following steps:

\textit{Step 1. $B\mbb G_m$.} In topology we have a nice presentation of $B\mbb C^\times \simeq BS^1$
as the infinite dimensional projective space $\mbb C\mbb P^\infty=\colim_n \mbb C\mbb P^n$. In the algebraic setting it is not true that $B\mbb G_m$ is the colimit of $\mbb P^n$'s; nevertheless, one can show that the map $\colim_n \mbb P^n \ra B\mbb G_m$, classifying the line bundle $\mc O(1)$, induces an equivalence on all cohomology theories of our interest (see \Cref{prismatoc_coh_of_Btori}). From this one deduces $\mbb F_p$-local acyclicity of $B\mbb G_m$ from one of $\mbb P^n$. Similarly one covers the case of $BT$.

\textit{Step 2. $[\mbb A^n/T]$ for a conical linear action (see \Cref{defn: conical action}).} Here by an explicit computation of Hodge cohomology we show that the projection $[\mathbb A^n/T]\ra BT$ induces an equivalence on all cohomology of our interest (see \Cref{subsect:prismatic_coh_of_An_quot_T}). Then $\mbb F_p$-local acyclicity of $[\mathbb A^n/T]$ reduces to the one of $BT$.

\textit{Step 3. $BB$.} Here one reduces to the previous case by considering the \v Cech simplicial object associated to the smooth cover $BT\ra BB$. Namely, the individual terms are exactly of the form $[\mbb A^n/T]$ (with conical linear action) and one gets that $BB$ is $\mbb F_p$-locally acyclic by smooth descent (see \Cref{subsect:case_of_BB}). 

\textit{Step 4. $[X/B]$ with $X$ smooth and proper.} This is the \textit{most} difficult step. In fact, the case of $[X/B]$ is where most of the technical tools developed in this paper (mainly, in \Cref{sect:etale_sheaves}) are being used. Most importantly, derived \'etale sheaves on stacks (\Cref{etale_sheaves_on_stacks}), derived local systems (\Cref{subsect:local systems}) and rigid analytic stacks with \'etale sheaves on them (\Cref{adic_local_systems}). The idea of the proof of $\mbb F_p$-acyclicity is to study (derived) pushforward of the constant sheaf $\mbb F_p$ under the morphism $[X/B]\ra BB$. The machinery from \Cref{sect:etale_sheaves} allows to compare it with the "analytification" of the analogous pushforward on the algebraic side (more precisely, for a proper morphism $f\colon \mstack X\ra \mstack Y$ we prove the "pushforward-commutes-with-analytification" statement, see \Cref{adic_base_change}). We also show that if the map is smooth and schematic, the derived pushforward of the constant sheaf on the algebraic side is a derived local system (\Cref{functoriality_of_loc_sys}), generalizing a classical theorem of Deligne. Then to finish the proof one notices that since $B_{\mbb C_p}$ is connected, $BB_{\mbb C_p}$ is \'etale simply connected and as such does not have any non-trivial abelian local systems. This allows to reduce to the $\mbb F_p$-local acyclicity of $BB$.

\textit{Step 5. $BG$ and $[X/G]$ with $X$ smooth and proper.} The statement for $BG$ is then obtained by considering the \v Cech simplicial object for the smooth cover $BB\ra BG$. Namely, its terms are given by quotients of products of flag varieties (that are smooth and proper) by the diagonal action of $B$ and we reduce to Step 4 by smooth descent. The case $[X/G]$ is covered similarly, by considering $[X/B]\ra [X/G]$ (see \Cref{sec: comparison for [X/P]}).

As a corollary of \Cref{thm:main_result} we establish the following generalized form of Totaro's conjecture:
\begin{thm*}[\ref{thm:main_application}] (Generalized form of Totaro's conjecture).
Let $X$ be a smooth proper scheme over $\mathcal O_K$ equipped with an action of a connected reductive group scheme $G$. Then for any isomorphism\footnote{We note that the resulting topological space $X(\mbb C)$ can depend on the choice of $\iota$ (see \cite{serre1964exemples}), however by Artin comparison $H^n(X(\mathbb C), \mathbb F_p)$, and more generally $H^n_{G(\mathbb C)}(X(\mathbb C), \mathbb F_p)$, does not.} $\iota\colon  \mbb C_p \xra{\sim} \mathbb C$ and all $n\in \mathbb Z_{\ge 0}$ we have an inequality
$$\dim_{\mathbb F_p} H_{G(\mathbb C)}^n(X(\mathbb C), \mathbb F_p) \le \dim_k H_{\dR}^n([X/G]_k/k),$$
where $H^n_{G(\mathbb C)}(X(\mathbb C), \mathbb F_p)$ denotes the $G(\mathbb C)$-equivariant singular cohomology of $X(\mathbb C)$.
\end{thm*}
\noindent Note that plugging in $X=\Spec \mc O_K$ one precisely gets the original Totaro's conjecture (\ref{intro_Totaros_conjecture}).

Moreover, from \Cref{thm:main_result} it follows that for $\mstack X=[X/G]$ with $X$ smooth and proper and $G$ reductive one has $H^n_\et(\widehat{\mstack X}_{\mbb C_p},\mbb Z_p)=H^n_\et({\mstack X}_{\mbb C_p},\mbb Z_p)$ and $H^n_\et(\widehat{\mstack X}_{\mbb C_p},\mbb Q_p)=H^n_\et({\mstack X}_{\mbb C_p},\mbb Q_p)$. In particular, the results from \Cref{secintro:p-adic Hodge theory for Hodge-proper stacks} apply with $H^n_\et(\widehat{\mstack X}_{\mbb C_p},\mbb Q_p)$ replaced by $H^n_\et({\mstack X}_{\mbb C_p},\mbb Q_p)$: namely, one has a $(\phi,G_K)$-equivariant isomorphism 
$$
H^n_\et({[X/G]}_{\mbb C_p}, \mathbb Q_p)\otimes_{\mbb Q_p} B_\crys \simeq H^n_{\crys}({[X/G]}_k/W(k))[\tfrac{1}{p}]\otimes_{K_0} B_\crys,
$$
a filtered $G_K$-equivariant isomorphism 
$$
H^n_\et({[X/G]}_{\mbb C_p}, \mathbb Q_p)\otimes_{\mbb Q_p} B_\dR \xymatrix{\ar[r]^\sim &} H^n_\dR([X/G]/K)\otimes_K B_\dR
$$
and the Hodge-Tate decomposition 
$$
H^n_{\et}({[X/G]}_{\mbb C_p},\mbb Q_p)\otimes_{\mbb Q_p}\mbb C_p \simeq \bigoplus_{i+j=n} H^j({[X/G]}_{K},\wedge^i\mbb L_{\mstack [X/G]/K})\otimes_{K}\mbb C_p(-i).
$$
This can be seen as a $G$-equivariant version of the classical rational $p$-adic Hodge theory.

It is also true that $BG$ is locally $\mbb F_\ell$-acyclic for $l\neq p$. This holds more generally for the quotients $[X/G]$ where $X$ itself is $\mbb F_\ell$-acyclic (see \Cref{prop: local F_l-acyclicity of quotients}); this now includes not only smooth proper $\mc O_{\mbb C_p}$-schemes, but also complements in those to smooth proper subschemes (\Cref{lem: example of F-l acyclic scheme}). 

\begin{rem}
	Let for simplicity $G$ be a split reductive group over $\mbb Z_p$. Note that the set $\widehat{G}_{\mbb Q_p}(\mbb Q_p)$ is identified with $G(\mbb Z_p)\subset G(\mbb Q_p)$ while $G^{\an}_{\mbb Q_p}(\mbb Q_p)=G(\mbb Q_p)$. Thus, as pointed out in \cite[Remark 2.3]{BhattLi}, one can think of $\widehat{G}_{\mbb Q_p}\subset G^{\an}_{\mbb Q_p}$ as a sort of "maximal compact subgroup". Over $\mbb C$, the topological group $G(\mbb C)$ contracts to its maximal compact subgroup $K\subset G(\mbb C)$, in particular, the natural map of classifying spaces $\lambda\colon BK\ra BG(\mbb C)$ is a homotopy equivalence. If $G$ is connected, this is equivalent to $\lambda$ inducing an equivalence on the singular $\mbb F_\ell$-cohomology for any $\ell$. Thus, one can interpret the simultaneous $\mbb F_p$ and $\mbb F_\ell$-local acyclicity of $BG$ as a $p$-adic analytic analogue of this fact.
\end{rem}

It would be interesting to understand what essential features of the geometry of $[X/G]$ were actually used in the proof of \Cref{thm:main_result}. In particular, it is not clear whether Hodge-properness of a smooth stack $\mstack X$ is enough to guarantee $\mbb F_p$-acyclicity. We don't think this is true in general, but we think that Hodge-properness at least should imply $\mbb Q_p$-acyclicity. Namely,

\begin{conj*}[\ref{conj:Hodge proper stacks}]
	Let $\mstack X$ be a smooth Hodge-proper stack over $\mc O_K$. Then the natural map 
	$$
	\Upsilon_{\mstack X}\colon \RG_\et({\mstack X}_{\mbb C_p},\mbb Q_p)\tto \RG_\et(\widehat{\mstack X}_{\mbb C_p},\mbb Q_p)
	$$
	is an equivalence.
\end{conj*}

One reason for believing in this conjecture is that dimensions of individual cohomology of both sides are the same (see \Cref{cor: dimensions of \'etale cohomology is the same}). There are many more examples of Hodge-proper stacks even among the quotient stacks than what is covered by \Cref{thm:main_result} (see \Cref{sec:more examples}); one can also give some sufficient conditions for Hodge-properness in terms a $\Theta$-stratification (a useful notion introduced by Halpern-Leistner in \cite{HL-instability} and \cite{Halpern-Leistner_Theta}) without appealing to the quotient stack presentation (see \Cref{ex: Theta-stratified}). We are planning to attack \Cref{conj:Hodge proper stacks} in the follow-up paper. If \Cref{conj:Hodge proper stacks} is true, the results of \Cref{secintro:p-adic Hodge theory for Hodge-proper stacks} would hold true with $H^n_\et(\widehat{ \mstack X}_{\mbb C_p},\mbb Q_p)$ replaced by $H^n_\et({ \mstack X}_{\mbb C_p},\mbb Q_p)$.

Returning to the question about $\mbb F_p$-acyclicity, we think that some further assumptions that are of more geometric nature are needed. Here is a potential option. Following\footnote{Strictly speaking, the definition of formal properness in \cite{HL_RelaxedProperness} is more involved, but is morally similar, so we kept the name.} \cite{HL_RelaxedProperness}, let's call an $\mc O_K$-stack $\mstack X$ \textit{formally proper} if it satisfies ($p$-adic) formal GAGA: namely, the natural $p$-completion functor $\Coh(\mstack X)\xra{\sim}\Coh(\widehat{ \mstack X})$ (with $\Coh(\mstack X)$ denoting the bounded \textit{derived} category of coherent sheaves) is an equivalence. We conjecture that this property is enough:
\begin{conj*}[\ref{etale_conjecture}]
Let $\mstack X$ be a smooth formally proper stack over $\mc O_K$. Then the natural map  
	$$
\Upsilon_{\mstack X}\colon\RG_\et({\mstack X}_{\mbb C_p},\mbb F_p)\tto \RG_\et(\widehat{\mstack X}_{\mbb C_p},\mbb F_p)
$$
is an equivalence.
\end{conj*}
Our work shows that \Cref{etale_conjecture} is true if $\mstack X=[X/G]$ for $X$ smooth and proper or if $\mstack X$ is itself proper. It's known that proper stacks are formally proper (\cite{lim2019grothendieck}, \cite{HL_RelaxedProperness}), besides that quite a few examples and properties of formally proper stacks (including $BG$ for $G$ split reductive) were given and proven in \cite{HL_RelaxedProperness}, and some work in this direction was also done by Ben Lim in his thesis \cite{lim2020algebraization}. In particular, a conical resolution is formally proper (see \Cref{rem:conical resolutions refined conjecture}) and thus, \Cref{etale_conjecture} would imply \cite[Conjecture 5.2.3]{Kubrak_Travkin}, as well as its refined version \Cref{conj: conical resolutions} (which is a natural analogue of Totaro's inequality in the conical resolution case). In general \Cref{etale_conjecture} would show that for any smooth formally proper stack $\mstack X$ one has the inequality 
$$
\dim_{\mathbb F_p} H_\sing^i(\Pi_{\infty}{\mstack X}(\mbb C), \mathbb F_p) \le \dim_{k} H^i_\dR(\mstack X_k/k),
$$
where $\Pi_{\infty}{\mstack X}(\mbb C)$ is the underlying homotopy type of $\mstack X_{\mbb C}$ (see \Cref{constr:underlying homotopy type}).

The motivation for \Cref{etale_conjecture} comes from the following observation: if $\mstack X$ is formally proper, then $\Upsilon_{\mstack X}$ at least induces an isomorphism on $H^1$. Indeed, both sides can be identified with the set of $\mbb F_p$-\'etale covers of the corresponding geometric stacks, which then can be reinterpreted purely in terms of the category of coherent sheaves on $\mstack X$ and $\widehat{\mstack X}$.

\subsection{Prismatic cohomology of $BG$ and $\Ainf$-Chern classes}
For more details see \Cref{sect_applications}. In \cite{Totaro_deRhamBG} Totaro showed that if $p$ is a non-torsion prime for a split reductive group $G$, then $H^*_\dR(BG/\mathbb Z_{(p)})$  is a polynomial ring on generators of degrees equal to 2 times the fundamental degrees of $G$ and as such is isomorphic to the singular cohomology $H^*_\sing(BG(\mathbb C),\mathbb Z_{(p)})$. However this isomorphism is not canonical and depends on a particular choice of polynomial generators.

From Totaro's result, under the same restriction on $p$ one can compute the prismatic cohomology ring of $BG$ as a graded algebra in the category of Breuil-Kisin modules. Namely, one has (see \Cref{cor:prismatic cohomology of BG as a ring})
$$
H^*_\Prism(BG/\mf S)\simeq \Sym_{\mf S}(\mf S\{-e_1\}[-2e_1]\oplus\cdots \oplus \mf S\{-e_n\}[-2e_n]),
$$
where $e_i$ are the so-called fundamental degrees of $G$.
Moreover, a choice of a uniformizer $\pi\in\mc O_K$ gives a natural way to associate to a given set of polynomial generators of $H^*_\sing(BG(\mathbb C),\mathbb Z_{p})$ a set of polynomial generators of $H^*_\dR(BG/\mc O_K)$ (see \Cref{rem: realizing de Rham as a deformation}). This makes Totaro's isomorphism more canonical, at least after replacing the localization $\mbb Z_{(p)}$ by the completion $\mbb Z_p$ (or, more generally $\mc O_K$). In particular when $\mc O_K=\mbb Z_p$ (or more generally $W(k)$) one preferable choice of $\pi$ could be $p$.

In the case of vector bundles (corresponding to $G=\GL_n$) one can also ask whether one has a nice theory of prismatic characteristic classes. We start to set up such a theory in \Cref{sec:characteristic classes}, defining certain polynomial generators $c_i^\Ainf$ of $H^*_\Prism(B\GL_n/\Ainf)$, however to make the construction independent of any choices, one needs to work with a twisted version (see \Cref{constr: Chern classes} for more details)
$$
H^*_{\Prism,{\mr tw}}(B\GL_n/\Ainf)\coloneqq\bigoplus_{i=0}^\infty \left(H^{2i}_{\Prism^{(1)}}(B\GL_n/\Ainf)\{i\}\right).
$$ Given a variety $X$ with a vector bundle $E$ one can then define the corresponding $i$-th $\Ainf$-Chern class of $E$ as the pull-back of $c_i^\Ainf$ under the corresponding map $X\ra B\GL_n$. Our construction of  $c_i^\Ainf$ uses the topological Chern classes as an input, and thus all the standard nice properties (like product formula) are more or less automatic.  As usual, one can also prove the "projective bundle formula", namely given a vector bundle $E\ra X$ of rank $n$, after a choice of a trivialization of a twist $\Ainf\{1\}$ one gets (see \Cref{splitting_principle} and Remarks \ref{rem: projective bundle formula}, \ref{rem: trivialized projective bundle formula})
$$
H^*_\Prism(\mbb P_X(E)/\Ainf)\simeq H^*_\Prism(X/\Ainf)[\xi]/(\sum_{i=0}^n c_i^\Ainf(E)\cdot \xi^{n-i}=0).
$$
We expect these classes to agree with de Rham and crystalline Chern classes via the de Rham and crystalline comparisons, in this work,  however, we only prove that they differ by an invertible constant in $\mbb Z_p$. We plan to return to this question in a sequel. This should also follow from the work of Bhatt-Lurie, where the local construction for the prismatic Chern classes will be given and which is now in preparation. Nevertheless, the comparison in the above form is still enough to get some non-trivial applications. For example, if $H^{2n}_{\Prism}(X/\Ainf)$ does not have $\mu$-torsion, then non-vanishing of the de Rham Chern classes of the reduction $E_k$ implies that the corresponding topological vector bundle $E(\mbb C)$ is non-trivial (see \Cref{cor: top Chern classes are  0 => de Rham are} and \Cref{rem: etale Chern are 0 is the same as topological}). We also prove a statement on the equality of $p$-adic valuations of Chern numbers in different theories (\Cref{prop: p-adic valuations of Chern numbers}); we plan to show that they are in fact equal in the sequel (and, again, it would follow from the work of Bhatt-Lurie). 

\subsection{Plan of the paper}

In this work we freely use the theory of higher geometric stacks developed in \cite[Section 1.3.3]{TV_HAGII} or \cite[Chapter 2.4]{GaitsRozI}. For the reader's convenience we review all necessary definitions and results about geometric stacks and quasi-coherent sheaves on them in \Cref{section_geometric_stacks_and_QCoh}. In a short \Cref{sec: Ainf} we recall the definition of Fontaine's infinitesimal period ring $\Ainf$ and some important elements of it.

We start by recalling the notion of Hodge-proper stacks (introduced by the authors in \cite{KubrakPrikhodko_HdR}) in \Cref{section:Hodge_proper_stacks}. The first two subsections (\ref{sect:Some properties of Coh}, \ref{sect:Hodge-proper stacks}) contain various technical results about bounded below coherent modules and basic properties of Hodge-proper stacks respectfully. In \Cref{sec: examples} we demonstrate Hodge-properness of the main example of interest in this work, namely a global quotient of a smooth proper scheme $X$ by an action of a parabolic subgroup $P$ of a reductive group. We then also give some other examples (like the ones coming from conical resolutions) in \Cref{sec:more examples}.

In \Cref{section:p_adic_cohomology_for_stacks} we extend the de Rham, crystalline and prismatic cohomology to stacks and study relations between them. In \Cref{sect:reminder on prismatic stuff} we recall some essential properties of prismatic cohomology functor for schemes. In \Cref{subsect:prism_for_stacks} we define the prismatic cohomology of stacks and prove some of its basic properties. Then in \Cref{sec: Hodge and de Rham cohomology} we recall the constructions of the Hodge and de Rham cohomology and prove the de Rham comparison for prismatic cohomology of stacks. In a short \Cref{sec:Hodge-Tate filtration} we construct the Hodge-Tate filtration on (the mod $I$-reduction) prismatic cohomology of stacks, which, as in the case of schemes, plays an important technical role in various proofs. In \Cref{subsect:crystalline_for_stacks} we give a definition of the crystalline cohomology of stacks, show that this definition agrees with Olsson's construction in the case of $1$-Artin stacks, and then prove the crystalline comparison.

In \Cref{sect:etale_sheaves} we develop some rudiments of \'etale sheaf theory on geometric stacks both in algebraic and rigid-analytic contexts that are necessary for this work. In \Cref{etale_sheaves_on_stacks} we introduce \'etale sheaves on Artin stacks and prove a few basic results about them, most importantly the smooth base change theorem (\Cref{et_smooth_basechange_algebraic}). In \Cref{subsect:local systems} we study the full subcategory of derived \'etale local systems inside it and show that the pushforward under a smooth proper schematic map sends (bounded below) local systems to local systems (\Cref{functoriality_of_loc_sys}). In \Cref{adic_local_systems} we develop a similar theory for geometric rigid-analytic stacks and extend some results of Huber relating \'etale sheaves on algebraic stacks and their analytifications. The highlight of this section is \Cref{adic_base_change} which states that for constructible sheaves on the algebraic side "analytification" commutes with pushforward under proper maps. The proofs in \Cref{sect:etale_sheaves} consist largely of systematic reductions to the case of schemes by descent, but they are not completely formal.

In \Cref{sec_etale_cohomology_of_stacks} we study \'etale cohomology of two variants of generic fibers: the usual algebraic generic fiber and the Raynaud's generic fiber. In \Cref{sect:etale_cohomology_on_stacks}, using the results from \Cref{adic_local_systems} we construct a natural comparison map $\Upsilon_{\mstack X}$ between \'etale cohomology of the two versions. We formulate a stacky analog of Artin's comparison (\Cref{Artins_comparison_stacks}) for the algebraic side and establish some properties of $\Upsilon_{\mstack X}$ (most importantly, \Cref{prop: Lambda-acyclicity translates via proper smooth maps}). We also formulate a conjecture (\ref{etale_conjecture}) about the generality in which this map is an equivalence. We call this property \emdef{local acyclicity}.  We then study its basic properties in the case of $\mbb F_\ell$-coefficients (for $\ell \neq p)$ in \Cref{sec: Algebraic nearby cycles and local }, where we relate it to the cohomology of the sheaf of the (algebraic) vanishing cycles. The last subsection (\ref{subsect:etale_comparison}) consists of several parts. In \Cref{subsubsect:etale_comparison} we prove the \'etale comparison which expresses the \'etale cohomology of the Raynaud generic fiber in terms of the prismatic cohomology and an enhancement of that in the Hodge-proper case. In \Cref{sssection:C_cris_for_stacks} we then deduce that the rational \'etale cohomology of Raynaud generic fiber of a Hodge-proper stack are crystalline Galois representations and give a formula for the associated Breuil-Kisin module in terms of prismatic cohomology. In \Cref{sssection: Hodge-Tate decomposition} we show the Hodge-to-de Rham degeneration for the generic fiber $\mstack X_K$ and also establish the Hodge-Tate decomposition for $H^n(\mstack X_{\mbb C_p},\mbb Q_p)\otimes_{\mbb Q_p}\mbb C_p$. In \Cref{sssect:inequality of dimension} we show that the length of $\mbb Z/p^n$-\'etale cohomology of the Raynaud generic fiber of a Hodge-proper stack $\mstack X$ is bounded above by the length of $W_n(k)$-crystalline cohomology of the special fiber of $\mstack X$, analogously to a similar result in \cite{BMS1}.

\Cref{sec: diff_etale} is devoted to the proof of local acyclicity of some global quotient stacks. In subsections (\ref{sec:de Rham cohomology of BG_m}, \ref{prismatoc_coh_of_Btori}) we compute the de Rham and prismatic cohomology of a classifying stack of a split torus $T$ by approximating it by products of projective spaces $\mathbb P^n$. In \Cref{subsect:prismatic_coh_of_An_quot_T} we show that the pullback along the structure map $[\mathbb A^n/T] \to BT$ induces an equivalence on the prismatic and \'etale cohomology of both generic fibers. We then use this in \Cref{subsect:case_of_BB} to deduce that the local acyclicity of the classifying stack $BB$ of a Borel subgroup $B$ in a reductive group $G$. In the very short \Cref{ssec:BPvsBL } we deduce that prismatic cohomology of $BP$ is equivalent to the prismatic cohomology of $BL$ for a parabolic subgroup $P\subset L$ and a Levi $L\subset P$ from an analogous result of Totaro. Finally, in subsection \ref{sec: comparison for [X/P]}, using the \'etale sheaf machinery developed in \Cref{sect:etale_sheaves}, we deduce local acyclicity of $[X/P]$, where $X$ is smooth proper scheme and $P\subset G$ is a parabolic subgroup. By combining the results of this and the previous sections we deduce the generalized form of Totaro's conjecture, see \Cref{thm:main_application}.

Finally, in \Cref{sect_applications} we give several other applications. In \Cref{sec: de Rham and prismatic cohomology of BG} we compute the prismatic cohomology of $BG$ for $G$ split reductive under the assumption that $p$ is non-torsion: it is (non-canonically) isomorphic to the free symmetric algebra on $\oplus_i \mathfrak S\{-e_i\}[-2e_i]$ in the category of Breuil-Kisin modules, where $e_i$ are the so-called fundamental degrees of $G$. A choice of a uniformizer $\pi\in\mc O_K$ then gives a preferable choice of polynomial generators (see \Cref{rem: realizing de Rham as a deformation}) which realizes the de Rham cohomology of $BG$ as a concrete deformation of the singular cohomology. In the next two subsections (\ref{sec:characteristic classes}, \ref{sec:Integrals and Chern numbers}) we start to set up the theory of $\Ainf$-valued Chern classes. We establish some comparison results with \'etale, crystalline, and de Rham Chern classes up to an invertible constant from $\mathbb Z_p$. This then gives an analogous compatibility for Chern numbers in all the theories. In the last \Cref{sec: conical resolutions} we discuss some application to conical resolutions studied by Travkin and the first author in \cite{Kubrak_Travkin}.

\subsection{Notations and conventions}\label{sect:notations}
\begin{enumerate}[wide,itemindent=*]
\item We will freely use the language of higher categories, modeled e.g. by quasi-categories of \cite{Lur_HTT}. If not explicitly stated otherwise all categories are assumed to be $(\infty,1)$ and all (co-)limits are homotopy ones. The $(\infty, 1)$-category of Kan complexes will be denoted by $\Type$ and we will call it \emph{the category of spaces}. By $\Lan_i F$ and $\Ran_i F$ we will denote left and right Kan extensions of a functor $F$ along $i$ (see e.g. \cite[Definition 4.3.2.2]{Lur_HTT} for more details).

\item For a commutative ring $R$ by $\DMod{R}$ we will denote the canonical $(\infty, 1)$-enhancement of the triangulated unbounded derived category of the abelian category of $R$-modules $\UMod{R}$. All tensor product, pullback and pushforward functors are implicitly derived.

\item In this work by Artin stacks we always mean (higher) Artin stacks in the sense of \cite[Section 1.3.3]{TV_HAGII} or \cite[Chapter 2.4]{GaitsRozI}: these are sheaves in \'etale topology admitting a smooth $(n-1)$-representable atlas for some $n\ge 0$ (an inductively defined notion, see \Cref{subsect: Artin stacks} for more details). We stress that we (mostly) work with non-derived Artin stacks, i.e. they are defined on the category of \emph{ordinary} commutative rings. When we need to emphasize a precise dependence on $n$ (usually in inductive arguments) we say that $\mstack X$ is an $n$-Artin stack. Sometimes we also need to impose some finiteness conditions on a stack (like quasi-compact, finite type, et.c.). For a reminder see \Cref{sec:stacks_finiteness conditions}.

\item For a stack $\mstack X$ we will denote by $\QCoh(\mstack X)$ the category of \emph{quasi-coherent sheaves on $\mstack X$} defined as the limit $\lim_{\Spec A \to \mstack X} \DMod{A}$ over all affine schemes $\Spec A$ mapping to $\mstack X$ (see \cite[Chapter 3.1]{GaitsRozI} for more details). Note that $\QCoh(\mstack X)$ admits a natural $t$-structure such that $\mathcal F \in \QCoh(\mstack X)^{\le 0}$ if and only if $x^*(\mstack F) \in \DMod{A}^{\le 0}$ for any $A$-point $x\in \mstack X(A)$. Moreover, by \cite[Chapter 3, Corollary 1.5.7]{GaitsRozI}, if $\mstack X$ is an Artin stack then $\QCoh(\mstack X)$ is left- and right-complete (i.e. Postnikov's and Whitehead's towers converge) and the truncation functors commute with filtered colimits. See \Cref{sec:Quasi-coherent sheaves} for more details.

\item For an affine group scheme $G$ over a ring $R$, given a representation $M$ (i.e. a comodule over the corresponding Hopf algebra $R[G]$) we denote by $\RG(G,M)\in \DMod{A}$ the \textit{rational cohomology} complex of $G$, namely the derived functor of $G$-invariants $M\mapsto M^G$. By flat descent, for $G$ flat over $R$, the abelian category $\mr{Rep}(G)^\heartsuit\coloneqq \Rep_G(\DMod{R})^\heartsuit$ is identified with $\QCoh(BG)^\heartsuit$ and $\RG(G,M)\simeq\RG(BG,M)$. More generally, if $G$ acts on an affine variety $X=\Spec A$, the bounded below category $\QCoh^+([X/G])$ can be identified with the (bounded below) derived category of $G$-equivariant $A$-modules (see \Cref{derived_vs_derived}). 

\item We also use a lot the formalism of cotangent complex $\mbb L_{\mstack X/R}$ and its exterior powers. See \Cref{sec: cotangent complex}. 

\item In the paper we also use $I$-adic completions and these are usually assumed to be derived (we refer to \cite[Tag 091N]{StacksProject} for a recollection of basic properties of derived completions). In all considerations the completion ideal $I$ will be (at least locally) finitely generated. We denote the category of derived $I$-complete $A$-modules by $\DMod{A}_{\widehat I}$. We also use notions of $I$-completely (faithfully) flat modules and modules of some $I$-complete Tor-amplitude. Namely, we say that a complex $M\in \DMod{A}$ has \textit{$I$-complete Tor-amplitude $[a,b]$} if for any $I$-torsion module $N$ the (derived) tensor product $M\otimes_A N$ only has cohomology in degrees $[a,b]$. $M$ is called \textit{$I$-completely flat} if it has $I$-complete Tor-amplitude $[0,0]$ and \textit{$I$-completely faithfully flat} if the (automatically classical) $A/I$-module $M\otimes_A A/I$ is faithfully flat.
\end{enumerate}

\paragraph{Acknowledgments.}
We would like to thank Roman Bezrukavnikov, Bhargav Bhatt, Chris Brav, Nikolay Konovalov, Daniel Halpern-Leistner, Shizhang Li, Ben Lim, Davesh Maulik, Matthew Morrow, Sasha Petrov, Peter Scholze, Burt Totaro and Bogdan Zavyalov for fruitful and encouraging discussions during our work on this project. We would also like to thank Georgy Belousov, Chris Brav, Nikolay Konovalov, Sasha Petrov and Alexandra Utiralova for their useful comments on the draft of our text. Special thanks goes to Peter Scholze for looking through the whole draft and making many helpful remarks about various parts of the text. 

The first author would like to express his gratitude to the Massachusetts Institute of Technology for the wonderful time that he spent there as a graduate student and where most of his work in relation to this project was carried out.
The second author was partially supported by Laboratory of Mirror Symmetry NRU HSE, RF Government grant, ag. \textnumero 14.641.31.0001.

\section{Hodge-proper stacks}\label{section:Hodge_proper_stacks}
In this short section we recall the definition and basic results about what we call Hodge-proper stacks introduced in \cite{KubrakPrikhodko_HdR}. The first subsection contains various technical results about bounded below coherent modules. In the second subsection we recall the definition and basic properties of Hodge-proper and cohomologically proper stacks. In the third subsection we prove Hodge-properness of the main example in this work, namely a global quotient of a smooth proper scheme $X$ by an action of a parabolic subgroup $P$ of a reductive group. Finally, in the last subsection we give some other examples of Hodge-proper stacks mainly following \cite[Section 3]{KubrakPrikhodko_HdR}. For a recollection on Artin stacks and quasi-coherent sheaves on them we refer the reader to \Cref{section_geometric_stacks_and_QCoh}.

\subsection{Some properties of \texorpdfstring{$\Coh^+$}{Coh+}}\label{sect:Some properties of Coh}
Fix a Noetherian base ring $R$. Recall that the category $\DMod{R}$ has a natural $t$-structure.
\begin{defn}\label{defn: bounded below coherent sheaves}
	As usual we will denote by $\Coh^+(R)$ the full subcategory spanned by complexes $X$ of $R$-modules such that $H^{\ll 0}(X) \simeq 0$ and such that $H^i(X)$ are finitely generated $R$-modules for all $i\in \mathbb Z$. Following the usual notations we will call such complexes \emdef{bounded below coherent}.
\end{defn}

\noindent We will frequently use the following basic properties of $\Coh^+(R)$:
\begin{prop}[\cite{KubrakPrikhodko_HdR}, Proposition 1.2.3]\label{nperf_basics}
	Let $R$ be a Noetherian ring. Then:
	\begin{enumerate}
		\item The category $\Coh^+(R)$ is closed under finite (co)limits and retracts. In particular $\Coh^+(R)$ is a stable subcategory of $\DMod{R}$.
		
		\item For each $n\in \mathbb Z$ the category $\Coh^{\ge n}(R) \coloneqq \Coh^+(R) \cap \DMod{R}^{\ge n}$ is closed under totalizations.
	\end{enumerate}
	
\end{prop}
\begin{ex}
	Let $G$ be a topological group, such that the underlying space of $G$ is a finite CW-complex (e.g. $G$ might be just a finite discrete group). Then the co-chains complex $C^*(BG, R)$ is a bounded below coherent $R$-module. To see this note that $BG \simeq | B_\bullet G|$, where $B_n G :=  G^{\times n}$, and hence $C^*(BG, R) \simeq \Tot C^*(G^{\times \bullet}, R)$. Now, by assumption on $G$, for all $n\in \mathbb Z_{\ge 0}$ the complexes $C^*(G^n, R)$ are perfect (and hence bounded below coherent), so we conclude by the second part of \Cref{nperf_basics}.
\end{ex}

The property of being bounded below coherent may be checked flat locally:
\begin{prop}\label{ncoh_is_fpqc_local}
	Let $R\ra R^\prime$ be a faithfully flat map. Then a module $X\in \DMod{R}$ is bounded below coherent if and only if $H^i(X\otimes_R R^\prime)$ is a finitely presentable $R^\prime$-module for all $i\in \mathbb Z$ and $H^{\ll 0}(X\otimes_R R^\prime) \simeq 0$.
	\begin{rem}
		We stress that $R^\prime$ is not assumed to be Noetherian. If $R^\prime$ is Noetherian, the proposition says that $X$ is bounded below coherent $R$-module if and only if $M\otimes_R R^\prime$ is bounded below coherent $R^\prime$-module.
	\end{rem}
	
	\begin{proof}
		By flatness of $R^\prime$ over $R$, we have $H^i(X\otimes_R R^\prime) \simeq H^i(X)\otimes_R R^\prime$. Hence it is enough to prove that a (classical) $R$-module $M$ is finitely presentable if and only if $M\otimes_R R^\prime$ is finitely presentable over $R^\prime$. This is proved e.g. in \cite[Tag 03C4, Lemma 10.82.2(2)]{StacksProject}.
	\end{proof}
\end{prop}
Bounded below coherent complexes are also preserved under base change if we assume that the map is of finite Tor-amplitude:
\begin{lem}\label{lem: base change preserves coh^+ sometimes}
	Let $R\ra R'$ be a map of Noetherian rings of finite Tor-amplitude. Then given $X\in \Coh^+(R)$ we have $X\otimes_R R'\in \Coh^+(R')$.
\end{lem}
\begin{proof}
	Without loss of generality we can assume $X\in \Coh^+(R)^{\ge 0}$. The tensor product $X\otimes_R R'$ has a filtration $(\tau^{\le n} X)\otimes_R R'$ induced by the Postnikov filtration on $X$. Let the Tor-amplitude of $R\ra R'$ be equal to $[-k+1,0]$; then $H^n(X\otimes_R R')$ is the same as $H^n((\tau^{\le n+k} X)\otimes_R R')$ for some big enough $n$. In particular, $X\otimes_R R'$ is bounded below and it is enough to show that each cohomology module is finitely generated. Filtering $\tau^{\le n+k} X$ further, by \Cref{nperf_basics}(1), it is then enough to show the statement for $X$ classical: $X\in\Coh^+(R)^{\heartsuit}$. Any such $X$ has a (possibly infinite) free resolution $\ldots \ra R^{\oplus n_2} \ra R^{\oplus n_1}\ra 0 \simeq X$ with all terms being free $R$-modules of finite rank. Tensoring up with $R'$ we get a resolution of $X\otimes_R R'$ with the same property, in particular we see that all cohomology modules of $X\otimes_R R'$ are finitely generated. 
\end{proof}

Let $I \subseteq R$ be an ideal defining a Cartier divisor in $\Spec R$ (i.e. $I$ is a line bundle on $\Spec R$). Recall that a complex $X\in \DMod{R}$ is called (derived) $I$-complete if the natural map $X\ra \prolim X\otimes_R R/I^n$ is an equivalence.
Later we will need the following $\bmod{\ I}$ criterion of coherentness:
\begin{prop}\label{neraly_coherent_module_xi}
	Let $R$ be a Noetherian ring and let $I \subseteq R$ be an ideal defining a Cartier divisor in $\Spec R$ and assume that $R$ is derived $I$-complete. Then an $I$-complete complex $X \in \DMod{R}$ is bounded below coherent if and only if $X/I := X \otimes_R R/I$ is bounded below coherent as an $R/I$-complex of module. 
\end{prop}
The proof consists of a series of lemmas. Below we assume that $R$ and $I$ are as in the statement of the proposition. We start with the following connectivity estimate: 
\begin{prop}\label{connectiveity_modulo}
Let $X$ be an $I$-complete complex of $R$-module. Then:
\begin{enumerate}[label=(\arabic*)]
	\item $X$ is $n$-connective if and only if $X/I$ is $n$-connective.
	
	\item If $X/I$ is $n$-coconnective, then so is $X$.
\end{enumerate}

\begin{proof}
	We will only prove the first part, the proof of the second one being similar but slightly easier. Since a tensor product of an $n$-connective and an $m$-connective modules is $(n+m)$-connective we have that if $X$ is $n$-connective then so is $X/I \simeq X\otimes_R R/I$. Conversely assume $X/I$ is $n$-connective for some $n\in \mathbb Z$. Shifting if necessary we can assume $n=0$. Now using a co-fiber sequence $X/I \to X/I^k \to X/I^{k-1}$ we find $H^{>0}(X/I^k) \simeq 0$ for all $k\ge 1$. Since $X\simeq \prolim X/I^k$ for any $i$ we have the Milnor exact sequence
	$$\xymatrix{0 \ar[r] & {\prolim}^1 H^{i-1}(X/I^k) \ar[r] & H^i(X) \ar[r] & \prolim H^i(X/I^k) \ar[r] & 0.}$$
	Hence to prove $H^{>0}(X) \simeq 0$ it is enough to prove that the ${\prolim}^1$-term on the left vanishes. Note that $R/I^k$ is of $\Tor$-amplitude $1$ (it admits a $2$-terms resolution $I^k \to R$), so for $i>1$ we have $H^{i-1}(X/I^k) \simeq 0$ for all $k$. For $i=1$ notice that since $H^{-1}(X/I) \simeq 0$ the maps $H^0(X/I^k) \to H^0(X/I^{k-1})$ are surjective for any $k$, hence ${\prolim}^1 H^0(X/I^k) \simeq 0$ by Mittag-Leffler.
\end{proof}
\end{prop}

Note that if $X$ and $Y$ are derived $I$-complete, then so is the cofiber $Y/X$ of any map $X\to Y$.
\begin{lem}\label{a_bit_surprising}
Let $X,Y$ be connective $I$-complete $R$-modules and let $\phi\colon X \to Y$ be a map inducing surjection on $H^0$ modulo $I$. Then $H^0(\phi)\colon H^0(X)\to H^0(Y)$ is also surjective.

\begin{proof}
It is enough to prove that $H^0(Y/X) \simeq 0$. But by \cref{connectiveity_modulo}, this is equivalent to $(Y/X)/I \simeq (Y/I)/(X/I)$ being connected, which is equivalent to the assumption that $\pi_0(\phi/I)$ is surjective.
\end{proof}
\end{lem}

Below, if $M$ is a classical module we also denote by $M/I$ the \emph{ordinary} reduction modulo $\xi$. We hope this doesn't lead to any confusion.
\begin{lem}
In notations above let $X\coloneqq M[0]$ be a discrete derived $I$-complete $R$-module. Then $X$ is coherent if and only if $H^0(X/I)\simeq M/I$ is coherent as an $R/\xi$-module.

\begin{proof}
	If $X$ is coherent, i.e. $M$ is a finitely generated $R$-module, then $M/I$ is also finitely generated over $R/I$. For the other implication we need to prove that $M$ is finitely generated as an $R$-module. Let $\phi\colon R^{\oplus n}\to M$ be any map which induces a surjection $R^{\oplus n}/I \surj M/I$ (such a map exists since $M/I$ is a finitely generated $R/I$-module). By applying the previous lemma to $\phi[0]\colon R^{\oplus n}[0]\to M[0]$, we get that the induced map $H^0(\phi[0]) \colon R^{\oplus n} \to M$ is also surjective.
\end{proof}
\end{lem}

\begin{proof}[Proof of \Cref{neraly_coherent_module_xi}]
 Let $X\in \DMod{R}$. For each $i\in \mathbb Z$ we have the universal coefficients exact sequence
$$\xymatrix{0\ar[r] & H^i(X)/I \ar[r] & H^i(X/I) \ar[r] & H^{i+1}(X)[I]\otimes_R I \ar[r] & 0.}$$
Since coherent modules are closed under extensions, it follows that $X/I$ is bounded below coherent if $X$ is.

Conversely, assume that $X/I$ is bounded below coherent. By \Cref{connectiveity_modulo}, $X$ is bounded below. By shifting if necessary, we can assume that $H^{<0}(X) \simeq 0$. We claim that $H^0(X)$ is coherent. This follows from the previous lemma applied to the module $M\coloneqq H^0(X)$ by noting that $M/I \simeq H^0(X)/I$ is a submodule of a coherent module $H^0(X/I)$, hence coherent.

Finally, if we denote $X^\prime\coloneqq \tau^{>0} X$, we see that $X^\prime/I\simeq \cofib(H^0(X)[0]/I \to X/I)$, being a cofiber of a map of bounded below coherent $R/I$-modules, is bounded below coherent. By the same argument as above $H^0(X^\prime[1]) \simeq H^1(X)$ is finitely generated. Going on, one sees that all $H^i(X)$ are finitely generated.
\end{proof}
We end the subsection with a couple of general facts that will be useful later. Let $I\subset R$ be a finitely-generated ideal and assume $R$ is Noetherian and derived (eq. classically) $I$-adically complete. Recall that by {\cite[Proposition 3.4.4]{BS_proetale}} $X\in \DMod{R}$ is derived $I$-adically complete if and only if all its cohomology modules $H^i(X)$ are.

\begin{lem}\label{lem: bounded below coherent are complete}
	Assume that $R$ is as above. Then any $X\in \Coh^+(R)$ is derived $I$-complete.
	
	\begin{proof}
By \cite[Lemma 3.4.14]{BS_proetale} the classical derived $I$-complete modules form a Serre abelian category. Thus, since $H^i(X)$ is a finitely generated $R$-module, it is derived $I$-complete for any $i\in \mbb Z$. We then are done by \cite[Proposition 3.4.4]{BS_proetale} mentioned above.
	\end{proof}
\end{lem}

\begin{lem}\label{lem:the tensor product remains complete}
	Assume that $R$ is as above, but also regular. Let $R\ra R'$ be a map and assume that $R'$ is also derived $I$-complete (but not necessarily Noetherian). Then for any complex $X\in\Coh(R)^+$ the tensor product $X\otimes_R R'$ is $I$-complete:
	$$
	X\otimes_R R^\prime \xymatrix{\ar[r]^\sim&} X\widehat\otimes_R R^\prime.
	$$

\begin{proof}
Since $R$ is Noetherian and regular the Tor-amplitude of $R'$ (and thus also the $I$-complete Tor-amplitude) is finite; let both lie in $[-d,0]$ for some $d\ge 0$.  Via the same Noetherian and regular assumptions one has an equivalence $\Coh(R)\simeq \DMod{R}^\perf$ and thus the truncation $\tau^{\le n}X$ belongs to $\DMod{R}^\perf$ for any $n$. Consequently, the tensor product $\tau^{\le n}X\otimes_R - $ commutes with all limits and $\tau^{\le n}X\otimes_{R} R'$ is already $I$-complete. Then, by both usual and $I$-complete Tor-finiteness of $R'$ over $R$, we get equivalences
	$$
	X\otimes_{R} R' \xymatrix{& \ar[l]_\sim} \colim_n \left((\tau^{\le n}X)\otimes_{R} R'\right)  \xymatrix{\ar[r]^\sim &} \colim_n \left((\tau^{\le n}X)\widehat\otimes_{R} R'\right)  \xymatrix{\ar[r]^\sim &} X\widehat\otimes_{R} R'.
	$$ Indeed, $\cofib (\tau^{\le n}X\otimes_{R} R'\ra X\otimes_{R} R')$ is given by $\tau^{>n}X \otimes_{R} R'$ and lies in $\DMod{R}^{> n-d}$ and thus becomes zero as $n\ra \infty$. A similar argument works for the completed tensor product using that $\tau^{>n}X \widehat\otimes_{R} R'\in \DMod{R}^{> n-d}$.
\end{proof}
\end{lem}

\subsection{Hodge-proper stacks}\label{sect:Hodge-proper stacks}
In this section we still keep $R$ to be Noetherian. We define the class of stacks that we will mostly be interested in in this work. 
\begin{defn}[\cite{KubrakPrikhodko_HdR}, Definition 1.2.5.]\label{def:Hodge_proper}
	A smooth quasi-compact quasi-separated Artin stack $\mstack X$ over $R$ is called \emdef{Hodge-proper} if for any $i\in \mathbb Z_{\ge 0}$ the module of global sections $R\Gamma(\mstack X, \wedge^i \mathbb L_{\mstack X/R})$ lies in $\Coh^+(R)$. In other words, the complex $R\Gamma(\mstack X, \wedge^i \mathbb L_{\mstack X/R})\in \DMod{R}$ should be bounded from below and $H^j(\mstack X, \wedge^i \mathbb L_{\mstack X/R})\in \UMod{R}$ should be a finitely generated $R$-module for all $j\in \mathbb Z_{\ge 0}$.
\end{defn}

\begin{ex}
	Let $X$ be a smooth proper scheme over $R$. Then $\wedge^i \mathbb L_{X/R} \simeq \Omega_{X/R}^i$ is coherent and $H^j(X, \Omega_{X/R}^i)$ is finitely generated for all $i$ and $j$ since $X$ is proper. Hence $X$ is cohomologically proper.
\end{ex}

We use the name of Hodge in the definition above for the simple reason that 
$$
\RG_{\Hdg}(\mstack X/R)\coloneqq \bigoplus_{i=0}^\infty \RG(\mstack X, \wedge^i\mathbb L_{\mstack X/R})[-i]
$$
is a natural extension of the Hodge cohomology to smooth Artin stacks. Indeed, for smooth proper scheme $X$ we get
$$
H^n_\Hdg(X/R)\simeq \bigoplus_{i+j=n} H^{j}(X,\Omega^i_{X/R}).
$$
Thus, Hodge-proper stack is "proper" from the point of view of its Hodge cohomology.

We postpone a detailed discussion of the Hodge (and also de Rham) cohomology till \Cref{sec: Hodge and de Rham cohomology}. We also provide some more general examples of Hodge-proper stacks in Sections \ref{sec: examples} and \ref{sec:more examples}. As for now, we will just prove some general statements about them. First of all, it is easy to see that the property of being Hodge-proper is flat-local on the base:
\begin{prop}\label{lem:cohomologically proper is flat local}
	Let $R \to R^\prime$ be faithfully flat map of Noetherian rings. Then if $\mstack X^\prime \coloneqq \mstack X\times_{\Spec R} \Spec R^\prime$ is Hodge-proper over $R^\prime$, then $\mstack X$ is Hodge-proper over $R$.
	
	\begin{proof}
		Let $q\colon \mstack X^\prime \to \mstack X$ be the canonical projection. Then by the flat base change \Cref{QCoh_base_change}, we have
		$$R\Gamma(\mstack X, \mathcal F)\otimes_R R^\prime \simeq R\Gamma(\mstack X^\prime, q^*\mathcal F)$$
		for any quasi-coherent sheaf $\mathcal F \in \QCoh(\mstack X)$. In particular, the complex $R\Gamma(\mstack X, \wedge^i \mathbb L_{\mstack X/R})\otimes_R R^\prime \simeq R\Gamma(\mstack X^\prime, \wedge^i \mathbb L_{\mstack X^\prime/R^\prime})$ is bounded below coherent by assumption on $\mstack X^\prime$. We conclude by \Cref{ncoh_is_fpqc_local}.
	\end{proof}
\end{prop}
Under some additional restrictions Hodge-properness is also preserved under flat base change
\begin{prop}\label{prop: Hodge proprness survives some base-changes}
	Let $R \to R^\prime$ be a map of Noetherian rings that has finite Tor-amplitude. Let $\mstack X$ be a Hodge-proper stack over $R$. Then $\mstack X^\prime \coloneqq \mstack X\times_{\Spec R} \Spec R^\prime$ is Hodge-proper over $R^\prime$.
\end{prop}
\begin{proof}
	Let $f\colon \mstack X' \ra \mstack X$ be the natural map. By \cite[Proposition 1.1.8]{KubrakPrikhodko_HdR}, the induced map $\RG(\mstack X,\wedge^j\mathbb L_{\mstack X/R})\otimes_R R' \ra \RG(\mstack X',\wedge^j\mathbb L_{\mstack X'/R'})$ is an equivalence for any $j$. Then we are done by \Cref{lem: base change preserves coh^+ sometimes}.
\end{proof}

Finally, we have the following smooth descent property:
\begin{prop}\label{lem:Hodge-proper}
	Let $\pi\colon \mstack U \surj \mstack X$ be a smooth quasi-compact quasi-separated surjective map of smooth Artin stacks and assume that all terms of the corresponding \v Cech nerve $\mstack U_n$ are Hodge-proper. Then $\mstack X$ is Hodge-proper.
	\begin{proof}
		By the smooth descent for the cotangent complex (\Cref{flat_descent_for_cotangent_compl}), we have
		$$R\Gamma(\mstack X, \wedge^i \mathbb L_{\mstack X/R}) \simeq \Tot R\Gamma(\mstack U_\bullet, \wedge^i \mathbb L_{\mstack U_\bullet /R}).$$
		By assumption, all complexes $R\Gamma(\mstack U_n, \wedge^i \mathbb L_{\mstack U_n/R})$ are bounded below coherent and cohomologically bounded below by $0$ (by smoothness). So $R\Gamma(\mstack X, \wedge^i \mathbb L_{\mstack X/R})$ is also bounded below coherent by \Cref{nperf_basics}. 
	\end{proof}
\end{prop}

One can also introduce the following stronger variant of properness where instead of wedge powers of the cotangent complex we consider all coherent sheaves. We introduce the relative version right away.
\begin{defn}\label{def:coh_proper}
	Let $\mstack X$ be a finitely presentable Artin stack. A sheaf $\mathcal F\in \QCoh(\mstack X)$ is called \emph{bounded below coherent} if $\mathcal H^{\ll 0}(\mathcal F) \simeq 0$ and $\mathcal H^i(\mathcal F)$ is coherent for all $i\in \mathbb Z$. We will denote the full subcategory of $\QCoh(\mstack X)$ consisting of bounded below coherent sheaves by $\Coh^+(\mstack X)$.
\end{defn}
\begin{defn}\label{def:cohomologically proper morphism}
	A morphism $f\colon \mstack X \to \mstack Y$ of finitely presentable Artin stacks is called \emph{cohomologically proper} if the induced functor $f_*\colon \QCoh(\mstack X) \to \QCoh(\mstack Y)$ sends $\Coh^+(\mstack X)$ to $\Coh^+(\mstack Y)$. A finitely presentable Artin stack $\mstack X$ is called \emph{cohomologically proper} if the structure morphism $\mstack X \to \Spec R$ is cohomologically proper.
\end{defn}
\begin{rem}\label{rem: coh proper enough to check on the heart}
	By the left exactness of $f_*$, it is enough to prove that $f_*(\Coh(\mstack X)^\heartsuit) \subset \Coh^+(\mstack Y)$.
\end{rem}
We have the following basic properties of cohomologically proper morphisms:
\begin{prop}[{\cite[Proposition 2.2.4]{KubrakPrikhodko_HdR}}] \label{properties_of_coh_prop_morps}
	In the notations above we have:
	\begin{enumerate}
		\item The class of cohomologically proper morphisms is closed under compositions.
		
		\item Let $f\colon \mstack X \to \mstack Y$ be a cohomologically proper morphism and assume that $\mstack X$ is Noetherian. Then for any open quasi-compact embedding $\mstack U \inj \mstack Y$ the pullback $\mstack U \times_{\mstack Y} \mstack X$ is cohomologically proper over $\mstack U$.
		
		\item Let $f\colon \mstack X \to \mstack Y$ be a quasi-compact quasi-separated morphism such that for some smooth cover $\pi\colon \mstack U\ra \mstack Y$ the pull-back $f_{\mstack U}\colon \mstack X\times_{\mstack Y} \mstack U \to \mstack U$ is cohomologically proper. Then $f$ is cohomologically proper.
	\end{enumerate}
\end{prop}

The previous notion is stronger then the Hodge-properness:
\begin{prop}\label{coh_proper_stronger}
	Let $\mstack X$ be a smooth cohomologically proper Artin stack. Then $\mstack X$ is Hodge-proper.
	
	\begin{proof}
		By Propositions \ref{cotangent_of_smooth_stack} and \ref{derived_tensor_finiteness} all exterior powers of the cotangent complex of $\mstack X$ are perfect, hence coherent. So, by assumption on $\mstack X$, all complexes $R\Gamma(\mstack X, \wedge^i \mathbb L_{\mstack X/R})$ are bounded below coherent.
	\end{proof}
\end{prop}

One also has the following descent property, that is analogous to \Cref{lem:Hodge-proper}, but allows more covers:
\begin{prop}[{\cite[Proposition 2.2.14]{KubrakPrikhodko_HdR}}]\label{lem:Cech cohomologically proper}
	Let $\pi_\bullet\colon \mstack U_{\bullet} \surj \mstack X$ be a flat hypercover by finitely presentable Artin stacks and assume that all $\mstack U_n$ are cohomologically proper. Then $\mstack X$ is cohomologically proper.
\end{prop}

\subsection{Main example: quotients of proper schemes by reductive groups}\label{sec: examples} 
In this section we give an example of a Hodge-proper stack which is  most important for this work: namely the global quotient $[X/G]$ of a smooth proper scheme $X$ by an action of a reductive group scheme $G$. The fact that $[X/G]$ is Hodge-proper in this case is also a consequence of \cite[Theorem 3.1.4]{KubrakPrikhodko_HdR}, but for completeness of the exposition here we give another argument, which does not rely on hard cohomology vanishing results from \cite{FranjouVDKallen_PowerReductivity}\footnote{However, we have to assume that $X$ is proper instead of a significantly less restrictive condition in \cite[Theorem 3.1.4]{KubrakPrikhodko_HdR}}.

For the rest of this section let $R$ be a Noetherian ring and let be $X$ a smooth proper $R$-scheme. Let $G$ be a reductive group $R$-scheme acting on $X$. By \Cref{coh_proper_stronger}, it is enough to prove that $[X/G]$ is cohomologically proper. Note that since $X$ is proper we have $\RG(X,\mc F)\in \Coh(R)$ for any coherent sheaf $\mc F\in \Coh(X)$.
\begin{rem}
	The statement is more or less obvious if the base ring $R$ is a $\mathbb Q$-algebra: in this case $G$ is linearly reductive and one recovers the underlying complex of $R\Gamma([X/G], \mathcal F) \simeq R\Gamma(X, \mathcal F)^G$ as a direct summand of $R\Gamma(X, \mathcal F)$ and thus is also coherent. However, if $R$ is just an algebra over $\mathbb Z$ one needs to come up with some other argument.
\end{rem}
This other argument will occupy the rest of this section (see also \Cref{BP_is_SCP} for a slight generalization). It will consist of a several step reduction to the case of a torus which is linearly reductive. First note that by \Cref{properties_of_coh_prop_morps} and the following assertion, it is enough to consider the case of split $G$:
\begin{prop}[{\cite[Lemma 5.1.3]{Conrad_Reductive}}]\label{prop:Conrad}
	Let $G$ be a reductive group scheme over $R$. Then there exists an \'etale cover $R\ra R^\prime$ such that $G_{R^\prime}$ is split.
\end{prop}
After that the proof consists of four steps: first, we treat the case of a split torus $T$ (which is easy, since $T$ is linearly reductive), then, the case of a quotient $[\mbb A^n/T]$ by linear conical action, from which we deduce the case of Borel subgroup $B\subseteq G$. Finally, from that we deduce the general case of any parabolic subgroup $P\subseteq G$.

\paragraph{Torus.}
Let $T$ be a split algebraic torus $T\simeq \mathbb G_m^d$ over $R$. Let $X^*(T)\simeq \mathbb Z^d$ denote the character lattice. Then by \Cref{derived_vs_derived}, we have $\QCoh(BT)^+\simeq D^+(\Rep_T^\heartsuit)$  and $\Coh(BT)\subset \QCoh(BT)^+$ can be identified with the subcategory of objects in $D^+(\Rep_T^\heartsuit)$ whose underlying $R$-module lies in $\Coh(R)$. Moreover, the category  $\Rep_T^\heartsuit$ is equivalent to the category of $X^*(T)$-graded $R$-modules and the global sections functor $\RG\colon \Coh(BT)\ra \DMod{R}$ is given in terms of $D^+(\Rep_T^\heartsuit)$ by sending a $X^*(T)$-graded coherent complex $M_*$ to its 0-th graded component $M_0$. Since $M_0$ is a direct summand of $M$ it is clear that $BT$ is cohomologically proper.

Let now $V\simeq R^n$ be a representation of $T$. As any other representation of $T$, it is decomposes as a direct sum of characters, say $V=V_{\chi_1}\oplus\cdots \oplus V_{\chi_k}$. We can consider the corresponding affine space $\mathbb A^n = \Spec \Sym_{R} (V^\vee)$, where $V^\vee\coloneqq \Hom_{R}(V,R)$ is the dual representation. There is a natural action of $T$ on $\mathbb A^n$ and we consider the corresponding quotient stack $[\mathbb A^n/T]$. Note that the action of $T$ on $\mathbb A^n$ corresponds to the natural $X^*(T)$-grading on $\Sym_{R}(V^\vee)$. Note also that, via \Cref{derived_vs_derived}, $\Coh([\mathbb A^n/T])$ can be identified with the full subcategory of the derived category of $X^*(T)$-graded $\Sym_{R}(V^\vee)$-modules consisting of objects which are coherent as $\Sym_{R}(V^\vee)$-modules. The global sections functor in this terms is again given by taking the $X^*(T)$-$0$-th graded component of the underlying complex.

We make the following crucial assumption on the action:

\begin{defn}\label{defn: conical action}
	The $T$-representation $V$ is called \textit{conical} if $V^T=0$ and the commutative monoid $R(V)=\mathbb N\cdot\chi_1+\cdots +\mathbb N\cdot \chi_n\subset X^*(T)$, spanned by $\chi_1,\ldots,\chi_n$, does not contain a copy of $\mathbb Z$ inside. 
\end{defn}
If $V$ is conical then for any $\alpha\in X^*(T)$ the corresponding graded component $\Sym_{R}(V^\vee)_\alpha$ has finite rank over $R$.\footnote{From the definition it follows that $\mbb R_{>0} \cdot\chi_1+\cdots +\mathbb R_{>0}\cdot \chi_n\subset X^*(T)_{\mbb R}$ is a proper cone: namely, it is contained in an open half space $\{v\in X^*(T)_{\mbb R} | h(v)>0\}$ for some linear functional $h\in (X^*(T)_{\mbb R})^\vee$. We have $h(\chi_i)>0$ for any $i$ and $h$ gives a strictly positive $\mbb R$-grading on $\Sym_{R}^{>0}(V^\vee)$ by pairing $X^*(T)$-weight with $h$. Rescaling $h$ we can assume $h(\chi_i)>1$, then for any $d\in \mbb Z$ the $R$-submodule $\Sym_{R}^{>0}(V^\vee)_{h<d}$ is contained in $\oplus_{i=1}^d\Sym_{R}^{i}(V^\vee)$ and, therefore, is finitely generated over $R$. Finally $\Sym_{R}^{>0}(V^\vee)_\alpha \subset \Sym_{R}^{>0}(V^\vee)_{h<\lceil h(\alpha) \rceil}$.}  
\begin{prop}\label{lem:A^n/T is cohomologically proper}
	Assume $V$ is a conical $T$-representation and let $\mathbb A^n = \Spec \Sym_{R} (V^\vee)$. Then the corresponding quotient stack $[\mathbb A^n/T]$ is cohomologically proper.
	
	\begin{proof}
		We use \Cref{rem: coh proper enough to check on the heart} and the  description of the heart $\Coh([\mathbb A^n/T])^{\heartsuit}\subset \Coh([\mathbb A^n/T])$ as finitely generated $X^*(T)$-graded modules over $\Sym_{R}(V^\vee)$. Since global sections are given by taking the 0-th graded component it is enough to show that for any finitely generated (classical) graded $\Sym_{R}(V^\vee)$-module $M_*$ the component $M_0$ is finitely generated over $R$. Choosing homogeneous generators of $M$ we get a surjection 
		$$
		\bigoplus_{i=1}^k \Sym_{R}(V^\vee)(-\alpha_i) \xymatrix{\ar@{->>}[r]&} M,
		$$ where $\alpha_i$ are the degrees of the generators and $(-\alpha_i)$ denotes the grading shift. This provides a surjection $\oplus_{i=1}^k \Sym_{R}(V^\vee)_{-\alpha_i} \surj M_0$ with the left module being finitely generated over $R$.
	\end{proof}
\end{prop}
\begin{rem}\label{U as an example}
	Given a split reductive group $G$ with a maximal torus $T$ an example of a conical representation $V$ is given by the Lie algebra $\mf u\coloneqq \Lie(U)$, where $U$ is the unipotent radical of any Borel subgroup $B\subset G$, such that $B$ contains $T$. We have a splitting of $\mf u$ into a sum of characters, which are given by positive roots $\alpha\in \Delta$ (with respect to the chosen $B$). Moreover, if we consider the $T$-action on $U$ by conjugation, one has a $T$-equivariant identification $\Lie(U)\xra{\sim} U$ as schemes (\cite[Chapter II,1.7]{Jantzen}). Thus, by \Cref{lem:A^n/T is cohomologically proper}, $[U/T]$ is cohomologically proper. This also applies to the diagonal action of $T$ on $U^n$ by conjugation.
\end{rem}

\begin{cor}
	For any $n\in \mathbb Z_{\ge 1}$ the iterated classifying stack $B^n T$ is cohomologically proper.
	
	\begin{proof}
		For $n=1$ this is a special case of \Cref{lem:A^n/T is cohomologically proper}. The general case follows by induction using the cover $\Spec R \ra B^n T$ and \Cref{lem:Cech cohomologically proper}.
	\end{proof}
\end{cor}

\paragraph{Borel subgroup.}
Let $B\subset G$ be a Borel subgroup containing $T$ and $U\subset B$ be the unipotent radical. We have $B \simeq T\ltimes U$.

\begin{prop}\label{prop:BB is coh proper}
	The classifying stack $BB$ is cohomologically proper.
	
	\begin{proof}
		We use the smooth cover $\pi\colon BT\to BB$ and the fact that $B \simeq T\ltimes U$. The corresponding \v{C}ech simplicial object $\check{C}(\pi)_\bullet$ of $\pi\colon BT\to BB$ is given as follows: $\check{C}(\pi)_n$ is identified with $[T\backslash \underbrace{B\times_T B\times_T\cdots\times_T B}_n/T]$ with the coboundary maps given by the associated projections. On the other hand, there is an isomorphism 
		$$
		[T\backslash \underbrace{B\times_T B\times_T\cdots\times_T B}_n/T]\simeq [(\underbrace{U\times U\times \cdots \times U}_n)/T],
		$$
		where $T$ acts on $U\times\cdots\times U$ via diagonal conjugation: $t\circ(u_1,\ldots,u_n)=(tu_1t^{-1}, \ldots, tu_nt^{-1})$. The corresponding map is induced by the map of schemes
		$$[(b_1, \ldots, b_n)] \in T\backslash B\times_T B\times_T\cdots\times_T B \mapsto ([b_1b_2\ldots b_{n-1} b_n], \ldots, [b_{n-1}b_n],[b_n])\in (U\times U\times \cdots \times U),
		$$ 
		where $[b]$ denotes a class of $b\in B$ in $[T\backslash B]\simeq U$. It is clear that this map is an isomorphism and the remaining $T$-action on the left translates exactly to the diagonal action of $T$ on the right.
		
		\Cref{U as an example} tells us that the quotient stack $(U\times U\times \cdots \times U)/T$ is cohomologically proper, thus, $\check{C}(\pi)_n$ is cohomologically proper for all $n$. Applying \Cref{lem:Cech cohomologically proper}, it follows that $BB$ is cohomologically proper.
	\end{proof}
\end{prop}
\begin{cor}\label{X_quot_B_is_coh_proper}
	Let $X$ be a smooth proper $R$-scheme equipped with an action of $B$. Then $[X/B]$ is cohomologically proper.
	
	\begin{proof}
		By assumption on $X$, the structure map $[X/B] \to BB$ is proper, and therefore, is cohomologically proper. Since the composition of cohomologically proper maps is cohomologically proper, we can conclude by the previous proposition.
	\end{proof}
\end{cor}

\paragraph{Any parabolic subgroup.} Finally, we can pass to the case of $BP$, where $P\subset G$ is any parabolic subgroup. In particular, one can take $P$ to be the whole group $G$.
\begin{prop}\label{BP_is_SCP}
	Let $P\subset G$ be a parabolic subgroup. Then $BP$ is cohomologically proper.
	
	\begin{proof}
		We use the cover $\pi\colon BB\ra BP$. The terms $\check{C}(\pi)_n$ of the corresponding \v{C}ech simplicial object are given by 
		$$
		\check{C}(\pi)_n \simeq [B\backslash\underbrace{P\times_B \cdots \times_B P}_n/B].
		$$
		Similarly to the formula in the proof of \Cref{prop:BB is coh proper}, one can identify $(P\times_B \cdots \times_B P)/B\simeq P/B \times \ldots P/B$. The variety $P/B$ is the flag variety for the Levi subgroup $L\subset P$ and, as such, is smooth and proper. Thus, it follows from \Cref{X_quot_B_is_coh_proper} that $
		\check{C}(\pi)_n$ is cohomologically proper for all $n$. By \Cref{lem:Cech cohomologically proper}, $BP$ is cohomologically proper as well.
	\end{proof}
\end{prop}
Similarly to the proof of \Cref{X_quot_B_is_coh_proper}, one deduces:
\begin{cor}
	Let $X$ be a smooth proper $R$-scheme and let $P$ be a parabolic subgroup in a reductive $R$-group scheme. Then $[X/P]$ is cohomologically proper.
\end{cor}
\subsection{Other examples}\label{sec:more examples}
Here we give some other examples of Hodge-proper stacks, mainly following \cite[Section 3]{KubrakPrikhodko_HdR}. We assume that the base ring $R$ is Noetherian and regular. Regularity is mainly important only for \Cref{ex: Theta-stratified}, but since we are mostly interested in the case $R=\mc O_K$, which is Noetherian and regular anyways, we make this assumption throughout the section.

\begin{ex}[Classifying stacks]
	By \cite[Proposition 2.3.6]{KubrakPrikhodko_HdR}, if $R=F$ is an algebraically closed field of characteristic 0, the classifying stack $BG$ is cohomologically proper for \textit{any} $G$. However, the situation is quite different in positive (or mixed) characteristic: many of the classifying stacks stop being Hodge-proper. This is very well illustrated by the example of $B{\mbb G_a}$; in this case over $\mbb F_p$ with $p>2$ one has an isomorphism (see \cite[Lemma 4.22 and Proposition 4.27]{Jantzen})
	$$
	H^*(B\mathbb G_a,\mc O_{B\mathbb G_a})\simeq \wedge^*_{\mbb F_p}(\mbb F_p v_1\oplus \mbb F_p v_p \oplus \mbb F_p v_{p^2}\oplus \ldots )\otimes_{\mbb F_p} \Sym^*_{\mbb F_p}(\mbb F_p w_p\oplus \mbb F_p w_{p^2} \oplus \mbb F_p w_{p^3}\oplus \ldots ),
	$$
	where $v_{p^i}$ have degree 1 and $w_{p^i}$ have degree $2$. In particular, $H^j(B\mathbb G_a,\mc O_{B\mathbb G_a})$ is an infinite-dimensional $\mbb F_p$-vector space for all $j>0$. The reader can also find an explicit computation of $H^*_{\mbb F_p}(B\mathbb G_a,\mc O_{B\mathbb G_a})$ over $\mathbb Z$ in \cite[Appendix A]{KubrakPrikhodko_HdR}.
\end{ex}

\begin{ex}[Quotients of conical resolutions by the $\mbb G_m$-action]\label{ex: conical resolutions} This example is motivated by \cite[(generalized form of) Conjecture 5.2.3]{Kubrak_Travkin}. We remind the definition of a conical resolution:
\begin{defn}\label{def: conical resolution}
		A \textit{conical resolution} over $R$ is given by:
		\begin{itemize}
			\item a proper birational morphism $\pi\colon X\ra Y$ of $R$-schemes,
			\item actions of $\mbb G_m$ on $X$ and $Y$ that are intertwined by $\pi$,
		\end{itemize}
		that satisfy the following conditions:
		\begin{itemize}
			\item $X$ is smooth over $R$;
			\item the action on $Y$ is conical, namely $Y\simeq \Spec A$ is affine, the induced $\mbb Z$-grading on $A$ (coming from the $\mbb G_m$-action) is nonpositive\footnote{Alternatively, one can ask for the grading to be nonnegative; these two options are intertwined by the automorphism of $\mbb G_m$ sending $t$ to $t^{-1}$. Note that in \cite{Kubrak_Travkin} the nonnegative convention on the weight is used.} and $A_0$ (which is also isomorphic to $A/A_{<0}$) is isomorphic to $R$.
		\end{itemize}	
	\end{defn}
	
	In this case one has a proper map $[X/\mbb G_m]\ra [Y/\mbb G_m]$. Moreover $\Coh([Y/\mbb G_m])^{\heartsuit}$ is identified (via \Cref{derived_vs_derived}) with the abelian category of finitely-generated $\mbb Z$-graded modules over $A$. Since $A$ is negatively graded (with $A_0\simeq R$ having finite rank over $R$), by an argument similar to \Cref{lem:A^n/T is cohomologically proper}, one sees that $[Y/\mbb G_m]$ is cohomologically proper. Thus, so is $[X/\mbb G_m]$.
\end{ex}

\begin{ex}[Hodge-proper schemes]
	There are examples of Hodge-proper schemes that are not proper (see \cite[Section 2.3.3]{KubrakPrikhodko_HdR}). However, we still do not know an example of such a scheme over $\mbb F_p$ (in characteristic $p$ using \cite[Section 2.3.3]{KubrakPrikhodko_HdR} one can only produce such a scheme over $\mbb F_p(t)$). We also note that a cohomologically proper quasi-compact separated scheme is automatically proper \cite[Corollary 3.10]{Hall_GAGA}.
\end{ex}

The next example shows that, in fact, a much larger class of quotient stacks, than what we considered in \Cref{sec: examples}, is Hodge-proper:
\begin{ex}[More quotients by reductive groups]\label{ex: quotients by reductive groups} Here we assume $R$ is a finitely generated algebra over $\mathbb Z$ (see \Cref{rem: more general base} for how to allow more general base if we are only interested in Hodge-properness). Let $G$ be a split reductive group over $R$ and let $X$ be a finite type $R$-scheme with a $G$-action. Assume that this action is locally linear, namely, that there is a $G$-invariant affine cover $U_i=\Spec A_i$ of $X$. Note that each $A_i$ is finitely generated over $R$. Then the categorical quotient $X/\!\!/G$ is a finite type scheme that is glued out of $U_i/\!\!/G\coloneqq \Spec A_i^G$. By \cite[Theorem 3]{FranjouVDKallen_PowerReductivity}, each $A_i^G$ is finitely generated over $R$ and, by \cite[Proposition 57]{FranjouVDKallen_PowerReductivity}\footnote{This is where the assumption on $R$ to be finite type over $\mbb Z$ is important, see also the footnote to \cite[Theorem 3.0.1]{KubrakPrikhodko_HdR}.}, each cohomology group $H^j(G,M)$ for any $G$-equivariant finitely generated $A_i$-module $M$ is a finitely generated $A_i^G$-module. It follows (see the proof of \cite[Proposition 3.1.2]{KubrakPrikhodko_HdR}) that the natural morphism $[X/G]\surj X/\!\!/G$ (induced by $[\Spec A_i/G]\surj \Spec A_i^G$) is cohomologically proper. In particular, we get the following
	\begin{prop}Let $X$ be a finite type $R$-scheme with a locally linear action of a split reductive group $G$. Then, if the categorical quotient $X/\!\!/G$ is proper, the quotient stack $[X/G]$ is cohomologically (and thus Hodge-) proper.
	\end{prop}	
	This applies for example when $X=\Spec A$ is affine and $A^G$ is a finitely-generated $R$-module. See also \cite[Examples 3.1.6 and 3.1.7]{KubrakPrikhodko_HdR}. Using \Cref{lem:cohomologically proper is flat local} and \Cref{prop:Conrad}, one can also get rid of the splitness assumption on $G$.
\end{ex}	

\begin{rem}\label{rem: more general base}
	Let $R$ be a regular $\mathbb Z$-flat Noetherian ring. In this case, by Popescu's theorem, $R$ is a filtered colimit of smooth $\mbb Z$-subalgebras $P\subset R$. Applying the spreading out \cite[Corollary 2.1.14]{KubrakPrikhodko_HdR}, we get that any finitely presentable smooth Artin stack $\mstack X$ over $R$ comes as a base change of a smooth finitely presentable Artin stack $\mstack Y$ over some $P\subset R$; moreover, if $\mstack X$ was given as $[X/G]$, spreading $X$ to some scheme $Y$ (over a possibly bigger $P$) endowed with a $G$-action, we can assume that $\mstack Y$ is of the form $[Y/G]$. If the action on $X$ was locally linear, then spreading the $G$-invariant affine cover (as in \cite[Proposition 3.1.2]{KubrakPrikhodko_HdR}), we can also assume that the action on $Y$ is locally linear. 
	Note that $P$ has finite Krull dimension and is regular, thus, its Tor-dimension is also finite and we can apply \Cref{prop: Hodge proprness survives some base-changes}: namely, $\mstack X$ is Hodge-proper over $R$, if $\mstack Y$ is Hodge-proper over $P$. Thus \Cref{ex: quotients by reductive groups} also applies when $R$ is as in this remark. In particular one can take $R=\mc O_K$.
\end{rem}

\begin{ex}[$\Theta$-stratified stacks]\label{ex: Theta-stratified}
	The notion of $\Theta$-stratification (as defined by Halpern-Leistner in \cite{HL-instability} and \cite{Halpern-Leistner_Theta}) often allows to reduce the question of Hodge-properness of a stack $\mstack X$ to the so-called $\Theta$-strata (provided $\mstack X$ has a full $\Theta$-stratification). Namely (see \cite[Proposition 3.3.2]{KubrakPrikhodko_HdR}), if a smooth Artin stack $\mstack X$ with affine diagonal is endowed with a finite $\Theta$-stratification and all $\Theta$-strata $\mstack \{X^{ss},\mstack S_\alpha\}$ are cohomologically proper, then $\mstack X$ itself is cohomologically proper. We also note that one can replace the unstable strata in this criterion by their centra $\mstack Z_\alpha\subset \mstack S_\alpha$. The particular examples then include even more general quotient stacks $[X/G]$ where the action is Kempf-Ness-complete (see \cite[Proposition 3.3.2]{KubrakPrikhodko_HdR}).
\end{ex}

\section{De Rham, crystalline, and prismatic cohomology of stacks}\label{section:p_adic_cohomology_for_stacks}
In this section we extend various crystalline-like $p$-adic cohomology theories from schemes to stacks and study the relations between them. In  \Cref{subsect:prism_for_stacks} we define the prismatic cohomology of stacks and prove its basic properties. Then, in \Cref{sec: Hodge and de Rham cohomology} we recall the constructions of the Hodge and de Rham cohomology and prove the de Rham comparison for the prismatic cohomology of stacks. In \Cref{subsect:crystalline_for_stacks} we also give a definition of the crystalline cohomology, show that this definition agrees with Olsson's construction in the case of a $1$-Artin stack, and then prove the crystalline comparison.

\subsection{Prismatic cohomology of schemes: reminder} \label{sect:reminder on prismatic stuff}
In this section we briefly remind the notions of prismatic site and prismatic cohomology for formal schemes following the work of Bhatt and Scholze \cite{BS_prisms}.
\begin{defn}
	A \emdef{$\delta$-ring} is a $p$-typical $\lambda$-ring, i.e. it is a $\mathbb Z_{(p)}$-algebra equipped with a map of sets $\delta\colon A \to A$ such that $\delta(0)=\delta(1)=0$ and 
	$$\delta(x+y) = \delta(x) + \delta(y) + \frac{x^p + y^p - (x+y)^p}{p} \quad\text{and}\quad \delta(xy) = x^p\delta(x) + y^p\delta(y) + p\delta(x)\delta(y)$$
	for all $x,y\in A$. Every $\delta$-ring $A$ admits a natural ring endomorphism
	$$\phi(x) := x^p + p\delta(x).$$
	lifting the absolute Frobenius on $A/(p)$. Note that if $A$ is $p$-torsion free, a $\delta$-structure on $A$ is just the same as a choice of such lift $\phi$. 
\end{defn}
\begin{defn}[{\cite[Definition 3.2]{BS_prisms}}] 
	A \emdef{prism} is a pair $(A, I)$, where $A$ is a $\delta$-ring and $I\subset A$ is an ideal defining a Cartier divisor in $\Spec A$, that satisfies the following two conditions
	\begin{itemize}
		\item $A$ is derived $(p, I)$-adically complete\footnote{In \cite{BS_prisms} a slightly less restrictive condition $(p,I)\in \rad(A)$ is used; however all prisms that appear in our work are in fact $(p,I)$-complete.}.
		
		\item  $p \in I + \phi(I)A$; geometrically this means that the divisors $V(I)$ and $V(\phi(I))$ intersect only in characteristic $p$.
	\end{itemize}
\end{defn}
In particular $I$ defines a line bundle on $\Spec A$; it follows that both $I$ and $A/I$ are perfect $A$-modules.
\begin{defn}
	\noindent A prism $(A,I)$ is called \begin{itemize}
		\item[--] \emdef{bounded},  if $A/I$ has bounded $p^\infty$-torsion, i.e. $A[p^\infty]=A[p^n]$ for some $n$. \item[--]\emdef{perfect},  if $\phi$ is an automorphism.
		\item[--]\emdef{orientable}, if $I$ is principal.
	\end{itemize}
\end{defn}
\begin{defn}\label{defn: distinguished}
	An element $d$ in a $\delta$-ring $A$ is called \emdef{distinguished} if $\delta(d)$ is invertible.
\end{defn}

If a $\delta$-ring $A$ is derived $(p,d)$-complete, then $(A,(d))$ is a prism if and only if $d$ is distinguished (\cite[Lemma 2.25]{BS_prisms}). Reversely, for any bounded prism $(A,I)$ there is a (derived) $(p,I)$-completely faithfully flat map $A\ra A'$ such that  $I\cdot A'=(d)$ for a distinguished element $d\in A'$ (\cite[Lemma 3.7(4)]{BS_prisms}). Note that if $(A,(d))$ is a prism, then by definition $d$ is a nonzerodivisor. 
We refer the reader to \cite[Section 3]{BS_prisms} for a more thorough discussion of prisms. We will be mostly interested in the following types of prisms:
\begin{ex}\label{ex: examples of prisms}\begin{enumerate}
		\item (\textit{crystalline}). Let $A$ be a $p$-torsion free, $p$-complete ring and let $I=(p)$; $\delta$-structure on $A$ is then the same as a lift of Frobenius $\phi$. Also, $p$ is always distinguished since $\delta(p)=1-p^{p-1}\in \mbb Z_p^\times$. Particular examples of such prisms include $(W(k),(p))$ and $(A_\crys, (p))$. 
		\item (\textit{Breuil-Kisin}). Let $K/\mbb Q_p$ be a finite extension and let $k$ be the residue field. Choose a uniformizer $\pi\in \mc O_K$, and let $I$ be the kernel of the natural surjection $\mf S\coloneqq W(k)[[u]]\surj \mc O_K$  which sends $u$ to $\pi$. The pair $(\mf S,I)$ is a prism, where $\mf S$ is regarded as a $\delta$-ring via the Frobenius $\phi_{\mathfrak S}\colon \mf S \ra \mf S$ which lifts the natural Frobenius on $W(k)$ and sends $u$ to $u^p$. We denote by $E(u)$ the unique generator of $I$ which is a monic polynomial. Note that since $\pi$ is a uniformizer the polynomial $E(u)$ is Eisenstein. 

		\item (The Fontaine's ring $\Ainf$). For more details on the ring $\Ainf$ see \Cref{sec: Ainf}. The pair $(\Ainf,\ker (\theta))$ defines a prism; the $\delta$-structure on $\Ainf$ is induced by the Frobenius $\phi$. The element $\xi$ (see \ref{rem: xi}) is a generator for $\ker (\theta)$.
		Also, given a choice of $\pi^\flat=(\pi,\pi^{\frac{1}{p}},\pi^{\frac{1}{p^2}},\ldots)\in \mc O_{\mbb C_p}^\flat$ of a compatible system of $p^n$-th roots of $\pi$ for a uniformizer $\pi\in \mc O_K$ we obtain a natural map of prisms $(\mf S,(E(u)))\ra (\Ainf, \ker (\theta))$ sending $u$ to $[\pi^\flat]$.

	\end{enumerate}
	\smallskip We note that all of these prisms are bounded and orientable.
\end{ex}
\begin{defn}[{\cite[Definition  4.1]{BS_prisms}}]
	Fix a bounded prism $(A, I)$. The \emdef{prismatic site $(\mathfrak X/A)_\Prism$ of a smooth $p$-adic formal scheme $\mathfrak X$ over $A/I$} is the category consisting of maps of bounded prisms $g\colon (A, I) \to (B, IB)$ and a map of $p$-adic formal schemes $f\colon \Spf(B/IB) \to \mathfrak X$ making the diagram below commutative
	$$\xymatrix{
		\Spf(B/IB) \ar@/^1pc/[rr] \ar[r]_-f & \mathfrak X \ar[d] & \Spf(B)\ar[d]^g\\
		& \Spf(A/I) \ar[r] & \Spf(A).
	}$$
	Morphisms in $(\mathfrak X/A)_\Prism$ are the natural ones. We endow $(\mathfrak X/A)_\Prism$ with the faithfully flat topology, where a morphism
	$$\big(\mathfrak X \leftarrow \Spf(C/IC) \rightarrow \Spf(C)\big) \xymatrix{\ar[r] &} \big(\mathfrak X \leftarrow \Spf(B/IB) \rightarrow \Spf(B)\big)$$
	in $(\Spf X/A)_\Prism$ is called (faithfully) flat if $C$ is $(p, IB)$-completely (faithfully) flat over $B$.
\end{defn}
\begin{defn}[Prismatic cohomology]\label{defn: prismatic cohomology}
	Fix a prism $(A, I)$ and let $\mathfrak X$ be a smooth $p$-adic formal $A/I$-scheme. \emdef{The prismatic cohomology $R\Gamma_\Prism(\mathfrak X / A)$ of $\mathfrak X$ over $A$} is defined as the derived global sections of the sheaf 
	$$\big(\mathfrak X \leftarrow \Spf(B/IB) \rightarrow \Spf(B)\big) \mapsto B,$$ i.e., as the homotopy limit
	$$R\Gamma_\Prism(\mathfrak X / A) \coloneqq \lim_{\big(\mathfrak X \leftarrow \Spf(B/IB) \rightarrow \Spf(B)\big) \in (\mathfrak X/A)_\Prism} B.$$
	Note that $R\Gamma_\Prism(\mathfrak X / A)$ is naturally an $E_\infty$-algebra over $A$ and since all terms in the diagram above are derived $(p, I)$-complete, the same holds for $R\Gamma_\Prism(\mf X / A)$.
\end{defn}

Note that $R\Gamma_\Prism(\mathfrak X / A)$ has a natural $\phi_A$-linear Frobenius endomorphism $$R\Gamma_\Prism(\mathfrak X / A) \xymatrix{\ar[r]&} \phi_{A*}R\Gamma_\Prism(\mathfrak X / A),$$ coming from the (automatically $\phi_A$-linear) maps $\phi_B\colon B\ra \phi_{A*}B$ for any object of $(\mathfrak X/A)_\Prism$ by taking the limit. We can also consider its linearization 
$$
\phi_{\Prism}\colon \phi_A^* R\Gamma_\Prism(\mathfrak X / A)\xymatrix{\ar[r]&} R\Gamma_\Prism(\mathfrak X / A).
$$

Finally, we also mention that by \cite[Corollary 4.11]{BS_prisms}, for any map $(A/I)\ra (B/J)$ of bounded prisms one has the base change
$$
\RG_{\Prism/A}(\mf X/A)\widehat{\otimes}_A B\simeq \RG_{\Prism/B}(\mf X_{B/J}/B),
$$
where the tensor product is $(p,I)$-completed.

With these notations we have:
\begin{thm}[{\cite[Theorem 11.2]{BMS2}}]\label{main_BMS2}
	Let $\mathfrak X$ be a smooth formal scheme over $\mathcal O_K$. Consider the Breuil-Kisin prism $(\mf S, (E(u)))$, corresponding to some choice of the uniformizer $\pi\in \mc  O_K$. Also recall the $\Ainf$-prism $(\Ainf, \ker (\theta))$. Then:
	\begin{enumerate}[label=(\arabic*)]
		\item If $\mf X$ is proper, $\RG_\Prism(\mf X/\mf S)$ is a perfect complex of $\mf S$-modules.
		
		\item The Frobenius $\phi$ on the prismatic cohomology
		$$\textstyle \phi_{\mathfrak S}^*R\Gamma_\Prism(\mathfrak X/\mathfrak S)[\frac{1}{E}] \xymatrix{\ar[r]^\phi_\sim &} R\Gamma_\Prism(\mathfrak X/\mathfrak S)[\frac{1}{E}]$$
		becomes an equivalence after localizing at $E(u)$.
		
		\item(de Rham comparison) There is a canonical equivalence
		$$\phi_{\mathfrak S}^*R\Gamma_\Prism(\mathfrak X/\mathfrak S)\otimes_{\mathfrak S} \mathcal O_K \simeq R\Gamma_\dR(\mathfrak X / \mathcal O_K).$$
		\item(Hodge-Tate comparison) For affine $\mf X$ there is a canonical identification $H^i(\RG_{\Prism}(\mf X/\mf S)\otimes_{\mf S}\mc O_K)\simeq \Omega^i_{\mf X/\mc O_K}\{-i\}$, where $M\{-i\}\coloneqq M\otimes_{\mf S} ((I/I^2)^\vee)^{\otimes i}$ is the corresponding (Breuil-Kisin) twist.
		\item(Crystalline comparison) There is a canonical $\phi$-equivariant equivalence
		$$\phi_{\mathfrak S}^*R\Gamma_\Prism(\mathfrak X/\mathfrak S)\otimes_{\mathfrak S} W(k) \simeq R\Gamma_{\crys}(\mathfrak X_k / W(k)).$$
		
		\item(\'Etale comparison) There are canonical equivalences
		$$R\Gamma_{\et}({\mathfrak X}_{\mbb C_p}, \mathbb Z/p^n) \simeq \left(R\Gamma_\Prism(\mathfrak X_{\mathcal O_{\mbb C_p}} / A_\iinf)/p^n[\tfrac{1}{\xi}]\right)^{\phi = 1} \quad \text{and}\quad R\Gamma_{\et}({\mathfrak X}_{\mbb C_p}, \mathbb Z_p) \simeq \left(R\Gamma_\Prism(\mathfrak X/ \mathfrak S)\widehat{\otimes}_{\mathfrak S} W(\mbb C_p^\flat) \right)^{\phi = 1},$$
		where ${\mathfrak X}_{\mbb C_p}$ denotes the Raynaud's generic fiber of ${\mathfrak X}$. If $\mf X$ is also proper, for any $i\ge 0$ one has a $(\phi,G_{K_\infty})$-equivariant isomorphism
		$$
		H^i_\Prism(\mathfrak X/ \mathfrak S){\otimes}_{\mathfrak S} W(\mbb C_p^\flat)\simeq H^i_{\et}({\mathfrak X}_{\mbb C_p}, \mathbb Z_p).
		$$
	\end{enumerate}
\end{thm}
\begin{rem}
	Note that if $\mf X=\widehat X$ is a completion of a smooth proper scheme $X$ over $\mc O_K$, then by \cite[Theorems 3.8.1 and 4.4.1]{Huber_adicSpaces} one also has an isomorphism 
	$$
	R\Gamma_{\et}({\mathfrak X}_{\mbb C_p}, \mathbb Z_p)\simeq R\Gamma_{\et}({X}_{\mbb C_p}, \mathbb Z_p)\simeq\footnote{Due to the invariance of the \'etale cohomology under a base change of the ground algebraically closed field of char 0.} R\Gamma_{\et}({X}_{\ol{\mbb Q}_p}, \mathbb Z_p).
	$$
	Thus, \Cref{main_BMS2} can indeed be applied in the classical algebraic setting.
\end{rem}

%
%

\paragraph{Derived prismatic cohomology}
Here we also review a slightly more general notion of derived prismatic cohomology from \cite[Section 7.2]{BS_prisms}. Let $(A,I)$ be a base bounded prism. Recall that in \Cref{defn: prismatic cohomology} for a smooth $p$-adic formal scheme $\mf X$ over $\Spf A/I$ we defined its prismatic cohomology $\RG_{\Prism}(\mf X/A)$.

\begin{defn}
For a smooth affine scheme $S=\Spec R$ over $A/I$ define its prismatic cohomology $\Prism_{R/A}\in \DMod A$ as $\RG_{\Prism}(\Spf \widehat R/A)$, where $\widehat R$ is the $p$-adic completion of $R$.
\end{defn}
Recall that $\Prism_{R/A}$ is $(p,I)$-adically complete and it has a natural Frobenius $
\phi_A^* R\Gamma_\Prism(\mathfrak X / A) \to R\Gamma_\Prism(\mathfrak X / A)
$, which corresponds by adjunction to a map $R\Gamma_\Prism(\mathfrak X / A) \to \phi_{A*}R\Gamma_\Prism(\mathfrak X / A)$. These maps are functorial, i.e. they give rise to a functor
$$
\Prism_{-/A}\colon \Aff^{\sm,\op}_{A/I} \xymatrix{\ar[r]&} \left({\DMod{A}}_{\widehat{(p,I)}}\right)^{\phi_{A*}}_{\mr{lax}},
$$
where ${\DMod{A}}_{\widehat{(p,I)}}\subset {\DMod{A}}$ denotes the full subcategory spanned by derived $(p,I)$-complete objects and $({\DMod{A}}_{\widehat{(p,I)}})^{\phi_{A*}}_{\mr{lax}}$ is the lax equalizer of $\Id_{{\DMod{A}}_{\widehat{(p,I)}}}$ and $\phi_{A*}$, which is defined as the limit of the following diagram
$$\xymatrix{
 && \Fun(\Delta^1, {\DMod{A}}_{\widehat{(p,I)}}) \ar[d] \\
{\DMod{A}}_{\widehat{(p,I)}} \ar[rr]^-{(\Id, \phi_{A*})} && {\DMod{A}}_{\widehat{(p,I)}} \times {\DMod{A}}_{\widehat{(p,I)}}.
}$$
An object of $({\DMod{A}}_{\widehat{(p,I)}})^{\phi_{A*}}_{\mr{lax}}$ is given by a complex $M$ with a morphism $M\ra \phi_{A*}M$, a 1-morphism is given by a commutative square between those and so on.

One can extend the functor $\Prism_{-/A}$ to all affine schemes via left Kan extension along the embedding $\Aff^{\sm,\op}_{A/I}\hookrightarrow \Aff^{\op}_{A/I}$ (see \cite[Construction 7.6]{BS_prisms}). More concretely, the value $\Prism_{Q/A}$ on an arbitrary $A/I$-algebra $Q$ is given by the colimit 
$$
\Prism_{Q/A}\simeq \underset{P\text{ smooth}}{\colim_{P\ra Q}} \Prism_{P/A}
$$
taken in the category $({\DMod{A}}_{\widehat{(p,I)}})^{\phi_{A*}}_{\mr{lax}}$. Unwinding the definitions, one finds that the underlying object $\Prism_{Q/A}\in {\DMod{A}}$ is given by the derived  $(p,I)$-adic completion of the same colimit taken in $\DMod A$. By construction one has a natural map $\Prism_{Q/A}\ra \phi_{A*}\Prism_{Q/A}$, functorial in $Q$. The resulting functor is called the \emdef{derived prismatic cohomology}.

Given a morphism of bounded prisms $(A,I)\ra (B,J)$, we have a base change equivalence
$$
\Prism_{Q/A}\widehat{\otimes}_A B \xymatrix{\ar[r]^-\sim &} \Prism_{Q_{B/J}/B},
$$
where $Q_{B/J}$ is the derived $p$-completion of the derived tensor product $Q\otimes_{A/I} B/J$ and the tensor product above is derived $(p, I)$-complete.

Also, from the Hodge-Tate comparison it follows that the functor of derived prismatic cohomology still satisfies \'etale hyperdescent:
\begin{lem}\label{descent_prism_for_schemes}
Let $(A,I)$ be a bounded prism. Then the presheaf
$$\Prism_{-/ A} \colon \Aff_{A/I}^{\op} \xymatrix{\ar[r] &} \DMod{A}$$
is an \'etale hypersheaf.

\begin{proof}
Let $Q^{-1} \to Q^\bullet$ be an \'etale hypercover in $\Aff_{A/I}^{\op}$. We want to prove that the natural map
$$\alpha\colon \Prism_{Q^{-1}/A} \xymatrix{\ar[r] &} \Tot \Prism_{Q^\bullet/A}$$
is an equivalence. Since both sides are derived $I$-adically complete and since the $A$-module $A/I$ is perfect (hence the tensor product functor $-\otimes_A A/I$ preserves limits), by derived Nakayama's lemma it is enough to prove that $\alpha\otimes_{A} A/I$ is an equivalence. Recall that by the Hodge-Tate comparison the complex $\Prism_{Q^i/A}\otimes_A A/I$ admits an increasing filtration with associated graded pieces being the $p$-adic completion of $\wedge^j \mathbb L_{Q^i/(A/I)}\otimes_{A/I} ((I/I^2)^\vee)^{\otimes j}[-j]$. But all maps $Q^{-1} \to Q^i$ are \'etale, so the base-change map
$$Q^i\widehat{\otimes}_{Q^{-1}} (\Prism_{Q^{-1}/A}\otimes_A A/I) \xymatrix{\ar[r] &} \Prism_{Q^i/A}\otimes_A A/I$$
is an equivalence. Hence $\alpha\otimes_A A/I$ is an equivalence by flat descent.
\end{proof}
\end{lem}

\subsection{Prismatic cohomology of stacks}\label{subsect:prism_for_stacks}
We fix a base bounded prism $(A,I)$. Let $\mstack X$ be a prestack over $A/I$.  
\begin{defn}
We define:
\begin{itemize}[wide]
\setlength\itemindent{\labelwidth+\labelsep}
\item \emdef{Prismatic cohomology}
$R\Gamma_\Prism(\mstack X/A)$ of $\mstack X$ as the value of the right Kan extension of the (derived) prismatic cohomology functor
$$\Prism_{-/A} \colon \Aff_{A/I}^\op \xymatrix{\ar[r] & } \left({\DMod{A}}_{\widehat{(p,I)}}\right)^{\phi_{A*}}_{\mr{lax}}$$
along the Yoneda embedding $\Aff^{\op}_{A/I} \inj \PStk_{A/I}^\op$. By the functoriality of Kan extension, for a morphism $f\colon \mstack X \ra \mstack Y$ of prestacks we get a pullback map $f^{-1}\colon R\Gamma_\Prism(\mstack Y/A)\ra R\Gamma_\Prism(\mstack X/A)$ on the prismatic cohomology.

\item \emdef{Twisted prismatic cohomology} $R\Gamma_{\Prism\fr }(\mstack X/A)$ as the right Kan extension of the twist 
$$
\Prism_{-/A}{\fr}\colon \Aff_{A/I}^\op \xymatrix{\ar[r] & } {\DMod{A}}_{\widehat{(p,I)}},
$$
where $\Prism_{-/A}{\fr}\coloneqq \Prism_{-/A}\widehat{\otimes}_{A,\phi_A}A$.
\end{itemize}
\end{defn}

Unwinding the definitions one finds that the underlying object $R\Gamma_\Prism(\mstack X/A)\in \DMod{A}$ is computed as the value of the right Kan extension of the corresponding underlying functor $\Prism_{-/A} \colon \Aff_{A/I}^\op \ra \DMod{A}$. Note that the $(p,I)$-adic completion is not needed here since the right Kan extension is automatically complete. We have a natural map $R\Gamma_\Prism(\mstack X/A)\ra \phi_{A*}R\Gamma_\Prism(\mstack X/A)$. By adjunction we also obtain a map 
$$
R\Gamma_\Prism(\mstack X/A)\otimes_{A,\phi_A} A \xymatrix{\ar[r] &} R\Gamma_\Prism(\mstack X/A)
$$
which factors through the completed tensor product $R\Gamma_\Prism(\mstack X/A)\fr\coloneqq R\Gamma_\Prism(\mstack X/A)\widehat{\otimes}_{A,\phi_A} A$ thus producing a map
$$
\phi_{\Prism}\colon R\Gamma_\Prism(\mstack X/A)\fr \xymatrix{\ar[r] &} R\Gamma_\Prism(\mstack X/A).
$$

Note that the transformation $\Prism_{-/A}{\fr}\ra \Prism_{-/A}$ induced by Frobenius $\Prism_{-/A}\ra \phi_{A*} \Prism_{-/A}$ also produces a map $R\Gamma_{\Prism\fr}(\mstack X/A) \ra R\Gamma_{\Prism}(\mstack X/A)$ which fits into the commutative diagram
$$
\xymatrix{R\Gamma_{\Prism\fr}(\mstack X/A) \ar[rr]^{\phi_{\Prism\fr}}\ar@{-->}[rd]&& R\Gamma_{\Prism}(\mstack X/A).\\
	&R\Gamma_{\Prism}(\mstack X/A)\fr\ar[ru]^{\phi_{\Prism}}&}
$$

\begin{rem} In \Cref{prop: twisted_cohomology_is_a_twist} we will see that if $\mstack X$ is smooth, under certain finiteness assumption on $A$, the dotted arrow is an equivalence. Given a general prism, the twisted version ${R\Gamma_{\Prism\fr}(\mstack X/A)}$ is better suited for establishing the comparison theorems. The subtlety here is that pullbacks (in particular, Frobenius pullback) do not necessarily commute with limits.
\end{rem}

\begin{rem}
	For completeness of the exposition we mention that in an analogous way we can also define the prismatic cohomology of derived prestacks (as well as its twisted version), by, first, left Kan extending $\Prism_{-/A}$ from $\Aff_{A/I}^\op$ to all derived affine schemes (i.e. the dual category to the category of simplicial $A/I$-algebras) and, then, right Kan extending to all derived prestacks afterwards. 
\end{rem}

\begin{rem}\label{rem:on_Gamma_prism}
	From the descent properties for $\Prism_{-/A}$ (see \Cref{descent_prism_for_schemes}) it follows that:
	\begin{itemize}
		\item For any prestack $\mstack X$ over $A/I$ we have $\RG_\Prism(\mstack X/A)\simeq \RG_\Prism(L_{\et}\mstack X/A)$, where $L_{\et}\colon \PStk_{A/I} \ra \Stk_{A/I}^{\et}$ is the \'etale hypersheafification functor.
		
		\item In particular, the functor 
		$$R\Gamma_\Prism(- /A)\colon \PStk_{A/I}^\op \xymatrix{\ar[r] & } \DMod{A}$$
		of prismatic cohomology satisfies \'etale hyperdescent. In other words if $\mstack U_\bullet \ra \mstack X$ is an \'etale hypercover of stacks, the natural map 
		$$
		R\Gamma_\Prism(\mstack X/A) \xymatrix{\ar[r]^-\sim &} \Tot R\Gamma_\Prism(\mstack U_\bullet/A)
		$$ 
		is an equivalence. As usual this applies to the \v Cech object associated to any smooth cover $\mstack U \surj \mstack X$ as a particular case.
		
		\item Similar assertions hold for the twisted version $R\Gamma_{\Prism\fr}(- /A)$.
	\end{itemize}
\end{rem}

\begin{lem}\label{lem: enough to take smooth guys}
	Let $\mstack X$ be a smooth Artin stack over $A/I$. Then 
	$$
	\RG_{\Prism}(\mstack X/A)\simeq \underset{P \text{ smooth over }A/I}{\lim_{(\Spec P\to \mstack X) }} \Prism_{P/A} \qquad \mbox{and} \qquad \RG_{\Prism\fr}(\mstack X/A)\simeq \underset{P \text{ smooth over }A/I}{\lim_{(\Spec P\to \mstack X) }} \Prism_{P/A}\fr,
	$$
	where the limit is taken only over the smooth maps $\Spec P\ra \mstack X$. 
	\begin{proof}
		The proof is standard and we will only give a sketch. Let $\Aff_{\sm}^\op\subset \Aff^\op$, $(\Aff/\mstack X)_{\sm}^\op\subset (\Aff/\mstack X)^\op$, $\Stk_{\sm}^{k-\Art,\op}\subset \Stk^{k-\Art,\op}$ and  $(\Stk^{k-\Art}/\mstack X)_{\sm}^\op\subset (\Stk^{k-\Art}/\mstack X)^\op$ be the full subcategories of the corresponding categories (over $A/I$) spanned by smooth objects. Then we can reformulate the statement as follows: for a smooth $k$-Artin stack $\mstack X$ the complex $\RG_\Prism(\mstack X/A)$ is the value of the right Kan extension of $\RG_\Prism(-/A)$ along $\Aff_{\sm}^\op \ra \Stk^{k-\Art,\op}_\sm$. This can be shown by induction on $k$. Indeed the case $k=0$ is obvious since $(\Aff/\mstack X)^\op$ has an initial object given by $\mstack X\xra{\id_{\mstack X}} \mstack X\in (\Aff/\mstack X)_{\sm}^\op $. For the induction step, let $\mstack X$ be smooth $k$-Artin. By the transitivity of right Kan extensions and the induction assumption we have
		$$
		\Ran_{\Aff_{\sm}^\op \ra \Stk^{k-\Art,\op}_\sm} \RG_\Prism(-/A) \simeq \Ran_{\Stk^{(k-1)-\Art,\op}_\sm \ra \Stk^{k-\Art,\op}_\sm} \RG_\Prism (-/A).
		$$
		For a smooth atlas $U\ra \mstack X$ with the \v Cech object $\mstack U_\bullet$ by descent we have an equivalence 
		$$
		R\Gamma_\Prism(\mstack X/A) \xymatrix{\ar[r]^-\sim &} \Tot R\Gamma_\Prism(\mstack U_\bullet/A),
		$$
		where each $\mstack U_i$ is smooth $(k-1)$-Artin, so by induction $\RG_\Prism(\mstack U_i/A)$ is given by the value of the right Kan extension $\Ran_{\Stk^{(k-1)-\Art,\op}_\sm \ra \Stk^{k-\Art,\op}_\sm} \RG_\Prism (-/A)$. Thus we get the statement, since $\Ran_{\Stk^{(k-1)-\Art,\op}_\sm \ra \Stk^{k-\Art,\op}_\sm} \RG_\Prism (-/A)$ also satisfies smooth descent.
	\end{proof}
\end{lem}

In particular we have a corollary:
\begin{cor}\label{cor_prism_coconnective_for_smooth}
If $\mstack X$ is a smooth Artin stack then  $\RG_{\Prism}(\mstack X/A)\in \DMod{A}^{\ge 0}$.
\end{cor}

We will now compare the two versions of twisted prismatic cohomology under a certain finiteness condition on the base prism:
\begin{prop}\label{prop: twisted_cohomology_is_a_twist}
	Let $\mstack X$ be a smooth qcqs Artin stack over $A/I$ and assume that the module $\phi_{A*}A$ has a finite $(p,I)$-complete Tor-amplitude. Then the natural map
	$$
	\xymatrix{\RG_{\Prism\fr}(\mstack X/A) \ar[r]& \RG_{\Prism}(\mstack X/A)\fr} 
	$$
	is an equivalence. 
	
	\begin{proof}
		Since $\phi_A\colon A\ra A$ has a finite $(p,I)$-complete Tor-amplitude, the functor $M\mapsto M\fr[-i]= M\widehat \otimes_{A,\phi_A} A[-i]$ is left $t$-exact for some $i\ge 0$. Thus, by  \Cref{left_exact_preserve_totalizations_of_uniformly_bounded_below} we get that $-\fr[-i]$, and thus also $-\fr$, commutes with totalizations of 0-coconnected objects. 
		
		In the case $\mstack X=\Spec P$ is a smooth affine scheme the statement holds trivially, since both sides are equal to $\Prism_{P/A}\fr$. Given an-Artin stack $\mstack X$, by \cite[Theorem 4.7]{Pridham_ArtinHypercovers}, there exists a hypercover $|U_\bullet|\xra{\sim} \mstack X$ with $U_i$ being smooth affine schemes and by descent we get
		$$
		R\Gamma_{\Prism\fr}(\mstack X/A)\xra{\sim} \Tot R\Gamma_{\Prism\fr}(U_\bullet/A)\simeq \Tot R\Gamma_{\Prism}(U_\bullet/A)\fr\xleftarrow{\sim} \RG_{\Prism}(\mstack X/A)\fr.
		$$
	\end{proof}
\end{prop}

\begin{ex}
	Though the condition on the prism in \Cref{prop: twisted_cohomology_is_a_twist} may seem not very restrictive, it is not satisfied in some rather natural examples. In particular, for $(A_\crys,(p))$ the $p$-complete Tor-amplitude of $\phi_{A_\crys *}A_\crys$ is infinite.
\end{ex}
\begin{rem}\label{rem: drop qcqs assumption in the perfect case}
	Given a (possibly infinite) disjoint union of smooth affine schemes $X=\sqcup_{i\in I} X_i$ one has
	$$\RG_{\Prism}(X/A)\simeq \prod_{i\in I} \RG_{\Prism}(X_i/A) \qquad\text{and}\qquad \RG_{\Prism^{(1)}}(X/A)\simeq \prod_{i\in I} \RG_{\Prism^{(1)}}(X_i/A)\simeq  \prod_{i\in I} \RG_{\Prism}(X_i/A)^{(1)}.$$ 
	Assume that $\phi_{A*}A$ is perfect as an $A$-module. Then, since the functor $-\otimes_A\phi_{A*} A$ preserves limits, we have $\RG_{\Prism^{(1)}}(X/A)^{(1)}\simeq \RG_{\Prism}(X/A)^{(1)}$. Considering the hypercover by unions of smooth affine schemes provided by \cite[Theorem 4.7]{Pridham_ArtinHypercovers} for any Artin stack $\mstack X$ and arguing as before we see that one can drop the qcqs assumption on $\mstack X$ in \Cref{prop: twisted_cohomology_is_a_twist} if $\phi_{A*}A$ is a perfect $A$-module.
\end{rem}

In the case of schemes an important property of the Frobenius $\phi_{\Prism}$ is that it becomes an equivalence after we invert $I$. In the case of stacks this holds as well, at least if we consider $\phi_{\Prism\fr}\colon R\Gamma_{\Prism\fr}(\mstack X/A)\ra R\Gamma_\Prism(\mstack X/A)$.
\begin{prop}\label{prop: Frobenius is an iso after inverting d}
	Let $(A,I)$ be a bounded prism. Let $\mstack X$ be a smooth qcqs Artin stack over $A/I$. Then the induced map
	$$\xymatrix{\phi_{\Prism\fr}[I^{-1}] \colon R\Gamma_{\Prism\fr}(\mstack X/A)[I^{-1}] \ar[r] & R\Gamma_\Prism(\mstack X / A)[I^{-1}]}$$
	is an equivalence.
	
	\begin{proof}
		Let $\mstack X$ be a smooth affine scheme. Then $\phi_{\Prism\fr}[I^{-1}]$ is an equivalence by \cite[Theorem 1.8 (6)]{BS_prisms}. For a general smooth qcqs Artin stack we use a hypercover $U_\bullet\ra \mstack X$ provided by \cite[Theorem 4.7]{Pridham_ArtinHypercovers}. Since by descent the natural maps
		$$R\Gamma_\Prism(\mstack X / A) \xymatrix{\ar[r]^\sim & } \Tot R\Gamma_\Prism(U_\bullet / A) \quad \mbox{and} \quad R\Gamma_{\Prism\fr}(\mstack X / A) \xymatrix{\ar[r]^\sim & } \Tot R\Gamma_{\Prism\fr}(U_\bullet / A)$$
		are equivalences, it is enough to prove that the localization $-[I^{-1}]$ commutes with totalizations of coconnected complexes. This follows from \Cref{left_exact_preserve_totalizations_of_uniformly_bounded_below}, since $A \to  A[I^{-1}]$ is flat.
	\end{proof}
\end{prop}

Under the same Tor-finiteness assumptions on $A$ as in \Cref{prop: twisted_cohomology_is_a_twist}, this also applies to $\phi_{\Prism}$:
\begin{cor}\label{cor: Frobenius is a twist if tor-dimension is finite}
	In the context of \Cref{prop: Frobenius is an iso after inverting d} assume that $\phi_{A*}A$ has a finite $(p,I)$-complete Tor-amplitude. Then the map
	$$
	\xymatrix{\phi_{\Prism}[I^{-1}] \colon R\Gamma_{\Prism}(\mstack X/A)\fr[I^{-1}] \ar[r] & R\Gamma_\Prism(\mstack X / A)[I^{-1}]}
	$$
	is an equivalence.
	
	\begin{proof}
		Follows from Propositions \ref{prop: Frobenius is an iso after inverting d} and \ref{prop: twisted_cohomology_is_a_twist}.
	\end{proof}
\end{cor}

\begin{ex}\label{ex: Frobeinus is iso Breuil-Kisin}
	Let's consider the Breuil-Kisin prism $(\mf S, E(u))$. Note that in this case $\mf S=W(k)[[u]]$ is a free $\mf S$-module via $\phi_{\mf S}\colon \mf S\ra\mf S$, $u\mapsto u^p$ with a basis given by $1,u,\ldots,u^{p-1}$. Thus, $\phi_{\mf S*}\mf S$ is flat and \Cref{prop: twisted_cohomology_is_a_twist} applies to $(\mf S, E(u))$. Moreover, since $\phi_{\mf S*}\mf S$ is free of finite rank, the completed tensor product $-\widehat{\otimes}_{\mf S, \phi_{\mf S}}\mf S$ coincides with the usual one $-{\otimes}_{\mf S, \phi_{\mf S}}\mf S$. Thus in this case for any smooth Artin stack $\mstack X$ over $\mc O_K$ the equivalence in \Cref{cor: Frobenius is a twist if tor-dimension is finite} restricts to analogous isomorphism on the individual cohomology
	$$
	\xymatrix{\phi_{\Prism}[\frac{1}{E}]\colon H^i_\Prism(\mstack X/\mf S)\fr[\frac{1}{E}] \ar[r]^(0.58){\sim} & H^i_\Prism(\mstack X/\mf S)[\frac{1}{E}]}.
	$$

	The situation is even better if we consider a perfect prism $(A,d)$, since $\phi_{A*}A$ is not just free of finite rank, but isomorphic to $A$. In this case we also get that the Frobenius induces isomorphisms
	$$
	\xymatrix{\phi_{\Prism}[\frac{1}{d}]\colon H^i_\Prism(\mstack X/A)\fr[\frac{1}{d}] \ar[r]^(0.58){\sim} & H^i_\Prism(\mstack X/A)[\frac{1}{d}]}.
	$$
\end{ex}	

Finally we show that (under a similar $(p,I)$-complete Tor-finiteness assumption) one has base change with respect to the morphisms of prisms.
\begin{prop}\label{prismatic_basechange}
	Let $\mstack X$ be a smooth qcqs Artin stack over $A/I$ and let $(A,I)\ra (B,IB)$ be a morphism of bounded prisms of finite $(p,I)$-complete Tor-amplitude. Then the natural map
	$$
	\RG_{\Prism}(\mstack X/A)\widehat{\otimes}_A B\xymatrix{\ar[r]^\sim &}\RG_{\Prism}(\mstack X_{B/IB}/B) 
	$$
	is an equivalence; here the tensor product is $(p,IB)$-completed.

\begin{proof}
	The proof is analogous to \Cref{prop: twisted_cohomology_is_a_twist}: since $A\ra B$ is of finite $(p,I)$-complete Tor-amplitude, the functor $-\widehat{\otimes}_A B[-i]$ is left $t$-exact for some $i$ and thus $-\widehat{\otimes}_A B$ commutes with totalizations. Taking a hypercover by smooth representable affine schemes, we reduce to this case, which is covered by \cite[Corollary 4.11]{BS_prisms}. 
\end{proof}
\end{prop}
\begin{rem}\label{rem: base change for twisted prismatic cohomology}
	Similarly, under the same assumptions on $\mstack X$ and $(A,I)\ra (B,IB)$ one has base change for the twisted prismatic cohomology:
	$$
	\RG_{\Prism^{(1)}}(\mstack X_{B/IB}/B)\simeq \RG_{\Prism^{(1)}}(\mstack X/A)\widehat{\otimes}_A B.
	$$
\end{rem}

\begin{rem}\label{rem: prismatic base change in the perfect case}
	Note that if $B$ is perfect as an $A$-module the completed tensor product in \Cref{prismatic_basechange} coincides with the usual one. Also, using the commutation of $\otimes_A B$ with limits as in \Cref{rem: drop qcqs assumption in the perfect case} we see that if $B$ is perfect over $A$ one can drop the qcqs assumption on $\mstack X$ in \Cref{prismatic_basechange}. 
\end{rem}

In what follows we will also need the K\"unneth formula for prismatic cohomology:
\begin{prop}[K\"unneth formula]\label{Kunneth_for_Prism_stacks}
Let $\mstack X$ and $\mstack Y$ be a pair of smooth quasi-compact quasi-separated smooth Artin stacks over $A/I$. Assume that the $(p, I)$-completed tensor product functor $\widehat\otimes_A$ is of finite cohomological amplitude. Then the natural map
$$R\Gamma_\Prism(\mstack X/A)\cotimes_A R\Gamma_\Prism(\mstack Y/A) \xymatrix{\ar[r] &} R\Gamma_\Prism(\mstack X\times \mstack Y/A),$$
in an equivalence.

\begin{proof}
By the universal property of coproduct pullbacks induce a map of $E_\infty$-algebras
$$R\Gamma_\Prism(\mstack X/A) \otimes_A R\Gamma_\Prism(\mstack Y/A) \xymatrix{\ar[r] &} R\Gamma_\Prism(\mstack X \times \mstack Y/A),$$
which factors through $(p, I)$-completion $R\Gamma_\Prism(\mstack X/A)\cotimes_A R\Gamma_\Prism(\mstack Y/A) \to R\Gamma_\Prism(\mstack X \times \mstack Y/A)$ since the right hand side is $(p, I)$-complete.

Note that by \Cref{cor_prism_coconnective_for_smooth} $R\Gamma_\Prism(\mstack X/A)$ and $R\Gamma_\Prism(\mstack Y/A)$ are $0$-coconnective and hence $R\Gamma_\Prism(\mstack X/A)/(p,I)$ and $R\Gamma_\Prism(\mstack Y/A)/(p, I)$ are at most $(-2)$-coconnective. It follows by our assumption on $\widehat\otimes_A$ that the $(p,I)$-completed tensor product functor
$$R\Gamma_\Prism(\mstack X/A)\cotimes_A - \colon \DMod{A} \tto \DMod{A}$$
is left $t$-exact up to a shift, in particular it preserves totalizations of $0$-coconnective objects. A similar assertion holds for the functor $-\cotimes_{A} R\Gamma_\Prism(\mstack Y/A)$. Let now $U_\bullet$ and $V_\bullet$ be smooth affine hypercovers of $\mstack X$ and $\mstack Y$ respectively. Then by smooth descent for prismatic cohomology and the previous discussion the natural maps
\begin{align*}
R\Gamma_\Prism(\mstack X\times \mstack Y/A) & \tto \Tot R\Gamma_\Prism(U_\bullet \times V_\bullet/A), \\
R\Gamma_\Prism(\mstack X/A) \cotimes_{A} R\Gamma_\Prism(\mstack Y/A) & \tto \Tot\left( (R\Gamma_\Prism(U_\bullet/A)\cotimes_{A} R\Gamma_\Prism(V_\bullet/A) \right)
\end{align*}
are equivalences. It follows that it is enough to prove that for a pair of smooth affine $A/I$-schemes $U, V$ that the natural map
$$\alpha\colon R\Gamma_\Prism(U/A)\cotimes_A R\Gamma_\Prism(V/A) \xymatrix{\ar[r] &} R\Gamma_\Prism(U\times V/A)$$
is an equivalence.

Since both parts are $I$-complete, it is enough to prove that the mod $I$ map
$$(R\Gamma_\Prism(U/A)\otimes_A A/I)\cotimes_{A/I} (R\Gamma_\Prism(V/A)\otimes_A A/I) \xymatrix{\ar[r] &} R\Gamma_\Prism(U\times V/A)\otimes_A A/I$$
is an equivalence. But by the Hodge-Tate comparison for schemes and smoothness of $U$ the complex $R\Gamma_\Prism(U/A)\otimes_A A/I$ admits a finite filtration with associated graded pieces $(\Omega^i_{U})^\wedge_p[-i]$ and similarly for $V$ and $U\times V$. The result then follows from the K\"unneth formula for $p$-completed Hodge cohomology, which in its turn follows from the usual K\"unneth isomorphism for smooth affine $A/I$-schemes
$$\wedge^*_U(\Omega^1_U) \otimes_{A/I} \wedge^*_V(\Omega^1_V) \xymatrix{\ar[r]^-\sim&} \wedge^*_{U\times V}(\Omega^1_{U\times V})$$
and the fact that the $p$-completion functor $-_{p}^\wedge$ is symmetric monoidal.
\end{proof}
\end{prop}

\subsection{Hodge and de Rham cohomology, $p$-completed and not}
\label{sec: Hodge and de Rham cohomology}
In this section we define the de Rham and Hodge cohomology and also their derived (as well as $p$-completed derived) versions. One can look at \cite[Section 1.1]{KubrakPrikhodko_HdR} for a more detailed treatment of the classical version. For the rest of this section fix a base ring $R$.
\begin{defn}[Hodge cohomology]
	Let $\mstack X$ be an Artin stack over $R$. Define \emdef{Hodge cohomology $R\Gamma_\Hdg(\mstack X/R)$ of $\mstack X$} to be
	$$R\Gamma_\Hdg(\mstack X/R) \coloneqq \bigoplus_{p \ge 0} R\Gamma\left(\mstack X, \wedge^p \mathbb L_{\mstack X/R}[-p]\right),$$
	where $\mathbb L_{\mstack X/R}$ is the cotangent complex of $\mstack X$ over $R$ and $\wedge^p \mathbb L_{\mstack X/R}$ is its $p$-th derived exterior power (see \Cref{stacky_tensor_functos}). Note that by Propositions \ref{cotangent_of_smooth_stack} and \ref{derived_tensor_finiteness}, for a smooth $\mstack X$, the wedge power $\wedge^p \mathbb L_{\mstack X/R}$ is a perfect sheaf. For a fixed $n\in \mathbb Z$ we will also denote
	$$H^n_\Hdg(\mstack X/R) := H^nR\Gamma_\Hdg(\mstack X/R) \simeq \bigoplus_{p+q=n} H^{p,q}(\mstack X/R), \quad\text{where}\quad H^{p,q}(\mstack X/R) := H^q\left(\mstack X, \wedge^p \mathbb L_{\mstack X/R}\right).$$
\end{defn}
\begin{notation}
	Let $U := \Spec P$ be a smooth affine $R$-scheme. The algebraic de Rham complex of $S$ over $R$
	$$\xymatrix{P \ar[r]^-d & \Omega^1_{P/R} \ar[r]^-d & \Omega^2_{P/R} \ar[r]^-d & \ldots}$$
	will be denoted by $\Omega_{U/R, \dR}^\bullet$. We define $R\Gamma_\dR(U/R) \coloneqq \Omega_{U/R, \dR}^\bullet \in \DMod{R}$.
\end{notation}
\begin{defn}[de Rham cohomology]\label{def: de Rham}
	Let $\mstack X$ be a smooth Artin stack over $R$. Define the \emph{de Rham cohomology $R\Gamma_\dR(\mstack X/R)$ of $\mstack X$} to be
	$$R\Gamma_\dR(\mstack X/R) := \lim_{Y\in \Aff^{\sm, \op}_{/\mstack X}} R\Gamma_\dR(Y/R),$$
	where $\Aff^{\sm}_{/\mstack X}$ is the full subcategory of stacks over $\mstack X$ consisting of affine $R$-schemes that are smooth over $\mstack X$. We will also denote $H^nR\Gamma_\dR(\mstack X/R)$ by $H^n_\dR(\mstack X/R)$. If the base ring $R$ is clear from the context, we will denote $R\Gamma_\Hdg(X/R)$ and $R\Gamma_\dR(\mstack X/R)$ just by $R\Gamma_\Hdg(X)$ and $R\Gamma_\dR(\mstack X)$ respectively.
\end{defn}

For smooth $\mstack X$ the Hodge cohomology admits a description similar to our definition of the de Rham cohomology:
\begin{prop}\label{Hodge_as_limit}
	For any $p\in \mathbb Z_{\ge 0}$ the natural map
	\begin{align*}
		R\Gamma(\mstack X, \wedge^p \mathbb L_{\mstack X/R}) \xymatrix{\ar[r]^\sim &} \lim_{S\in \Aff^{\sm, \op}_{/\mstack X}} R\Gamma(S, \wedge^p \mathbb L_{S/R})
	\end{align*}
	is an equivalence.
	
	\begin{proof}
		The cotangent complex and its exterior powers satisfies \'etale descent and then we are done by an argument analogous to \Cref{lem: enough to take smooth guys}.
	\end{proof}
\end{prop}

Hodge-proper stacks (\Cref{def:Hodge_proper}) are particularly nice from the point of view of de Rham cohomology:
\begin{prop}\label{Hodge_filtration_for_stacks}
	Let $\mstack X/R$ be a smooth Artin stack. Then
	\begin{enumerate}[label=(\arabic*)]
		\item There exists a complete (decreasing) \emph{Hodge filtration $F^\bullet R\Gamma_\dR(\mstack X/R)$} such that $\gr F^\bullet R\Gamma_\dR(\mstack X/R) \simeq R\Gamma_\Hdg(\mstack X/R)$.
		
		\item Suppose that $R$ is Noetherian. If $\mstack X$ is Hodge-proper over $R$, then the complex $R\Gamma_\dR(\mstack X/R)$ is a bounded below coherent $R$-module.
	\end{enumerate}
	
	\begin{proof}
		This is formal, see \cite[Corollary 1.1.6]{KubrakPrikhodko_HdR} and \cite[Proposition 1.2.6]{KubrakPrikhodko_HdR} for details.
	\end{proof}
\end{prop}

There is also base change for the de Rham cohomology under certain assumptions on the base:
\begin{prop}[Base-change]\label{Hodge_and_deRham_basechange}
	Let $\mstack X$ be a smooth Artin stack over $R$ and let $R\to R^\prime$ be a ring homomorphism such that $R^\prime$ considered as an $R$-module has finite Tor-amplitude (e.g. if $R^\prime$ is flat or perfect as an $R$-module). Then for $\mstack X^\prime := \mstack X \otimes_R R^\prime$ the natural map $R\Gamma_\dR(\mstack X / R) \otimes_R R^\prime \to R\Gamma_\dR(\mstack X^\prime / R^\prime)$ is a filtered equivalence. In particular, for each $p\in \mathbb Z_{\ge 0}$ the natural map $R\Gamma(\mstack X, \wedge^p \mathbb L_{\mstack X/R}) \otimes_R R^\prime \to R\Gamma(\mstack X^\prime, \mathbb \wedge^p \mbb L_{\mstack X^\prime/R^\prime})$ is an equivalence.
	
	\begin{proof}
		See \cite[Proposition 1.1.8]{KubrakPrikhodko_HdR}.
	\end{proof}
\end{prop}

\paragraph{Derived de Rham cohomology for affine schemes.}
This subsection contains a short technical digression on derived de Rham cohomology which will be used below. For the rest of the subsection fix a commutative base ring $R$. Let $\Alg_{E_\infty, R/}\coloneqq \Alg_{E_\infty}(\DMod{R})$ be the category of $E_\infty$-algebras in $\DMod{R}$. Recall that the functor of derived de Rham cohomology
$$R\Gamma_{L\dR}(-/R) \colon \CAlg_{R/}^\Delta \xymatrix{\ar[r] &} \Alg_{E_\infty, R/}$$
is defined as the left Kan extension of the functor
$$\Omega_{-/R, \dR}^\bullet\colon \Poly^{\fg}_R \xymatrix{\ar[r] &} \Alg_{E_\infty, R/}$$
along the embedding $\Poly^\fg_R \inj \CAlg_{R/}^\Delta$, where $\Poly^\fg_R$ is the category of finitely generated polynomial $R$-algebras. Using the observation that $\CAlg_{R/}^\Delta$ is the sifted completion of $\Poly^\fg_R$ one can see that $R\Gamma_{L\dR}(-/R)$ is the unique functor that agrees with $\Omega_{-/R, \dR}^\bullet$ on $\Poly_R$ and commutes with filtered colimits and geometric realizations. In particular, given a simplicial algebra $A_\bullet\in \CAlg_{R/}^\Delta$ one can explicitly compute $R\Gamma_{L\dR}(A_\bullet/R)$ by finding a polynomial resolution $P_\bullet \sim A_\bullet$ and taking the geometric realization $|\Omega^\bullet_{P_\bullet/R, \dR}|$, where, if $P$ is an infinitely generated, we put $\Omega^\bullet_{P/R, \dR}$ to be equal to the colimit of $\Omega^\bullet_{Q/R, \dR}$ over all $Q\subset P$ such that $Q\in \Poly^\fg_R$. If $A\in \CAlg_{R/}$ is discrete we have a particularly nice functorial \emph{bar resolution} $P_\bullet\ra A$ given by $P_0\coloneqq R[A]$, $P_1\coloneqq R[R[A]]$, etc. with the natural transition maps. We also note that the forgetful functor $\Alg_{E_\infty}({R}) \to \DMod{R}$ commutes with sifted colimits, thus the underlying object of $R\Gamma_{L\dR}(A_\bullet/R)$ is computed by the same procedure but with values in $\DMod{R}$.

\begin{prop}\label{I_myself_was_a_bit_surprised}
	The functor of derived de Rham cohomology
	$$R\Gamma_{L\dR}(-/R) \colon \CAlg_{R/}^\Delta \xymatrix{\ar[r] &} \Alg_{E_\infty, R/}$$
	preserves all small colimits.
	
	\begin{proof}
		By construction $R\Gamma_{L\dR}(-/R)$ preserves filtered colimits, so it is enough to prove that for a span of simplicial $R$-algebras $B \leftarrow A \rightarrow C$ the natural map
		$$R\Gamma_{L\dR}(B/R)\otimes_{R\Gamma_{L\dR}(A/R)} R\Gamma_{L\dR}(C/R) \xymatrix{\ar[r] &} R\Gamma_{L\dR}(B\otimes_A C/R)$$
		is an equivalence. The pushout $B\otimes_A C$ can be rewritten as a geometric realization of simplicial diagram
		$$\xymatrix{\ldots \ar@<+1.35ex>[r] \ar@<+.45ex>[r] \ar@<-.45ex>[r] \ar@<-1.35ex>[r] & B\otimes_R A \otimes_R A \otimes C \ar@<+.9ex>[r] \ar[r] \ar@<-.9ex>[r] & B\otimes_R A \otimes_R C \ar@<-.45ex>[r] \ar@<.45ex>[r] & B\otimes_R C}, $$
		thus, since $R\Gamma_{L\dR}(-/R)$ also preserves geometric realizations, it is left to prove that $R\Gamma_{L\dR}(-/R)$ commutes with binary coproducts, i.e. to establish the special case when $A=R$. Since tensor products commute with colimits this reduces to the case of finitely generated polynomial $R$-algebras by representing $B$ and $C$ as sifted colimits of the latter. Finally, the polynomial case can be checked by a direct computation.
	\end{proof}
\end{prop}

We have a natural decreasing (however not necessarily complete) Hodge filtration $\Fil^{\ge n}_\Hdg \RG_{L\dR}(-/R)$ induced by the Hodge filtration on $\Omega^\bullet_{-/R,\dR}$. It's associated graded is given by the wedge powers of the cotangent complex. In particular, for any $A$ we have a natural projection $R\Gamma_{L\dR}(A/R)\rightarrow A\simeq R\Gamma_{L\dR}(A/R)/\Fil^{\ge 1}_\Hdg \RG_{L\dR}(-/R)$. The following assertion compares the absolute and relative derived de Rham cohomology:
\begin{lem}\label{self_consistent}
	Let $A \to B$ be a map of (discrete) commutative $R$-algebras. Then there is a natural equivalence
	\begin{align}\label{eq:relative_dR_vs_absolute}
		R\Gamma_{L\dR}(B/R)\otimes_{R\Gamma_{L\dR}(A/R)} A \simeq R\Gamma_{L\dR}(B/A).
	\end{align}
	
\begin{proof} 
We first construct the equivalence in the case when $A\in \Poly^\fg_R$ and $B\in \Poly^\fg_A \subset \Poly^\fg_R$. The natural map $\Omega^1_{B/R} \to \Omega^1_{B/A}$ induces a morphism of dg-algebras $\Omega^\bullet_{B/R,\dR}
\to \Omega^\bullet_{B/A,\dR}$. By construction the restriction of this morphism along $\Omega^\bullet_{A/R,\dR} \to \Omega^\bullet_{B/R,\dR}$ factors through $\Omega^\bullet_{A/A,\dR} \simeq A$. It follows that there is a natural comparison map
$$\Omega^\bullet_{B/R,\dR} \otimes_{\Omega^\bullet_{A/R,\dR}} A \tto \Omega^\bullet_{B/A,\dR}.$$
By construction the map above is a filtered morphism of $E_\infty$-algebras, where the filtration on the right is the Hodge filtration and on the left it is obtained from Hodge filtrations on $\Omega^\bullet_{A/R,\dR}$ and $\Omega^\bullet_{B/R,\dR}$ and the trivial filtration on $A$. Note that since $A,B$ are polynomial algebras these filtrations are finite and thus complete. Hence it is enough to show that the map above induces an equivalence on the associated graded pieces, which can be proved on the level of cohomology. But note that by our assumptions $H^*(\Omega^\bullet_{B/R})$ is a free graded $H^*(\Omega^\bullet_{A/R})$-module, hence the $\Tor$-spectral sequence computing
$$H^*(\Omega^\bullet_{B/R} \otimes_{\Omega^\bullet_{A/R}} A)$$
degenerates on the first page and there is a natural isomorphism
$$H^*(\Omega^\bullet_{B/R} \otimes_{\Omega^\bullet_{A/R}} A) \simeq H^*(\Omega^\bullet_{B/R}) \otimes_{H^*(\Omega^\bullet_{A/R})} A.$$
Finally, it is easy to see that the natural map
$$H^*(\Omega^\bullet_{B/R}) \otimes_{H^*(\Omega^\bullet_{A/R})} A \tto H^*(\Omega_{B/A}^\bullet)$$
is an isomorphism under our assumptions on $A$ and $B$.

This gives a functorial equivalence in the case of finitely generated polynomial algebras. We extend it to infinitely generated polynomial algebras by passing to the colimit over $A_i\ra B_i \subset A\ra B$ with $A_i$ and $B_i$ being finitely generated.

In the general case observe that for a pair of algebra maps $R\to S \to T$ the bar resolution of $T$ over $R$ maps naturally to the bar resolution of $T$ over $S$. Let $A_\bullet$ be a bar resolution of $A$ over $R$ and let $B_{i,\bullet}$ be a bar resolution of $B$ over $A_i$. By the previous observation $B_{\bullet, \bullet}$ naturally form a bi-simplicial $R$-algebra. Passing to the diagonal simplicial algebra $B_i \coloneqq B_{i,i}$ we obtain a natural map $A_\bullet \to B_\bullet$ such that
\begin{itemize}
\item All $A_i$ are free $R$-algebras.

\item All $B_i$ are free $A_i$-algebras (in particular they are also free as $R$-algebras).

\item The geometric realization of $A_\bullet \to B_\bullet$ is $A \to B$.
\end{itemize}
Now, since $R\Gamma_{L\dR}(-/R)$ preserves geometric realizations and since colimits commute, we have 
$$R\Gamma_{L\dR}(B/R)\otimes_{R\Gamma(A/R)} A \simeq \left| R\Gamma_{L\dR}(B_\bullet/R)\otimes_{R\Gamma(A_\bullet/R)} A_\bullet \right| \simeq \left| \Omega_{B_\bullet/R,\dR}^\bullet \otimes_{\Omega_{A_\bullet/R,\dR}^\bullet} A_\bullet \right| \xra{\sim} \left| \Omega_{B_\bullet/A_\bullet,\dR}^\bullet\right|,$$
where in the last step we use the equivalence constructed in the polynomial case.

Next note that since both sides of \eqref{eq:relative_dR_vs_absolute} preserve sifted colimits in $B$, we can assume that $B$ is a polynomial $A$-algebra. Then
$$\left| \Omega_{B_\bullet/A_\bullet,\dR}^\bullet\right| \simeq \left| \Omega_{B_\bullet/A_\bullet,\dR}^\bullet\right|\otimes_A A \simeq \left| \Omega_{B_\bullet/A_\bullet,\dR}^\bullet\otimes_{A_\bullet} A\right| \simeq \left| \Omega_{B/A,\dR}^\bullet\right| \simeq R\Gamma_{L\dR}(B/A).$$

The naturality of the equivalence follows from the functoriality of bar resolutions as well the isomorphism in the polynomial case.
\end{proof}
\end{lem}
Using the observation above one can extend the definition of the relative de Rham cohomology to an arbitrary simplicial base:
\begin{defn}\label{def: relative de Rham cohomology}
	Let $A\to B$ be a map of simplicial $R$-algebras. Define
	$$R\Gamma_{L\dR}(B/A) \coloneqq R\Gamma_{L\dR}(B/R)\otimes_{R\Gamma_{L\dR}(A/R)} A.$$
	By \Cref{self_consistent} this definition agrees with the usual one when $A$ is discrete.
\end{defn}

This relative notion is rather well-behaved, for example it satisfies base change:
\begin{prop}[Base change]\label{yet_another_base_change}
	Let $A\to B$ be a map of simplicial algebras. Then there is a natural equivalence
	$$R\Gamma_{L\dR}(B/A)\otimes_{A} C \simeq R\Gamma_{L\dR}(B\otimes_A C/C).$$
	
	\begin{proof}
		By definition the left hand side is equivalent to $R\Gamma_{L\dR}(B/R)\otimes_{R\Gamma_{L\dR}(A/R)} C$ and the right hand side is equivalent to $R\Gamma_{L\dR}(B\otimes_A C/R)\otimes_{R\Gamma_{L\dR}(C/R)} C$. Since by \Cref{I_myself_was_a_bit_surprised}
		$$R\Gamma_{L\dR}(B\otimes_A C/R) \simeq R\Gamma_{L\dR}(B/R)\otimes_{R\Gamma_{L\dR}(A/R)} R\Gamma_{L\dR}(C/R)$$
		we obtain the desired result.
	\end{proof}
\end{prop}
\begin{rem}
	Note that it follows from \Cref{yet_another_base_change} that the definition of $R\Gamma_{L\dR}(B/A)$ in \ref{def: relative de Rham cohomology} does not depend on the choice of the base ring $R$ and only on the homomorphism $A\ra B$.
\end{rem}

In the case the base ring $R$ is discrete and of characteristic $p$, there is the (increasing) conjugate filtration $\Fil^{\mr{cnj}}_\bullet R\Gamma_{L\dR}(B/R)$ on the derived de Rham cohomology, which is induced by the conjugate filtration on $\Omega_{-/R,\dR}^\bullet$. By the derived Cartier isomorphism (\cite[Proposition 3.5]{Bhatt_pAdicdR}) its associated graded is identified with $\oplus_i\wedge^i_R\mbb L_{B^{{(1)}_R}/R}$ where $B^{{(1)}_R}\coloneqq B\otimes_{R,\phi_R} R$ is the Frobenius twist relative to $R$. One can also extend this to the relative situation as follows:
\begin{construction}[Conjugate filtration on $\RG_{L\dR}(B/A)$]\label{constr: conjugate filtration in the relative case}
	Let $A\ra B$ be a homomorphism of simplicial algebras and let's endow $\RG_{L\dR}(B/R)$ and $\RG_{L\dR}(A/R)$ with the conjugate filtration and $A$ with the trivial filtration $\Fil_{<0}A=0$, $A=\Fil_0A=\Fil_1A=\ldots$. The natural homomorphisms $\RG_{L\dR}(A/R) \ra \RG_{L\dR}(B/R)$ and $\RG_{L\dR}(A/R)\ra A$ are filtered and this gives a natural filtration\footnote{The filtration on the tensor product $A\otimes_C B$ for filtered $R$-algebras $A,B,C$ with $C\ra A$ and $C\ra B$ being filtered homomorphisms is best described via the Rees construction $\mc R\colon \DMod{R}^{\mbb Z-\Fil} \xra{\sim} \QCoh([\mbb A^1_R/\mbb G_{m,R}])$. The tensor product $\mc R(A)\otimes_{\mc R(B)} \mc R(C)$ then has the underlying $R$-module (computed as the fiber at $1\in \mbb A^1(R)$) given by $A\otimes_B C$ and the associated graded of the natural filtration (given by the order of vanishing of the corresponding $\mbb G_m$-equivariant sections) is computed as the fiber at $0\in \mbb A^1(R)$ and is equal to the tensor product of the associated gradeds: $\gr A\otimes_{\gr B}\gr C$.}  on the tensor product 
	$
	\RG_{L\dR}(B/A)\simeq \RG_{L\dR}(B/R)\otimes_{\RG_{L\dR}(A/R)} A$, which we also denote by $\Fil^{\mr{cnj}}_\bullet \RG_{L\dR}(B/A)$. Its associated graded is given by the tensor product of the associated gradeds
	$$
	\gr^{\mr{cnj}}_\bullet \RG_{L\dR}(B/A) \simeq  \left(\oplus_i \wedge^i_R\mbb L_{B^{(1)}/R}\right)\otimes_{\left(\oplus_i \wedge^i_R\mbb L_{A^{(1)}/R}\right)}A,
	$$
	where the morphism $\oplus_i \wedge^i_R\mbb L_{A^{(1)}/R}\ra A$ quotients through $\oplus_i \wedge^i_R\mbb L_{A^{(1)}/R} \surj A^{{(1)}_R}\xra{\phi_{A/R}} A$ with the second arrow given by the relative Frobenius. It is not hard to see (e.g. from an explicit computation in the polynomial case) that 
	$$
	\left(\oplus_i \wedge^i_R\mbb L_{B^{(1)}/R}\right)\otimes_{\left(\oplus_i \wedge^i_R\mbb L_{A^{(1)}/R}\right)}A^{(1)}\simeq \oplus_i \wedge^i_R\mbb L_{B^{(1)}/A^{(1)}},
	$$ and thus by base change 
	$$
	\gr^{\mr{cnj}}_\bullet \RG_{L\dR}(B/A) \simeq \left(\oplus_i \wedge^i_R\mbb L_{B^{(1)}/A^{(1)}}\right)\otimes_{A^{(1)},\phi_{A/R}}A\simeq \oplus_i \wedge^i_R\mbb L_{B^{{(1)}_A}/A},
	$$
	where $B^{{(1)}_A}\coloneqq B\otimes_{A,\phi_A}A\simeq B^{{(1)}_R}\otimes_{A^{{(1)}_R},\phi_{A/R}}A$ is the Frobenius twist relative to $A$.
\end{construction}

\paragraph{The $p$-completed version.}
For the rest of this subsection fix a classically $p$-complete base ring $R$ with bounded $p^\infty$-torsion; note that this condition forces $R$ to be derived $p$-complete. For a smooth affine $R$-scheme $\Spec P$ we can consider the derived $p$-completion $(\Omega_{P/R, \dR}^\bullet)^\wedge_p$ of the de Rham complex. Note that by \cite[Lemma 4.7]{BMS2} the derived completion $\widehat P\coloneqq P^\wedge_p$ of $P$ is a classically $p$-complete ring. Since by smoothness of $P$, each $\Omega_{P/R}^i$ is a finitely generated projective $P$-module, it follows from the same lemma that the derived $p$-completion $\widehat \Omega_{P/R}^i\coloneqq (\Omega_{P/R}^i)^\wedge_p$ is also discrete and classically $p$-complete. It follows further that $(\Omega_{P/R, \dR}^\bullet)^\wedge_p$ has a model as a complex 
$$
\widehat \Omega_{P/R, \dR}^\bullet \coloneqq \xymatrix{\widehat{P} \ar[r]^-d & \widehat \Omega^1_{P/R} \ar[r]^-d & \widehat \Omega^2_{P/R} \ar[r]^-d & \ldots}, 
$$
with the differential induced by the de Rham differential on $\Omega_{P/R, \dR}^\bullet$. 

\begin{rem}
	The complex $\widehat \Omega_{P/R, \dR}^\bullet $ can be considered as a natural extension of the relative de Rham complex of the formal scheme $\Spf \widehat P$ to the not necessarily Noetherian case.
\end{rem}

This construction can be extended to a general prestack over $R$ in the same way we did for the prismatic cohomology:
\begin{construction}[$p$-completed derived  de Rham cohomology]\label{constr:p-completed derived de Rham}
	Let $R$ be a $p$-complete ring with bounded $p^\infty$-torsion. Define the functor
	$$R\Gamma_{L\dR^\wedge_p}(-/R) \colon \Aff_R^\op \xymatrix{\ar[r] &} {\DMod{R}}_{\widehat{p}}$$
	of \emdef{$p$-completed derived de Rham cohomology} as a left Kan extension of the functor $P\mapsto \widehat \Omega_{P/R, \dR}$ along the embedding $\Poly^\fg_R \inj \CAlg_R$. The functor
	$$R\Gamma_{L\dR^\wedge_p}(-/R) \colon \PStk_R^\op \xymatrix{\ar[r] &} {\DMod{R}}_{\widehat{p}}$$
	is then defined as a right Kan extension along the Yoneda's embedding $\Aff^\op_R \inj \PStk_R^\op$.
\end{construction}
\begin{rem}
	Note that the functor $R\Gamma_{L\dR^\wedge_p}(-/R)\otimes_{\mbb Z}\mbb F_p\colon \Aff_R^\op \ra {\DMod{R}}_{\widehat{p}}$ is identified with the relative derived de Rham cohomology $R\Gamma_{L\dR}(-\otimes_{\mbb Z}\mbb F_p /R\otimes_{\mbb Z}\mbb F_p )$, where the simplicial algebra $R\otimes_{\mbb Z}\mbb F_p$ is not necessarily discrete.
\end{rem}

The next proposition shows that this construction satisfies \'etale hyperdescent. The key observation of the proof goes back to the work of Bhatt (\cite{Bhatt_pAdicdR}, see also \cite[Example 5.12]{BMS2}).
\begin{lem}
	The functor of $p$-complete derived de Rham cohomology $R\Gamma_{L \dR^\wedge_p}(-/R)\colon \PStk_R^\op \to {\DMod{R}}_{\widehat{p}}$ factors through the \'etale hypersheafification functor $L_\et\colon \PStk_R \to \Stk_{R,\et}$. In particular it satisfies \'etale hyperdescent.
	
	\begin{proof}
		It is enough to prove that $R\Gamma_{L\dR^\wedge_p}(-/R) \colon \Aff_R^\op \xymatrix{\ar[r] &} {\DMod{R}}_{\widehat{p}}$ satisfies \'etale hyperdescent. Let $ Q^{-1} \to Q^\bullet$ be an \'etale hypercover in $\Aff_R^\op\simeq \Alg_R$ and let
		$$\alpha \colon R\Gamma_{L\dR^\wedge_p}({\Spec Q^{-1}}/R) \xymatrix{\ar[r] &} \Tot R\Gamma_{L\dR^\wedge_p}(\Spec Q^\bullet/R)$$
		be the natural map. By $p$-completeness it is enough to prove that $\alpha \otimes_{\mathbb Z} \mathbb F_p$ is an equivalence, thus it is enough to show the same statement for $R\Gamma_{L\dR}$, and $R$ and $Q$ are simplicial algebras of characteristic $p$. In this case note that by the derived (\cite[Proposition 3.5]{Bhatt_pAdicdR} and \Cref{constr: conjugate filtration in the relative case}) Cartier isomorphism $R\Gamma_{L\dR}(Q^i/R)$ is an algebra over the Frobenius twist $Q^{i(1)}\coloneqq Q^i\otimes_{R,\phi_R}R$. Moreover, the natural base change map
		$$R\Gamma_{L\dR}({\Spec Q^{-1}}/R)\otimes_{Q^{-1(1)}} Q^{i(1)} \xymatrix{\ar[r] &} R\Gamma_{L\dR}(\Spec Q^i/R)$$
		is an equivalence for any $i$: indeed, by the derived Cartier isomorphism  both sides admit an increasing filtration with the associated graded pieces $\wedge^j \mathbb L_{Q^{i(1)}/R}$ and the map $Q^{-1}\ra Q^{i}$ (thus also $Q^{-1(1)}\ra Q^{i(1)}$) is \'etale. The result follows by flat descent.
	\end{proof}
\end{lem}
\begin{rem}
	There is also an obvious non $p$-completed version of derived de Rham cohomology. A similar descent result holds for its Hodge-completion, but it seems somewhat hard to control the non-completed version outside of the smooth case.
\end{rem}

In the case of a smooth Artin stack the $p$-completed derived de Rham cohomology is exactly the derived $p$-completion of the non-derived version (\Cref{def: de Rham}):
\begin{prop}\label{prop: derived complete is the completion for smooth Artin stacks}
	Let $\mstack X$ be a smooth Artin stack over $R$. Then there is a natural equivalence
	$$R\Gamma_{L\dR^\wedge_p}({\mstack X}/R)\xymatrix{\ar[r]^\sim &}R\Gamma_\dR(\mstack X/R)_{p}^\wedge.$$
	
	\begin{proof}
		By the universal property of the left Kan extension there is a natural transformation
		\begin{align}\label{eq:DeRhamVsDeRham}
			R\Gamma_{L\dR^\wedge_p}(-/R) \xymatrix{\ar[r] &} R\Gamma_\dR(-/R)^\wedge_{p},
		\end{align}
		as functors on smooth affine schemes. Since for a smooth Artin stack both sides are given by right Kan extension from smooth affine schemes (for the left hand side this follows by an argument analogous to \Cref{lem: enough to take smooth guys}, while for the right hand side it follows by definition of de Rham cohomology and the fact that completion commutes with limits), it is enough to show that this transformation is an equivalence when computed on any smooth affine $R$-scheme $\Spec A$.
		
		Let now $\tilde R$ be a $\mathbb Z_p$-flat derived $p$-adically complete (discrete) algebra with a surjective map $\tilde R \surj R$. To construct such $\tilde R$ one can take any surjection $\tilde R^\prime \surj R$ from a $p$-torsion free algebra $\tilde R^\prime$ and define $\tilde R$ to be the derived $p$-adic completion of $\tilde R^\prime$. Note that since $\tilde R^\prime$ is $p$-torsion free its derived $p$-adic completion is equivalent to the usual one. To see that the induced map $\tilde R \to R$ is surjective note that the obstruction to the surjectivity is $H^1(I_p^{\wedge})$, where $I := \ker(\tilde R^\prime \surj R)$. But $I$, being a submodule of $\tilde R^\prime$, is $p$-torsion free itself, thus $H^1(I_p^{\wedge}) \simeq 0$.
		
		Then, by \cite[Tag 07M8]{StacksProject} there exists a smooth $\tilde R$-algebra $\tilde A$ lifting $R \to A$. By the base change for the usual de Rham cohomology of smooth algebras and by \Cref{yet_another_base_change} we have
		$$R\Gamma_\dR(\tilde A/\tilde R)_p^\wedge \widehat \otimes_{\tilde R} R \simeq R\Gamma_\dR(A/R)_p^\wedge \qquad\quad R\Gamma_{L\dR_p^\wedge}(\tilde A/\tilde R)\widehat\otimes_{\tilde R} R \simeq R\Gamma_{L\dR_p^\wedge}(A/R),$$
		thus, by replacing $R$ by $\tilde R$, we can assume that $R$ is flat over $\mathbb Z_p$ and $R\otimes_{\mbb Z_p} \mbb F_p$ is a discrete algebra.
		
		Since both sides of \eqref{eq:DeRhamVsDeRham} are derived $p$-adically complete, it is enough to prove the equivalence modulo $p$. By base change applied to both sides, it is enough to show that if $A$ is a smooth $R$-algebra and $R$ is of characteristic $p$, then the natural map
		$$
		\RG_{L\dR}(\Spec A/R)  \xymatrix{\ar[r] &}  \Omega_{A/R,\dR}^\bullet
		$$ is an equivalence. Both sides admit the conjugate filtration, which is exhaustive since $A$ is smooth; the morphism above induces an equivalence on associated graded pieces, hence is an equivalence itself.
	\end{proof}
\end{prop}
If we further assume that $\mstack X$ is Hodge-proper there is no need to complete the right hand side:
\begin{cor}\label{cor: de Rham coh of Hodge-proper are complete}
	Assume $R$ is Noetherian and let $\mstack X$ be a smooth Hodge-proper stack over $R$. Then there is a natural equivalence
	$$
	\RG_{L\dR^\wedge_p}(\mstack X/R)\xymatrix{\ar[r]^\sim &} \RG_{\dR}(\mstack X/R).
	$$
	
\begin{proof}
	By \Cref{prop: derived complete is the completion for smooth Artin stacks} it is enough to show $\RG_{\dR}(\mstack X/R)$ is derived $p$-complete. This follows from \Cref{Hodge_filtration_for_stacks} and \Cref{lem: bounded below coherent are complete}.
\end{proof}
\end{cor}

\paragraph{The de Rham comparison.} Now we can state the de Rham comparison for the prismatic cohomology. Let $A$ be a bounded prism. Recall that by \cite[Theorem 1.8 (3)]{BS_prisms} for a smooth $p$-adic formal scheme $\mf X$ over $\Spf A/I$ we have the de Rham comparison: $\RG_{\dR}(\mf X/(A/I))\simeq\RG_{\Prism}(\mf X/A)\widehat \otimes_{A,\phi_A}A/I$. 

Since for a smooth (non-complete) affine $A/I$-scheme $\Spec P$ we defined the prismatic cohomology $\RG_{\Prism}(\Spec P/A)$ as the prismatic cohomology $\RG_{\Prism}(\Spf \widehat{P}/A)$ of its $p$-adic completion, the comparison we get is with the de Rham cohomology $\RG_{\dR}(\Spf \widehat{P} /(A/I))\simeq \widehat\Omega_{P/(A/I),\dR}^\bullet$ of the formal completion as well. It extends formally to any prestack:
\begin{prop}[de Rham comparison]\label{prop: de Rham comparison}
	Let $(A,I)$ be a bounded prism. Let $\mstack X$ be a prestack over $A/I$. Then there is natural equivalence
	$$
	\RG_{\Prism\fr}(\mstack X/A)\otimes_{A}A/I\simeq \RG_{L\dR^\wedge_p}({\mstack X}/(A/I)).
	$$
	
	\begin{proof}
		Since $A/I$ is a perfect $A$-module, the completed tensor product functor $-\widehat\otimes_A A/I$ coincides with the usual one and commutes both with limits and colimits. Since both sides are defined as the composition of the left and the right Kan extensions to the same categories starting from $\Aff^{\sm,\op}$, it is in enough to construct the equivalence in the smooth affine case, which is covered by \cite[Theorem 1.8 (3)]{BS_prisms}.
	\end{proof}
\end{prop}

\begin{cor}\label{cor: de Rham comparison for untwisted coh}
	In \Cref{prop: de Rham comparison} assume that $\mstack X$ is a smooth Artin stack over $A/I$ and assume that the module $\phi_{A*}A$ has finite $(p,I)$-complete Tor-amplitude. Then there is a natural equivalence 
	$$
	\RG_{\Prism}(\mstack X/A)\widehat\otimes_{A,\phi_A}A/I\simeq \RG_{L\dR^\wedge_p}({\mstack X}/(A/I)).
	$$
\end{cor}
\begin{proof}
	Follows from \Cref{prop: twisted_cohomology_is_a_twist}.
\end{proof}

In the Hodge-proper case this also gives a comparison with the usual de Rham cohomology (\Cref{def: de Rham}):
\begin{cor}\label{cor: twisted prismatic cohomology are bounded below coherent}
Assume that $A/I$ is Noetherian and let $\mstack X$ be a smooth Hodge-proper Artin stack. Then 
$$
\RG_{\Prism\fr}(\mstack X/A)\otimes_{A}A/I\simeq \RG_{\dR}(\mstack X/(A/I)).
$$
As a consequence,
$
\RG_{\Prism\fr}(\mstack X/A)\in \Coh^+(A).
$

\begin{proof}
The first assertion follows from the de Rham comparison and \Cref{cor: de Rham coh of Hodge-proper are complete} which asserts that in the Hodge-proper case $\RG_{L\dR^\wedge_p}({\mstack X}/(A/I))\simeq \RG_{\dR}(\mstack X/(A/I))$. For the second one, note that by \cite[Lemma 10.96.5]{StacksProject} $A$ is also Noetherian, thus $\Coh^+(A)$ makes sense. Then the statement follows from  \Cref{neraly_coherent_module_xi} since $\RG_{\Prism\fr}(\mstack X/A)$ is derived $I$-adically complete and $\RG_{\dR}(\mstack X/(A/I))\in \Coh^+(A/I)$ by Hodge-properness of $\mstack X$.
\end{proof}
\end{cor}

\begin{ex}
	Let $\mstack X$ be a smooth Hodge-proper stack over $\mc O_K$ and let $(\mf S,E(u))$ be the Breuil-Kisin prism. The ring $\mf S=\mbb Z_p[[u]]$ is obviously Noetherian and $\phi_{\mf S*}$ is free over $\mf S$, thus both \Cref{cor: de Rham comparison for untwisted coh} and \Cref{cor: twisted prismatic cohomology are bounded below coherent} apply. In particular we get that each $H^i_{\Prism\fr}(\mstack X/\mf S)$ is a finitely generated $\mbb Z_p[[u]]$-module and that it has a "$\phi_{\mf S}$-untwist" given by $H^i_{\Prism}(\mstack X/\mf S)$. This gives a certain extra information about the torsion in $H^i_{\dR}(\mstack X/\mc O_K)$. In particular, as in \cite[Remark 1.4]{BMS2}, for $K=\mbb Q_p(\zeta_p)$ we get that $\RG_\dR(\mstack X/\mc O_K)$ comes as a base-change from a $\mbb Z_p$-module and thus the length of $H^i_{\dR}(\mstack X/\mc O_K)_{\mr{tors}}$ in this case is always divisible by $p$.
\end{ex}

\subsection{Hodge-Tate filtration}\label{sec:Hodge-Tate filtration}
\begin{construction}[Hodge-Tate filtration]
	Let $(A,I)$ be a bounded prism and let $S$ be a smooth affine scheme over $A/I$.  Recall that the Hodge-Tate comparison (\Cref{main_BMS2}(4)) asserts that there is a functorial isomorphism
	$$H^i(R\Gamma_\Prism(S/A)\otimes_A A/I) \simeq (\Omega_{S/(A/I)}^i)^\wedge_p\{-i\},$$
	where $M\{-i\}\coloneqq M\otimes_{A/I}((I/I^2)^\vee)^{\otimes n}$. The $A/I$-module $(I/I^2)^\vee $ is perfect, and since $A/I$ is itself a perfect $A$-module, the tensor product $-\otimes_{A}((I/I^2)^\vee)^{\otimes i}$ commutes with all limits and colimits. 
	It follows that the Postnikov filtration on $R\Gamma_\Prism(S/A)\otimes_A A/I$ induces an increasing filtration $\Fil_{\HT}$ on $R\Gamma_\Prism(\mstack X/A)\otimes_R A/I$ for any prestack $\mstack X$, which we will call the \emdef{Hodge-Tate filtration}. Note that since the $p$-adic completion commutes with limits and also by smooth descent for cotangent complex (\Cref{flat_descent_for_cotangent_compl}), in the case $\mstack X$ is a smooth quasi-compact quasi-separated stack over $A/I$, the associated graded pieces of the Hodge-Tate filtration on $R\Gamma_\Prism(\mstack X/A)\otimes_A A/I$ are given by $\RG(\mstack X,\wedge^i \mathbb L_{\mstack X/(A/I)})^\wedge_p\{-i\}[-i]$.  
\end{construction}
\begin{notation}
	Let $R\Gamma_{\Prism/I}(\mstack X/A)\coloneqq R\Gamma_{\Prism}(\mstack X/A)\otimes_A A/I$. We call $R\Gamma_{\Prism/I}(\mstack X/A)\in \DMod{A/I}$ the \textit{Hodge-Tate cohomology} of $\mstack X$.
\end{notation}
\begin{prop}\label{prop: HT_filtration_is_exaustive}
	Let $\mstack X$ be a smooth Artin stack over $A/I$. Then the Hodge-Tate filtration on $R\Gamma_{\Prism/I}(\mstack X/A)$ is exhaustive.
	
	\begin{proof}
		Let $S$ be a smooth affine $A/I$-scheme. Then by definition the cofiber $(R\Gamma_{\Prism/I}(S/A))/\Fil_{\HT}^{\le i}$ is $i$-coconnective. Since $i$-coconnective objects are closed under limits, the cofiber of the map
		$$\Fil^{\le i}_\HT R\Gamma_{\Prism/I}(\mstack X/A) = \lim_{S\in \Aff^\sm_{/\mstack X}} \Fil^{\le i}_\HT R\Gamma_{\Prism/I}(S/A) \xymatrix{\ar[r] &} \lim_{S\in \Aff^\sm_{/\mstack X}} R\Gamma_{\Prism/I}(S/A) \simeq R\Gamma_{\Prism/I}(\mstack X/A)$$
		is $i$-coconnective. It follows that the map 
		$$
		\colim_i \left(\Fil^{\le i}_\HT R\Gamma_{\Prism/I}(\mstack X/A)\right) \xymatrix{\ar[r]&} R\Gamma_{\Prism/I}(\mstack X/A)
		$$
		is $\infty$-coconnective, hence is an equivalence.
	\end{proof}
\end{prop}

\begin{cor}\label{cor: prismatic cohomology are bounded below coherent}
	Assume that $A/I$ is Noetherian and let $\mstack X$ be a Hodge-proper stack over $A/I$. Then 
	$$
	\gr^i_{\HT}(R\Gamma_{\Prism/I}(\mstack X/A))\simeq \RG(\mstack X,\wedge^i \mathbb L_{\mstack X/(A/I)})\{-i\}[-i].
	$$
	Consequently,
	$
	\RG_\Prism(\mstack X/A)\in \Coh^+(A).
	$ 
\end{cor}
\begin{proof}
	By the definition of Hodge-properness $\RG(\mstack X,\wedge^i \mathbb L_{\mstack X/(A/I)})\in \Coh^+(A/I)$ and so, by \Cref{lem: bounded below coherent are complete}, it is derived $p$-complete. This gives the first assertion. It also follows that $\Fil^{\le i}_{\HT}(\RG_\Prism(\mstack X/A)\otimes_A A/I)\in \Coh^+(A/I)$ for any $i$. Since the cofiber of the map $\Fil^{\le i}_\HT R\Gamma_\Prism(\mstack X/A)\otimes_A A/I \ra R\Gamma_\Prism(\mstack X/A)\otimes_A A/I$ is $i$-coconnected, we get that $R\Gamma_\Prism(\mstack X/A)\otimes_A A/I\in \Coh^+(A/I)$. Finally, $A$ is also Noetherian by \cite[Lemma 10.96.5]{StacksProject}, and then $\RG_{\Prism}(\mstack X/A)\in \Coh^+(A)$ by \Cref{neraly_coherent_module_xi}. 
\end{proof}

\subsection{Crystalline cohomology}\label{subsect:crystalline_for_stacks}
Crystalline cohomology for $1$-stacks was previously defined and studied by Olsson in \cite{Olson_CrysCoh}, where he used them (among other things) to prove Fontaine's $C_{\mathrm{st}}$-conjecture. In this section we compare his construction with the right Kan extension of the functor of crystalline cohomology from the category of schemes (the latter is exactly what we are going to obtain in the stacky version of the crystalline comparison). Similar comparison is also discussed in \cite[Section 2]{Mondal}. For the rest of this section fix a base pd-scheme $(S, I, \gamma)$.

Let $\Aff^{\pd}_{/S}$ denote the category consisting of pairs $(U\inj T, \delta)$, where $U\inj T$ is a closed embedding of an affine $S/I$-scheme $U$ into an affine $S$-scheme $T$, and $\delta$ is a pd-structure on the ideal of $U$ in $T$ compatible with the fixed pd-structure on $I$. Note that there is a natural forgetful functor $\Aff^{\pd}_{/S} \to \Aff_{S/I}$ sending $(U\inj T, \delta)$ to $U$.

\begin{defn}[{\cite[Section 1.3]{Olson_CrysCoh}}]\label{def:crys_Olsson}
	
	For a stack $\mstack X$ over $S/I$ we define \emdef{lisse-crystalline site $\mathrm{Crys}(\mstack X/S)$} of $\mstack X$ over $S$ as a fibered product $\Aff^\pd_{/S} \times_{\Aff_{S/I}} \Aff^\sm_{/\mstack X}$. Unwinding the definition, one sees that an object of $\mathrm{Crys}(\mstack X/S)$ is a triple $(U\to \mstack X, U\inj T, \delta)$, where $U\to \mstack X$ is a smooth map and $T$ is a pd-thickening of $U$ over $S$. There is a sheaf of rings $\mathcal O_{\mstack X_\crys}$ on $\mathrm{Crys}(\mstack X/S)$ sending $(U, T, \delta)$ to $H^0(T, \mathcal O_T)$. Define
	$$R\Gamma(\mstack X_\crys, \mathcal O_{\mstack X_\crys}) := \lim\limits_{(U,T, \delta)\in \mathrm{Crys}(\mstack X/S)^\op} \Gamma(T, \mathcal O_T).$$
\end{defn}
\begin{rem}
	Being more accurate, in his definition Olsson allows $U$ and $T$ to be arbitrary (i.e. not necessary affine) schemes. This definition gives the same answer as \Cref{def:crys_Olsson} by Zariski descent.
\end{rem}

On the other hand we could define crystalline cohomology via right Kan extension:
\begin{defn}\label{def:crys}
	Let $\mstack X$ be a smooth Artin stack over $S/I$. The \emdef{crystalline cohomology $R\Gamma_{\crys}(\mstack X /S)$ of $\mstack X$} is defined as
	$$R\Gamma_{\crys}(\mstack X /S) \coloneqq \lim\limits_{U \in (\Aff_{/\mstack X}^\sm)^\op} R\Gamma_{\crys}(U/S),$$
	where the limit is taken over the diagram of affine schemes smooth over $\mstack X$.
\end{defn}


And, in the case $\mstack X$ is smooth, Definitions \ref{def:crys_Olsson} and \ref{def:crys} produce the same answer:
\begin{prop}\label{our_crys_versus_olsons}
	Let $\mstack X$ be a smooth Artin stack over $S/I$. There exists a natural equivalence
	$$R\Gamma(\mstack X_{\crys}, \mathcal O_{\mstack X_\crys}) \simeq R\Gamma_\crys(\mstack X/S).$$
	
	\begin{proof}
		Let $\pi\colon \mathrm{Crys}(\mstack X/S) \to \Aff_{/\mstack X}^\sm$ be the natural projection. By general properties of Kan extensions we can compute Olsson's crystalline cohomology $R\Gamma(\mstack X_{\crys}, \mathcal O_{\mstack X_\crys})$ by first right Kan extending the functor $\mathcal O_{\mstack X_\crys}$ along $\pi$ and then by computing the limit of $\Ran_\pi \mathcal O_{\mstack X_\crys}$ over $\Aff_{/\mstack X}^\sm$. For $U \in \Aff_{/\mstack X}^\sm$ we have
		$$(\Ran_\pi \mathcal O_{\mstack X_\crys})(U) \simeq \lim\limits_{(V, T, \delta) \in (\pi/U)^\op} \Gamma(T, \mathcal O_T).$$
		Note that an object in $\pi/U$ is a pair $(V\to U, V\inj T)$, where $V\to U$ is arbitrary morphism such that the composition $V\to U \to \mstack X$ is smooth and $T$ is a pd-thickening of $U$.
		
		There is a full subcategory $\mathrm{Crys}^\prime(U/S)$ in $\pi/U$ consisting of pairs $(V \to U, T)$ where $V\to U$ is an identity morphism. We claim that $\mathrm{Crys}^\prime(U/S)$ is cofinal in $\pi/U$. By Quillen's Theorem A (see e.g. \cite[Theorem 4.1.3.1]{Lur_HTT}) it is enough to prove that for a fixed $P:=(V\to U, T)$ the comma category $P/\mathrm{Crys}^\prime(U/S)$ is non-empty and weakly contractible. Now, $P/\mathrm{Crys}^\prime(U/S)$ consists of pd-thickenings $T^\prime$ of $U$ such that the diagram
		$$\xymatrix{
			V \ar[r]\ar[d] & U \ar[d]\\
			T \ar[r] & T^\prime
		}$$
		is commutative. But by \cite[Lemma 5.11]{BO_NotesOnCrys} the pushout $U\coprod_V T$ admits a pd-structure, so $P/\mathrm{Crys}^\prime(U/S)$ is non-empty. To see that it is contractible, note that $P/\mathrm{Crys}^\prime(U/S)$ has products given by pd-envelope of $U$ diagonally embedded in $T^\prime\times T^{\prime\prime}$.
		
		It follows that
		$$(\Ran_\pi \mathcal O_{\mstack X, \crys})(U) \simeq \lim\limits_{(T,\delta) \in \mathrm{Crys}^\prime(U/W(k))^\op} \Gamma(T, \mathcal O_T) \simeq R\Gamma_\crys(U/S)$$
		and
		$$R\Gamma(\mstack X_{\crys}, \mathcal O_{\mstack X_\crys}) = \lim_{U\in (\Aff_{/\mstack X}^\sm)^\op} (\Ran_\pi \mathcal O_{\mstack X_\crys})(U) \simeq \lim_{U\in (\Aff_{/\mstack X}^\sm)^\op} R\Gamma_\crys(U/S) = R\Gamma_\crys(\mstack X/S),$$
		finishing the proof.
	\end{proof}
\end{prop}

We can now formulate the crystalline comparison. Recall that if $(A,p)$ is a prism, then $A$ is $p$-torsion free and $(p)\subset A$ has a natural a pd-structure. 

\begin{prop}[Crystalline comparison]\label{prop: cristalline comparison}
Let $(A,p)$ be a prism. Let $\mstack X$ be a smooth Artin stack over $A/p$. Then
$$
\RG_{\Prism\fr}(\mstack X/A)\simeq \RG_\crys(\mstack X/A).
$$
	
\begin{proof}
	This follows from the limit presentations in \Cref{lem: enough to take smooth guys} and \Cref{def:crys}, and the comparison \cite[Theorem 1.8 (1)]{BS_prisms} in the case of a smooth affine scheme.
\end{proof}
\end{prop}


\begin{ex}
	In \cite{Mondal} it is shown that for a finite group $G$ over a perfect field $k$ of characteristic $p$ the corresponding Dieudonn\'e module $M(G)$ can be expressed as $H^2_{\crys}(BG/k)$ with the natural Frobenius action on it.	
\end{ex}

\section{\'Etale sheaves on geometric stacks}\label{sect:etale_sheaves}
In this section we develop rudiments of \'etale sheaf theory on geometric stacks both in algebraic and rigid-analytic contexts, and show that the usual relations between them remain valid in our more general stacky setting. We refer the interested reader to \cite{YifengWeizhe_SixOperationsFormalism} for a more exhaustive treatment of the formalism of six operations for sheaves on Artin stacks. We don't need the full power of  six operations and in any case the results of loc. cit. are not quite adapted to our situation, so we decided to develop the theory independently from scratch, taking the known results for schemes as an input. Recall that by our conventions (\Cref{subsect: Artin stacks}) we consider stacks with respect to the \'etale topology.

\subsection{\'Etale sheaves}\label{etale_sheaves_on_stacks}
Let $\Lambda$ be a finite torsion ring such that $|\Lambda|$ is invertible in $R$. In this section we develop rudiments of the theory of \'etale sheaves on Artin stacks and show that many familiar results about sheaves on schemes remain valid in this generality.

\begin{defn}\label{def_sheaves}
Let $(\mathcal C, \tau)$ be an $(\infty, 1)$-site (see \cite[Definition 6.2.2.1]{Lur_HTT}) and let $\fcat D$ be an $(\infty, 1)$-category. A functor $\mathcal F\colon \mathcal C^\op \to \fcat D$ is called a \emdef{sheaf} if for every $U \in \mathcal C$ and a covering sieve $S \inj \mathcal C_{/U}$ the limit $\lim_{V \in S} \mathcal F(V)$ exists and the natural map
$$F(U) \xymatrix{\ar[r] &} \lim_{V \in S} \mathcal F(V)$$
is an equivalence. We will denote the full subcategory of $\Fun(\mathcal C^\op, \fcat D)$ spanned by $\tau$-sheaves by $\Shv_\tau(\mathcal C, \fcat D)$. If $\fcat D$ admits all small limits, the right Kan extension induces an equivalence
$$\Shv_\tau(\mathcal C, \fcat D) \areq \mathrm{RFun}(\Shv_\tau(\mathcal C, \Type)^\op, \fcat D),$$
where $\mathrm{RFun}$ denotes the category of small limit preserving functors. Moreover, if $\fcat D$ is presentable, by \cite[Proposition 4.8.1.17]{Lur_HA}
$$\mathrm{RFun}(\Shv_\tau(\mathcal C, \Type)^\op, \fcat D) \simeq \Shv_\tau(\mathcal C, \Type) \otimes \fcat D,$$
where $-\otimes -$ is the Lurie's tensor product of presentable categories (see \cite[Section 4.8.1]{Lur_HA}). In particular, if $\fcat D$ is presentable, then so is $\Shv_\tau(\mathcal C, \fcat D)$.
\end{defn}

For an affine $R$-scheme $T$ let ${\et}/T$ be the small \'etale site of $T$ and let
$$\Shv_{\et}(T, \Lambda) \coloneqq  \Shv_\et(\et/T, \DMod{\Lambda})$$
be the $(\infty,1)$-category of sheaves of $\Lambda$-modules.
\begin{rem}
The category $\Shv_\et(T, \DMod{\Lambda})$ admits a natural $t$-structure: a sheaf $\mathcal F$ lies in subcategory $\Shv_\et(T, \DMod{\Lambda})^{\le 0}$ if an only if for all $i>0$ the cohomology sheaves $\mathcal H^i(\mathcal F)$ vanish. The heart of this $t$-structure is naturally equivalent to the usual abelian category of \'etale $\Lambda$-modules on $T$. In fact by \cite[Corollary 2.1.2.4]{Lur_SAG} this inclusion induces an equivalence $D^+\Shv_\et(T, \UMod{\Lambda}) \areq \Shv_\et(T, \Lambda)^+$. The full derived category $D\Shv_\et(T, \UMod{\Lambda})$ can be identified with the full subcategory of $\Shv_\et(T, \Lambda)$ spanned by hypercomplete sheaves.
\end{rem}

\begin{construction}
Let $\mstack X$ be a prestack over $R$. We define the category of \emdef{$\Lambda$-\'etale sheaves on $\mstack X$} as
$$\Shv_{\et}(\mstack X, \Lambda) \coloneqq \lim\limits_{(T\ra \mstack X) \in (\Aff_R/\mstack X)^\op} \Shv_{\et}(T, \Lambda),$$
where the limit is taken in $\PrL_\Lambda$. For a sheaf $\mathcal F\in \Shv_{\et}(\mstack X, \Lambda)$ and a map $x\colon T\to \mstack X$ we will sometimes denote the canonical projection $x^{-1}\mathcal F\in\Shv_{\et}(T, \Lambda)$ by $\mathcal F_{|T}$. 

The \'etale cohomology of $\mstack X$ with coefficients in $\mc F$ is defined as the global sections of the usual \'etale cohomology over the category $(\Aff_R/\mstack X)^\op$:
$$
\RG_\et(\mstack X,\mc F)\coloneqq \lim \limits_{(T\ra \mstack X) \in (\Aff_R/\mstack X)^\op}\RG_\et(T,\mc F_{|T}).
$$
\end{construction}

\begin{ex}[The constant sheaf]\label{ex: the constant sheaf} Taking the constant sheaf $\ul \Lambda \in \Shv_{\et}(T, \Lambda)$ for all $T\ra \mstack X$ gives a well-defined object $\ul \Lambda\in \Shv_{\et}(\mstack X, \Lambda)$, which is \textit{the constant sheaf} associated to $\Lambda$. More generally, given any complex $M\in\DMod{\Lambda}$ we can associate to it an object $\ul M \in \Shv_{\et}(\mstack X, \Lambda)$ by taking the sheafification of the constant preasheaf associated to $M$ for each $T\ra \mstack X$. We put $\RG_\et(\mstack X,M)\coloneqq \RG_\et(\mstack X,\ul M)$.
\end{ex}

Similarly to the setting of quasi-coherent sheaves (\Cref{sec:Quasi-coherent sheaves}), we can extend the association $\mstack X \mapsto \Shv_\et(\mstack X,\Lambda)$ to a functor $$\Shv_{\et}(-, \Lambda)^{-1}\colon \PStk_R^\op \to \Prs^{\mathrm L, \otimes}_\Lambda$$ as the right Kan extension of the functor $T \mapsto \Shv_{\et}(T, \Lambda)$ along the inclusion $\Aff_{R}\subset \PStk_R$. The functor $\Shv_{\et}(-, \Lambda)^{-1}$ enjoys the following properties:

\begin{prop}\label{alg_etale_shaeves_basics}
We have:
\begin{enumerate}[label=(\arabic*)]
\item There is a functor
$$\Shv_\et(-, \Lambda)_* \colon \PStk_R \xymatrix{\ar[r] &} \Prs^{\mathrm R}_\Lambda$$
agreeing with $\Shv_{\et}(-, \Lambda)^{-1}$ on objects and such that for each morphism $f\colon \mstack X \to \mstack Y$ the functor
$$f_*\colon \Shv_{\et}(\mstack X, \Lambda) \xymatrix{\ar[r] &} \Shv_{\et}(\mstack Y, \Lambda)$$
is the right adjoint of $f^{-1}\colon \Shv_\et(\mstack Y, \Lambda) \xymatrix{\ar[r] &} \Shv_\et(\mstack X, \Lambda)$.

\item Let $\mstack X_\bullet \colon I \to \PStk_R$ be a small diagram of prestacks. Then the natural map
$$\Shv_{\et}(\colim_I \mstack X_i, \Lambda) \xymatrix{\ar[r] &} \lim_I \Shv_{\et}(\mstack X_i, \Lambda)$$
is an equivalence.

\item Let $\mstack X \to \mstack Y$ be an \'etale equivalence, i.e. a morphism of presheaves inducing equivalence after the \'etale sheafification. Then the induced map $\Shv_{\et}(\mstack X, \Lambda) \to \Shv_{\et}(\mstack Y, \Lambda)$ is an equivalence.

\item Let $\mstack U \to \mstack X$ be a surjective in \'etale topology map of prestacks and let $p_\bullet\colon \mstack U_\bullet \to \mstack X$ be the corresponding \v Cech nerve. Then the natural functor
$$\Shv_{\et}(\mstack X, \Lambda) \xymatrix{\ar[r] &} \Tot \Shv_{\et}(\mstack U_\bullet, \Lambda)$$
is an equivalence. In particular, for $\mathcal F \in \Shv_\et(\mstack X, \Lambda)$ the natural maps
$$\mathcal F \xymatrix{\ar[r] &} \Tot p_{\bullet *} p^{-1}_\bullet \mathcal F \qquad\qquad R\Gamma_\et(\mstack X, \mathcal F) \xymatrix{\ar[r] &} R\Gamma_\et(\mstack U_\bullet, p_\bullet^{-1} \mathcal F)$$
are equivalences and the pullback functor $p^{-1}$ is conservative.
\end{enumerate}

\begin{proof}
The first point follows from the adjoint functor theorem, the second one is the general property of the right Kan extensions, and the last one follows formally from the second one and the third one, since $|\mstack U_\bullet| \to \mstack X$ induces an equivalence after sheafification. To prove the third point by \cite[Proposition A.3.3.1]{Lur_SAG} it is enough to verify the following two assertions:
\begin{itemize}
\item The natural map
$$\Shv_\et\left(X\coprod Y, \Lambda\right) \tto \Shv_\et(X, \Lambda) \prod \Shv_\et(Y, \Lambda)$$
is an equivalence for any pair of affine schemes $X, Y$.

\item For any \'etale surjections $U \surj X$ of affine schemes with the corresponding \v Cech nerve $U_\bullet$ and natural projections $p_n\colon U_n \to X$ the natural functor induced by pullbacks
$$\Shv_\et(X, \Lambda) \xymatrix{\ar[r] &} \Tot \Shv_\et(U_\bullet, \Lambda) \qquad \mathcal F \xymatrix{\ar@{|->}[r] &} p_\bullet^{-1} \mathcal F$$
is an equivalence.
\end{itemize}

The first point is clear. For the second one note that the comparison functor from above admits a right adjoint $\mathcal F^\bullet \mapsto \Tot p_{\bullet *} \mathcal F^\bullet$. We claim that the unit of adjunction $\epsilon \colon \mathcal F \to \Tot p_{\bullet *} p_\bullet^{-1} \mathcal F$ is an equivalence. Since the pullback functor $p_0^{-1}$ is conservative it is enough to prove that $p_0^{-1}(\epsilon)$ is an equivalence. But since limits of sheaves are computed pointwise, $p_0^{-1}$ preserves totalizations. Since $p_0^{-1}(\mathcal F) \to p_0^{-1}(p_{\bullet *} p_\bullet^{-1} \mathcal F)$ is a split augmented cosimplicial object, $p_0^{-1}(\epsilon)$ is an equivalence. It is left to prove that $\Tot p_{\bullet *}$ is conservative. Let $\mathcal F^\bullet \in \Tot \Shv_\et(U_\bullet, \Lambda)$. Arguing as before, we see that $p_0^{-1} p_{\bullet *} \mathcal F^\bullet \simeq \mathcal F^0$. Hence if $\Tot p_{\bullet *} \mathcal F^\bullet \simeq 0$, then $\mathcal F^0 \simeq 0$. But the forgetful functor $\Tot \Shv_\et(U_\bullet, \Lambda) \to \Shv_\et(U_0, \Lambda)$ is conservative, so $\mathcal F^\bullet \simeq 0$.
\end{proof}
\end{prop}

\begin{ex}[Smooth descent of \'etale sheaves on Artin stacks]
	Let $\mstack X$ be an Artin stack and let $p\colon \mstack U\to \mstack X$ be an atlas. The morphism $p$ is an \'etale surjection, thus $|\mstack U_\bullet|\xra{\sim} \mstack X$ and by \Cref{alg_etale_shaeves_basics}(3) we have
	$$\Shv_{\et}(\mstack X, \Lambda) \xymatrix{\ar[r]^\sim &} \Tot \Shv_{\et}(\mstack U_\bullet, \Lambda),$$
	where $\mstack U_\bullet$ is the \v{C}ech nerve of $p$. This applies in particular to a smooth map of schemes $p\colon U\ra X$ which is surjective in the usual sense (as a map of schemes): indeed, any such morphism admits a section \'etale locally and thus is a surjection in the \'etale topology.
\end{ex}
	
\smallskip The category $\Shv_{\et}(-, \Lambda)$ admits a natural $t$-structure:
\begin{prop}\label{t_structure_on_et_sheaves}
We have:
\begin{enumerate}[label=(\arabic*)]
\item Let $\mstack X$ be a prestack. Then the category $\Shv_{\et}(\mstack X, \Lambda)$ admits a $t$-structure such that $\mathcal F \in \Shv_{\et}^{\le 0}(\mstack X, \Lambda)$ (resp. $\mathcal F \in \Shv_{\et}^{\ge 0}(\mstack X, \Lambda)$) if and only if for each map $T \in \Aff_{R/\mstack X}$ the restriction $\mathcal F_{|T}$ lands to $\Shv_{\et}^{\le 0}(T, \Lambda)$ (resp. $\Shv_{\et}^{\ge 0}(T, \Lambda)$).

\item This $t$-structure is right complete (see \Cref{defn:t_complete}).

\item Let $f\colon \mstack X \to \mstack Y$ be a morphism of prestacks. Then the pullback functor $f^{-1}$ is $t$-exact and $f_*$ being its right adjoint is left $t$-exact.
\end{enumerate}

\begin{proof}
First point follows from \Cref{lem: t-structure in the limit}. The second assertion follows from \Cref{limits_of_t_structures} and the fact that since filtered colimits are exact in $\Shv_\et(T, \Mod_\Lambda)$, the derived category $\Shv_\et(T, \Lambda)$ is right $t$-complete for all affine scheme $T$. The last statement follows from $t$-exactness of the translation morphisms in the diagram defining $\Shv_\et$.
\end{proof}
\end{prop}
\begin{rem}
The left $t$-completeness of $\Shv_\et(\mstack X, \Lambda)$ is subtle, unless $\mstack X$ has bounded $\Lambda$-cohomological dimension (see \cite[Section 1.3.3]{Lur_SAG}). See \cite{AkhilClausen_Hyperdescent} for the systematic study of this question.
\end{rem}

\begin{rem}[Coskeletal hyperdescent] \label{rem: hyperdescent for etale cohomology}
Recall that a hypercover $X_\bullet$ is called \emdef{$n$-coskeletal} if it is right Kan extended from $\Delta_{\le n}^\op$. Note that in any topos for $n$-skeletal hypercover $X_\bullet$ the natural map $|X_\bullet| \to X_{-1}$ is an equivalence (this follows by induction from the case $n=0$ which is \v Cech decent, see \cite[Corollary 6.2.3.5]{Lur_HTT}). In particular, if $p_\bullet\colon \mstack U_\bullet \ra \mstack X$ is an $n$-coskeletal hypercover of prestacks in \'etale topology, then 
$$
\RG_\et(\mstack X,\mc F) \xymatrix{\ar[r]^\sim &} \Tot \RG_\et(\mstack U_\bullet,p_\bullet^{-1}\mc F).$$	
\end{rem}

\smallskip
As usual we will denote by $\Shv_{\et}^+(\mstack X, \Lambda)$ the union of $\Shv_{\et}^{\ge n}(\mstack X, \Lambda)$ for $n\in\mbb Z$. This is the category of \textit{bounded below \'etale sheaves}.

We now pass to a more detailed discussion of the pushforward functor $f_*$ for a morphism $f\colon \mstack X\ra \mstack Y$. Though we were able to define it formally, we do not yet know how to "compute" it: namely to describe its values at various points $x\colon T\ra \mstack Y$. Using descent it is enough to restrict only to those $x$ that are smooth. In the case of schemes such a description then is given by smooth base change:
\begin{thm}[Smooth base change, {\cite[Expos\'e XVI, Th\'eor\`em 1.1]{SGA4_3}}]\label{smooth_bc_schemes}
Let
$$\xymatrix{
X_2 \ar[r]^q\ar[d]^g & X_1 \ar[d]^f \\
Y_2 \ar[r]^p & Y_1
}$$
be a pullback square of schemes such that $p$ is smooth and $f$ is quasi-compact and quasi-separated. Then the natural map of functors $p^{-1}\circ f_* \to g_* \circ q^{-1}$ renders the diagram below commutative
\begin{align*}
\xymatrix{
\Shv_{\et}^+(X_2, \Lambda) \ar[d]^{g_*} && \ar[ll]_-{q^{-1}} \Shv_{\et}^+(X_1, \Lambda) \ar[d]^{f_*}\\
\Shv_{\et}^+(Y_2, \Lambda) && \ar[ll]_{p^{-1}} \Shv_{\et}^+(Y_1, \Lambda).
}
\end{align*}
\end{thm}
The rest of this section will be devoted to the proof of the analogous statement for Artin stacks. We will more or less formally reduce to the case of schemes, details of the argument are provided below. During the proof we will need the following notion and its basic properties from \cite[Section 4.7.4]{Lur_HA}:
\begin{defn}\label{def:adjointable_sq}
A commutative diagram of categories
$$\xymatrix{
\fcat C_2 & \ar[l]_q \fcat C_1 \\
\fcat D_2 \ar[u]_{F_2} & \fcat D_1\ar[u]_{F_1}\ar[l]_p
}$$
is called \emdef{right adjointable} if both $F_1, F_2$ admit right adjoints $G_1,G_2$ and the natural map
$$p\circ G_1 \xymatrix{\ar[r] &} G_2 \circ F_2 \circ p \circ G_1 \simeq G_2 \circ q \circ F_1 \circ G_1 \xymatrix{\ar[r] &} G_2 \circ q$$
is an equivalence.
\end{defn}

We will also need some auxiliary results about interaction of totalization with a $t$-structure.
\begin{lem}\label{totalization_of_coconectives}
Let $\fcat C$ be a stable category with a $t$-structure. Let $X^\bullet$ be a cosimplicial object of $\fcat C$ such that each term $X^i$ lies in $\fcat C^{\ge 0}$. Then the fiber of the natural map $\Tot(X^\bullet) \to \Tot^{\le n}(X^\bullet)$ lies in $\fcat C^{>n}$.

\begin{proof}
The fiber $\Tot^{n+1}(X^\bullet)$ of $\Tot^{\le n}(X^\bullet) \to \Tot^{\le n+1}(X^\bullet)$ is a retract of $X^{n+1}[-n-1]$. Since $X^{n+1}\in \fcat C^{\ge 0}$, we find that $\Tot^{n+1}(X^\bullet) \in \fcat C^{> n}$. It follows the fiber $F_n$ of the natural map $\Tot(X^\bullet) \to \Tot^{\le n}(X^\bullet)$ admits a decreasing filtration with associated graded pieces $\Tot^i(X^\bullet), i>n$ from $\fcat C^{>n}$. Since $\fcat C^{>n}$ is closed under extensions and limits, $F_n\in \fcat C^{>n}$.
\end{proof}
\end{lem}
\begin{cor}\label{left_exact_preserve_totalizations_of_uniformly_bounded_below}
Let $F\colon \fcat C \to \fcat D$ be a left $t$-exact functor between stable categories equipped with $t$-structures. Then if $\fcat D$ is right $t$-complete, then the natural map $F(\Tot X^\bullet) \to \Tot F(X^\bullet)$ is an equivalence for any cosimplicial object $X^\bullet \in \fcat C^\Delta$ such that $X^n \in \fcat C^{\ge 0}$ for all $n \in \mathbb Z_{\ge 0}$.

\begin{proof}
By the previous proposition the natural map $\Tot(X^\bullet) \to \Tot^{\le n}(X^\bullet)$ is $n$-coconnective. Since $F$ is left $t$-exact, the same holds for $F(\Tot(X^\bullet)) \to F(\Tot^{\le n}(X^\bullet)) \simeq \Tot^{\le n}(F(X^\bullet))$. It follows that the limit map $F(\Tot(X^\bullet)) \to \Tot(f(X^\bullet))$ is $\infty$-coconnective and hence is an equivalence by assumption on $\fcat D$.
\end{proof}
\end{cor}

This leads to the following corollary for the pull-back functor $f^{-1}\colon \Shv_{\et}(\mstack Y, \Lambda)\ra \Shv_{\et}(\mstack X, \Lambda)$:
\begin{cor}\label{pullback_preserves_totalizations_of_coconnectives}
Let $f\colon \mstack X \to \mstack Y$ be a map between $R$-prestacks. Then the functor $f^{-1}\colon \Shv_{\et}(\mstack Y, \Lambda)\ra \Shv_{\et}(\mstack X, \Lambda)$ preserves totalizations of uniformly bounded below sheaves.
\end{cor}

\smallskip We are now ready to prove the following:
\begin{prop}[Smooth base change for Artin stacks]\label{et_smooth_basechange_algebraic}
Let
$$\xymatrix{
\mstack X_2 \ar[r]^q\ar[d]^g & \mstack X_1 \ar[d]^f \\
\mstack Y_2 \ar[r]^p & \mstack Y_1
}$$
be a pullback square of Artin stacks such that $f$ is a quasi-compact quasi-separated (see Definitions \ref{quasicompact_Artin} and \ref{quasiseparated_Artin}) and $p$ is smooth. Then the following diagram is right adjointable
\begin{align}\label{etale_bc_eq1}
\xymatrix{
\Shv_{\et}^+(\mstack X_2, \Lambda) && \ar[ll]_-{q^{-1}} \Shv_{\et}^+(\mstack X_1, \Lambda) \\
\Shv_{\et}^+(\mstack Y_2, \Lambda) \ar[u]_{g^{-1}} && \ar[ll]_-{p^{-1}} \Shv_{\et}^+(\mstack Y_1, \Lambda) \ar[u]_{f^{-1}}.
}
\end{align}

\begin{proof}
We first treat a special case when $\mstack Y_1$ is affine, then deduce the general statement.
\begin{enumerate}[wide, topsep=\parskip, parsep=\parskip, itemsep=\parskip, label=\underline{Step \arabic*.}]
\item Assume that $\mstack X_1, \mstack X_2, \mstack Y_2$ are arbitrary Artin stacks, but $\mstack Y_1$ is an affine scheme. Let $\mstack X_1$ be $k$-Artin. We will prove the statement by induction on $k$. Since the question is flat-local on $\mstack Y_2$, we can assume that $\mstack Y_2$ is an affine scheme. Note that in the $k=-1$ case by quasi-compactness of $f$ we know that $\mstack X_1$ is an affine scheme (as opposed to a possibly infinite disjoint union of affines). Hence the base of induction $k=-1$ follows from \Cref{smooth_bc_schemes}.

For $k\ge 0$, by quasi-compactnesson of $f$, there exists a smooth affine atlas $U\surj \mstack X_1$, and, by quasi-separatedness of $f$, all terms in the corresponding \v Cech nerve $U_\bullet$ are quasi-compact and quasi-separated. Set $V:=\mstack X_2\times_{\mstack X_1} U$ and let $V_\bullet$ be the \v{C}ech nerves of $V\surj \mstack X_2$. Also let $f_i\colon U_i \to \mstack X_1 \xrightarrow{f} \mstack Y_1$, $g_i\colon V_i \to \mstack X_2 \xrightarrow{g} \mstack Y_2$ and $q_i\colon V_i\to U_i$ be the canonical maps (in particular $q_{-1} = q$):
$$\xymatrix{
V_i \ar[r]^{q_i} \ar[d]\ar@/_1pc/[dd]_{g_i} & U_i \ar[d] \ar@/^1pc/[dd]^{f_i}\\
\mstack X_2 \ar[r]^q \ar[d]^g & \mstack X_1\ar[d]_f \\
\mstack Y_2 \ar[r]^p & \mstack Y_1.
}$$
Let $\mathcal F$ be a bounded below \'etale sheaf on $\mstack X_1$. Then by descent and the fact that $f_*$ and $g_*$ (being right adjoints) commute with limits, we have
$$p^{-1} f_*\mathcal F \simeq p^{-1}\Tot(f_{\bullet*} \mathcal F^\bullet) \qquad\text{and}\qquad g_*(q^{-1} \mathcal F) \simeq \Tot(g_{\bullet*} (q^{-1} \mathcal F)^\bullet) \simeq \Tot(g_{\bullet*} q_\bullet^{-1} \mathcal F^\bullet).$$
Now since $\mstack X_1$ is $k$-Artin, all elements in the \v{C}ech nerves $U_\bullet, V_\bullet$ are $(k-1)$-Artin. It follows by induction that
$$\Tot(g_{\bullet*} q_\bullet^{-1} \mathcal F^\bullet) \simeq \Tot(p^{-1}f_{\bullet*}\mathcal F^\bullet).$$
Hence it is enough to prove that the natural map
$$p^{-1}\Tot(f_{\bullet*} \mathcal F^\bullet) \xymatrix{\ar[r] &} \Tot(p^{-1}f_{\bullet*}\mathcal F^\bullet)$$
is an equivalence. But since the pushforward functors are right $t$-exact, the simplicial diagram $f_{\bullet*}\mathcal F^\bullet$ is uniformly bounded below, so we conclude by \Cref{pullback_preserves_totalizations_of_coconnectives}.

\item Now assume that $\mstack X_1, \mstack X_2, \mstack Y_1, \mstack Y_2$ are arbitrary Artin stacks. Let $\chi\colon \Delta^1\times \Delta^1 \to \Stk$ denote the square \eqref{etale_bc_eq1}. We want to prove that the diagram of categories $\Shv_\et^{+}(\chi)^{-1}$ is right adjointable. By co-Yoneda lemma $\mstack Y_1 \simeq \colim S_\alpha$, where $S_\alpha$ are affine schemes. Let $\chi_\alpha \coloneqq \chi \times_{\mstack Y_1} S_\alpha$. Since colimits in topoi are universal (see \cite[Theorem 6.1.0.6]{Lur_HTT}) the natural map $\colim \chi_\alpha \to \chi$ is an equivalence. But by the previous step all squares $\Shv_\et^{+}(\chi_\alpha)^{-1}$ are right adjointable. Since by \cite[Corollary 4.7.4.18 (2)]{Lur_HA} the category of right adjointable diagrams is closed under limits, the square $\Shv_\et^+(\chi)^{-1} \simeq \lim_\alpha \Shv_\et^+(\chi_\alpha)^{-1}$ is also right adjointable.\qedhere
\end{enumerate}
\end{proof}
\end{prop}

\begin{ex}
	Let $f\colon \mstack X\ra \mstack Y$ be a qcqs schematic map of Artin stacks, let $\mc F\in \Shv_{\et}(\mstack X,\Lambda)$ be a sheaf and consider a smooth atlas $p\colon U \ra \mstack X$ which we assume is also schematic. Then we get that the pull-back $p^{-1}f_*\mc F$ can be described as $g_*q^{-1}\mc F$ where $$
	\xymatrix{
		\mstack X\times_{\mstack Y} U \ar[r]^q\ar[d]^g & \mstack X \ar[d]^f \\
		U \ar[r]^p & \mstack Y.
	}
	$$ Note that the fiber product $V\coloneqq \mstack X\times_{\mstack Y} U$ is represented by a (non-derived) scheme since $p$ is smooth and that $g_*\colon \Shv_{\et}(V,\Lambda) \ra \Shv_{\et}(U,\Lambda)$ is the usual pushforward for (the derived categories) of \'etale sheaves on schemes. 
\end{ex}

\subsection{Local systems and constructible sheaves}\label{subsect:local systems}
Consider a map of schemes $f\colon X\ra Y$ and an \'etale sheaf $\mc F$ on $X$, and let's say that we are interested in computing the global sections $\RG_\et(X,\mc F)$. Then it be can helpful to look at $\RG_\et(Y, f_*\mc F)$ instead: namely one gets a filtration on $\RG_\et(X,\mc F)$ with the associated graded given by $\bigoplus_i \RG_\et(Y,R^if_*F)[-i]$. The situation is particularly nice if $\mc F$ is a local system and $f$ is smooth and proper: then the individual cohomology sheaves $R^if_*\mc F$ are also local systems. If $Y$ is (\'etale) simply-connected, then all local systems are trivial and the sum above is given by the global sections of constant sheaves (with some cohomological shifts). We would like to generalize this picture to the context of Artin stacks.

Let $R$ and $\Lambda$ be as in the previous section. First recall the following standard notions:
\begin{defn}\label{def: etale local systems for schemes}
Let $X$ be a scheme. An (abelian) \'etale sheaf $\mathcal F \in \Shv_\et^\heartsuit(X, \Lambda)$ of $\Lambda$-modules is called
\begin{itemize}
\item an \emdef{\'etale local system} or an \emdef{\'etale locally constant sheaf} if there exists an \'etale covering $\{T_i\}$ of $X$ such that for all $i$ the restriction $\mathcal F_{|T_i}$ is a constant sheaf of finitely generated $\Lambda$-modules.

\item a \emdef{constructible sheaf} if for each affine open $U$ of $X$ there exists a finite stratification $\eset = U_0 \subseteq U_1 \subseteq \ldots \subseteq U_n = U$ by quasi-compact open subsets $U_i$ such that the restriction of $\mathcal F$ on each $U_i\setminus U_{i-1}$ is an \'etale locally constant sheaf of finitely generated $\Lambda$-modules.
\end{itemize}
An object $\mathcal F$ of the derived category $\Shv_\et(X, \Lambda)$ is called a local system (resp. constructible sheaf) if all of its cohomology sheaves $\mathcal H^i(\mathcal F)\in \Shv_\et^{\heartsuit}(X, \Lambda)$ are local systems (resp. constructible sheaves).
\end{defn}

We extend this definition to Artin stacks in a natural way. We restrict to the bounded below setting from the start.
\begin{defn}
Let $\mstack X$ be an Artin stack. A sheaf $\mathcal F \in \Shv_\et^+(\mstack X, \Lambda)$ is called
\begin{itemize}
\item a \emdef{bounded below \'etale local system} if for every affine scheme $T$ mapping to $\mstack X$ all cohomology sheaves $\mathcal H^i(\mathcal F_{|T}) \in \Shv_\et^\heartsuit(T, {\Lambda})$ are local systems. We will denote the full subcategory of such sheaves by $\mathscr L_{\et}^+(X, \Lambda)$.

\item \emdef{bounded below constructible} if for every affine scheme $T$ mapping to $\mstack X$ all cohomology sheaves $\mathcal H^i(\mathcal F_{|T}) \in \Shv_\et^\heartsuit(T, {\Lambda})$ are constructible. We will denote the full subcategory of \'etale sheaves on $\mstack X$ spanned by constructible sheaves by $\Shv^+_{\et, \constr}(\mstack X, \Lambda)$.
\end{itemize}
Both $\mathcal L_\et^+(\mstack X, \Lambda)$ and $\Shv^+_{\et, \constr}(\mstack X, \Lambda)$ are stable subcategories of $\Shv_\et(\mstack X, \Lambda)$ closed under retracts and extensions and we have a chain of full embeddings 
$$\mathcal L_\et^+(\mstack X, \Lambda) \subseteq \Shv^+_{\et, \constr}(\mstack X, \Lambda) \subseteq \Shv_\et^+(\mstack X, \Lambda).$$
Both $\mathscr L_{\et}^+(\mstack X, \Lambda)$ and $\Shv^+_{\et, \constr}(\mstack X, \Lambda)$ admit natural $t$-structures determined by 
$$\mathscr L_{\et}^{\ge 0}(\mstack X, \Lambda)\coloneqq \Shv_{\et}^{\ge 0}(\mstack X, \Lambda)\cap \mathscr L_{\et}^+(\mstack X, \Lambda) \qquad \Shv_{\et, \constr}^{\ge 0}(\mstack X, \Lambda)\coloneqq \Shv_{\et}^{\ge 0}(\mstack X, \Lambda)\cap \Shv^+_{\et, \constr}(\mstack X, \Lambda).$$
\end{defn}

\smallskip One can characterize $\mathscr L_{\et}^+(\mstack X, \Lambda)$ and $\mathscr L_{\et}^{\ge 0}(\mstack X, \Lambda)$ in terms of a surjective cover:
\begin{lem}\label{local_sys_via_atlas}
Let $p \colon \mstack U\to \mstack X$ be a surjective morphism of Artin stack. Then
\begin{enumerate}
\item An \'etale sheaf $\mathcal F\in \Shv_{\et}(\mstack X, \Lambda)$ is a bounded below local system (resp. constructible sheaf) if and only if $p^{-1}(\mathcal F)$.

\item An \'etale local system $\mathcal L$ lies in $\mathscr L_{\et}^{\ge 0}(\mstack X, \Lambda)$ (resp. $\mathscr L_{\et}^{\le 0}(\mstack X, \Lambda)$) if and only if $p^{-1}(\mathcal L) \in \mathscr L_{\et}^{\ge 0}(\mstack U, \Lambda)$ (resp. $p^{-1}(\mathcal L) \in \mathscr L_{\et}^{\le 0}(\mstack U, \Lambda)$). Similar statement holds for the bounded below constructible categories.
\end{enumerate}

\begin{proof}
We will only prove the first part for the local system case, the proof of other assertions is similar. Let $X \to \mstack X$ be a map, where $X$ is a disjoint union of affines. Then, since $\mstack U\surj \mstack X$ is surjective, for some \'etale cover $V \to X$ there exists a dashed arrow making the diagram below commutative
$$\xymatrix{
V \ar@{-->}[r]\ar@{->>}[d] & \mstack U \ar@{->>}[d]\\
X \ar[r] & \mstack X
}$$
Then $\mathcal F_{|V}$ is a local system because it is a local system already restricted on $U$. On the other hand since the property of being a local system on a scheme can be checked  \'etale locally, we conclude that $\mathcal F_{|X}$ is a local system.
\end{proof}
\end{lem}
\begin{cor}\label{functoriality_of_loc_sys}
Let $f\colon \mstack X \to \mstack Y$ be a morphism of Artin stacks. Then:
\begin{enumerate}
\item The pullback functor $f^{-1}\colon \Shv_{\et}(\mstack Y, \Lambda) \to \Shv_{\et}(\mstack X, \Lambda)$ preserves the (bounded below) constructible subcategory and local systems.

\item If $f$ is schematic proper (see \Cref{local_on_target_examples}) the pushforward functor $f_*\colon \Shv_\et(\mstack X, \Lambda) \to \Shv_\et(\mstack Y, \Lambda)$ preserves the (bounded below) constructible subcategory. In particular, $f_{*| \Shv_{\et, \constr}^+(\mstack X, \Lambda)}$ is right adjoint to $f^{-1}_{|\Shv_{\et, \constr}^+(\mstack Y, \Lambda)}$.

\item Let $f\colon \mstack X \to \mstack Y$ be an arbitrary morphism between stacks of finite type (see \Cref{def:finite_type}) over a regular base scheme of dimension $\le 1$. Then the pushforward functor $f_* \colon \Shv_{\et}(\mstack X, \Lambda) \to \Shv_\et(\mstack Y, \Lambda)$ preserves the (bounded below) constructible subcategory.

\item If $f$ is smooth and schematic proper then $f_*$ preserves the subcategory of bounded below local systems. In particular, $f_{*| \mathcal L_{\et}^+(\mstack X, \Lambda)}$ is right adjoint to $f^{-1}_{|\mathcal L_{\et}^+(\mstack Y, \Lambda)}$.
\end{enumerate}

\begin{proof}
\begin{enumerate}[wide,itemsep=\parskip]
\item Let $U\surj \mstack Y$ be a smooth atlas and consider the following commutative diagram
$$\xymatrix{
V \ar@{->>}[rd]\ar@/^/[rrd] \ar@/_/[rdd]_p & & \\
& \mstack X\times_{\mstack Y} U \ar[r]\ar@{->>}[d] & U \ar[d] \\
& \mstack X \ar@{->>}[r] & \mstack Y,
}$$
where $V \surj \mstack X\times_{\mstack Y} U$ is a smooth cover by a union of affine schemes, which exists since a fibered product of Artin stacks is again an Artin stack. Now for any $\mathcal F \in \mathscr L_{\et}^+(\mstack Y, \Lambda)$ (resp. in $\Shv_{\et, \constr}^+(\mstack Y, \Lambda)$) we have that $\mathcal F_{|U}$ is a local system (resp. constructible) and hence $(f^{-1}\mathcal F)_{|V} \simeq p^{-1}\mathcal F_{|U}$ is a local system (resp. constructible) as well. We conclude by \Cref{local_sys_via_atlas}.

\item Let $k$ be such that $\mstack X, \mstack Y$ are at most $k$-Artin. We will prove the statement by induction on $k$, the base of induction $k=0$ being \cite[Expos\'e XIV, Th\'eor\`em 1.1]{SGA4_3}. Assume the statement is proved for $k^\prime < k$. Let us take a smooth atlas $U\surj \mstack Y$. By \Cref{local_sys_via_atlas} it is enough to prove that the restriction $(f_*\mathcal \mathcal F)_{|U}$ is constructible. Consider the fiber square
$$\xymatrix{
\mstack X_U\ar[r]\ar[d]^g & \mstack X \ar[d]^f \\
U \ar@{->>}[r] & \mstack Y.
}$$
By the base change we have $(f_*\mathcal \mathcal F)_{|U} \simeq g_*(\mathcal F_{|V})$. By the equivalence above it is enough to prove $g_*\mathcal G$ is constructible for $\mathcal G := \mathcal F_{|V}$. Let $V\surj \mstack X_U$ be an atlas, $V_\bullet$ be the corresponding \v{C}ech nerve, and $g_i\colon V_i\to \mstack X_U \xrightarrow{g} U$ be the canonical morphisms (note that all of $V_i$ are at most $(k-1)$-Artin stacks).  Then by descent we have $g_*(\mathcal G) \simeq \Tot(g_{\bullet*}\mathcal G^\bullet)$. By induction assumption all $g_{i*}\mathcal G^i$ are constructible. By shifting if necessary we can assume $\mathcal F \in \Shv_{et, \constr}^{\ge 0}(\mstack X_U, \Lambda)$ and thus so are all the terms in the cosimplicial diagram $g_{\bullet*}(\mathcal G^\bullet)$. Hence by \Cref{pullback_preserves_totalizations_of_coconnectives} we have 
$$\mathcal H^n(g_*\mathcal G) \simeq \mathcal H^n(\Tot^{\le n}(g_{\bullet*}\mathcal G^\bullet)).$$
But since $\Tot^{\le n}$ is a finite limit, all $\Tot^{\le n}(g_{\bullet*}\mathcal G^\bullet)$ are constructible sheaves. It follows that all cohomology sheaves of $g_*\mathcal G$ are constructible, and so $g_*\mathcal G$ is constructible as desired.

\item Same reduction to the case of schemes, which is covered by \cite[Chapitre 7, Th\'eor\`eme 1.1]{SGA4_12}.

\item Again the same reduction as in the second point, but using \cite[Expos\'e XVI, Corollaire 2.2]{SGA4_3} for the case of schemes.\qedhere
\end{enumerate}
\end{proof}
\end{cor}

From \Cref{alg_etale_shaeves_basics} and \Cref{local_sys_via_atlas} we obtain the following corollary:
\begin{cor}
Let $\mstack U \surj \mstack X$ be a surjection of Artin stacks and let $\mstack U_\bullet$ be the corresponding \v Cech diagram. Then the natural functors
$$\Shv_{\et, \constr}^+(\mstack X, \Lambda)\xymatrix{\ar[r] &} \Tot \Shv_{\et, \constr}^+(\mstack U_\bullet, \Lambda) \qquad\qquad \mathcalr L_{\et}^+(\mstack X, \Lambda) \xymatrix{\ar[r] &} \Tot \mathcalr L_{\et}^+(\mstack U_\bullet, \Lambda)$$
are $t$-exact equivalences.
\end{cor}

The following result describes the heart of the $t$-structure on $\Shv_{\et}(-, \Lambda)$ in the simply-connected case. From now on we assume $R=C$ is an algebraically closed field.
\begin{defn}
A prestack $p\colon \mstack X\to\Spec C$ is called \emdef{$\Lambda$-\'etale connected} if the natural map 
$$\Mod_\Lambda^{\fg} \simeq \xymatrix{\mathscr L_{\et}^\heartsuit(\Spec C, \Lambda) \ar[r]^-{p^{-1}} & \mathscr L_{\et}^\heartsuit(\mstack X, \Lambda)}$$
is fully faithful. A prestack $\mstack X$ is called \emdef{$\Lambda$-\'etale simply connected} if the functor above is an equivalence.
\end{defn}
\begin{ex}
Let $\Lambda \neq 0$. Then a scheme $X$ over $C$ is $\Lambda$-\'etale connected if and only if it is connected in Zariski topology and it is $\Lambda$-\'etale simply connected if and only if $\pi_1^{\et}(X, x) \simeq *$ for any (equivalently for all) geometric point $x$ in $X$.
\end{ex}
\begin{prop}\label{simply_connected}
Let $\mstack U \surj \mstack X$ be a surjective map between Artin stacks and let $p\colon \mstack X \to \Spec C$ be a structure map.
\begin{enumerate}
\item If $\mstack U$ is $\Lambda$-\'etale connected, then the functor of a constant local system
$$\Mod_\Lambda^\fg \simeq \xymatrix{\mathscr L_{\et}^\heartsuit(\Spec C, \Lambda) \ar[r]^-{p^{-1}} & \mathscr L_{\et}^\heartsuit(\mstack X, \Lambda)}$$
is fully-faithful, i.e. $\mstack X$ is $\Lambda$-\'etale connected as well.

\item If $\mstack U_0 := \mstack U$ is $\Lambda$-\'etale simply connected and $\mstack U_1 := \mstack U\times_{\mstack X} \mstack U$ is $\Lambda$-\'etale connected, then $p^{-1}$ is an equivalence, i.e. $\mstack X$ is $\Lambda$-\'etale simply connected.
\end{enumerate}

\begin{proof}
\begin{enumerate}[wide, topsep=\parskip, itemsep=\parskip, parsep=\parskip]\setlength\itemindent{\labelwidth+\labelsep}

\item Since $\mathcal L^\heartsuit(\mstack X, \Lambda)$ is a $1$-category, by descent we know that the pullback functor $\mathcal L^\heartsuit(\mstack X, \Lambda) \to \mathcal L^\heartsuit(\mstack U, \Lambda)$ is faithful. Since the composition
$$\Mod_\Lambda^\fg \simeq \xymatrix{\mathscr L_{\et}^\heartsuit(\Spec C, \Lambda) \ar[r]^-{p^{-1}} & \mathscr L_{\et}^\heartsuit(\mstack X, \Lambda) \ar[r] & \mathscr L_{\et}^\heartsuit(\mstack U, \Lambda)}$$
is fully-faithful, $p^{-1}$ is fully-faithful as well.

\setlength\itemindent{\parindent}\item By the previous part we know that the pullback $p^{-1}\colon \Mod_\Lambda^\fg \to \mathscr L(\mstack X, \Lambda)^\heartsuit$ is fully-faithful, hence it is enough to prove essential surjectivity. For this end let $\mathcal L$ be a local system on $\mstack X$. We will prove that $\mathcal L$ is constant by descent. Since $\mstack U_0$ is simply connected the restriction $\mathcal L_{|\mstack U_0} \simeq \underline M$ for some $M \in \Mod_\Lambda^\fg$. Let $q_1, q_1\colon \mstack U_1 \to \mstack U_0$ be natural projections. Since $\mstack U_1$ is \'etale connected, the descent datum $\alpha\colon q_0^{-1} \underline M \areq q_1^{-1} \underline M$ is pulled back from $\Spec C$, say $\alpha = \underline \phi$ for some $\phi \colon M\areq M$. We want to prove $\phi = \Id_M$.

Let $q_{01}, q_{02}, q_{12}\colon \mstack U_2 := \mstack U \times_{\mstack X} \mstack U \times_{\mstack X} \mstack U \to \mstack U\times_{\mstack X} \mstack U$ be natural projections. The cocylce condition then reads as
$$\underline \phi = q_{02}^{-1}(\underline \phi) = q_{12}^{-1}(\underline \phi) \circ q_{01}^{-1}(\underline \phi) = \underline \phi \circ \underline \phi.$$
Since $\underline \phi$ is invertible it follows that $\underline \phi = \Id_{\underline M}$. Let $L$ be an algebraically closed field extension of $C$ such that $\mstack U_2(L) \not \simeq \eset$. Then since the pullback functor $\mathcal L^\heartsuit_\et(\Spec C, \Lambda) \to \mathcal L^\heartsuit_\et(\Spec L, \Lambda)$ is faithful (in fact an equivalence), the constant local system functor $\mathcal L_\et^\heartsuit(\Spec C, \Lambda) \to \mathcal L_\et^\heartsuit(\mstack U_2, \Lambda)$ is faithful as well. Hence $\phi = \Id_M$.\qedhere
\end{enumerate}
\end{proof}
\end{prop}
\begin{ex}\label{ex: BG is etale simply-connected if G is connected}
Let $G$ be a connected $C$-group scheme. Then the \v{C}ech nerve $\Spec C \to BG$ satisfies conditions of the second part of the proposition above, hence $BG$ is \'etale simply connected. If $G$ is abelian, then by induction $B^n G$ are \'etale simply connected for all $n\ge 1$.
\end{ex}


\subsection{Rigid analytic \'etale sheaves}\label{adic_local_systems}
For the rest of this subsection we fix the following notation:
\begin{notation}
Let $K$ be a complete non-Archimedean field with the ring of integers $\mathcal O_K \coloneqq \{f\in K \mid |f| \le 1\}$. We fix a pseudo-uniformizer $\pi \in \mathcal O_K$, i.e. an element with the property $0 < |\pi| < 1$.
\end{notation}
In this subsection for a $K$-rigid analytic stack $\mstack X$ (see \Cref{defn:rigid_an_stacks} below) we construct the corresponding category of \'etale sheaves on $\mstack X$ following a path similar to \Cref{sect:etale_sheaves}. The resulting category $\Shv_{\et}(\mstack X,  \Lambda)$ behaves very similarly to what we had in the algebraic setting. If $\mstack X\simeq {(\widehat{\mstack X_{\mc O_K}})}_K $ comes as the Raynaud generic fiber of an Artin $\mathcal O_K$-stack $\mstack X_{\mathcal O_K}$, we also discuss the relationship of $\Shv_\et(\mstack X, \Lambda)$ with the category $\Shv_\et(\mstack X_K, \Lambda)$ of \'etale sheaves on the algebraic generic fiber $\mstack X_K\coloneqq \mstack X_{\mathcal O_K}\times_{\Spec \mathcal O_K} \Spec K$. First recall the following definition:
\begin{defn} A \emdef{Tate algebra} $K\langle x_1,\ldots , x_n \rangle$ is the $K$-subalgebra of the ring $K[[x_1,\ldots,x_n]]$ of formal power series consisting of those $f=\sum_I f_Ix^I$ such that $\lim_{I\ra \infty}|f_I|\ra 0$. The ring $K\langle x_1,\ldots , x_n \rangle$ has the natural \emdef{Gauss norm} given by $\|f\|\coloneqq \max_I |f_I|$, making it a Banach algebra. \emdef{An affinoid $K$-algebra} $A$ by definition is a Banach algebra that is a quotient of some $K\langle x_1,\ldots , x_n \rangle$ by a closed ideal $I$. A \emdef{morphism of affinoid algebras} is a homomorphism that is bounded. 
\end{defn}
The Tate algebra $K\langle x_1,\ldots , x_n \rangle$ can be thought of as the algebra of analytic functions on the closed polydisk $\mathbb D_1^n\subset K^n$ and affinoid algebras are supposed to be rings of analytic functions on closed subvarieties in those.
We denote the opposite category of affinoid $K$-algebras by $\Afd_K$; this is the category of \emdef{affinoid spaces over} $K$. We refer to \cite{Berkovich} or \cite{Bosch_RigidAn} (and partially \cite{Huber_adicSpaces}) for the foundations of rigid analytic geometry which is built up starting from this notion. We denote the category of rigid analytic spaces over $K$ by $\mr{Rig}_K$. We have an embedding $\Afd_K\ra \mr{Rig}_K$, given by $A\mapsto \Sp(A)$. We put $\mbb B^n_K\coloneqq \Sp K\langle x_1,\ldots , x_n \rangle$, a closed rigid analytic polydisk of radius $1$.

One has a notion of an \'etale map and \'etale cover (the latter being a little subtle and using a concept of admissibility to ensure some quasicompactness properties) of affinoid spaces.
\begin{defn}[{\cite[Definition 2.1.1]{Huber_adicSpaces}}]
A morphism of rigid spaces $f\colon X \to Y$ is called \emdef{\'etale} if for any point $x\in X$ the induced map of local rings $\mathcal O_{Y,f(x)} \to \mathcal O_{X, x}$ is flat and unramified. An \'etale morphism $f\colon X \to Y$ is called \emdef{strongly surjective} if for any quasi-compact admissible open subset $V$ of $Y$ there exists a finite collection of admissible open subsets $U_1, \ldots , U_n$ of $X$ such that $V = f(\bigcup_i U_i)$. For a rigid analytic space we define a \emdef{rigid analytic \'etale site $\ET_{/X}^\an$} to be the category of rigid analytic spaces \'etale over $X$ equipped with the Grothendieck topology generated by strongly surjective \'etale morphisms.
\end{defn}

There is also notion of a smooth map:
\begin{defn}[{\cite[Definition 2.1]{BoschLutkebohmertRaynaud_FaRGIII}}]
The map $f\colon X \to Y$ of affinoids is called \emdef{smooth} if for any point $x\in X$ there is a neighborhood $V\subset X$ such that $f_{|V}$ can be factored as
$$\xymatrix{
V\ar@{-->}[r]^(0.4){g}\ar[dr]& Y\times\mbb B^n_K\ar[d]^p\\
& Y
}$$
with $g$ \'etale and $p$ the natural projection.
\end{defn}

Obviously, \'etale maps are smooth and we get a well-defined geometric context ($\Afd_K^\op$, \'etale, smooth). Following \cite{PortaYu_GAGA} we introduce rigid analytic (pre)stacks as follows:
\begin{defn}\label{defn:rigid_an_stacks}
Let $K$ be a complete non-Archimedean field. A category of \emph{locally finite type $K$-analytic prestacks $\PStk_K^{\an,\lft}$} is defined as a category of presheaves on the category of affinoid $K$-analytic spaces $\Afd_K$:
$$\PStk_K^{\an,\lft} := \Fun(\Afd_K^\op, \Type).$$
The category $\Stk_K^{\an,\lft}$ of \emdef{locally finite type $K$-rigid analytic Artin stacks} is defined as the full subcategory of $\PStk_K^{\an,\lft}$ consisting of geometric stacks (see \Cref{def:geometric_stacks}) with respect to the context ($\Afd_K^\op$, \'etale, smooth).
\end{defn}

As promised, now we define the category of \'etale sheaves on a rigid-analytic stack. Given an affinoid $T\in \Afd_K$,  for a commutative ring $\Lambda$ we will denote the corresponding $(\infty, 1)$-sheaf category for $\ET_{/T}^{\mr{an}}$ by $\Shv_\et(T, \Lambda)$. Given a map $f\colon T_1\ra T_2$ we have the pull-back functor $f^{-1}\colon \Shv_{\et}(T_2, \Lambda)\ra \Shv_{\et}(T_1, \Lambda)$. 
\begin{construction}
Let $\Lambda$ be a ring. We define the functor of rigid analytic \'etale sheaves as the right Kan extension of the functor
$$\Shv_\et(-, \Lambda)^{-1}\colon \Afd_K^\op \xymatrix{\ar[r] &} {\Prs}^{\mathrm L}_\Lambda$$
along the Yoneda embedding $\Afd^\op_K \inj {(\PStk^{\an, \lft}_K)}^\op$.
\end{construction}
The proof of the following properties is similar to that for the \'etale sheaves in the algebraic context (see  \Cref{etale_sheaves_on_stacks}):
\begin{prop}
In the notation above we have:
\begin{enumerate}
\item For any rigid analytic prestack $\mathpzc X$ over $K$ the category $\Shv_{\et}(\mathpzc X, \Lambda)$ is presentable and $\Lambda$-linear.

\item The assignment $\mathpzc X\mapsto \Shv_{\et}(\mathpzc X, \Lambda)$ lifts to a functor
$$\Shv_{\et}(-, \Lambda)^{-1}\colon {(\PStk^{\an, \lft}_K)}^\op \xymatrix{\ar[r] &} {\Pr}^{\mathrm L}_{\Lambda}$$
and by adjoint functor theorem there is also a functor
$$\Shv_{\et}(-, \Lambda)_*\colon \PStk^{\an, \lft}_K \xymatrix{\ar[r] &} \PrR_{\Lambda}.$$
For a morphism $f\colon \mathpzc X \to \mathpzc Y$ of prestacks we will denote the induced pullback and pushforward functors by $f^{-1}$ and $f_*$ respectively.

\item The functor $\Shv_{\et}(-, \Lambda)^{-1}$, being a right Kan extension, transforms colimits of prestacks into limits of categories.

\item Since the functor $\Shv_{\et}(-, \Lambda)^{-1}_{|\Afd_K^\op}\colon \Afd_K^\op \to {\Pr}^{\mathrm L}_{\Lambda}$ is (by definition) a sheaf in the \'etale topology, so is the functor $\Shv_{\et}(-, \Lambda)$. In particular if $\epsilon \colon \mstack X \to L_{\et} \mstack X$ is the \'etale sheafification map, the induced functor $\epsilon^{-1}$ is an equivalence.

\item For each prestack $\mathpzc X$ the category $\Shv_{\et}(\mathpzc X, \Lambda)$ admits a natural right complete $t$-structure determined by the property that a sheaf $\mathcal F$ belongs to $\Shv_{\et}^{\ge 0}(\mathpzc X, \Lambda)$ if and only if its restriction $\mathcal F_S$ to any $K$-affinoid $S$ mapping to $X$ lands to $\Shv_{\et}^{\ge 0}(\widehat S, \Lambda)$.

\item For a morphism of stacks $f\colon \mathpzc X \to \mathpzc Y$ the pullback functor $f^{-1}$ is $t$-exact and the pushforward functor $f_*$ is left $t$-exact.
\end{enumerate}
\end{prop}

\smallskip  We now turn to studying the relation between algebraic and rigid analytic \'etale sheaves. First recall that to a finite type $K$-scheme $X$ one can attach a rigid analytic space $X^\an\in \mr{Rig}_K$, the so-called \emdef{analytification of $X$}, characterized by the property that to give a map of rigid analytic spaces $Y \to X^\an\in \mr{Rig}_K$ is the same as to give a map of locally ringed $G$-spaces $Y \to X$ (see \cite[Section 5.4]{Bosch_RigidAn} for more details). On the other hand, if $\mathfrak X$ is a flat finite type formal $\pi$-adic $\mathcal O_K$-scheme, one can attach to $\mathfrak X$ the so-called \emdef{Raynaud generic fiber $\mathfrak X_K$ of $\mathfrak X$}. Explicitly, if $\mathfrak X \simeq \Spf \widehat A$ for some $\pi$-adically complete finitely generated flat $\mathcal O_K$-algebra $\widehat A$, then $\widehat A\otimes_{\mathcal O_K} K$ is an affinoid $K$-algebra and $\mathfrak X_K$ is just $\mathrm{Sp}(\widehat A\otimes_{\mathcal O_K} K)$. We note that both construction commute with finite limits.

Thus we see that there are two natural ways to associate a rigid analytic space to a finite type flat $\mathcal O_K$-scheme $X$: one can take the algebraic generic fiber $X_K$ of $X$ and then consider its analytification $X_K^\an\coloneqq (X_K)^\an$, or one can take a $\pi$-adic completion $\widehat X$ of $X$ and then consider its Raynaud generic fiber $\widehat X\coloneqq (\widehat X)_K$. By the universal property of $X_K^\an$ these two constructions are related to each other:
\begin{prop}\label{map_from_Raynaud_to_usual_generic_fiber}
For a finite type flat $\mathcal O_K$-scheme $S$ there are natural maps of locally ringed $G$-spaces
$$\xymatrix{\widehat S_K \ar[r]^{\psi_S} & S_K^\an \ar[r]^{\phi_S} & S_K.}$$
Moreover, the pullbacks along these maps induce maps of the corresponding \'etale sites. In particular there is a natural pullback functor on the categories of \'etale sheaves:
$$\lambda_S^{-1}\colon \Shv_\et(S_K, \Lambda) \xymatrix{\ar[r] &} \Shv_\et(\widehat S_K, \Lambda).$$
\end{prop}

If $S$ is proper then $\psi_S$ is an isomorphism. Also, note that by a non-archimedean analogue of Artin's comparison (\cite[Theorem 3.8.1]{Huber_adicSpaces}) $\phi_S^{-1}$ induces an isomorphism on cohomology (at least when applied to constructible sheaves). On the opposite, $\psi_S^{-1}$ changes cohomology quite drastically, unless $S$ is proper.

Now we are going to generalize this picture to stacks:
\begin{construction}
Since the Raynaud generic fiber construction $\Spec A\mapsto \Sp(\widehat A\otimes_{\mc O_K} K)$ sends \'etale maps/covers to \'etale maps/covers, it extends to a map between the big sites
$$
\widehat{-}_K\colon \ET(\Aff^{\fg}_{\mc O_K}) \xymatrix{\ar[r] &} \ET(\Afd_{K}).
$$
We denote by
$$
\widehat{-}_K\colon \PStk^{\lft}_{\mc O_K} \xymatrix{\ar[r] &} \PStk^{\lft,\an}_{K}
$$
its left Kan extension. It then descends to the corresponding categories of stacks
$$
\widehat{-}_K\colon \Stk^{\lft}_{\mc O_K} \xymatrix{\ar[r] &} \Stk^{\lft,\an}_{K}.
$$
Moreover, since smooth maps are also preserved under the Raynaud generic fiber functor, geometric stacks go to geometric stacks.
\end{construction}

Situation with the functor $(-)^\an_K$ is slightly more complicated. The problem is that $(\Spec A)^\an_K\in \mr{Rig}_K$ is usually not an affinoid space, e.g. in the case $A=\mc O_K[x]$ we get the rigid analytic affine line $(\mbb A^1_K)^\an$ which is not quasi-compact. However, the construction $\Spec A\ra (\Spec A)^\an_K$ still sends \'etale covers/maps on the algebraic side to \'etale covers/maps on the rigid analytic side.
\begin{construction}
Let $\mr{Rig}_K\ra \Stk^{\lft,\an}_{K}$ be the left Kan extending the Yoneda embedding $\Afd_K\hookrightarrow \Stk^{\lft,\an}_{K}$ via the embedding $\Afd_K\hookrightarrow \mr{Rig}_K$. It once again sends \'etale maps/covers to \'etale maps/surjections. Left Kan extending the resulting functor $\Aff_K^\fp \to \Stk^{\lfp,\an}_{K}$ we obtain a functor
$$
(-)^\an_K\colon \Stk^{\lft}_{K} \xymatrix{\ar[r] &} \Stk^{\lft,\an}_{K}.
$$
\end{construction}

We will need the following technical lemma:
\begin{lem}\label{pass from algebraic to analytic}
Let $\mstack X$ be a finitely presentable flat Artin stack over $\mathcal O_K$. Then the natural maps
\begin{align*}
\Shv_\et(\mstack X_K, \Lambda) \xymatrix{\ar[r] &} \lim_{S \in \Aff_{\mathcal O_K/\mstack X}^{\ft,\fl,\op}} \Shv_\et(S_K, \Lambda), \\ 
\Shv_\et(\widehat{\mstack X}_K, \Lambda) \xymatrix{\ar[r] &} \lim_{S \in \Aff_{\mathcal O_K/\mstack X}^{\ft,\fl,\op}} \Shv_\et(\widehat S_K, \Lambda), \\
\Shv_\et(\mstack X_K^\an, \Lambda) \xymatrix{\ar[r] &} \lim_{S \in \Aff_{\mathcal O_K/\mstack X}^{\ft,\fl,\op}} \Shv_\et(S^\an_K, \Lambda)
\end{align*}
are equivalences.

\begin{proof}
Let $\mstack X$ be $n$-Artin. We will prove the statement by induction on $n$, the base of induction $n=-1$ being obvious since the category $\Aff_{\mathcal O_K/\mstack X}$ has a final object in this case. To make an inductive step note that in each case both parts satisfy \'etale descent. Hence, if $U\surj \mstack X$ is a smooth atlas of $\mstack X$, the statement follows by induction from the corresponding statement for the elements of the \v Cech nerve $U_\bullet$ and descent.
\end{proof}
\end{lem}

We are now ready to define functors $\lambda^{-1}$, $\phi^{-1}$ and $\psi^{-1}$ for arbitrary geometric stacks:
\begin{construction}\label{construction:et_algebraic_vs_Raynaud}
The map of \'etale sites from \Cref{map_from_Raynaud_to_usual_generic_fiber} induces a natural transformation 
$$\lambda_-^{-1} \colon \Shv_\et(-_K, \Lambda) \xymatrix{\ar[r] &} \Shv_\et(\widehat -_K, \Lambda)$$
of functors from $\Aff_{\mathcal O_K}^{\ft, \mathrm{fl}, \op}$ to ${\Pr}^{\mathrm L}_\Lambda$.
Passing to the right Kan extensions and using the equivalences from \Cref{pass from algebraic to analytic}, for a finitely presentable flat Artin $\mathcal O_K$-stack $\mstack X$ we obtain a functorial natural transformation
$$\lambda_{\mstack X}^{-1}\colon \Shv_\et(\mstack X_K, \Lambda) \xymatrix{\ar[r] &} \Shv_\et(\widehat{\mstack X}_K, \Lambda).$$
We will also denote the right adjoint of $\lambda_{\mstack X}^{-1}$ by $\lambda_{\mstack X*}$. Similarly, we can also define pairs of adjoint functors
$$
\xymatrix{\Shv_\et(\mstack X_K, \Lambda)\ar@<.5ex>[r]^{\phi_X^{-1}}& \Shv_\et(\mstack X^\an_K, \Lambda)\ar@<.5ex>[l]^{\phi_{X*}}}\qquad\text{and}\qquad\xymatrix{\Shv_\et(\mstack X^\an_K, \Lambda)\ar@<.5ex>[r]^{\psi_X^{-1}}& \Shv_\et(\widehat{\mstack X}_K, \Lambda) \ar@<.5ex>[l]^{\psi_{X*}}}.
$$
\end{construction}

If $f\colon \mstack X\to \mstack Y$ is a map of finitely presentable flat Artin $\mathcal O_K$-stacks we have a natural equivalence $\lambda_{\mstack X}^{-1} \circ f_K^{-1} \simeq \widehat f_K^{-1} \circ \lambda_{\mstack Y}$. Our goal now will be to show that the corresponding square of categories is not only commutative, but right adjoinable when restricted to the category of constructible sheaves. We will do this in two steps (basically decomposing $\lambda$ as $\phi\circ \psi$).

\begin{prop}\label{commutation with psi}
Let $f_K\colon \mathpzc X_K \to \mathpzc Y_K$ be a quasi-compact quasi-separated morphism between locally finite type Artin $K$-stacks. Then the diagram
$$\xymatrix{
\Shv_{\et}^+({\mstack X}^\an_K, \Lambda) & \ar[l]_-{\phi_{\mstack X}^{-1}} \Shv_{\et, \constr}^+(\mstack X_K, \Lambda) \\
\Shv_{\et}^+({\mstack Y}^\an_K, \Lambda) \ar[u]_{ (f_{K}^{\an})^{-1}} & \ar[l]_-{\phi_{\mstack Y}^{-1}}\ar[u]_{f_K^{-1}} \Shv_{\et, \constr}^+(\mstack Y_K, \Lambda).
}$$
is right adjointable (see \Cref{def:adjointable_sq}).

\begin{proof}
First assume that $\mstack Y$ is a finitely presentable (equivalently finite type in our case) scheme and that $\mstack X$ is $n$-Artin. We will prove the statement by induction on $n$. The base of induction $n=0$ is covered by the theorem of Huber \cite[Theorem 3.8.1]{Huber_adicSpaces}. For $n>0$ let $U\surj \mstack X$ be a smooth atlas and let $\mstack U_\bullet$ be the corresponding \v Cech nerve. By induction all squares
$$\xymatrix{
\Shv_{\et}^+(\mstack U^\an_{i}, \Lambda) & \ar[l]_-{\phi_{U_i}^{-1}} \Shv_{\et, \constr}^+(\mstack U_i, \Lambda) \\
\Shv_{\et}^+({\mstack Y}^\an_K, \Lambda) \ar[u] & \ar[l]_-{\phi_{\mstack Y}^{-1}}\ar[u] \Shv_{\et, \constr}^+(\mstack Y_K, \Lambda).
}$$
are right adjointable. Moreover, by \'etale descent
$$\Shv_{\et, \constr}^+(\mstack X_K, \Lambda) \xymatrix{\ar[r]^\sim &} \Tot \Shv_{\et, \constr}^+(\mstack U_{\bullet}, \Lambda) \qquad \text{and}\qquad \Shv_{\et}^+(\mstack X_K^{\an}, \Lambda) \xymatrix{\ar[r]^\sim &} \Tot \Shv_{\et}^+(\mstack U^\an_{\bullet}, \Lambda).$$
We conclude, since by \cite[Corollary 4.7.4.18.]{Lur_HA} the category of left adjoint functors and right adjointable squares as morphisms is closed under limits.

By co-Yoneda's lemma a general $\mstack Y_K$ can be represented as a colimit of affine finitely presentable $K$-schemes. Since by construction both functors 
$$\Shv_{\et, \constr}^+(-, \Lambda),\  \Shv_{\et}^+(-^\an, \Lambda) \colon \Stk^{\lft,\op}_K \xymatrix{\ar[r] &} \Cat_\infty$$
preserve limits, the result follows by \cite[Corollary 4.7.4.18.]{Lur_HA} again.
\end{proof}
\end{prop}

To deal with $\psi$ we will use proper base change for adic morphisms (in particular, properness assumption on $f$ will be essential). To use it we will need the following lemma. 
\begin{lem}
Let $f\colon X\to Y$ be a proper morphism of finite type flat $\mathcal O_K$-schemes. Then the commutative square
$$\xymatrix{
\widehat X_K \ar[r]^{\psi_X}\ar[d]^{\widehat f_K} & X_K^\an \ar[d]^{f^\an_K}\\
\widehat Y_K \ar[r]^{\psi_Y} & Y_K^\an
}$$
is a pullback square of rigid analytic spaces.

\begin{proof}
By Nagata's compactification theorem there exist compactifications $\overline X, \overline Y$ of $X$ and $Y$ respectively. By replacing $\ol X$ with the closure of the graph $\Gamma_f\subset \ol X\times \ol Y$ we can choose $\ol X$ such that there is a commutative square
$$\xymatrix{
X \ar[r]\ar[d]^f & \overline X \ar[d]^{\overline f} \\
Y \ar[r] & \overline Y
}$$
with $\ol f$ being proper. It is also easy to see that this commutative square is fibered (since $f$ was already proper and thus taking closure of $\Gamma_f$ does not affect fibers over $Y\subset \ol Y$).

Consider now a commutative diagram of rigid spaces
$$\xymatrix{
& X_K^\an \ar[rr] \ar[dd] && \overline X_K^\an \ar[dd]^{\overline f_K^\an}\\
\widehat X_K \ar[ru]^{\psi_X}\ar[rr]\ar[dd]^{\widehat f_K} && \widehat{\overline X}_K \ar[dd] \ar[ru]\\
& Y_K^\an \ar[rr] && \overline Y_K^\an \\
\widehat Y_K \ar[rr]\ar[ru] && \ar[ru] \widehat{\overline Y}_K.
}$$
The back and front faces of this diagram are pullback: this follows from the fact that both the Raynaud generic fiber and the analytification commute with fibered products. Moreover, since $\overline X$ and $\overline Y$ are proper, we have $\widehat{\overline X}_K \simeq \overline X_K^\an$ and $\widehat{\overline Y}_K \simeq \overline Y_K^\an$. Hence in the right face horizontal maps are isomorphisms and in particular it is fibered. It follows that in the commutative diagram
$$\xymatrix{
\widehat X_K \ar[r]\ar[d]^{\widehat f_K} & X_K^\an \ar[r]\ar[d]^{f_K} & \overline X_K^\an \ar[d]^{\overline f_K} \\
\widehat Y_K \ar[r] & Y_K^\an \ar[r] & \overline Y_K^\an
}$$
big and right hand squares are fibered. It follows that the left hand square is also fibered.
\end{proof}
\end{lem}
\begin{prop}\label{adic_proper_base_change}
Let $f\colon \mathpzc X \to \mathpzc Y$ be a proper morphism (see \Cref{def:proper morphism}) of finitely presentable flat Artin $\mathcal O_K$-stacks. Then the diagram
\begin{align*}
\xymatrix{
\Shv_{\et}^+(\widehat{\mstack X}_K, \Lambda) & \ar[l]_-{\psi_{\mstack X}^{-1}} \Shv_{\et}^+(\mstack X^\an_K, \Lambda) \\
\Shv_{\et}^+(\widehat{\mstack Y}_K, \Lambda) \ar[u]_{\widehat f_K^{-1}} & \ar[l]_-{\psi_{\mstack Y}^{-1}}\ar[u]_{f_K^{-1}} \Shv_{\et}^+(\mstack Y^\an_K, \Lambda).
}
\end{align*}
is right adjointable (see \Cref{def:adjointable_sq}).

\begin{proof}
First assume that both $X$ and $Y$ are schemes. By the lemma above the commutative square
$$\xymatrix{
\widehat X_K \ar[r]^{\psi_X}\ar[d]^{\widehat f_K} & X_K^\an \ar[d]^{f^\an_K}\\
\widehat Y_K \ar[r]^{\psi_Y} & Y_K^\an
}$$
is a fibered square of adic spaces. Hence the result in this case follows from the proper base change for adic spaces \cite[Theorem 4.4.1]{Huber_adicSpaces}. Note that $\psi_Y$ is an open embedding, thus $\dim\!.\!\tr(\psi_Y)=0$ and the mentioned base change result can indeed be applied.

Next assume that $\mstack X$ is $n$-Artin but $\mstack Y$ is a scheme. Since $f$ is proper we can choose a surjection $U \to \mstack X$ from a scheme $U$ such that all terms of the corresponding \v Cech nerve $\mstack U_k$ are proper over $\mstack Y$. Note that by \'etale descent we have
$$\Shv_{\et}^+(\widehat{\mstack X}_K, \Lambda) \xymatrix{\ar[r]^\sim &} \Tot \Shv_{\et}^+(\widehat{\mstack U}_{\bullet, K}, \Lambda)$$
and similar equivalence holds for $\mstack X_K^\an$ (see the proof of \Cref{commutation with psi}). Since all $\mstack U_i$ are at most $(n-1)$-Artin by induction the proposition holds for the map $\mstack U_i \to \mstack Y$, i.e. all squares
$$\xymatrix{
\Shv_{\et}^+(\widehat{\mstack U}_{i, K}, \Lambda) & \ar[l]_-{\psi_{\mstack U_k}^{-1}} \Shv_{\et}^+(\mstack U^\an_{i,K}, \Lambda) \\
\Shv_{\et}^+(\widehat{\mstack Y}_K, \Lambda) \ar[u]_{\widehat f_K^{-1}} & \ar[l]_-{\psi_{\mstack Y}^{-1}}\ar[u]_{f_K^{-1}} \Shv_{\et}^+(\mstack Y^\an_K, \Lambda).
}$$
We deduce that the proposition holds for the map $\mstack X \to \mstack Y$ since right adjointable squares are closed under limits.

For the general case let $\mstack X$ be arbitrary and let $\mstack Y$ be $n$-Artin. Let $V\surj \mstack Y$ be an atlas and let $\mstack V_i$ be the corresponding \v Cech cover. Since all $\mstack V_i$ are at most $(n-1)$-Artin the proposition holds for all maps $\mstack U_i \coloneqq \mstack X\times_{\mstack Y} \mstack V_i \to \mstack V_i$. Again we conclude by \'etale descent and since right adjointable squares are closed under limits.
\end{proof}
\end{prop}

\begin{cor}\label{adic_base_change}
Let $f\colon \mathpzc X \to \mathpzc Y$ be a proper morphism of finitely presentable flat Artin $\mathcal O_K$-stacks. Then the diagram
\begin{align*}
\xymatrix{
\Shv_{\et}^+(\widehat{\mstack X}_K, \Lambda) & \ar[l]_-{\lambda_{\mstack X}^{-1}} \Shv_{\et, \constr}^+(\mstack X_K, \Lambda) \\
\Shv_{\et}^+(\widehat{\mstack Y}_K, \Lambda) \ar[u]_{\widehat f_K^{-1}} & \ar[l]_-{\lambda_{\mstack Y}^{-1}}\ar[u]_{f_K^{-1}} \Shv_{\et, \constr}^+(\mstack Y_K, \Lambda).
}
\end{align*}
is right adjointable (see \Cref{def:adjointable_sq}).

\begin{proof}
Follows by composing squares in Propositions \ref{commutation with psi} and \ref{adic_proper_base_change}.
\end{proof}
\end{cor}

\section{\'Etale cohomology of stacks}\label{sec_etale_cohomology_of_stacks}
In this section we study various \'etale cohomology theories for Artin stacks and prove the \'etale comparison theorem for prismatic cohomology. In \Cref{sect:etale_cohomology_on_stacks} we defined the two variants of the \'etale cohomology: one of the usual algebraic and that of the Raynaud generic fiber. Using results from \Cref{adic_local_systems} we construct a natural comparison map between these two versions and formulate a conjecture (\ref{etale_conjecture}) on when this map is an equivalence. \Cref{subsect:etale_comparison} is devoted to the proof of the \'etale comparison for the prismatic cohomology which relates it to the \'etale cohomology of the Raynaud's generic fiber. As an application, we show that in the Hodge-proper case the analytic $\mbb Z_p$-\'etale cohomology groups are finitely generated over $\mbb Z_p$ and that over $\mc O_K$, after inverting $p$, the corresponding Galois representations are crystalline. We also deduce an inequality on lengths of the crystalline cohomology of the special fiber and the \'etale cohomology of the Raynaud's generic fiber of a Hodge-proper stack similar to the one in \cite[Theorem 1.1(ii)]{BMS1}. As another applications we deduce Hodge-to-de Rham degeneration for Hodge-proper stacks admitting a smooth proper model over $\mathcal O_K$ and Hodge-Tate decomposition of the \'etale cohomology of Raynaud generic fiber.

\subsection{\'Etale cohomology of algebraic and Raynaud generic fibers}\label{sect:etale_cohomology_on_stacks}
In this subsection we define two versions of \'etale $\mathbb Z_p$-cohomology of the generic fiber of a stack $\mstack X$ over $\mathcal O_K$. They come naturally from the two candidates for the (geometric) generic fiber, one being the algebraic stack $\mstack X_{\mbb C_p}$ given by the \textit{algebraic} generic fiber of $\mstack X$, and the other being the Raynaud generic fiber $\widehat{\mstack X}_{\mbb C_p}$: this is the rigid analytic stack that we have defined in \Cref{adic_local_systems}. The \'etale cohomology of the first are easier to describe in practice and has a clear "topological" interpretation via Artin's comparison (see \Cref{Artins_comparison_stacks}). The \'etale cohomology of the Raynaud generic fiber a priory behaves rather badly, in particular the individual cohomology groups are not finitely generated, unless we make some (either literal or relaxed) properness assumptions on the stack. As we will see shortly in \Cref{etale_comparison}, in general the prismatic cohomology only compares with the cohomology of the Raynaud general fiber (and that's why we are forced to consider it). 


\paragraph{The algebraic version.}
Fix a base ring $R$ and let $\Lambda$ be a torsion ring. Recall the category $\Shv_\et(\mstack X, \Lambda)$ that we defined in \Cref{etale_sheaves_on_stacks} for any prestack $\mstack X\in \PStk_{R}$. Let $\ul{\Lambda}\in \Shv_\et(\mstack X, \Lambda)$ be the constant sheaf (see \Cref{ex: the constant sheaf}).

\begin{defn}\label{defn: algebraic etale cohomology}
	\begin{enumerate}[wide]
		\item The \'etale cohomology $R\Gamma_\et(\mstack X, \Lambda)$ is defined as
		$$R\Gamma_\et(\mstack X, \Lambda) \coloneqq \End_{\Shv_\et(\mstack X, \Lambda)}(\underline \Lambda)\in \DMod{\Lambda}.$$
		Unraveling the definitions we have a functorial isomorphism
		$$
		\End_{\Shv_\et(\mstack X, \Lambda)}(\underline \Lambda)\simeq \lim_{(T\ra \mstack X)\in(\Aff_{R}/\mstack X)^\op} \End_{\Shv_\et(T, \Lambda)}(\underline \Lambda) \simeq \lim_{(T\ra \mstack X)\in(\Aff_{R}/\mstack X)^\op} \RG_{\et}(T,\Lambda).
		$$
		Thus, $R\Gamma_\et(\mstack X, \Lambda)$ is the value of the right Kan extension of the functor $\RG_{\et}(-,\Lambda)$ along $\Aff_{R}\subset \PStk_{R}$ and as such naturally extends to a functor 
		$$
		\RG_{\et}(-,\Lambda)\colon \PStk_R^{\op}\xymatrix{\ar[r]&}
		\DMod{\Lambda}.$$
		In particular, for a morphism $f\colon\mstack X\ra \mstack Y$ we have a functorial pull-back map $f^{-1}\colon \RG_{\et}(\mstack Y,\Lambda)\ra \RG_{\et}(\mstack X,\Lambda)$.
		\item More generally, given an \'etale sheaf $\mc F\in \Shv_\et(\mstack X, \Lambda)$ we define 
		$$
		\RG_{\et}(\mstack X,\mc F)\coloneqq \Hom_{\Shv_\et(\mstack X, \Lambda)}(\ul{\Lambda},\mc F).
		$$
		Once again, unraveling the definitions, we get a formula expressing it as a limit
		$$
		\RG_{\et}(\mstack X,\mc F)\simeq  \lim_{(f:T\ra \mstack X)\in(\Aff_{R}/\mstack X)^\op} \Hom_{\Shv_\et(T, \Lambda)}(\ul{\Lambda},f^*\mc F) \simeq  \lim_{(f:T\ra \mstack X)\in(\Aff_{R}/\mstack X)^\op} \RG_{\et}(T,f^*\mc F).
		$$
		
		\item We also extend the definition to $\mbb Z_p$-coefficients by a somewhat ad hoc formula
		$$
		R\Gamma_\et(\mstack X, \mathbb Z_p) \coloneqq \prolim R\Gamma_\et(\mstack X, \mathbb Z/p^n).
		$$
		Similarly, for any $k\in \mbb Z$ we put 
		$$
		R\Gamma_\et(\mstack X, \mathbb Z_p(k)) \coloneqq \prolim R\Gamma_\et(\mstack X, \mu_{p^n}^{\otimes k}).
		$$
	\end{enumerate}
\end{defn}

\begin{rem}
	For a scheme $X$ our $R\Gamma_\et(X, \mathbb Z_p(k))$ is equivalent to the continuous \'etale cohomology of Jennsen \cite{Jannsen_ContEtale} (essentially by definition) or to the cohomology of the sheaf $\mathbb Z_p(k)$ in the pro-\'etale site of $X$ (by \cite[Proposition 5.6.2]{BS_proetale}). Moreover, if all cohomology groups $H^i_\et(X, \mathbb Z/p^n(k))$ are finite, e.g. if $X$ is a variety over an algebraically closed field, the Mittag-Leffler condition is automatically satisfied, so by Milnor's exact sequence $H^iR\Gamma_\et(X, \mathbb Z_p(k)) \simeq \prolim H^i_\et(X, \mathbb Z/p^n(k))$.
\end{rem}

As in the case of schemes, for $R$ being the field of complex numbers $\mathbb C$, the \'etale cohomology of a stack can be interpreted as the singular cohomology of the underlying homotopy type of its complex points (which we now define).
\begin{lem}\label{etale_descent_for_underlying_homotopy_type}
	Let $\Pi_\infty -(\mathbb C)\colon \Aff_{\mathbb C}^\ft \to \Type$ be a functor defined as a composition
	$$\xymatrix{\Aff_{\mathbb C}^\ft \ar[r]^-{-(\mathbb C)} & \Top \ar[r]^{\Pi_\infty} & \Type,}$$
	where the first functor sends an affine scheme $X$ to its set of complex points $X(\mathbb C)$ equipped with the natural analytic topology and the second functor sends a topological space to its underlying homotopy type. Then $\Pi_\infty-(\mathbb C)$ satisfies \'etale descent.
	
	\begin{proof}
		Let $U\surj X$ be an \'etale cover. We need to prove that the induced map
		$$|\Pi_\infty U_\bullet(\mathbb C)| \xymatrix{\ar[r] &} \Pi_\infty X(\mathbb C)$$
		is a homotopy equivalence, where $U_\bullet$ is the \v Cech nerve of $\pi$.
		
		Note that the analytification functor $-(\mathbb C)$ preserves fibered products and sends \'etale morphisms to local homeomorphisms. Hence, it is enough to prove that for a local homeomorphism $V \surj Y$ of topological spaces the induced map $|V_\bullet| \to Y$ (where $V_\bullet$ is a \v Cech nerve of $V \surj Y$ formed in the category $\Top$ of topological spaces) is an equivalence. This is \cite[Corollary 1.5]{DuggerIsaksen_HypercoversAndA1HomotopyRealiation}.
	\end{proof}
\end{lem}
\begin{construction}\label{constr:underlying homotopy type}
	Define the functor of the \emdef{underlying homotopy type} of a prestack of locally finite type 
	$$\Pi_\infty-(\mathbb C)\colon \PStk_{\mathbb C}^\lft\coloneqq \Fun(\CAlg_{\mathbb C/}^\ft, \Type) \xymatrix{\ar[r] &} \Type$$
	as the left Kan extension of the functor $\Pi_\infty - (\mathbb C)\colon \Aff_{\mathbb C}^\ft \to \Type$ (considered in the previous lemma) along the Yoneda embedding $\Aff_\mathbb C^\ft \inj \PStk_{\mathbb C}^\lft$.
	
	Since $\Pi_\infty - (\mathbb C)\colon \Aff_{\mathbb C}^\ft \to \Type$ evidently preserves finite coproducts and satisfies \'etale descent by the previous lemma, we see that $\Pi_\infty -(\mathbb C)$ factors through a functor $\Stk_{\mathbb C, \et}^\lft \to \Type$, which (by a slight abuse of notation) we will also denote by $\Pi_\infty-(\mathbb C)$.
\end{construction}
\begin{ex}\label{ex: BG anlytification}
	By construction the functor $\Pi_\infty-(\mathbb C)\colon \Stk^\lft_{\mathbb C, \et} \to \Type$ commutes with colimits. In particular, for an algebraic group $G$ over $\mathbb C$ we have that $\Pi_\infty (BG)(\mathbb C)$ is naturally equivalent to the topological classifying space of the topological group $G(\mathbb C)$. More generally, for a $\mathbb C$-scheme $X$ equipped with an action of $G$ there is a natural equivalence $\Pi_\infty [X/G](\mathbb C) \simeq X(\mathbb C)_{hG(\mathbb C)}$, where $-_{hG(\mathbb C)}\colon \Type^{BG(\mathbb C)} \to \Type$ is the Borel's homotopy quotient by $G(\mathbb C)$ functor.
\end{ex}
\begin{prop}[Artin's comparison]\label{Artins_comparison_stacks}
	Let $\mstack X$ be an Artin $\mathbb C$-stack locally of finite type and let $\Lambda$ be a finite abelian group. Then there is a natural equivalence
	$$R\Gamma_\et(\mstack X, \Lambda) \simeq C^*(\Pi_\infty \mstack X(\mathbb C), \Lambda),$$
	where $C^*(-, \Lambda) \colon \Type^\op \to \DMod{\Lambda}$ is the singular $\Lambda$-cochains functor. Also, 
	$$
	R\Gamma_\et(\mstack X, \mbb Z_p) \simeq C^*(\Pi_\infty \mstack X(\mathbb C), \mbb Z_p).
	$$
	
\begin{proof}
Recall that since $\mathbb Z_p \simeq \prolim \mathbb Z/p^n$ in the derived category of abelian groups, for any homotopy type $K$ we have
$$C^*(K, \mathbb Z_p) = \Hom_{\DMod{Z}}(C_*(K, \mathbb Z), \mathbb Z_p) \simeq \prolim \Hom_{\DMod{\mathbb Z}}(C_*(K, \mathbb Z), \mathbb Z/p^n) = \prolim C^*(K, \mathbb Z/p^n).$$
Since we have defined $R\Gamma_\et(\mstack X, \mathbb Z_p)$ as $\prolim R\Gamma_\et(\mstack X, \mathbb Z/p^n)$, the last part of the proposition follows from the case of torsion coefficients.

In the latter case both sides satisfy \'etale descent: the left hand side essentially by definition and the right hand side by \Cref{etale_descent_for_underlying_homotopy_type}. Hence by \cite[Theorem 4.7]{Pridham_ArtinHypercovers} it is enough to prove the equivalence for $\mstack X$ being an affine finite type $\mathbb C$-scheme. This is a classical theorem of Artin, see \cite[Expos\'e XVI, Theorem 4.1]{SGA4_3}.
\end{proof}
\end{prop}

\begin{rem}
	This applies in particular to the  $p$-adic setting. Namely, let $\mstack X$ be an Artin stack over $\mc O_{\mbb C_p}$. Recall that there is an abstract isomorphism $\iota\colon \mathbb C_p\simeq \mathbb C$ because both fields are of characteristic 0, algebraically closed and of the same cardinality. Then by \Cref{Artins_comparison_stacks} we get 
	$$
	\RG_\et(\mstack X_{\mbb C_p},\Lambda)\simeq \RG_\et(\mstack X_{\mbb C},\Lambda)\simeq C^*(\Pi_\infty \mstack X(\mathbb C), \Lambda).
	$$
	Note that this isomorphism is not canonical and depends on the choice of $\iota$.
\end{rem}

\begin{rem}\label{rem:algebraic etale cohomology is usually finite}
	We also record that for any finite type stack $\mstack X$ over $\mbb C_p$ the cohomology of $\RG_{\et}(\mstack X,\Lambda)\in \Coh^+(\Lambda)$. Indeed, by part (3) of \Cref{functoriality_of_loc_sys}, $p_*\ul \Lambda\in \DMod{\Lambda}$ for $p\colon \mstack X\ra \Spec \mbb C_p$ is constructible and so all of its cohomology groups are finitely generated.
\end{rem}	
\paragraph{The rigid analytic version.}
Unfortunately, the \'etale comparison from \Cref{main_BMS2} relates the prismatic cohomology with the \'etale cohomology of the Raynaud generic fiber instead of the algebraic one. A discussion on the relation between the two is postponed till the next subsection.

Meanwhile, the \'etale comparison will be formulated in a slightly more general setup than the one we had in \Cref{adic_local_systems}. We fix a base $p$-adically complete ring $R$ with bounded $p$-torsion.
\begin{notation}
	Let $U = \Spec A \in \Aff_{R}$ be a flat finite type affine scheme over $R$. Consider the corresponding formal scheme $\widehat{U}\coloneqq \Spf \widehat A$ and the adic space $\Spa(\widehat A,\widehat A)$, where $\widehat A = \prolim A/p^n$ is the $p$-adic completion. To $\widehat U$ we can associate its \emdef{Raynaud generic fiber} $\widehat U_{R[1/p]}$ given by $\widehat U_{R[1/p]} \coloneqq \Spa(\widehat A[1/p],\widehat A).$
	This is an adic space.
\end{notation}

In \cite{Huber_adicSpaces} under Noetherian assumptions on $\widehat A [1/p]$ the theory of \'etale cohomology of adic spaces was developed and studied in great detail. By \cite[Corollary 3.2.2]{Huber_adicSpaces} for finite constant coefficients one has an equivalence
$$
\xymatrix{\RG_{\et}(\widehat U_{R[\frac{1}{p}]},\Lambda)\simeq \RG_{\et}(\Spec \widehat A[\frac{1}{p}],\Lambda),}
$$
where on the right hand side we have the ordinary affine scheme $\Spec \widehat A[{1}/{p}]$. In the case we drop the Noetherian assumptions we can use this as a definition, as suggested in \cite{BhattMathew_Arc}:

\begin{notation}\label{not: adic cohomology in non-Noetherian case}	
	We define	
	$$
	R\Gamma_\et(\widehat - _{R[\frac{1}{p}]}, \Lambda)\colon \left(\Aff^{\fl,\ft}_{R}\right)^\op \xymatrix{\ar[r]&}\DMod{\Lambda},
	$$
	to be a natural functor that sends $\Spec A \in \Aff^{\fl,\ft}_{R}$ to $\RG_{\et}(\Spec \widehat A[{1}/{p}],\Lambda)$.
\end{notation}

\begin{defn}\begin{enumerate}[wide]
		\item
		We define the functor of \emdef{\'etale cohomology of the Raynaud generic fiber} (which we sometimes also call the \emdef{adic \'etale cohomology} for brevity)
		
		$$R\Gamma_\et(\widehat - _{R[\frac{1}{p}]}, \Lambda) \colon (\PStk_{R})^\op \xymatrix{\ar[r] &} \DMod{\Lambda}$$
		as the right Kan extension along the Yoneda embedding $\Aff^{\fl,\ft}_{R}\inj \PStk_{R}$ of the functor which sends an affine $R$-scheme $U$ to $R\Gamma_\et(\widehat U_{R[1/p]}, \Lambda)$. Explicitly, for a prestack $\mstack X$ we have:
		$$R\Gamma_{\et}(\widehat{\mstack X}_{R[\frac{1}{p}]}, \Lambda) \simeq \lim_{(\Spec A\ra \mstack X)\in (\Aff^{\fl,\ft}_R\!\!/\mstack X)^\op}\! \! \xymatrix{ R\Gamma_\et(\Spec \widehat A[\frac{1}{p}], \Lambda).}$$
		\item Given an \'etale sheaf $\mc F\in \Shv_\et({\mstack X}, \Lambda)$ we also define its adic \'etale cohomology $\RG(\widehat{\mstack X}_{R[1/p]},\mc F)$. Namely, for any $\Spec A\in \Aff^{\fl,\ft}_R$ one has a map $\mu_A\colon \Spec \widehat A[1/p]\ra \Spec A$. Then we can put
		$$
		\RG(\widehat{\mstack X}_{R[\frac{1}{p}]},\mc F)=\lim_{(\Spec A\ra \mstack X)\in (\Aff^{\fl,\ft}_R\!\!/\mstack X)^\op} \!\!\xymatrix{R\Gamma_\et(\Spec \widehat A[\frac{1}{p}], \mu_A^*\mc F|_{U}).}
		$$
		\item 
		For $p$-adic coefficients we also define $R\Gamma_\et(\widehat{\mstack X}_C, \mathbb Z_p) \coloneqq \prolim R\Gamma_\et(\widehat{\mstack X}_C, \mathbb Z/p^n)$.
	\end{enumerate}
\end{defn}

\begin{ex}
	Recall the setting of \Cref{adic_local_systems}. There, among other things, to a flat finitely presentable Artin stack $\mstack X$  over $\mc O_K$ we associated the rigid analytic stack $\widehat {\mstack X}_K$ over $K$. Note that by \cite[Proposition 2.1.4]{Huber_adicSpaces} for any flat finitely-presented $\mc O_K$-algebra $A$ the rigid \'etale cohomology $\RG_\et(\Sp \widehat A_K,\Lambda)$ compares with the adic \'etale cohomology. Thus, by \Cref{pass from algebraic to analytic} we get a formula for $R\Gamma_{\et}(\widehat{\mstack X}_K, \Lambda)$ analogous to the one in \Cref{defn: algebraic etale cohomology}, namely:
	$$
	R\Gamma_\et(\widehat{\mstack X}_K, \Lambda)\simeq \lim_{(U\ra \mstack X)\in(\Aff_{R}^{\fl, \ft}/\mstack X)^\op} \End_{\Shv_\et(\widehat U_{K}, \Lambda)}(\underline \Lambda) \simeq \End_{\Shv_\et(\widehat{\mstack X}_K, \Lambda)}(\underline \Lambda)\in \DMod{\Lambda}.
	$$
	Similarly, for a sheaf $\mc F\in \Shv_\et(\widehat{\mstack X}_K, \Lambda)$ we have
	$$
	R\Gamma_\et(\widehat{\mstack X}_K, \mc F)\simeq \Hom_{\Shv_\et(\widehat{\mstack X}_K, \Lambda)}(\underline \Lambda,\mc F)\in \DMod{\Lambda}.
	$$
	This also allows us to define
	$$
	R\Gamma_\et(\widehat{\mstack X}_K, \mbb Z_p(k))\coloneqq \prolim \RG_{\et}(\widehat{\mstack X}_K, \mu_{p^{n}}^{\otimes k}).
	$$
\end{ex}

\begin{rem}
	Let $\mstack X$ be a flat finitely presentable Artin stack over $\mc O_K$, where $K/\mbb Q_p$ is finite. Then the complex $\RG(\widehat{\mstack X}_{\mbb C_p},\mc F)$ for the algebraic closure $K\subset \mbb C_p$ can be considered as the cohomology of a version of nearby cycles of $\mc F$ relative to the special fiber of the morphism $\mstack X\ra \Spec \mc O_K$.
\end{rem}

\paragraph{Comparison of the two versions of \'etale cohomology.} For this section let's assume that our base ring $R$ is the ring of integers $\mc O_C$ in an  algebraically closed complete non-Archimedean field $C$ of mixed characteristic.

\begin{construction}\label{constr:alg_to_adic_et_comp}
	Let $\mstack X$ be a flat finitely presentable Artin $\mathcal O_C$-stack. Recall that by \Cref{construction:et_algebraic_vs_Raynaud} there is a natural pair of adjoint functors
	$$\lambda_{\mstack X}^{-1} \colon \Shv_\et(\mstack X_C, \Lambda) \rightleftarrows \Shv_\et(\widehat {\mstack X}_C, \Lambda)\!\ :\!\lambda_{\mstack X*}.$$
	Applying the global sections functor $\RG_{\et}(\mstack X,-)$ to the unit of adjunction $\underline \Lambda \to \lambda_{\mstack X*} \underline \Lambda$ and using the identification $\lambda_{\mstack X}^{-1}\ul\Lambda\simeq \ul\Lambda $ we obtain a natural comparison map
	$$\Upsilon_{\mstack X}\colon R\Gamma_\et(\mstack X_C, \Lambda) \xymatrix{\ar[r] &} R\Gamma_\et({\mstack X}_C,  \lambda_{\mstack X*}\ul{\Lambda})\simeq R\Gamma_\et(\widehat{\mstack X}_C, \Lambda).$$
\end{construction}
\begin{rem}\label{rem:Upsilon as base change}
	The map $\Upsilon_{\mstack X}$ can be also viewed as the base change morphism (see \Cref{adic_base_change}) 
	$$
	(\lambda_{\Spec C}^{-1}\circ p_{C*})(\ul{\Lambda})\to (\widehat p_{C*}\circ\lambda^{-1}_{\mstack X})(\ul{\Lambda})
	$$ for the constant sheaf $\ul \Lambda \in \Shv_{\et}(\mstack X_C)$ and the structure morphism $p\colon \mstack X\ra \Spec \mc O_C$. Indeed, the functor $\lambda\colon \et/\Spec C \xra{\sim} \et/\Sp C$ is an equivalence of sites since both of them are trivial. Thus the pull-back functor
	$$\lambda_{\Spec C}^{-1}\colon \Shv_\et(\Spec C, \Lambda) \xymatrix{\ar[r]^-\sim &} \Shv_\et(\Sp C, \Lambda)$$ is trivially an equivalence, with both categories being equivalent to $\DMod{\Lambda}$ via the global sections functor.
	Under this identification, the functors $p_{C*}$ and $\widehat{p}_{C*}$ are given by $\RG_\et(\mstack X_C,-)$ and  $\RG_\et(\widehat{\mstack X}_C,-)$ respectfully; this way the base change morphism gives us exactly the map $\Upsilon_{\mstack X}$.
	In particular, by \Cref{adic_base_change}
	$\Upsilon_{\mstack X}$ is an equivalence if $\mstack X$ is a proper stack.
\end{rem}
\begin{rem}
	We can also replace $\Upsilon_{\mstack X}$ by a pull-back map for a morphism of rigid-analytic stacks using the analytification $\mstack X_C^\an$. Namely recall (\Cref{construction:et_algebraic_vs_Raynaud}) that we have the following diagram of functors $$
	\xymatrix{\Shv_\et(\mstack X_C, \Lambda)\ar@<.5ex>[r]^{\phi_{\mstack X}^{-1}} \ar@/_2.0pc/[rr]_{\lambda_{\mstack X}^{-1}}& \Shv_\et(\mstack X^\an_C, \Lambda)\ar@<.5ex>[l]^{\phi_{X*}}\ar@<.5ex>[r]^{\psi_X^{-1}}& \Shv_\et(\widehat{\mstack X}_C, \Lambda) \ar@<.5ex>[l]^{\psi_{\mstack X*}}.}
	$$
	This gives a decomposition:  
	$$
	\xymatrix{
		\RG_{\et}(\mstack X_C,\Lambda) \ar@/_2.0pc/[rr]_{\Upsilon_{\mstack X}}\ar[r]_\sim^{\phi_{\mstack X}^{-1}}& \RG_\et(\mstack X^\an_C, \Lambda) \ar[r]^{\psi_{\mstack X}^{-1}}&\RG_\et(\widehat{\mstack X}_C, \Lambda).}
	$$
	Moreover, interpreting $\phi_{\mstack X}^{-1}$ as a base change morphism for $f_C\colon \mstack X_C\ra \Spec C$, by \Cref{commutation with psi} we see that $\phi_{\mstack X}^{-1}$ is an equivalence. Thus, we get that $\Upsilon_{\mstack X}$ is an equivalence if and only if $\psi_{\mstack X}^{-1}\colon \RG_\et(\mstack X^\an_C, \Lambda)\ra \RG_\et(\widehat{\mstack X}_C, \Lambda)$ is.
\end{rem}

\begin{rem}\label{rem: interpretation of Upsilon using arc-descent}
	Quite interestingly,  using $\mr{arc}_p$-descent, one can also give a purely algebraic interpretation of the fiber of $\Upsilon_{\mstack X}$ as the obstruction for certain local acyclicity of $\mstack X$ relative to $\Spec \mc O_C$. Namely, by \cite[Corollary 6.17]{BhattMathew_Arc}, for any $\mc O_C$-algebra $R$ we have the following fibered square	
	$$
	\xymatrix{\RG_{\et}(\Spec R,\Lambda)\ar[d]_{\iota_p^*}\ar[r]& \RG_{\et}(\Spec R[1/p],\Lambda)\ar[d]^{\Upsilon_{\Spec R}}\\
		\RG_{\et}(\Spec R/p,\Lambda)\ar[r] & \RG_{\et}(\Spec \widehat R[1/p],\Lambda)
	}
	$$
	with $\iota_p\colon \Spec R/p \ra \Spec R$ being the embedding of the closed fiber. Passing to the limit over $\Aff_{\mc O_C}/\mstack X$ we get an analogous fiber square:
	$$
	\xymatrix{\RG_{\et}(\mstack X,\Lambda)\ar[d]_{\iota_p^*}\ar[r]& \RG_{\et}(\mstack X_{C},\Lambda)\ar[d]^{\Upsilon_{\mstack X}}\\
		\RG_{\et}(\mstack X_{\mc O_C/p},\Lambda)\ar[r] & \RG_{\et}(\widehat {\mstack X}_C,\Lambda).}
	$$
	This induces an equivalence of fibers
	$$
	\fib{\iota_p^*}\simeq \fib \Upsilon_{\mstack X},
	$$
	and, in particular, $\Upsilon_{\mstack X}$ is an equivalence if and only if $\iota_p^*\colon \RG_{\et}(\mstack X,\Lambda)\ra\RG_{\et}(\mstack X_{\mc O_C/p},\Lambda)$ is.
\end{rem}

Previous remark suggests the following definition:
\begin{defn}
	A flat locally finitely presentable stack $\mstack X$ over $\mc O_C$ is called \emdef{$\Lambda$-locally acyclic} if the map
	$$\Upsilon_{\mstack X}\colon R\Gamma_\et(\mstack X_C, \Lambda) \xymatrix{\ar[r]^\sim &} R\Gamma_\et(\widehat{\mstack X}_C, \Lambda)$$
	is an equivalence. 
\end{defn}
\begin{ex}
	$\Spec \mc O_C$ is $\Lambda$-locally acyclic. By \Cref{rem:Upsilon as base change} any proper $\mc O_C$-scheme is also $\Lambda$-locally acyclic.
\end{ex}
\begin{rem}It is not hard to see that $\mstack X$ is $\Lambda$-locally acyclic if and only if it is $\mbb F_\ell$-locally acyclic for all $\ell$ that divide $|\Lambda|$.
\end{rem}	
The behavior for $\mbb F_\ell$-local acyclicity is very different depending on whether $\ell$ equals to $p$ or not. Situation is easier if $\ell\neq p$: then $\mbb F_\ell$-local acyclicity can be reformulated in terms of (a version of) the algebraic nearby cycles of $\mstack X$, see \Cref{prop: F-l-acyclicity in terms of the neraby cycles}. In particular some nice affine $\mc O_C$-schemes like $\mbb A^n$ or $\mbb G_m$ are $\mbb F_\ell$-locally acyclic. In this case $\mbb F_\ell$-local acyclicity also automatically enjoys some basic nice properties, like being closed under products:
\begin{lem}\label{lem: product of acyclic is acyclic}
	Let $\Lambda$ be a field and assume $\mr{char}\ \!  \Lambda \neq p$. Let $\mstack X,\mstack Y$ be two flat locally finitely presentable $\Lambda$-locally acyclic Artin $\mc O_C$-stacks. Then the product $\mstack X\times\mstack Y$ is $\Lambda$-locally acyclic.
\end{lem}
\begin{proof}We need to show that the map 
	$$
	\Upsilon_{\mstack X\times \mstack Y}\colon R\Gamma_\et((\mstack X\times \mstack Y)_C, \Lambda) \xymatrix{\ar[r]&} R\Gamma_\et(\widehat{(\mstack X\times \mstack Y)}_C, \Lambda)
	$$
	is an equivalence.	This follows from the K\"unneth formula applied both sides, which, as we now show, still holds true in the world of stacks. First of all, the K\"unneth formula holds if both $\mstack X$ and $\mstack Y$ are affine schemes: this is classical for the algebraic side and is given by \cite[Proposition 6.22]{BhattMathew_Arc} in the adic case, here we use the assumption on the characteristic of $\Lambda$. We now show that the natural map $\RG_\et(\mstack X_C, \Lambda)\otimes_{\Lambda}\RG_\et(\mstack Y_C, \Lambda)\longrightarrow R\Gamma_\et((\mstack X\times \mstack Y)_C, \Lambda)$ is an equivalence, the proof in the adic case is analogous. Since $\Lambda$ is a field the tensor product functor $-\otimes_{\Lambda}M$ with any complex $M \in \DMod{\Lambda}^{\ge 0}$ is left $t$-exact, and thus commutes with totalizations of $0$-coconnected objects. Taking smooth hypercovers $U_{\bullet_1}\ra \mstack X$ and $V_{\bullet_2}\ra \mstack Y$ by affine schemes, we get a hypercover $(U\times V)_\bullet\ra \mstack X\times \mstack Y$ where $(U\times V)_n \coloneqq U_n\times V_n$. By smooth descent and the remark above we get a chain of natural equivalences	
	\begin{align*}
		\RG_\et(\mstack X_C, \Lambda)\otimes_{\Lambda}\RG_\et(\mstack Y_C, \Lambda) \xymatrix{\ar[r]^\sim &}  \Tot_{\bullet_1} &\left( \RG_\et((U_{\bullet_1})_C, \Lambda)\right)\otimes_{\Lambda}\Tot_{\bullet_2}\left( \RG_\et((V_{\bullet_2})_C, \Lambda)\right)\xymatrix{\ar[r]_\sim^{\text{(remark)}} &}\\ 
		\!\!\!  \Tot_{\bullet_1} \left(\Tot_{\bullet_2}\left(\RG_\et((U_{\bullet_1})_C, \Lambda) \otimes_{\Lambda}\RG_\et((V_{\bullet_2})_C, \Lambda)\right)\right) &\xymatrix{\ar[r]^\sim &}\Tot_{\bullet=\bullet_1= \bullet_2}\left(\RG_\et((U_{\bullet})_C, \Lambda) \otimes_{\Lambda}\RG_\et((V_{\bullet})_C, \Lambda)\right)\xymatrix{\ar[r]^\sim &}\\
		\xymatrix{\ar[r]_\sim^{\text{K\"unneth}} &} \Tot_{\bullet}	(\RG_\et((U\times V&)_{\bullet,C},\Lambda)) \xymatrix{\ar[r]^\sim &} \RG_\et((\mstack X\times\mstack Y)_C, \Lambda),	
	\end{align*}
	which show that the map above is an equivalence. Here in the second line we also used that the limit over $\Delta\times \Delta$ can be computed by restricting to the diagonal $\Delta \hookrightarrow \Delta\times \Delta$.
\end{proof}

If $\ell=p$, then typically the map $\Upsilon_{\mstack X}$ is very far from being an equivalence (see \Cref{ex:F-p etale cohomology of rigid affine line} below). Nevertheless there are still two ways to show that a stack $\mstack X$ is $\mbb F_\ell$-locally acyclic no matter what $\ell$ is.  

First, by \'etale hyperdescent for algebraic and adic \'etale cohomology of Artin stacks (see \Cref{rem: hyperdescent for etale cohomology} for the algebraic case, the same reasoning then works for the adic situation as well) one can check $\Lambda$-local acyclicity on a hypercover: 
\begin{lem}\label{lem: descent for F-l acyclic stacks}
	Let $\mstack X$ be an $n$-Artin $\mc O_C$-stack locally of finite type and let $\mstack U_{\bullet}\ra \mstack X$ be a hypercover in \'etale topology such that each $\mstack U_i$ is also $n$-Artin and locally of finite type. Then if all $\mstack U_i$ are $\Lambda$-locally acyclic, so is $\mstack X$.
\end{lem}

Unfortunately, if $\ell=p$, this is not as useful as one could think since there are very few $\mbb F_p$-locally acyclic affine schemes, as the following example illustrates:
\begin{ex}\label{ex:F-p etale cohomology of rigid affine line}
	Let $X=\mathbb A^1$ be the affine line and let $\Lambda=\mbb F_p$. Then all higher \'etale cohomology of $\mathbb A^1_C$ vanish, in particular $H^1_{\et}(\mathbb A^1_C,\mbb F_p)\simeq 0$. On the other hand the first $\mathbb F_p$-cohomology of $\mbb D_C\simeq\widehat{\mathbb A}^1_C$ is infinite-dimensional over $\mathbb F_p$. Indeed, identifying $\mbb F_p$ with $\mu_p$ we can see this from the Kummer exact sequence $1\ra \mbb F_p\ra \mc O_{\mbb D_C}^\times \xra{-^p}\mc O_{{\mbb D_C}}^\times \ra 1 $ of sheaves on $\et/\mbb D_C$. Namely, we have $H^1_{\et}(\mbb D_C, \mc O_{\mbb D_C}^\times)\simeq \Pic(\mbb D_C)\simeq 0$ and so 
	$$
	H^1_\et(\mbb D_C,\mbb F_p)\simeq C\langle T\rangle^\times/(C\langle T\rangle^\times)^p.
	$$ This has $(1+\mf m_C\langle T\rangle)/(1+\mf m_C\langle T \rangle)^p$ as a subgroup and there is a well-defined surjective homomorphism 
	$$
	[-]_1\colon (1+\mf m_C\langle T\rangle)/(1+\mf m_C\langle T \rangle)^p \xymatrix{\ar@{->>}[r]&}\mf m_C/p\cdot \mf m_C,
	$$ 
	that sends a power series $1+a_1T+a_2T^2+\ldots$ to the first coeffient $a_1\! \pmod {p\cdot \mf m_C}$.
	The vector space $\mf m_C/p\cdot \mf m_C$ has infinite dimension over $\mbb F_p$ and thus we see that $H^1_\et(\mbb D_C,\mbb F_p)$ is very far from being 0.
\end{ex}

Another way to show that $\mstack X$ is $\mbb F_\ell$-locally acyclic is to map it smoothly and properly onto a stack $\mstack Y$ which we know already is $\mbb F_\ell$-locally acyclic. 
\begin{prop}\label{prop: Lambda-acyclicity translates via proper smooth maps}
	Let $\Lambda$ be a field. Let $f\colon \mstack X\ra \mstack Y$ be a smooth proper schematic map of Artin stacks over $\mc O_C$ and assume that $\mstack Y$ is $\Lambda$-locally acyclic and $\mstack Y_C$ is \'etale $1$-connected. Then ${\mstack X}$ is also $\Lambda$-locally acyclic.
	
	\begin{proof}
		Since $f$ is smooth and proper, by \Cref{adic_base_change} we have $\widehat f_{C*}\ul\Lambda_{\widehat{ \mstack X}_C}\simeq \lambda^{-1}_{\mstack Y}f_{C*}\ul\Lambda_{\mstack X_C}$ and it is enough to check that 
		$$
		\lambda^{-1}_{\mstack Y}\colon \RG(\mstack Y_C, f_{C*}\ul\Lambda_{\mstack X_C}) \xymatrix{\ar[r]&} \RG(\widehat{\mstack Y}_C, \widehat f_{C*}\ul\Lambda_{\widehat{ \mstack X}_C})
		$$ is an equivalence. By \Cref{functoriality_of_loc_sys} $f_{C*}\ul\Lambda\in \Shv_{\et}(\mstack X_C,\Lambda)^{\ge 0}$ is a derived local system. Since $\mstack Y$ is 1-connected, by \Cref{simply_connected} the Postnikov filtration of $f_{C*}\ul\Lambda$ has the associated graded pieces given by finite sums of shifts of the constant sheaf $\ul\Lambda_{\mstack Y_C}$. We conclude, since $\lambda^{-1}_{\mstack Y}$ is $t$-exact and $\mstack Y$ is $\Lambda$-locally acyclic.
	\end{proof}
\end{prop}

\begin{rem}
	Obviously, one can drop the \'etale 1-connectedness assumption in \Cref{prop: Lambda-acyclicity translates via proper smooth maps} if we assume that $\lambda_{\mstack Y}^{-1}$ is an equivalence not only for the constant sheaf but also for all local systems in the heart $\mathscr{L}_{\et}(\mstack Y,\Lambda)^{\heartsuit}$.
\end{rem}

In general, if we concentrate on the $\ell=p$ case, it seems that for $\mstack X$ to be $\mbb F_p$-locally acyclic one should impose some sort of (relaxed) properness condition on $\mstack X$.
For example, in \Cref{etale_coh_finiteness} we will deduce from the \'etale comparison that for a smooth Hodge-proper stack $\mstack X$ defined over $\mc O_K$ (with $K/\mathbb Q_p$ being finite) we have $\RG_\et(\widehat{\mstack X}_C, \mathbb F_p)\in \Coh^+(\mathbb F_p)$. In particular individual \'etale cohomology groups $H^i_\et(\widehat{\mstack X}_C,\mbb F_p)$ are finite, as well as are $H^i_\et(\mstack X,\mbb F_p)$ (by \Cref{rem:algebraic etale cohomology is usually finite}). This may give a slender hope that at least for Hodge-proper stacks the map $\Upsilon_{\mstack X}$ is an equivalence. On the other hand, \Cref{prop: Lambda-acyclicity translates via proper smooth maps} might be seen as a red flag, since it is not clear why Hodge-properness should persists under smooth proper maps.

Thus it seems reasonable to impose some further assumptions on $\mstack X$. One option is to assume that $\mstack X$ is cohomologically proper: at least this eliminates the objection above, since cohomological properness persists under any proper map. However, it is still not clear how to relate the algebraic \'etale cohomology to anything of coherent nature in this setting. Thus one might want to impose an even stronger condition which would give such a connection:

\begin{defn}(Similar to \cite[Definition 1.1.3]{HL_RelaxedProperness})
	A smooth Artin stack $\mstack X$ over $\mc O_C$ is called \emdef{formally proper} if it satisfies formal GAGA, namely if the natural $p$-completion functor
	$$
	\QCoh(\mstack X)^\perf \xymatrix{\ar[r]^\sim &} \prolim\QCoh(\mstack X_{\mc O_C/p^n})^\perf
	$$
	is an equivalence.
\end{defn}
\begin{rem}
	We call this formal GAGA because $\prolim\QCoh(\mstack X_{\mc O_C/p^n})^\perf$ can be seen as the category of perfect sheaves on the corresponding formal stack $\widehat{\mstack X}$.
\end{rem}
We make the following conjecture in this setup:
\begin{conj}\label{etale_conjecture}
	Let $\mstack X$ be a smooth formally proper Artin stack over $\mathcal O_K$. Then the comparison map
	$$
	\Upsilon_{\mstack X} \colon R\Gamma_\et(\mstack X_C, \mbb F_\ell)\xymatrix{\ar[r]^\sim &} R\Gamma_\et(\widehat{\mstack X}_C, \mbb F_\ell)
	$$
	is an equivalence for any $\ell$.
\end{conj}

\begin{rem}
	For now, if $\ell=p$ the class of stacks for which we can check that $\Upsilon_{\mstack X}$ is an equivalence are the quotient stacks $[X/G]$ under the assumption that $X$ is proper and $G$ is reductive (see \Cref{thm:main_result}). From \cite{HL_RelaxedProperness} one can deduce that all these stacks are formally proper. However the class of formally proper stacks is much bigger, in particular it includes the $\mbb G_m$-quotients of conical resolutions (at least under some projectivity assumptions this is covered by \cite[Proposition 4.2.3]{HL_RelaxedProperness}). More generally this situation fits well into the context of $\Theta$-stratifications where one could often use the associated semiorthogonal decompositions of $\QCoh(\mstack X)^\perf$ (constructed in \cite{Halpern-Leistner_Theta}, see also \cite[Section 3.3]{KubrakPrikhodko_HdR}) to reduce to quotients of smooth proper schemes by trivial actions. Some other examples also were considered by Ben Lim in his thesis \cite{lim2020algebraization}, where he shows, in particular, that the quotient stack $[\mbb A^n/\mr{SL}_n]$ is formally proper over a complete local ring in $\mr{char} \ \! p$.
\end{rem}

\subsection{Algebraic nearby cycles and local $\mathbb F_l$-acyclicity}\label{sec: Algebraic nearby cycles and local }
For this section our base is $\mc O_{C}$, and we fix a prime $\ell\ne p$. In particular $\ell$ is invertible in $\mc O_{C}$. We are going to study the map $\Upsilon_{\mstack X}$ for $\mbb F_\ell$-coefficients.

For an Artin stack $\mstack X$ consider $i_{\mstack X}\colon \mstack X_k \hookrightarrow \mstack X \leftarrow\mstack X_C \!\! \ :\! \!j_{\mstack X}$, the embeddings of the closed and generic fibers of $\mstack X \ra \Spec \mc O_C$. Note that both maps are qcqs since they are obtained by base change from $\Spec k \hookrightarrow \Spec \mc O_C \leftarrow \Spec C$. 

\begin{defn} Let $\mc F\in \Shv_{\et}(\mstack X,\mbb F_\ell)$ be an \'etale sheaf.
	\begin{enumerate}
		\item The \textit{nearby cycles sheaf} $\Psi_{\mstack X}\mc F\in \Shv_{\et}(\mstack X_k,\mbb F_\ell)$ is defined as 
		$$
		\Psi_{\mstack X}\mc F\coloneqq i_{\mstack X}^*j_{\mstack X*} j_{\mstack X}^*\mc F.
		$$
		
		\item The adjunction $\id \ra i_{\mstack X*} i_{\mstack X}^*$ gives a natural map $j_{\mstack X*} j_{\mstack X}^*\mc F \ra i_{\mstack X*} \Psi \mc F$ which induces a map 
		$$
		\eta_{\mstack X,\mc F}\colon \RG(\mstack X_C, j_{\mstack X}^*\mc F) \xymatrix{\ar[r]& \RG(\mstack X_k, \Psi_{\mstack X} \mc F)}.
		$$
	\end{enumerate}
\end{defn}

We start with the following simple base change lemma:

\begin{lem}\label{lem: base change for nearby cycles}
	Let $f\colon \mstack X \ra \mstack Y$ be a smooth morphism of Artin $\mc O_C$-stacks and let $\mc F\in \Shv_{\et}(\mstack Y,\mbb F_\ell)$. Then there is a natural equivalence $$
	f^*(\Psi_{\mstack Y}\mc F)\xymatrix{\ar[r]^\sim &} \Psi_{\mstack X}(f^*\mc F).
	$$
\end{lem}
\begin{proof}
	This follows from smooth base change for \'etale sheaves (\Cref{et_smooth_basechange_algebraic}) and the definition of $\Psi$.
\end{proof}

\begin{rem}\label{rem: nearby cycles in the smooth case}
	In the case when $\mstack X$ is a smooth Artin $\mc O_C$-stack and $\mc F= \ul{\mbb F_\ell}$ is constant we can apply \Cref{lem: base change for nearby cycles} to the structure morphism $p\colon \mstack X \ra \Spec \mc O_C$. The \'etale site of $\Spec \mc O_C$ is trivial and so $\Psi_{\Spec \mc O_C}\ul{\mbb F_\ell}\simeq \ul{\mbb F_\ell}\in \Shv_{\et}(\Spec k,\ul{\mbb F_\ell})$. Thus we get that 
	$$
	\Psi_{\mstack X}(\ul{\mbb F_\ell})\simeq p^*(\Psi_{\Spec \mc O_C}\ul{\mbb F_\ell})\simeq p^*\ul{\mbb F_\ell}\simeq \ul{\mbb F_\ell}\in \Shv_{\et}(\mstack X_k,\mbb F_\ell).
	$$
	In particular, $\RG(\mstack X_k, \Psi_{\mstack X} \ul{\mbb F_\ell})\simeq \RG(\mstack X_k, \mbb F_\ell)$ and the map $\eta_{\mstack X,\ul{\mbb F_\ell}}$ in this case gives a "specialization" map 
	$$
	\eta_{\mstack X}\coloneqq \eta_{\mstack X,\ul{\mbb F_\ell}}\colon \RG(\mstack X_C, \ul{\mbb F_\ell}) \xymatrix{\ar[r]&} \RG(\mstack X_k, {\mbb F_\ell}).
	$$
	
\end{rem}
\begin{rem}
	Note that from smooth base change for \'etale sheaves it also follows that for any $\mc F\in \Shv_{\et}(U,\mbb F_\ell)$ the assignment 
	$U \mapsto \RG(U_k,\Psi \mc F)$ defines a $\DMod{\mbb F_\ell}$-valued \'etale sheaf on $\Aff_{/\mc O_C}^\op$.
\end{rem}

The main result of this section is the following:
\begin{prop}\label{prop: F-l-acyclicity in terms of the neraby cycles}
	An Artin $\mc O_C$-stack $\mstack X$ locally of finite type is $\mbb F_\ell$-locally acyclic if and only if 
	$$
	\eta_{\mstack X}\coloneqq \eta_{\mstack X,\ul{\mbb F_\ell}}\colon \RG(\mstack X_C, \ul{\mbb F_\ell}) \xymatrix{\ar[r]&} \RG(\mstack X_k, \Psi_{\mstack X} \ul{\mbb F_\ell})$$ is an equivalence.
\end{prop}
\begin{proof}
	The key input of the proof is the result of Huber \cite[Theorem 3.5.13, Corollary 3.5.16]{Huber_adicSpaces} which, applied to our setting, states that for an $\mc O_C$-scheme $X$ locally of finite type one has a natural equivalence
	$$
	\iota_X\colon \RG(\widehat X_C,\mbb F_\ell)\xymatrix{\ar[r]^(.5)\sim &} \RG(X_k,\Psi\ul{\mbb F_\ell}).
	$$
	Moreover it is not hard to see that from the definition of the corresponding map in \cite{Huber_adicSpaces} one has a commutative triangle
	$$
	\xymatrix{\RG(X_C,\mbb F_\ell) \ar[r]^(0.45){\eta_{\mstack X}}\ar[d]_{\Upsilon_{\mstack X}}& \RG(X_k,\Psi_{X}\ul{\mbb F_\ell})\\
		\RG(\widehat X_C,\mbb F_\ell)\ar[ru]^\sim_{\iota_{\mstack X}}}.
	$$
	
	By smooth base change for nearby cycles (\Cref{lem: base change for nearby cycles}) and hyperdescent for \'etale cohomology (\Cref{rem: hyperdescent for etale cohomology}), for any  hypercover $\mstack U_\bullet \ra \mstack X$ with $U_i\ra \mstack X$ being smooth we have 
	$$
	\RG(\mstack X_k,\Psi_{\mstack X}\ul{\mbb F_\ell})\xymatrix{\ar[r]^(.5)\sim &} \Tot\RG({(\mstack U_\bullet)}_k,\Psi_{\mstack U_\bullet}\ul{\mbb F_\ell}).
	$$
	Picking a hypercover (using \cite[Theorem 4.7]{Pridham_ArtinHypercovers}), where all $\mstack U_i$ are disjoint unions of finite type affine schemes  we deduce that there is an equivalence 
	$$
	\iota_{\mstack X}\colon \RG(\widehat{\mstack X}_C,\ul{\mbb F_\ell})\xymatrix{\ar[r]^(.5)\sim &}\RG(\mstack X_k,\Psi_{\mstack X}\ul{\mbb F_\ell}),
	$$
	and, moreover, $\eta_{\mstack X}\simeq \iota_{\mstack X}\circ\Upsilon_{\mstack X}$. Hence the statement.
\end{proof}

Quite differently from the $l=p$ case there are many more examples of  $\mbb F_\ell$-locally acyclic schemes.
\begin{lem}\label{lem: example of F-l acyclic scheme}
	Let $U\subset X$ be an open subscheme of a smooth proper $\mc O_C$-scheme $X$, such that $Z\coloneqq X\setminus U$ is also smooth. Then $U$ is $\mbb F_\ell$-locally acyclic.
\end{lem}
\begin{proof}
	Let's denote by $f\colon Z \hookrightarrow X$ and $g\colon U \ra X$ the corresponding embeddings. Then one has a fiber sequence
	$$
	f_*f^!\ul{\mbb F_\ell}\xymatrix{\ar[r]&} \ul{\mbb F_\ell} \xymatrix{\ar[r]&} g_*\ul{\mbb F_\ell} 
	$$
	in $\Shv_{\et}(X,\mbb F_\ell)$. Moreover, by the smoothness assumption, we have $f^!\ul{\mbb F_\ell}\simeq \mbb F_\ell[-2d]$, where $d$ is the relative codimension of $Z$ in $X$ (here we also used that $C$ is algebraically closed to trivialize the Tate twist). It follows that if $\eta_Z$ and $\eta_X$ are equivalences, then $\eta_U$ is also an equivalence. Thus it is enough to restrict to the smooth proper case. But there everything follows from smooth proper base change.
\end{proof}

\begin{ex}\label{ex: A^1 and G_m are F-l acyclic}
	In particular, $\mbb A^1$ and $\mbb G_m$ are $\mbb F_\ell$-locally acyclic. Indeed one can compactify both by schemes by $\mbb P^1$ with complement given by either one or two copies of $\Spec \mc O_C$ respectively. Consequently, e.g. by the K\"unneth formula (proof of \Cref{lem: product of acyclic is acyclic}), any product $\mbb A^k\times \mbb G_m^l$ is also $\mbb F_\ell$-locally acyclic. Recall that we had a very different behavior in the case $l=p$ (\Cref{ex:F-p etale cohomology of rigid affine line}).
\end{ex}

In particular, from this it is rather easy to show that $BG$ for $G$ reductive is $\mbb F_\ell$-locally acyclic. 
\begin{prop}\label{prop: local F_l-acyclicity of quotients}
	Let $G$ be a reductive group over $\mc O_C$. Then $BG$ is $\mbb F_\ell$-locally acyclic. Consequently, any quotient $[X/G]$ of an $\mbb F_\ell$-locally acyclic $\mc O_C$-scheme $X$ (e.g. smooth and proper) is $\mbb F_\ell$-locally acyclic.
\end{prop}
\begin{proof}	
	Let $B\subset G$ be a Borel subgroup (note that $G$ is automatically split, since $\mc O_C$ has no nontrivial \'etale covers). By \Cref{ex: A^1 and G_m are F-l acyclic}, since $B\simeq T\times U$ (as schemes), $B$ and its self-products $B^n$ are $\mbb F_\ell$-locally acyclic; thus so is $BB$ (by considering the \v Cech cover $\Spec \mc O_C \ra BB$ and applying \Cref{lem: descent for F-l acyclic stacks}).
	
	Given any smooth proper $\mc O_C$-scheme $X$ with a $B$-action, the quotient stack $[X/B]$ is also $\mbb F_\ell$-locally acyclic. Indeed, the morphism $[X/B]\ra BB$ is smooth and proper, and $BB_C$ is \'etale 1-connected since $B_C$ is connected; thus we are done by \Cref{prop: Lambda-acyclicity translates via proper smooth maps}.

	Finally, considering a smooth proper cover $BB\ra BG$ we can consider the corresponding \v Cech object, all of whose terms are given by double-quotients $B\backslash(G\times_B \cdots \times_B G)/B$, where each $(G\times_B \cdots \times_B G)/B$ is a smooth proper scheme. Applying \Cref{lem: descent for F-l acyclic stacks} to this simplicial cover we get that $BG$ is $\mbb F_\ell$-locally acyclic. The case of $[X/G]$ then follows by considering the \v Cech object of the smooth cover $X\ra [X/G]$, all of whose terms are now given by products $X\times G^n$ and are $\mbb F_\ell$-locally acyclic by the K\"unneth formula.
\end{proof}
\begin{rem}
	In fact the reductive group $G$ itself is also $\mbb F_\ell$-locally acyclic. This can be seen by either considering a so-called wonderful compactification of $G$ and using \Cref{lem: example of F-l acyclic scheme} directly or by considering the map $q\colon G\ra G/B$. The map $q$ is an \'etale (and in fact Zariski) locally-trivial $B$-torsor; from this it is automatic that $q_*\mbb F_\ell$ is a local system and then the  $\mbb F_\ell$-local acyclicity of $G$ follows from $q_*\mbb F_\ell$-local acyclicity of $G/B$ (which in turn reduces to "local acyclicity" and proper base change for the smooth and proper morphism $G/B\ra \Spec \mc O_C$).
\end{rem}
\begin{ex}
	To give the reader a better feeling of $\mbb F_\ell$-local acyclicity we also provide some examples of non-$\mbb F_\ell$-locally acyclic stacks. First of all, if the closed fiber $\mstack X_k$ is empty then $\RG(\mstack X_k, \Psi_{\mstack X}\mc F)\simeq 0$ for any $\mc F$ and thus such $\mstack X$ can't be $\mbb F_\ell$-locally acyclic. As a slightly more elaborate example one can take $\mbb A^1$ over $\mc O_C$ and remove a point $0_k\in \mbb A^1_k \hookrightarrow \mbb A^1$ from the closed fiber: $X\coloneqq \mbb A^1\setminus\{0_k\}$. The scheme $X$ is still smooth, but the $\mbb F_\ell$-cohomology of the closed and generic fibers are different, thus by \Cref{rem: nearby cycles in the smooth case} $X$ is not $\mbb F_\ell$-locally acyclic. 
\end{ex}

\subsection{\'Etale comparison and applications}\label{subsect:etale_comparison}
\subsubsection{\'Etale comparison}\label{subsubsect:etale_comparison}
Let $(A,(d))$ be a perfect prism. In this section we prove the \'etale comparison: namely, we express the \'etale cohomology of the Raynaud generic fiber as Frobenius-fixed points of a suitable localization of the prismatic cohomology. We start with the $\mbb Z/p^n$-coefficients. 

Note that since $\phi_A\colon A\ra A$ is perfect the underlying $\mbb Z_p$-module of $\RG_{\Prism^{(1)}}(\mstack X/A)\simeq \RG_{\Prism}(\mstack X/A)^{(1)}$ (see \Cref{prop: twisted_cohomology_is_a_twist}) is isomorphic to $\RG_{\Prism}(\mstack X/A)$. Also note that since $\phi(d)-d^p \in pA$, we have $A[\tfrac{1}{d}]/p^n\simeq A[\tfrac{1}{\phi(d)}]/p^n$ and then consequently $\RG_{\Prism}(\mstack X/A)^{(1)}[\frac{1}{d}]/p^n \simeq \RG_{\Prism}(\mstack X/A)[\frac{1}{d}]/p^n$. Putting this together we see that for each $n$ the prismatic Frobenius gives a well-defined ($\mbb Z_p$-linear) endomorphism
$$
\phi_{\Prism}\colon \RG_{\Prism}(\mstack X/A)[\tfrac{1}{d}]/p^n \xymatrix{\ar[r]&} \RG_{\Prism}(\mstack X/A)[\tfrac{1}{d}]/p^n.
$$
For a complex $M\in\DMod{\mbb Z_p}$ with an endomorphism $\phi$ we denote by $M^{\phi=1}$ the homotopy $\phi$-fixed points $M^{\phi=1}\coloneqq \fib(\phi-\id_M\colon M \ra M)$.
\begin{prop}\label{etale_comparison}
	Let $(A,(d))$ be a perfect prism. Let $\mstack X$ be a smooth quasi-compact quasi-separated Artin stack over $R\coloneqq A/I$. Then for any $n \ge 0$ there is a canonical equivalence
	$$R\Gamma_{\et}(\widehat{\mstack X}_{{R[\frac{1}{p}]}}, \mathbb Z/p^n) \simeq \left(R\Gamma_\Prism(\mstack X/ A)[\tfrac{1}{d}]/p^n\right)^{\phi_{\Prism} = 1}. $$
	
	\begin{proof}
	By \Cref{left_exact_preserve_totalizations_of_uniformly_bounded_below} the localization $-[\frac{1}{d}]$ commutes with totalizations of $0$-coconnected objects. Being a shift of a limit and a limit, $-/p^n$ and $(-)^{\phi_{\Prism}=1}$ commute with any totalizations. Thus both sides satisfy \'etale (equivalently smooth) descent and, since they take values in the $\DMod{\mbb Z/p^n \mbb Z}^{\ge 0}\subset \DMod{\mbb Z/p^n \mbb Z}$, satisfy \'etale hyperdecent as well. By considering a smooth hypercover of $\mstack X$ by affine schemes provided by \cite[Theorem 4.7]{Pridham_ArtinHypercovers}, the result follows from the affine case \cite[Theorem 9.1]{BS_prisms}.
	\end{proof}
\end{prop}

Recall that by \cite[Corollary 2.31]{BS_prisms} for any perfect prism $(A, d)$ one in fact has $A=W(R^\flat)$ where $R^\flat\coloneqq A/p$ is perfect ring of characteristic $p$. In particular $A$ is necessarily $p$-torsion free. Also note that the reduction $\ol d\in R^\flat$ is non-zero (\cite[Lemma 2.33]{BS_prisms}). To pass to the limit by $n$ in \Cref{etale_comparison} we will need the following easy lemma:
\begin{lem}\label{lem: p-completion of A[1/d]}
	Let $(A,(d))$ be a perfect prism. Then the derived $p$-completion of $A[\frac{1}{d}]$ is classical and equivalent to $W(R^\flat[\frac{1}{\ol d}])$.
	
	\begin{proof}
		Note that since $A$ is $p$-torsion free, so is $A[\frac{1}{d}]$; thus the derived completion ${A[\frac{1}{d}]}^\wedge_p$ is classical. The element $d\in W(R^\flat[\frac{1}{\ol d}])$ is invertible modulo $p$ and since $W(R^\flat[\frac{1}{\ol d}])$ is (derived) $p$-complete, it is invertible. This gives a map ${A[\frac{1}{d}]}^\wedge_p\ra W(R^\flat[\frac{1}{\ol d}])$, which is an isomorphism mod $p$; thus so was the original map, since both sides are derived $p$-complete.
	\end{proof}
\end{lem}

\begin{cor}\label{cor: Z_p etale cohomology of rigid fiber}
	In notations of \Cref{etale_comparison} we also have an equivalence $$\xymatrix{R\Gamma_{\et}(\widehat{\mstack X}_{{R[\frac{1}{d}]}}, \mathbb Z_p) \simeq \left(R\Gamma_\Prism(\mstack X/ A)\widehat{\otimes}_A W(R^\flat[\frac{1}{\ol d}])\right)^{\phi_{\Prism} = 1}},$$ where the tensor product is $p$-completed.
	
	\begin{proof}
		Passing to the limit over $n$ in \Cref{etale_comparison} we get an equivalence $$\xymatrix{R\Gamma_{\et}(\widehat{\mstack X}_{{R[\frac{1}{d}]}}, \mathbb Z_p) \simeq \left(R\Gamma_\Prism(\mstack X/ A)\widehat{\otimes}_A {A[\frac{1}{d}]}^\wedge_p\right)^{\phi_{\Prism} = 1}}.$$
		The rest then follows from ${A[\frac{1}{d}]}^\wedge_p\simeq W(R^\flat[\frac{1}{\ol d}])$. 
	\end{proof}
\end{cor}

\begin{rem}
	Even though the \'etale comparison \cite[Theorem 9.1]{BS_prisms} applies to any formal scheme without any smoothness assumptions, the proof of \Cref{etale_comparison} does not immediately generalize to arbitrary (say flat) Artin stacks over $A/I$, unless $\mstack X$ has a hypercover by affine schemes with uniformly bounded below prismatic cohomology. This does not hold for a general affine scheme, but it does for spectra of (quasi)syntomic rings. In particular the same proof works for syntomic stacks.
\end{rem}

We now restrict to the case $A=\Ainf$, where we have $R=\mc O_{\mbb C_p}$ and ${A[\frac{1}{d}]}^\wedge_p\simeq W(\mbb C_p^\flat)$. 
We are going to deduce some finer formulas for the \'etale cohomology in the case when a stack comes from a Hodge-proper stack over $\mc O_K$. For this we fix a Breuil-Kisin prism $(\mf S,E(u))$ corresponding to a uniformizer $\pi\in \mc O_K$ and also fix an element $\pi^\flat\coloneqq (\pi, \pi^{1/p}, \pi^{1/p^2}, \ldots)\in \mathcal O_{\mbb C_p}^\flat$ defined by a compatible choice of $p^n$-th roots of $\pi$. Such $\pi^\flat$ defines a map $(\mf S,(E(u)))\ra (\Ainf, (\xi))$ by sending $u\in \mf S$ to $[\pi^\flat]$. This map lifts the natural embedding $\mc O_K\ra \mc O_{\mbb C_p}$. 

Let $\mstack X$ be a smooth qcqs Artin stack over $\mc O_K$. The $\mf S$-linear Frobenius $R\Gamma_{\Prism^{(1)}}(\mstack X / \mathfrak S) \ra R\Gamma_{\Prism}(\mstack X / \mathfrak S)$ defines a $W(\mbb C_p^\flat)$-linear map 
$$R\Gamma_{\Prism^{(1)}}(\mstack X / \mathfrak S)\otimes_{\mf S} W(\mbb C_p^\flat)\tto R\Gamma_\Prism(\mstack X / \mathfrak S)\otimes_{\mf S} W(\mbb C_p^\flat)$$ 
where the underlying $\mbb Z_p$-module of $R\Gamma_{\Prism^{(1)}}(\mstack X / \mathfrak S)\otimes_{\mf S} W(\mbb C_p^\flat)$ is the same as for $R\Gamma_\Prism(\mstack X / \mathfrak S)\otimes_{\mf S} W(\mbb C_p^\flat)$ since $\phi_{W(\mbb C_p^\flat)}\colon W(\mbb C_p^\flat)\xra{\sim} W(\mbb C_p^\flat)$ is an isomorphism. This defines a $\mbb Z_p$-linear endomorphism $\phi_{\Prism}$ of $R\Gamma_\Prism(\mstack X / \mathfrak S)\otimes_{\mf S} W(\mbb C_p^\flat)$.

\begin{prop}\label{cor: adic_etale_comparison2}
	Let $(\mf S,E(u))$ be a Breuil-Kisin prism. Let $\mstack X$ be a smooth Hodge-proper Artin stack over $\mathcal O_K\simeq \mf S/E(u)$. Then there is a canonical equivalence
	$$
	R\Gamma_{\et}(\widehat{\mstack X}_{{\mathbb C_p}}, \mathbb Z_p)\xymatrix{\ar[r]^\sim&} \left(R\Gamma_\Prism(\mstack X / \mathfrak S)\otimes_{\mathfrak S} W(\mathbb C_p^\flat)\right)^{\phi_{\Prism}=1}.
	$$
	Moreover, it induces a natural equivalence 	
	$$R\Gamma_{\et}(\widehat{\mstack X}_{{\mathbb C_p}}, \mathbb Z_p)\otimes_{\mathbb Z_p} W(\mathbb C_p^\flat) \xymatrix{\ar[r]^\sim&}  R\Gamma_\Prism(\mstack X / \mathfrak S)\otimes_{\mathfrak S} W(\mathbb C_p^\flat).$$
		\begin{proof}
		First note that by base change (\Cref{prismatic_basechange}) for a morphism of prisms $(\mf S,E(u))\ra (\Ainf, \widetilde{\xi})$ we have an equivalence $\RG_{\Prism}(\mstack X_{\mc O_{\mbb C_p}}/\Ainf)\simeq \RG_{\Prism}(\mstack X/\mf S)\widehat\otimes_{\mf S} \Ainf$. Then, applying \Cref{cor: Z_p etale cohomology of rigid fiber} to $\mstack X_{\mc O_{\mbb C_p}}$, we get 
		$$
		R\Gamma_{\et}(\widehat{\mstack X}_{{\mathbb C_p}}, \mathbb Z_p)\simeq \left(R\Gamma_\Prism(\mstack X / \mathfrak S)\widehat{\otimes}_{\mathfrak S} W(\mathbb C_p^\flat)\right)^{\phi_{\Prism}=1}.
		$$ 
		Since $\mstack X$ is Hodge-proper, by \Cref{cor: prismatic cohomology are bounded below coherent}, $R\Gamma_\Prism(\mstack X / \mathfrak S)\in \Coh^+(\mf S)$. By \Cref{lem: base change preserves coh^+ sometimes}, the tensor product $R\Gamma_\Prism(\mstack X / \mathfrak S){\otimes}_{\mathfrak S} W(\mathbb C_p^\flat)$ also lies in $ \Coh^+(W(\mathbb C_p^\flat))$, thus by \Cref{lem: bounded below coherent are complete} it is derived $p$-complete (and so coincides with the completed tensor product). This proves the first assertion.
		
		For the second statement it is enough to prove the following more general assertion: for an object $$(M,\phi_M)\in\left({\DMod{W(\mathbb C_p^\flat)}}_p^\wedge\right)^{\phi_{W(\mathbb C_p^\flat)*}}_{\mr{lax}}$$
		with the semi-linear morphism $\phi_M\colon M\ra \phi_{W(\mathbb C_p^\flat)*}M$ such that the underlying object $M$ lies in $\Coh^+(W(\mathbb C_p^\flat))$, the natural map
		$$M^{\phi_M=1}\otimes_{\mathbb Z_p} W(\mathbb C_p^\flat) \xymatrix{\ar[r] & } M$$
		is an equivalence. Note that both parts commute with colimits, hence it is enough to prove the statement for $M^{\le n}$. But since $W(\mathbb C_p^\flat)$ is regular and Noetherian, coherent $W(\mathbb C_p^\flat)$-modules are perfect, so the result follows from \cite[Lemma 8.5]{Bhatt_LiteBMS}.
	\end{proof}
\end{prop}

One can also give a formula for the \'etale cohomology of $\widehat{\mstack X}_{\mbb C_p}$ that would not depend on the auxiliary choice of the Breuil-Kisin prism $(\mf S, E(u))$.
\begin{cor}\label{cor: prismatic_vs_etale_torsion} For a Hodge-proper smooth Artin stack over $\mc O_K$ one also has an equivalence 
$$
	R\Gamma_{\et}(\widehat{\mstack X}_{{\mathbb C_p}}, \mathbb Z_p)\xymatrix{\ar[r]^\sim &} \left(R\Gamma_\Prism(\mstack X _{\mc O_{\mbb C_p}}/ \Ainf)\otimes_{\Ainf} W(\mathbb C_p^\flat)\right)^{\phi_{\Prism}=1}
$$
that in turn leads to an equivalence 
$$
R\Gamma_{\et}(\widehat{\mstack X}_{{\mathbb C_p}}, \mathbb Z_p)\otimes_{\mathbb Z_p} W(\mathbb C_p^\flat)  \xymatrix{\ar[r]^\sim &} R\Gamma_\Prism(\mstack X _{\mc O_{\mbb C_p}}/ \Ainf)\otimes_{\Ainf} W(\mathbb C_p^\flat).
$$
	
\begin{proof} For the first point, by \Cref{cor: Z_p etale cohomology of rigid fiber} we have $R\Gamma_{\et}(\widehat{\mstack X}_{{\mathbb C_p}}, \mathbb Z_p)\xra{\sim} \left(R\Gamma_\Prism(\mstack X _{\mc O_{\mbb C_p}}/ \Ainf)\widehat \otimes_{\Ainf} W(\mathbb C_p^\flat)\right)^{\phi_{\Prism}=1}$ and thus it is enough to see that the tensor product $R\Gamma_\Prism(\mstack X _{\mc O_{\mbb C_p}}/ \Ainf)\otimes_{\Ainf} W(\mathbb C_p^\flat)$ is already $p$-complete. For this we use $\RG_{\Prism}(\mstack X/\mf S)$ for any choice of $\mf S$. Then, by \Cref{lem:the tensor product remains complete} and base change $R\Gamma_\Prism(\mstack X _{\mc O_{\mbb C_p}}/ \Ainf)\simeq \RG_{\Prism}(\mstack X/\mf S)\otimes_{\mf S}\Ainf$ and so $R\Gamma_\Prism(\mstack X _{\mc O_{\mbb C_p}}/ \Ainf)\otimes_{\Ainf} W(\mathbb C_p^\flat)\simeq \RG_{\Prism}(\mstack X/\mf S)\otimes_{\mf S} W(\mathbb C_p^\flat)$ which is $p$-complete again by \Cref{lem:the tensor product remains complete}.
The second point is also clear from \Cref{cor: adic_etale_comparison2} since we have just identified the right hand side with $\RG_{\Prism}(\mstack X/\mf S)\otimes_{\mf S} W(\mathbb C_p^\flat)$.
\end{proof}
\end{cor}

\begin{rem}\label{rem: comparison for the twisted prismatic cohomology}
	One can prove a variant of \Cref{cor: prismatic_vs_etale_torsion} with the twisted version $\RG_{\Prism^{(1)}}(\mstack X _{\mc O_{\mbb C_p}}/ \Ainf)\otimes_{\Ainf} W(\mathbb C_p^\flat)$ in place of $\RG_{\Prism}(\mstack X / \Ainf)\otimes_{\Ainf} W(\mathbb C_p^\flat)$. Indeed, this does not change the statement in the first part of the proposition (since the underlying $\mbb Z_p$-modules are the same), but gives another equivalence in the second:
	$$
	R\Gamma_{\et}(\widehat{\mstack X}_{{\mathbb C_p}}, \mathbb Z_p)\otimes_{\mathbb Z_p} W(\mathbb C_p^\flat)\xymatrix{\ar[r]^\sim&} R\Gamma_{\Prism^{(1)}}(\mstack X _{\mc O_{\mbb C_p}}/ \Ainf)\otimes_{\Ainf} W(\mathbb C_p^\flat).
	$$
	One can then return to a similarly looking statement via the identification $$R\Gamma_{\Prism^{(1)}}(\mstack X _{\mc O_{\mbb C_p}}/ \Ainf){\otimes}_{\Ainf} W(\mathbb C_p^\flat)\xymatrix{\ar[r]^\sim&}  R\Gamma_{\Prism}(\mstack X _{\mc O_{\mbb C_p}}/ \Ainf){\otimes}_{\Ainf, \phi_\Ainf} W(\mathbb C_p^\flat)\xymatrix{\ar[r]_\sim^{\id\otimes \phi}&} R\Gamma_{\Prism}(\mstack X _{\mc O_{\mbb C_p}}/ \Ainf){\otimes}_{\Ainf} W(\mathbb C_p^\flat).$$
	The resulting equivalence $
	R\Gamma_{\et}(\widehat{\mstack X}_{{\mathbb C_p}}, \mathbb Z_p)\otimes_{\mathbb Z_p} W(\mathbb C_p^\flat)\xra{\sim} \RG_{\Prism}(\mstack X / \Ainf)\otimes_{\Ainf} W(\mathbb C_p^\flat)$ thus will differ from the one in \Cref{cor: prismatic_vs_etale_torsion} by composing with Frobenius on $W(\mathbb C_p^\flat)$.
\end{rem}

From these finer formulas we deduce that $p$-adic \'etale cohomology of a smooth Hodge-proper stack are finitely generated over $\mbb Z_p$:
\begin{cor}\label{etale_coh_finiteness}
	Let $\mstack X$ be a smooth Hodge-proper Artin stack over $\mathcal O_K$. Then $R\Gamma_{\et}(\widehat{\mstack X}_{{\mathbb C_p}}, \mathbb Z_p)$ (resp. $R\Gamma_{\et}(\widehat{\mstack X}_{{\mathbb C_p}}, \mathbb Z/p^n)$) is a bounded below coherent complex of $\mathbb Z_p$-(resp. $\mbb Z/p^n$-)modules.
	
	\begin{proof}
		By the \'etale comparison and \Cref{cor: prismatic cohomology are bounded below coherent} the complex $R\Gamma_\et(\widehat{\mstack X}_{{\mathbb C_p}}, \mathbb Z_p)\otimes W(\mathbb C_p^\flat)$ is bounded below coherent. Since the natural map $\mathbb Z_p \to W(\mathbb C_p^\flat)$ is faithfully flat, we conclude by \Cref{ncoh_is_fpqc_local}. Reducing modulo $p^n$ we get the statement for $R\Gamma_{\et}(\widehat{\mstack X}_{{\mathbb C_p}}, \mathbb Z/p^n)$ (e.g. using \Cref{lem: base change preserves coh^+ sometimes}).
	\end{proof}
\end{cor}

\begin{rem}\label{rem: adic cohomology of Hodge-proper stack are finitely generated}
	It follows that if $\mstack X$ is smooth Hodge-proper, each individual cohomology $H^i_\et(\widehat{\mstack X}_{{\mathbb C_p}}, \mathbb Z_p)$ (resp. $H^i_\et(\widehat{\mstack X}_{{\mathbb C_p}}, \mathbb Z/p^n)$) is a finitely generated $\mbb Z_p$-(resp. $\mbb Z/p^n$-)module. From this we also get that $$H^i_\et(\widehat{\mstack X}_{{\mathbb C_p}}, \mathbb Z_p)\xra{\sim} \lim_n H^i_\et(\widehat{\mstack X}_{{\mathbb C_p}}, \mathbb Z/p^n);$$ indeed  $R\Gamma_\et(\widehat{\mstack X}_{{\mathbb C_p}}, \mathbb Z_p)\xra{\sim}\lim\limits_{\leftarrow n} R\Gamma_\et(\widehat{\mstack X}_{{\mathbb C_p}}, \mathbb Z/p^n)$ by definition, but since all groups $\{H^i_\et(\widehat{\mstack X}_{{\mathbb C_p}}, \mathbb Z/p^n)\}_n$ are finite, they form a Mittag-Leffler system and thus the derived limit on all the cohomology agrees with the classical one.
\end{rem} 
\begin{rem}
	If $(\ell,p)=1$, one can also show that $R\Gamma_\et(\widehat{\mstack X}_{{\mathbb C_p}}, \mathbb F_\ell)\in \Coh^+(\mathbb F_\ell)$ for any $\mc O_C$-stack $\mstack X$ of finite type. Indeed, using smooth hyperdescent it is enough to show this for the case of a finite type affine scheme over $\mc O_C$. By \Cref{prop: F-l-acyclicity in terms of the neraby cycles} we have $R\Gamma_\et(\widehat{\mstack X}_{{\mathbb C_p}}, \mathbb F_\ell)\simeq R\Gamma_\et(\widehat{\mstack X}_{\ol{\mbb F}_p}, \Psi_{\mstack X}\mathbb F_\ell)$, the complex $\Psi_{\mstack X}\mathbb F_\ell\in \Shv_{\et}(\mstack X_{\ol{\mbb F}_p},\mbb F_\ell)$ is constructible by \cite[Th. finitude]{SGA4_12} and thus its cohomology is finite-dimensional (e.g. by \cite[Chapitre 7, Th\'eor\`eme 1.1]{SGA4_12}).
\end{rem}

\subsubsection{Fontaine's crystalline conjecture for stacks and Breuil-Kisin modules}\label{sssection:C_cris_for_stacks}
In this section we deduce an analogue of Fontaine's crystalline conjecture for stacks: namely we prove that the rational \'etale cohomology of the Raynaud generic fiber of a smooth Hodge-proper stack $\mstack X$ over $\mathcal O_K$ is a crystalline representation. We also show that the prismatic cohomology in this case provides a geometric description of the Breuil-Kisin functor applied to $H^n_\et(\widehat{\mstack X}_{\mbb C_p}, \mbb Z_p)$.

Let $K$ be a finite extension of $\mbb Q_p$ and $k$ be the residue field of $\mc O_K$. Let $G_K$ be the absolute Galois group of $K$. Let, as before, $\mf S\coloneqq W(k)[[u]]$ be the Kisin ring with a $\mbb Z_p$-linear endomorphism $\phi_{\mf S}\curvearrowright\mf S$, sending $u$ to $u^p$. Let $\pi\in \mc O_K$ be a uniformizer and $E(u)\in \mf S$ be the minimal polynomial of $\pi$ over $K_0\coloneqq W(k)[\frac{1}{p}]\subset K$. Recall the following definition:
\begin{defn}\label{defn: Breuil-Kisin module}
	\emdef{A Breuil-Kisin module $(M,\phi_M)$} is a data of a finitely generated $\mathfrak S$-module $M$ together with an isomorphism
	$$\phi_M\colon M^{(1)}[\tfrac{1}{E}] \xymatrix{\ar[r]^\sim&} M[\tfrac{1}{E}],$$
	where $M^{(1)} \coloneqq \phi^*_{\mf S} M$ denotes the Frobenius twist of $M$. The category of Breuil-Kisin modules will be denoted by $\Mod_{\mathfrak S}^\phi$.
\end{defn}

\begin{ex}\label{ex: prismatic cohomology of Hodge-proper stack is a BK-module}
	Let $\mstack X$ be a smooth Hodge-proper $\mc O_K$-stack. Then $H^i_\Prism(\mstack X/\mf S)$ has a natural structure of a Breuil-Kisin module. Indeed, by \Cref{cor: prismatic cohomology are bounded below coherent} each $H^i_\Prism(\mstack X/\mf S)$ is a finitely-generated $\mf S$-module and the Frobenius map induces an isomorphism	(see \Cref{ex: Frobeinus is iso Breuil-Kisin})    
	$$
	\phi_{\Prism}[\tfrac{1}{E}]\colon H^i_\Prism(\mstack X/\mf S)\fr[\tfrac{1}{E}] \xymatrix{\ar[r]^(0.5){\sim}&} H^i_\Prism(\mstack X/\mf S)[\tfrac{1}{E}].
	$$
	In particular this applies when $\mstack X$ is a smooth proper scheme.
\end{ex}

\begin{ex}\label{ex:stuff on BK-modules}
	The structure of a Breuil-Kisin module imposes some restrictions on the underlying structure of an $\mf S$-module. Namely, for any $(M,\phi_M)$ there is a canonical exact sequence of Breuil-Kisin modules 
	$$
	\xymatrix{0\ar[r]& (M_{\mr{tor}}, \phi_{M_{\mr{tor}}})\ar[r]& (M,\phi_M)\ar[r]& (M_{\mr{free}}, \phi_{M_{\mr{free}}}) \ar[r]& (\ol M, \phi_{\ol M}) \ar[r]& 0,}
	$$
	where $M_{\mr{tor}}$ is killed by a power of $p$, $M_{{\mr{free}} }$ is free over $\mf S$, $\ol M$ is killed by a power of $(p,E)$ (see e.g. \cite[Proposition 4.3]{BMS1}).
	From this it follows in particular that
	\begin{itemize}
		\item $M[\frac{1}{p}]$ is a free $\mf S [\frac{1}{p}]$-module.

		\item If $M$ is a torsion $\mf S$-module, it is killed by some power of $p$.
	\end{itemize}
\end{ex}

\medskip

From the discussion in \Cref{ex:stuff on BK-modules} we immediately get the following result comparing the dimensions of different (rationalized) $p$-adic cohomology theories:
\begin{prop}\label{prop: rationally all cohomology has the same dimension}
	Let $\mstack X$ be a smooth Hodge-proper $\mc O_K$-stack. Let $K_0\coloneqq W(k)[\frac{1}{p}]$ and let $H^i_\et(\widehat {\mstack X}_{\mbb C_p},\mbb Q_p)\coloneqq H^i_\et(\widehat {\mstack X}_{\mbb C_p},\mbb Z_p)[\tfrac{1}{p}]$. Then 
	$$
	\dim_{K} H^i_\dR(\mstack X_K)=\dim_{K_0}\left(H^i_\crys(\mstack X_k/W(k))[\tfrac{1}{p}]\right)= \dim_{\mbb Q_p}H^i_\et(\widehat {\mstack X}_{\mbb C_p},\mbb Q_p).
	$$

\begin{proof}
	First by \Cref{cor: de Rham coh of Hodge-proper are complete} one has $H^i_\dR(\mstack X_K)\simeq H^i_{L\dR^\wedge_p}(\mstack X/\mc O_K)[\frac{1}{p}]$. Also, both $\mf S/u$ and $\mf S/E(u)$ are perfect and thus the completed tensor product of $\RG_{\Prism^{(1)}}(\mstack X/\mf S)$ with either coincides with the non-completed one. By \Cref{ex:stuff on BK-modules} the cohomology of $\RG_{\Prism}(\mstack X/\mf S)[\frac{1}{p}]$ and, consequently, $\RG_{\Prism^{(1)}}(\mstack X/\mf S)[\frac{1}{p}]\simeq \RG_{\Prism}(\mstack X/\mf S)^{(1)}[\frac{1}{p}]$ are free $\mf S[\frac{1}{p}]$-modules. By de Rham, crystalline and \'etale comparisons (plus the universal coefficients formula) $H^i_{L\dR^\wedge_p}(\mstack X/\mc O_K)[\frac{1}{p}]$, $H^i_\crys(\mstack X_k/W(k))[\tfrac{1}{p}]$ and $H^i_\et(\widehat {\mstack X}_{\mbb C_p},\mbb Q_p)$ all have the same dimension given by the rank of $H^i_{\Prism^{(1)}}(\mstack X/\mf S)[\frac{1}{p}]$.
\end{proof}
\end{prop}

\begin{cor}\label{cor: dimensions of \'etale cohomology is the same}
	Let $\mstack X$ be a smooth Hodge-proper $\mc O_K$-stack. Then 
	$$
	\dim H^i_\et(\widehat {\mstack X}_{\mbb C_p},\mbb Q_p) = \dim H^i_\et({\mstack X}_{\mbb C_p},\mbb Q_p).
	$$

\begin{proof}
	Choose some isomorphism $\mbb C_p\simeq \mbb C$; it induces embeddings $K\inj \mbb C$ and $\mbb Q_p\inj \mbb C$. By base change $\dim_{K} H^i_\dR(\mstack X_K/K)=\dim_{\mbb C} H^i_\dR(\mstack X_{\mbb C}/\mathbb C)$. Recall that by Artin's comparison (\Cref{Artins_comparison_stacks}) $H^i_\et({\mstack X}_{\mbb C_p},\mbb Q_p)\simeq H^i_\sing(\Pi_{\infty}{\mstack X}(\mbb C),\mbb Q_p)$. By base change of coefficients $\dim_{\mbb Q_p} H^i_\sing(\Pi_{\infty}{\mstack X}(\mbb C),\mbb Q_p)= \dim_{\mbb C} H^i_\sing(\Pi_{\infty}{\mstack X}(\mbb C),\mbb C)$. 
	
	Both singular and de Rham cohomology satisfy smooth descent. For a smooth scheme $X$ over $\mbb C$ one has a natural equivalence $\RG_\sing({X}(\mbb C),\mbb C)\xra{\sim} \RG_\dR(X/\mathbb C)$, which induces an equivalence  $\RG_\sing(\Pi_\infty{\mstack X}(\mbb C),\mbb C)\xra{\sim}\RG_\dR(\mstack X_{\mbb C})$ in our setting. Altogether this shows that $\dim_{K} H^i_\dR(\mstack X_K)= \dim_{\mbb Q_p} H^i_\et({\mstack X}_{\mbb C_p},\mbb Q_p)$ and then the statement follows from \Cref{prop: rationally all cohomology has the same dimension}.
\end{proof}
\end{cor}
\begin{conj}\label{conj:Hodge proper stacks}
	The corollary above poses a natural question: is it true that for any Hodge-proper stack $\mstack X$ the rationalized map $$\lambda_{\mstack X}^{-1}[\tfrac{1}{p}]\colon H^i_\et({\mstack X}_{\mbb C_p},\mbb Q_p) \xymatrix{\ar[r]&} H^i_\et(\widehat{\mstack X}_{\mbb C_p},\mbb Q_p)$$ is an equivalence? We do not present a strong evidence for the validity of this here, but we believe that it should indeed be the case.	
\end{conj}	
The notion of Breuil-Kisin modules was originally motivated by the result of Kisin relating them to crystalline $G_K$-representations. To recall it we introduce some extra piece of notation.

\begin{defn}\label{defn:crystlline_and_deRham_reps}
Recall that a continuous $\mbb Q_p$-linear $G_K$-representation $V$ is called \emdef{crystalline} if the $K_0$-vector space $D_\crys(V)\coloneqq (V\otimes_{\mbb Q_p} B_\crys)^{G_K}$ has the same dimension as $V$.
\end{defn}

Here $B_\crys\coloneqq A_\crys[\frac{1}{\mu}]$ and $A_\crys$ is explicitly given by the $p$-adic completion of the (sub)ring $\Ainf[\frac{\xi^n}{n!}]_{n\ge 1}\subset \Ainf[\frac{1}{p}]$. $B_\crys$ is a $\mbb Q_p$-algebra since $\mu^{p-1}\in pA_\crys$. One also shows that the natural Frobenius on $\Ainf$ induces a Frobenius $\phi\colon B_{\crys} \ra B_{\crys}$.

Recall that the choice of a compatible family of $p^n$-th roots of $\pi$ gives an element $\pi^\flat\coloneqq (\pi, \pi^{1/p}, \pi^{1/p^2}, \ldots)\in \mathcal O_{\mbb C_p}^\flat$ which then induces a natural morphism $\mathfrak S \to \Ainf$ by sending $u$ to $[\pi^\flat]$. This induces a map of prisms $(\mf S,(E(u)))\ra (\Ainf, (\xi))$. Let $K_{\infty} \coloneqq K(\pi^{1/p^\infty})\subset \mbb C_p$ be the infinite extension of $K$ obtained by adjoining $\pi^{1/p^n}$ for all $n\ge 1$ and let $G_{K_\infty}$ be the the absolute Galois group of $K_\infty$. By composing $\mathfrak S \to \Ainf$ with the  natural map $A_{\inf} \to W(\mbb C_p^\flat)$ we obtain a homomorphism $\mathfrak S\to W(\mbb C_p^\flat)$ which is  $(\phi, G_{K_\infty})$-equivariant. In the formulation of \Cref{thm: Kisin theorem} below we will also use its Frobenius-twisted version $\mf S\xra{\phi_{\mf S}}\mf S\ra W(\mbb C_p^\flat)$; it is also $(\phi, G_{K_\infty})$-equivariant. Finally, let's denote by $\Rep_{G_K}^{\#,\crys}$ the category of $\mbb Z_p$-lattices $T$ of finite rank with a continuous $G_K$-action with the property that $T\otimes_{\mbb Z_p}\mbb Q_p$ is crystalline.
\begin{thm}[{\cite[Theorem 1.2.1]{Kisin_BKModules}}]
	\label{thm: Kisin theorem}
	There is a fully-faithful tensor (but \emph{not} exact!) functor $\BK\colon \Rep_{G_K}^{\#,\crys} \to \Mod_{\mathfrak S}^\phi$ such that 
	$$\textstyle (\BK(T)/u)[\frac{1}{p}] \simeq D_{\crys}(T\otimes_{\mathbb Z_p}\mathbb Q_p).$$
	Moreover $\BK(T)$ is characterized as the unique free Breuil-Kisin module such that there exists a $(\phi, G_{K_\infty})$-equivariant isomorphism
	$$\BK(T)\otimes_{\mathfrak S,\phi_{\mf S}} W(\mbb C_p^\flat) \simeq T\otimes_{\mathbb Z_p} W(\mbb C_p^\flat).$$
\end{thm}

\begin{rem}
	Compare with the convention used in \cite[Theorem 4.4]{BMS1}, namely there one considers a map $\mf S\ra W(\mbb C_p^\flat)$ sending $u$ to $[\pi^\flat]^p$. This is also related to the discussion in \Cref{rem: comparison for the twisted prismatic cohomology} and the fact that $\RG_\Ainf(X)\simeq \RG_{\Prism^{(1)}}(\mstack X/\Ainf)$ (see \cite[Example 1.9(2)]{BMS2}). In particular for the \'etale comparison in \cite{BMS1} one is forced to consider the twisted prismatic cohomology.
\end{rem}

\begin{ex}\label{ex: twists}
	The Tate twist $\mbb Z_p(1)\subset \mbb Q_p(1)$ is a nice example of a lattice in a crystalline representation. The corresponding Breuil-Kisin module $\mf S\{1\}\coloneqq \BK(\mbb Z_p(1))$ is described in a canonical but somewhat complicated form in \cite[Example 4.2]{BMS1}. Its underlying $\mf S$-module is (non-canonically!) isomorphic to $\mf S$ and $\phi_{\mf S\{1\}}$ sends $x\in \mf S\{1\}$ to $\frac{\alpha}{E(u)}\phi_{\mf S}(x)$ where $\alpha\in \mf S^\times$ is a certain unit depending on the choice of $E(u)$. On the other hand, there is an easy explicit description if we make base change via the composite map $\mf S\xra{\phi_{\mf S}}\mf S\ra \Ainf$ above, then the corresponding Breuil-Kisin-Fargues module (see \cite[Definition 4.22]{BMS1}) $\Ainf\{1\}\coloneqq \mf S\{1\}\otimes_{\mf S, \phi_{\mf S}} \Ainf$ is described as 
	$$
	\Ainf\{1\}\simeq \frac{1}{\mu}(\Ainf\otimes_{\mbb Z_p}\mbb Z_p(1)).
	$$
	This identification is $G_{K_\infty}$-equivariant, but in fact it can also be made $G_K$-equivariant if we use a more canonical description of $\Ainf\{1\}$ without appealing to $\mf S\{1\}$ whatsoever (see \cite[Example 4.24]{BMS1}). In particular this gives a natural $G_K$-equivariant map $\mbb Z_p(1)\ra \Ainf\{1\}$. By multiplicativity we also get an identification  of $\Ainf\{i\}\coloneqq \Ainf\{1\}^{\otimes_\Ainf i}$ with  $\frac{1}{\mu^i}(\Ainf\otimes_{\mbb Z_p}\mbb Z_p(i))$. The $G_K$-equivariant map $\mbb Z_p(i)\ra \Ainf\{i\}$ then exists if $i\ge 0$; also note that its image is $\phi$-invariant.
\end{ex}

By the famous Fontaine's crystalline conjecture (\cite[Conjecture $C_{\mr{cris}}$]{fontaine1982}), proved in {\cite[Theorem 0.2]{Tsuji_Cst} (also \cite[Theorem 1.1]{BMS1}}) for a smooth proper $\mc O_K$-scheme $X$ the \'etale cohomology $H^n_\et(\widehat X_{\mbb C_p},\mbb Q_p)$ is a crystalline $G_K$-representation. Assuming $H^n_\crys(\widehat X_k / W(k))$ is $p$-torsion free, the Breuil-Kisin module corresponding to $H^n_\et(\widehat X_{\mbb C_p},\mbb Z_p)$ can be described explicitly, namely, as shown in \cite{BMS2} by Bhatt, Morrow and Scholze, one has
$$
\BK(H^i_\et(\widehat X_{\mbb C_p},\mbb Z_p))\simeq H^i_\Prism(X/\mf S).
$$ 
Below we extend these results to smooth Hodge-proper stacks. For this we need a few preliminary lemmas.

Note that the maps $\mf S\ra W(\mbb C_p^\flat)$ and $\Ainf \ra W(\mbb C_p^\flat)$ are flat (the proof in \cite[Lemma 4.30]{BMS1} applies to both). Thus, from \Cref{cor: prismatic_vs_etale_torsion} (and \Cref{rem: comparison for the twisted prismatic cohomology}) for any $i$ we get an isomorphism
$$
H^i_{\et}(\widehat{\mstack X}_{{\mathbb C_p}}, \mathbb Z_p)\otimes_{\mathbb Z_p} W(\mathbb C_p^\flat) \xymatrix{\ar[r]^{\sim}&} H^i_{\Prism^{(1)}}(\mstack X_{\mc O_{\mbb C_p}} / \Ainf)\otimes_{\Ainf} W(\mathbb C_p^\flat).$$ Note that we have an embedding $\Ainf[\tfrac{1}{\mu}] \subset  W({\mbb C_p^\flat})$ (since $\mu\equiv \epsilon -1 \mod pW({\mbb C_p^\flat})$ is non-zero and thus invertible in $W({\mbb C_p^\flat})$).

\begin{lem}\label{lem: can descend to A_inf[1/mu]}
	Let $\mstack X$ be a smooth Hodge-proper $\mc O_K$-stack. Then the isomorphism $$H^i_{\et}(\widehat{\mstack X}_{{\mathbb C_p}}, \mathbb Z_p)\otimes_{\mathbb Z_p} W(\mathbb C_p^\flat) \xymatrix{\ar[r]^{\sim}&} H^i_{\Prism^{(1)}}(\mstack X_{\mc O_{\mbb C_p}} / \Ainf)\otimes_{\Ainf} W(\mathbb C_p^\flat)$$ restricts to an isomorphism
	$$H^i_{\et}(\widehat{\mstack X}_{{\mathbb C_p}}, \mathbb Z_p)\otimes_{\mathbb Z_p} \Ainf[\tfrac{1}{\mu}] \xymatrix{\ar[r]^-\sim&} H^i_{\Prism^{(1)}}(\mstack X_{\mc O_{\mbb C_p}} / \Ainf)\otimes_{\Ainf}\Ainf[\tfrac{1}{\mu}].$$
\end{lem}
\begin{proof}
	Take a Breuil-Kisin prism $(\mf S,E(u))$ with $\mf S/E(u)\simeq \mc O_K$. Since $H^i_\Prism(\mstack X / \mathfrak S)$ is a Breuil-Kisin module, the module $H^i_\Prism(\mstack X / \mathfrak S)[\tfrac{1}{p}]$ is free over $\mf S[\tfrac{1}{p}]$ (\Cref{ex:stuff on BK-modules}). It follows that $H^i_{\Prism^{(1)}}(\mstack X_{\mc O_{\mbb C_p}} / \Ainf)\simeq H^i_\Prism(\mstack X / \mathfrak S)\otimes_{\mathfrak S,\phi_{\mf S}}\Ainf$ is a Breuil-Kisin-Fargues module (see \cite[Definition 4.22]{BMS1}). The statement then follows from \cite[Lemma 4.26]{BMS1}.
\end{proof}
\begin{rem}\label{rem: can descend to A_inf[1/mu]}
	For us an important consequence of \Cref{lem: can descend to A_inf[1/mu]} is that by tensoring further with $B_{\crys}$ over $\Ainf[\frac{1}{\mu}]$ we also get an isomorphism $H^i_{\et}(\widehat{\mstack X}_{{\mathbb C_p}}, \mathbb Z_p)\otimes_{\mathbb Z_p} B_\crys \xra{\sim}H^i_{\Prism^{(1)}}(\mstack X_{\mc O_{\mbb C_p}} / \Ainf)\otimes_{\Ainf} B_\crys$. By construction this map is $\phi$-equivariant.
\end{rem}

\begin{rem}\label{rem: G_K-equivariance of etale comparison} Given a scheme $X$ (of locally finite type) over $\mc O_K$ and an element $\sigma\in G_K$ one has a natural equivalence 
	$$
	\sigma^*\colon\RG_{\et}(\widehat{X}_{\mbb C_p},\mbb Z_p) \xymatrix{\ar[r]^{\sim}&} \RG_{\et}(\widehat{X}_{\mbb C_p},\mbb Z_p)
	$$
induced by the action of $\sigma$ as an automorphism on $\widehat{\mstack X}_{\mbb C_p}$ (considered as an adic space). Similarly, if we consider $\RG_{\Prism^{(1)}}(X_{\mc O_{\mbb C_p}} / \Ainf)$, by base change with respect to $\sigma\colon \Ainf \xra{\sim} \Ainf$, we get a $\sigma$-linear equivalence  
$$
\sigma^*\colon\RG_{\Prism^{(1)}}(X_{\mc O_{\mbb C_p}} / \Ainf) \xymatrix{\ar[r]^{\sim}&} \RG_{\Prism^{(1)}}(X_{\mc O_{\mbb C_p}} / \Ainf)
$$
then also giving an equivalence 
$$
\sigma^*\colon\RG_{\Prism^{(1)}}(X_{\mc O_{\mbb C_p}} / \Ainf)\widehat{\otimes}_{\Ainf} W(\mbb C_p^\flat) \xymatrix{\ar[r]^{\sim}&} \RG_{\Prism^{(1)}}(X_{\mc O_{\mbb C_p}} / \Ainf)\widehat{\otimes}_{\Ainf} W(\mbb C_p^\flat)
$$
This defines an action of $G_K$ by derived autoequivalences and moreover by construction the \'etale comparison map 
$$
\RG_{\et}(\widehat{X}_{\mbb C_p},\mbb Z_p) \tto \RG_{\Prism^{(1)}}(X_{\mc O_{\mbb C_p}} / \Ainf)\widehat{\otimes}_{\Ainf} W(\mbb C_p^\flat)
$$
is $G_K$-equivariant. By applying the right Kan extension for any Artin stack $\mstack X$ over $\mc O_K$ we get $G_K$-actions on $\RG_{\et}(\widehat{\mstack X}_{\mbb C_p},\mbb Z_p)$ and $\RG_{\Prism^{(1)}}(\mstack X_{\mc O_{\mbb C_p}} / \Ainf)\widehat{\otimes}_{\Ainf} W(\mbb C_p^\flat)$, such that the \'etale comparison morphism 
$$
\RG_{\et}(\widehat{\mstack X}_{\mbb C_p},\mbb Z_p) \tto \RG_{\Prism^{(1)}}(\mstack X_{\mc O_{\mbb C_p}} / \Ainf)\widehat{\otimes}_{\Ainf} W(\mbb C_p^\flat)
$$
is $G_K$-equivariant. In particular the isomorphism $$
H^i_{\et}(\widehat{\mstack X}_{{\mathbb C_p}}, \mathbb Z_p)\otimes_{\mathbb Z_p} B_\crys \xymatrix{\ar[r]^{\sim}&} H^i_{\Prism^{(1)}}(\mstack X_{\mc O_{\mbb C_p}} / \Ainf)\otimes_{\Ainf} B_\crys
$$ that we established in \Cref{rem: can descend to A_inf[1/mu]} for Hodge-proper $\mstack X$ is also $G_K$-equivariant.
\end{rem}	

\begin{rem}
	Recall that if $\mstack X$ is smooth and Hodge-proper over $\mc O_K$, by \Cref{etale_coh_finiteness} each individual group $H^i_{\et}(\widehat{\mstack X}_{{\mathbb C_p}}, \mathbb Z/p^n)$ is finite. Thus, say by \cite{nikolov2003finite}, the corresponding $G_K$-action is automatically continuous. While using \cite{nikolov2003finite} might seem as a slight overkill, the continuity can also be seen as follows: namely one can consider $Rp_*\ul{\mbb Z/p^n}\in \Shv_{\et}(\Sp(K),\mbb Z/p^n)$ for the natural morphism $p\colon \widehat{\mstack X}_K\ra \Sp(K)$, where $\Shv_{\et}(\Sp(K),\mbb Z/p^n)^\heartsuit$ is identified exactly with the $\mbb Z/p^n$-modules with continuous $G_K$-action. The (generalizing) base change \cite{Huber_adicSpaces}, applied to $\Sp \mbb C_p\ra \Sp_K$ and extended to stacks (as in \Cref{adic_local_systems}) exactly identifies $H^i_{\et}(\widehat{\mstack X}_{{\mathbb C_p}}, \mathbb Z/p^n)$ with the above action and $R^ip_*\ul{\mbb Z/p^n}$. We then also have $H^i_\et(\widehat{\mstack X}_{{\mathbb C_p}}, \mathbb Z_p)\simeq  \lim_n H^i_\et(\widehat{\mstack X}_{{\mathbb C_p}}, \mathbb Z/p^n)$ (see \Cref{rem: adic cohomology of Hodge-proper stack are finitely generated}) and thus the $G_K$-action on $H^i_\et(\widehat{\mstack X}_{{\mathbb C_p}}, \mathbb Z_p)$ is also continuous. In particular this gives a natural structure of a continuous $G_K$-representation on $H^i_\et(\widehat{\mstack X}_{{\mathbb C_p}}, \mathbb Q_p)\coloneqq H^i_\et(\widehat{\mstack X}_{{\mathbb C_p}}, \mathbb Z_p)[\frac{1}{p}]$. 
\end{rem}

We now are going to show that $H^i_\et(\widehat{\mstack X}_{{\mathbb C_p}}, \mathbb Q_p)$ is a crystalline representation. We stress that in the proof we are going to use the twisted prismatic cohomology $\RG_{\Prism^{(1)}}(\mstack X/\mf S)$ and the corresponding \'etale comparison (see \Cref{rem: comparison for the twisted prismatic cohomology}).
\begin{thm}\label{thm: Fontaine's conjecture}
	Let $\mstack X$ be a smooth Hodge-proper stack over $\mathcal O_K$ and let $C=\mbb C_p$. Then the Galois representation $H^i_\et(\widehat{\mstack X}_C, \mathbb Q_p)$ is crystalline (see \Cref{defn:crystlline_and_deRham_reps}) for all $i\in \mathbb Z_{\ge0}$ and $D_\crys(H^i_\et(\widehat{\mstack X}_C, \mathbb Q_p)) \simeq H^i_\crys(\mstack X_k/ W(k))[\tfrac{1}{p}]$.
	
	\begin{proof} A slight problem that we face is that the morphism $\Ainf \ra A_{\crys}$ is not of finite $(p,I)$-complete Tor-amplitude and thus base change (\Cref{prismatic_basechange}) can't be applied directly. To remedy this we first fix an auxiliary Breuil-Kisin prism $(\mf S, E(u))$. Note that since $\mf S$ is regular and Noetherian any $\mf S$-module has finite $I$-complete Tor-amplitude for any $I$ (including $I=0$ and $I=(p,E(u))$).
		By base change for the twisted version of prismatic cohomology (\Cref{rem: base change for twisted prismatic cohomology}), applied to the composition $(\mf S,(E(u)))\ra (\Ainf,({\xi}))\ra (A_\crys,(p))$, we get an equivalence
		$$
		\RG_{\Prism^{(1)}}(\mstack X/\mf S)\widehat\otimes_{\mf S} A_\crys\xymatrix{\ar[r]^\sim &}R\Gamma_{\Prism^{(1)}}(\mstack X_{\mathcal O_C/p} / A_\crys)\xymatrix{\ar[r]^\sim &} \RG_\crys(\mstack X_{\mc O_{\mbb C_p}/p}/A_\crys),
		$$
		where the second arrow is given by the crystalline comparison (\Cref{prop: cristalline comparison}).
		By \Cref{lem:the tensor product remains complete} if $\mstack X$ is Hodge-proper the tensor product on the left does not need to be completed. 
		We also have 
		$$
		\RG_{\Prism^{(1)}}(\mstack X/\mf S)\otimes_{\mf S} A_\crys \simeq \RG_{\Prism^{(1)}}(\mstack X/\mf S)\otimes_{\mf S}\Ainf \otimes_{\Ainf} A_\crys \simeq \RG_{\Prism^{(1)}}(\mstack X_{\mc O_{\mbb C_p}}/\Ainf)\otimes_{\Ainf} A_\crys, 
		$$
		where we use again that $\RG_{\Prism^{(1)}}(\mstack X/\mf S)\otimes_{\mf S}\Ainf$ is already $(p,I)$-complete. From this we see that the base change morphism is an equivalence:
	$$
	\RG_{\Prism^{(1)}}(X_{\mc O_{\mbb C_p}}/\Ainf)\otimes_{\Ainf}A_\crys\xymatrix{\ar[r]^\sim &}R\Gamma_{\Prism^{(1)}}(\mstack X_{\mathcal O_C/p} / A_\crys).
	$$
	 Note that this morphism is $G_K$-equivariant.

		Localizing at $\mu\in A_\crys$ we get a $G_K$-equivariant equivalence 
		$$
		\RG_{\Prism^{(1)}}(X_{\mc O_{\mbb C_p}}/\Ainf)\otimes_{\Ainf} B_\crys \xymatrix{\ar[r]^\sim &}R\Gamma_\crys(\mstack X_{\mathcal O_C/p} / A_\crys)[\tfrac{1}{\mu}].
		$$

		Using the regularity and Noetherianness of $W(k)$ to deal with completions, by base-change (now applied to $(W(k),(p))\ra (A_\crys,(p))$) we also get a $G_K$-equivariant equivalence
		$$
		R\Gamma_{\Prism^{(1)}}(\mstack X_k / W(k))\otimes_{W(k)} B_\crys \xymatrix{\ar[r]^\sim &} R\Gamma_{\Prism^{(1)}}(\mstack X_{\mathcal O_C/p} / A_\crys)[\tfrac{1}{\mu}]
		$$
		in turn providing an equivalence
		$$R\Gamma_{\crys}(\mstack X_k / W(k))\otimes_{W(k)} B_\crys \xymatrix{\ar[r]^\sim &} R\Gamma(\mstack X_{\mathcal O_C/p} / A_\crys)[\tfrac{1}{\mu}]$$
		via the crystalline comparison.
		
		Recall that $B_\crys$ is a $\mbb Q_p$-algebra and that  $H^i_\Prism(\mstack X/\mf S)[\tfrac{1}{p}]$ (and thus $H^i_{\Prism^{(1)}}(\mstack X/\mf S)[\tfrac{1}{p}]$) is a free $\mathfrak S[\tfrac{1}{p}]$-module for any $i$; consequently the cohomology of the tensor product $\RG_{\Prism^{(1)}}(\mstack X/\mf S)\otimes_{\mf S} B_\crys$ are computed simply as $H^i_{\Prism^{(1)}}(\mstack X/\mf S)\otimes_{\mf S} B_\crys$\footnote{In fact it is also true that $B_\crys$ is a flat $\mf S$-module.}. It then follows that $H^i(\RG_{\Prism^{(1)}}(\mstack X_{\mc O_{\mbb C_p}}/\Ainf)\otimes_{\Ainf} B_\crys)\simeq H^i_{\Prism^{(1)}}(\mstack X_{\mc O_{\mbb C_p}}/\Ainf)\otimes_{\Ainf} B_\crys$. Even better, $H^i_\crys(\mstack X,W(k))[\frac{1}{p}]$ is a vector space over $K_0$ and so the cohomology of $R\Gamma_{\crys}(\mstack X_k / W(k))\otimes_{W(k)} B_\crys$ are given by $H^i_\crys(\mstack X,W(k))\otimes_{W(k)}B_\crys$. Using Remarks \ref{rem: can descend to A_inf[1/mu]} and \ref{rem: G_K-equivariance of etale comparison}, we thus obtain a $G_K$-equivariant isomorphism
		$$H^i_\et(\widehat{\mstack X}_C, \mathbb Q_p)\otimes_{\mathbb Q_p} B_\crys \xymatrix{\ar[r]^\sim&} H^i_\crys(\mstack X_k / W(k))[\tfrac{1}{p}]\otimes_{K_0} B_\crys$$
		 for any $i\ge 0$. It also respects $\phi$-actions (since all isomorphisms we used were essentially obtained by base change). Since the $G_K$-action on $H^i_\crys(\mstack X_k / W(k))[\tfrac{1}{p}]$ is trivial and $(B_\crys)^{G_K}\simeq K_0$ we get that
		$$D_\crys(H^i_\et(\widehat{\mstack X}_C, \mathbb Q_p))\coloneqq  \left((H^i_\et(\widehat{\mstack X}_C, \mathbb Q_p)\otimes_{\mathbb Q_p} B_\crys\right)^{G_K} \simeq H^i_\crys(\mstack X_k/W(k))[\tfrac{1}{p}].\qedhere$$
	\end{proof}
\end{thm}

\begin{rem}
	One can see that \Cref{thm: Fontaine's conjecture} (after appropriate replacement of $B_\crys$ and $\mbb C_p$) stays true in a greater generality where $\mc O_K$ is a ring of integers in a finite extension of $K_0=W(k)[\frac{1}{p}]$ where $k$ is a perfect field. Using \cite[Section 2.3.3]{KubrakPrikhodko_HdR}, one can construct examples of Hodge-proper non-proper schemes over such more general $\mc O_K$ (at least for $p\gg 0$). However for the construction in \textit{loc.cit.} it is necessary that $k$ has transcendence degree at least 1 over $\mbb F_p$. It would be interesting to find an example of a Hodge-proper non-proper scheme over $\mc O_K$ where $K$ is a finite extension of $\mbb Q_p$, for now we do not know one. 
\end{rem}

We end the subsection with some remarks concerning the Breuil-Kisin module associated to the \'etale cohomology $H^i_\et(\widehat{\mstack X}_{\mbb C_p},\mbb Z_p)$.

\begin{rem}\label{rem: the BK-module associated to the etale cohomology}
		By \Cref{thm: Fontaine's conjecture} the free quotient
		$$H^i_\et(\widehat{\mstack X}_{\mbb C_p},\mbb Z_p)_{\free}\coloneqq H^i(\widehat{\mstack X}_{\mbb C_p},\mbb Z_p)/H^i_\et(\widehat{ \mstack X}_{\mbb C_p},\mbb Z_p)_{\mr{tors}}\hookrightarrow H^i(\widehat{\mstack X}_{\mbb C_p},\mbb Q_p)$$ is a lattice in a crystalline representation. We claim that the prismatic cohomology gives a formula for the corresponding Breuil-Kisin module, namely:	
$$
\BK(H^i_\et(\widehat{\mstack X}_{\mbb C_p},\mbb Z_p)_{\free}) \simeq H^i_\Prism(\mstack X/\mf S)_{\free},
$$
where $H^i_\Prism(\mstack X/\mf S)_{\free}$ is the free Breuil-Kisin module associated to $H^i_\Prism(\mstack X/\mf S)$ (see \Cref{ex:stuff on BK-modules}).

Indeed, using \Cref{ex:stuff on BK-modules} and flatness of $\mf S \ra W(\mbb C_p^\flat)$ we get a short exact sequence 
$$
\xymatrix{
0\ar[r]& H^i_\Prism(\mstack X/\mf S)_{\mr{tors}}\otimes_{\mf S}W(\mbb C_p^\flat) \ar[r]&H^i_\Prism(\mstack X/\mf S)\otimes_{\mf S}W(\mbb C_p^\flat)\ar[r]& H^i_\Prism(\mstack X/\mf S)_{\free}\otimes_{\mf S}W(\mbb C_p^\flat) \ar[r]& 0}
$$
which is isomorphic to
$$
\xymatrix{
	0\ar[r]& H^i_\et(\widehat{\mstack X}_{\mbb C_p},\mbb Z_p)_{\mr{tors}}\otimes_{\mbb Z_p} W(\mbb C_p^\flat) \ar[r]&H^i_\et(\widehat{\mstack X}_{\mbb C_p},\mbb Z_p)\otimes_{\mbb Z_p}W(\mbb C_p^\flat)\ar[r]& H^i_\et(\widehat{\mstack X}_{\mbb C_p},\mbb Z_p)_{\free} \otimes_{\mbb Z_p}W(\mbb C_p^\flat) \ar[r]& 0}
$$
via the \'etale comparison (\Cref{cor: adic_etale_comparison2}). This shows in particular that $\BK(H^i_\et(\mstack X_{\mbb C_p},\mbb Z_p)_{\free})\simeq H^i_\Prism(\mstack X/\mf S)_{\free}$.
\end{rem}

The following lemma is an analogue of \cite[Corollary 4.20]{BMS1} that guarantees freeness of $H^i_\Prism(\mstack X/\mf S)$ under some assumptions.
\begin{lem}
Let $\mstack X$ be a Hodge-proper stacks over $\mathcal O_K$.
\begin{enumerate}[label=(\arabic*)]
\item If $H^i_\crys(\mstack X_k/W(k))$ is $p$-torsion free, then $H^i_\Prism(\mstack X/\mathfrak S)$ is free as an $\mathfrak S$-module.

\item If additionally $H^{i+1}_\crys(\mstack X_k/W(k))$ is $p$-torsion free, then the natural injection
$$H^i_{\Prism^{(1)}}(\mstack X/\mathfrak S) \otimes_{\mathfrak S} W(k) \xymatrix{\ar@{^(->}[r] &} H^i(R\Gamma_{\Prism^{(1)}}(\mstack X/\mathfrak S)\otimes^{\mbb L}_{\mathfrak S} W(k)) \simeq H^i_\crys(\mstack X_k / W(k))$$
is an isomorphism.
\end{enumerate}

\begin{proof} All tensor products are assumed to be non-derived (unless noted otherwise).
First note that $H^i_\Prism(\mstack X/\mathfrak S)$ is free over $\mf S$ if and only if $H^i_{\Prism^{(1)}}(\mstack X/\mathfrak S)$ is. Let more generally $C$ be a complex of $\mathfrak S$-modules such that $H^i(C)$ is finitely generated over $\mathfrak S$ and $H^i(C)[1/p]$ is free over $\mathfrak S[1/p]$ (e.g. $C=\RG_{\Prism^{(1)}}(\mstack X/\mathfrak S)$). We claim that if $H^i(C\otimes_{\mathfrak S}^{\mbb L} W(k))$ is $p$-torsion free then $H^i(C)$ is free. Indeed, since $H^i(C)\otimes_S W(k)$ is a submodule of $H^i(C\otimes^{\mbb L}_{\mathfrak S} W(k))$ it follows that $H^i(C)\otimes_{\mf S} W(k)$ is $p$-torsion free (equivalently free as a $W(k)$-module) as well. Hence it is enough to prove the following assertion: let $M$ be a finitely generated $\mathfrak S$-module such that
\begin{itemize}
\item $M[1/p]$ is free over $\mathfrak S[1/p]$.

\item $M\otimes_{\mathfrak S} W(k)$ is free over $W(k)$.
\end{itemize}
Then $M$ is free over $\mathfrak S$.

But under these assumptions
\begin{align*}
\dim_k M\otimes_{\mathfrak S} k & = \rank_{W(k)} M\otimes_{\mathfrak S} W(k) = \rank_{W(k)[1/p]} M\otimes_{\mathfrak S} W(k)[\tfrac{1}{p}] = \\
& = \rank_{\mathfrak S[1/p]} M\otimes_{\mathfrak S} \mathfrak S[\tfrac 1 p] = \dim_{\Frac \mathfrak S} M\otimes_{\mathfrak S} \Frac \mathfrak S.
\end{align*}
and thus $M$ is free by the semicontinuity of stalks.

To see the second part note that by the universal coefficients formula the cokernel of the natural injection $H^i_\Prism(\mstack X_k/\mathfrak S) \otimes_{\mathfrak S} W(k) \inj H^i(R\Gamma_\Prism(\mstack X/\mathfrak S)\otimes^{\mbb L}_{\mathfrak S} W(k))$
is isomorphic to $u$-torsion in $H^{i+1}_\Prism(\mstack X/\mathfrak S)$. But by our assumption and the previous discussion this is a free $\mathfrak S$-module, thus has no $u$-torsion.
\end{proof}
\end{lem}

\begin{rem}\label{rem: associated BK-module/refined}
	In particular from \Cref{rem: the BK-module associated to the etale cohomology} we get that if $H^i_\crys(\mstack X_k/W(k))$ is $p$-torsion free, then so is $H^i_\et(\widehat{\mstack X}_{\mbb C_p},\mbb Z_p)$ and we have $$
	\BK(H^i_\et(\widehat{\mstack X}_{\mbb C_p},\mbb Z_p))\simeq H^i_\Prism(\mstack X,\mf S).
	$$
\end{rem}

\subsubsection{Hodge-Tate decomposition}\label{sssection: Hodge-Tate decomposition}
In this section we also deduce the Hodge-Tate decomposition from the comparison established in \Cref{thm: Fontaine's conjecture}. In the case of smooth proper schemes this is a celebrated result of Faltings \cite[Theorem 4.1]{Faltings_p-adic_Hodge}. 

Recall that for a smooth stack $\mstack X$ over $A/I$ we have the Hodge-Tate filtration $\Fil_{\HT}$ on the reduction
$$\RG_{\Prism/I}(\mstack X/A)\coloneqq \RG_{\Prism}(\mstack X/A)\otimes_A A/I.$$
By \Cref{prop: HT_filtration_is_exaustive} it is exhaustive. Moreover, if $\mstack X$ is Hodge-proper, by \Cref{cor: prismatic cohomology are bounded below coherent} the associated graded pieces of this filtration are equivalent to $\RG(\mstack X,\wedge^i \mathbb L_{\mstack X/(A/I)})\{-i\}[-i]$. Note that this induces a (strongly convergent) \textit{Hodge-Tate} spectral sequence $E_{2}^{p,q}= H^{q,p}(\mstack X/(A/I))\{q\} \Rightarrow H^{p+q}_{\Prism/I}(\mstack X/A)$ where $H^{p,q}(\mstack X/(A/I))\coloneqq H^{q}(\mstack X,\wedge^p\mbb L_{\mstack X/(A/I)})$.

Let $\mbb C_p(i)\coloneqq \mbb C_p\otimes_{\mbb Z_p}\mbb Z_p(i)$ denote the $i$-th Hodge-Tate twist. One can consider the category $\Mod_{\mbb C_p}(\Rep_{G_K})$ of finite-dimensional $\mbb C_p$-vector spaces with a semi-linear continuous $G_K$-action. Such a representation is called \emdef{Hodge-Tate} if it is isomorphic to a direct sum of $\mbb C_p(i)$'s for various $i$. Recall that by a classical result of Tate there are neither nontrivial homomorphisms nor extensions between $\mbb C_p(i)$ and $\mbb C_p(j)$ if $i\neq j$.

\begin{prop}\label{prop: HT-degeneration}
	Let $\mstack X$ be a smooth Hodge-proper stack over $\mc O_K$ and let $(A,I)=(\mf S, (E(u)))$ be a Breuil-Kisin prism with $\mf S/E(u) \simeq \mc O_K$. Then the Hodge-Tate spectral sequence for $\mstack X$ degenerates rationally (after inverting $p$).

\begin{proof}
	Note that $\mf S/E(u)[\frac{1}{p}]\simeq K$ is a field and that by base change $H^{p,q}(\mstack X/\mc O_K)[\frac{1}{p}]\simeq H^{p,q}(\mstack X_K/K)$. The twist $\{i\}$ does nothing on the level of $\mc O_K$-modules since $I$ is principal. Thus it is enough check that for a given $n$ one has $\dim_K H^{n}_{\Prism/I}(\mstack X/\mf S)[\frac{1}{p}]=\sum_{p+q=n} \dim_K H^{p,q}(\mstack X_K/K)$. This is enough to do after base change to $(\Ainf, \ker \theta)$; here we again use \Cref{lem:the tensor product remains complete} and Hodge-properness of $\mstack X$ to deduce that  $\RG_{\Prism/I}(\mstack X/\mf S)\otimes_{\mf S}\Ainf$ is already $p$-adically complete to be able to apply \Cref{prismatic_basechange}. The Hodge-Tate filtration is functorial and thus all maps in the spectral sequence $E_{2}^{p,q}= H^{q,p}(\mstack X_{\mc O_{\mbb C_p}}/{\mc O_{\mbb C_p}})\{-q\} \Rightarrow H^{p+q}_{\Prism/I}(\mstack X_{\mc O_{\mbb C_p}}/\Ainf)$ are $G_K$-equivariant. Note that the $i$-th twist $(\Ainf/(\xi))\{i\}$ is given by the reduction of the Breuil-Kisin twist and thus is isomorphic to $\mc O_{\mbb C_p}(i)$ as a $G_K$-module. By base change we get $H^{q,p}(\mstack X_{\mc O_{\mbb C_p}}/{\mc O_{\mbb C_p}})\{q\}[\frac{1}{p}]\simeq H^{q,p}(\mstack X_{K}/K)\otimes_K \mbb C_p(-q)$. Thus after inverting $p$ the terms in the Hodge-Tate spectral sequence are given by Hodge-Tate twists with the value of the twist equal to minus the column number. Since there are no homomorphisms between different twists, all differentials vanish.
\end{proof}
\end{prop}

\begin{rem}\label{rem: the maps from HT-truncations are embeddings}
	Also note that since the Hodge-Tate spectral sequence degenerates rationally the natural maps $H^n(\Fil_\HT^{\le i}\RG_{\Prism/I}(\mstack X/\mf S)[\frac{1}{p}])\ra H^n_{\Prism/I}(\mstack X/
	\mf S)[\frac{1}{p}]$ are embeddings for all $i$ and $n$.
\end{rem}

\begin{rem}\label{rem: HT-decomposition for HT-truncations} From the proof of \Cref{prop: HT-degeneration} it also follows that for any $n$ one has a $G_K$-equivariant decomposition $$H^n(\Fil_\HT^{\le j}\RG_{\Prism/I}(\mstack X_{\mc O_{\mbb C_p}}/\Ainf)[\tfrac{1}{p}]) \simeq \oplus_{i=0}^j  H^{n-i}(\mstack X_K,\Lambda^i\mbb L_{\mstack X_K/K})\otimes_K\mbb C_p(-i).$$
\end{rem}

As a corollary we can establish the degeneration of the Hodge-de Rham spectral sequence for $\mstack X_K$:

\begin{thm}[Hodge-de Rham degeneration]\label{thm: HdR degeneration for stacks}
	Let $\mstack X$ be a smooth Hodge-proper stack over $\mc O_K$. Then for the generic fiber $\mstack X_K$ the Hodge-to-de Rham spectral sequence $E_2^{p,q}=H^{p}(\mstack X_K,\wedge^q \mbb L_{\mstack X_K/K})\Rightarrow H^{p+q}_\dR(\mstack X_K/K)$ degenerates.
\end{thm}
\begin{rem}
	Note that this gives a significant (though partial) strengthening of \cite[Theorem 1.4.3]{KubrakPrikhodko_HdR}, where the authors have shown Hodge-to-de Rham degeneration for the so-called Hodge-properly spreadable stacks. Indeed, \Cref{thm: HdR degeneration for stacks} says that instead of having a smooth Hodge-proper model over a finitely generated $\mbb Z$-algebra, for the Hodge-de Rham degeneration it is enough to have a model over $\mc O_K$ for some finite extension $K/\mbb Q_p$ over a single prime $p$. 
\end{rem}
\begin{proof}
	For a reminder on the Hodge-de Rham spectral sequence for stacks see \cite[Section 1.1]{KubrakPrikhodko_HdR}. It is enough to show that $\dim_K H^{n}_\dR(\mstack X_K/K)=\sum_{p+q=n} \dim_K H^{p}(\mstack X_K,\wedge^q \mbb L_{\mstack X_K/K})$. As in the proof of \Cref{prop: rationally all cohomology has the same dimension}, one has $\dim_K H^{n}_\dR(\mstack X_K/K)= \rk_{\mf S[\frac{1}{p}]} H^n_{\Prism^{(1)}}(\mstack X/\mf S)[\frac{1}{p}]$. Since $H^n_{\Prism}(\mstack X/\mf S)[\frac{1}{p}]$ is free over $\mf S[\frac{1}{p}]$ and $H^n_{\Prism^{(1)}}(\mstack X/\mf S)[\frac{1}{p}]\simeq H^n_{\Prism}(\mstack X/\mf S)^{(1)}[\frac{1}{p}]$ one has $\rk_{\mf S[\frac{1}{p}]} H^n_{\Prism^{(1)}}(\mstack X/\mf S)[\frac{1}{p}]=\rk_{\mf S[\frac{1}{p}]} H^n_{\Prism}(\mstack X/\mf S)[\frac{1}{p}]$. On the other hand, again by freeness, $ H^n_{\Prism/I}({\mstack X}/\mf S)[\frac{1}{p}]\simeq (H^n_{\Prism}({\mstack X}/\mf S)[\frac{1}{p}])/I$ and $\rk_{\mf S[\frac{1}{p}]} H^n_{\Prism}(\mstack X/\mf S)[\frac{1}{p}]=\dim_{K} H^n_{\Prism/I}({\mstack X}/\mf S)$. One then has $\dim_{K} H^n_{\Prism/I}({\mstack X}/\mf S)=\sum_{p+q=n} \dim_K H^{p}(\mstack X_K,\wedge^q \mbb L_{\mstack X_K/K})$ by the degeneration of the Hodge-Tate spectral sequence.
\end{proof}

In the previous section in \Cref{thm: Fontaine's conjecture} we proved that $D_{\crys}(H^n_\et(\mstack X_{\mbb C_p}, \mbb Q_p))\simeq H^n_{\crys}(\mstack X_k, W(k))[\frac{1}{p}]$.  Consider de Rham period rings $B_\dR^+\coloneqq (\Ainf[\frac{1}{p}])^\wedge_\xi$ and $B_\dR\coloneqq B_\dR^+[\frac{1}{\xi}]$; one can show that there is a natural $G_K$-equivariant embedding $B_{\crys}\hookrightarrow B_{\dR}$. Recall that $(B_\dR)^{G_K}\simeq K$ and that $B_\dR$ has a natural ascending filtration $\Fil^{\ge i}B_\dR\coloneqq \xi^{i} B_\dR^+\subset B_\dR$ which is preserved by the $G_K$-action. Along with $D_\crys$ one can also consider the functor $D_\dR(V)\coloneqq (V\otimes_{\mbb Q_p}B_\dR)^{G_K}$; $D_\dR(V)$ is a naturally a $K$-vector space. Moreover, if $V$ is crystalline then the natural map $D_\crys(V)\otimes_{K_0}K\ra D_\dR(V)$ is an isomorphism.

\begin{prop}\label{prop: D_dR for the etale cohomology of a smooth Hodge proper stack}
	Let $\mstack X$ be a smooth Hodge-proper stack over $\mc O_K$. Then $D_\dR(H^n(\widehat{\mstack X}_{\mbb C_p},\mbb Q_p))$ is isomorphic to $H^n_\dR(\mstack X_K/K)$ endowed with the Hodge filtration.

\begin{rem} It is not hard to show from the de Rham and crystalline comparisons that for a smooth Hodge-proper stack
$$H^n_\dR(\mstack X_K/K) \simeq H^n_\crys(\mstack X_k/W(k))\otimes_{W(k)} K.$$
In particular this shows that the underlying $K$-vector space of $D_\dR(H^n(\widehat{\mstack X}_{\mbb C_p},\mbb Q_p))$ is given by $H^n_\dR(\mstack X_K/K)$. The main content of the proposition is that this is also an isomorphism of filtered vector spaces, namely that the filtration on $D_\dR(H^n(\widehat{\mstack X}_{\mbb C_p},\mbb Q_p))$ coming from $B_\dR$ coincides with the Hodge filtration on $H^n_\dR(\mstack X_K/K)$.
\end{rem}

\begin{proof} Recall the Nygaard filtration $\Fil_N^{\ge j}\RG_{\Prism}(X/\mf S)^{(1)}$ on the twisted prismatic cohomology (\cite[Section 15]{BS_prisms}). By definition $\phi_{\Prism}$ induces a map $\phi_{\Prism}\colon \Fil_N^{\ge j}\RG_{\Prism}(X/\mf S)^{(1)} \ra \RG_{\Prism}(X/\mf S)$ which factors through the multiplication by $E(u)^j$ map 
	$$
	\xymatrix{\Fil_N^{\ge j}\RG_{\Prism}(X/\mf S)^{(1)} \ar[rr]^{\phi_{\Prism}}\ar@{-->}[rd]_{\ol\phi_{\Prism}}& &\RG_{\Prism}(X/\mf S)\\
		& \RG_{\Prism}(X/\mf S) \ar[ru]_{\cdot E(u)^j}.& 
	}
	$$
	Moreover $\ol\phi_{\Prism}$ induces an isomorphism $\gr^j_N\RG_{\Prism}(X/\mf S)^{(1)}\simeq \Fil^{\le j}_\HT \RG_{\Prism/I}(X/\mf S)\{j\}$. 
	
	Defining $\Fil_N^{\ge j}\RG_{\Prism^{(1)}}(\mstack X/\mf S)$ as the right Kan extension of $\Fil_N^{\ge j}\RG_{\Prism}(-/\mf S)^{(1)}$ from $\Aff_{\mathcal O_K}^\sm$ we also get maps 
	
	$$
	\xymatrix{\Fil_N^{\ge j}\RG_{\Prism^{(1)}}(\mstack X/\mf S) \ar[rr]^{\phi_{\Prism}}\ar@{-->}[rd]_{\ol\phi_{\Prism}}& &\RG_{\Prism}(\mstack X/\mf S)\\
		& \RG_{\Prism}(\mstack X/\mf S) \ar[ru]_{\cdot E(u)^j}.& 
	}
	$$
	with identifications $\gr^j_N\RG_{\Prism^{(1)}}(\mstack X/\mf S)\simeq \Fil^{\le j}_\HT \RG_{\Prism/I}(\mstack X/\mf S)\{j\}$. Let the filtration $F^{\ge j}_1H^n_{\Prism^{(1)}}(\mstack X/\mf S)$ be defined as the image of $H^n(\Fil_N^{\ge j}\RG_{\Prism^{(1)}}(\mstack X/\mf S))$ in $H^n_{\Prism^{(1)}}(\mstack X/\mf S)$.

	On the other hand there is also its own Nygaard filtration on $H^n_{\Prism}(\mstack X/\mf S)^{(1)}$; namely for $M\coloneqq H^n_{\Prism}(\mstack X/\mf S)$ we have the equivalence $\phi_\Prism \colon M^{(1)}[\frac{1}{E}]\xra{\sim} M[\frac{1}{E}]$ and one can put $F^{\ge j}_2 M^{(1)}\coloneqq \phi^{-1}_\Prism(E^j\cdot M)\cap M^{(1)}$. We claim that after inverting $p$ the two filtrations coincide: $F^{\ge j}_1=F^{\ge j}_2$ as filtrations on $M[\frac{1}{p}]$.

	Let's prove the claim. One has $F_1^{\ge j}\subset F_2^{\ge j}$ because $\phi_{\Prism}|_{F_{1}^{\ge j}}=E^j\cdot \ol \phi_{\Prism}$ and thus $\phi_{\Prism}(F_1^{\ge j})\subset E^j \cdot M[\frac{1}{p}]$. Moreover, $F_1^{\ge 0}=F_2^{\ge 0}=M^{(1)}$ and so it remains to show that if $x\in F_1^{\ge j}$ and $\phi_{\Prism}(x)\in E^{j+1}\cdot M[\frac{1}{p}]$ (which is equivalent to $x\in F_1^{\ge j+1}$) then in fact $x\in F_1^{\ge j+1}$; in other words it is enough to show that the natural map $s\colon \gr^j_1\ra E^j\cdot M[\frac{1}{p}]/E^{j+1}\cdot M[\frac{1}{p}]\simeq (M[\frac{1}{p}]/E)\{j\}$ induced by $\phi_{\Prism}$ is an embedding. We have a commutative diagram 
	$$
	\xymatrix{H^n(\gr^j_N\RG_{\Prism^{(1)}}[\tfrac{1}{p}])\ar[rr]^{\phi_\Prism}&& M/I[\frac{1}{p}]\{j\}\\
	H^n(\Fil^{\ge j}_N\RG_{\Prism^{(1)}}[\tfrac{1}{p}])/H^n(\Fil^{\ge j +1}_N\RG_{\Prism^{(1)}}[\tfrac{1}{p}])\ar[rr]^(.7)i\ar[u]^\delta&&\gr_1^jM^{(1)}\ar[u]^s	
	}	
	$$
	where $\delta$ is an embedding coming as the boundary map in the long exact sequence of cohomology and $i$ is the surjection coming from the definition of $F_1^j$. The map $\phi_\Prism$ is identified with the one induced on $H^n$ by $$\Fil^{\le j}_\HT\RG_{\Prism/I}(\mstack X/\mf S)[\frac{1}{p}]\{j\} \tto \RG_{\Prism/I}(\mstack X/\mf S)[\frac{1}{p}]\{j\}$$
	and is an embedding by \Cref{rem: the maps from HT-truncations are embeddings}. Since so is $\delta$ from the commutative diagram we get that $i$ is also an embedding and, being also surjective, an isomorphism. It follows then that $s$ is also an embedding. 
	 Thus we get $F_1^j=F_2^j$. 


	Recall that by \Cref{rem: the BK-module associated to the etale cohomology} $M_\free\simeq \BK(H^n_\et(\widehat{\mstack X}_{\mbb C_p},\mbb Z_p)_\free)$. Note also that by \Cref{ex:stuff on BK-modules} $M_\free[\frac{1}{p}]=M[\frac{1}{p}]$. By \cite[Theorem 1.2.1(1)]{Kisin_BKModules} (and also \Cref{rem: the BK-module associated to the etale cohomology}) the filtration on  $D_\dR(H^n_\et(\widehat{\mstack X}_{\mbb C_p},\mbb Q_p))\simeq\footnote{This is also given by \cite[Theorem 1.2.1(1)]{Kisin_BKModules}.} M^{(1)}\otimes_{\mf S} K\simeq H^n_\dR(\mstack X_K/K)$ is given by the image $\Fil_{N \text{mod} E}^*$ of $F^{\ge *}_1=F^{\ge *}_2$ modulo $E$.  Let $\Fil_\Hdg^* H^n_{\dR}(\mstack X_K/K)$ denote the Hodge filtration on $H^n_{\dR}(\mstack X_K/K)$. By \Cref{lem: reduction of Nygaard filtration maps to Hodge} one has $\Fil_{N \text{mod} E}^j\subset \Fil^j_\Hdg$. Thus it is enough to show the inclusion in the other way. 

	Note that we have a natural map of long exact sequences 
	$$
	\xymatrix{\ldots \ar[r]& H^n(\Fil_N^{\ge j+1}\RG_{\Prism^{(1)}}[\tfrac{1}{p}])
	\ar[r]\ar[d]&  H^n(\Fil_N^{\ge j}\RG_{\Prism^{(1)}}[\tfrac{1}{p}])\ar[r]\ar[d]&  H^n(\gr_N^{j}\RG_{\Prism^{(1)}}[\tfrac{1}{p}])\ar[d] \ar[r]&\ldots\\
	\ldots \ar[r] & H^n(\Fil_\Hdg^{\ge j+1}\RG_\dR[\tfrac{1}{p}])
	\ar[r]&  H^n(\Fil_\Hdg^{\ge j}\RG_\dR[\tfrac{1}{p}])\ar[r]&  H^n(\gr_\Hdg^{j}\RG_\dR[\tfrac{1}{p}])\ar[r]&\ldots .}
	$$
	where the arrows 
	$
	H^n(\Fil_N^{\ge *}\RG_{\Prism^{(1)}}[\tfrac{1}{p}]) \ra H^n(\Fil_\Hdg^{\ge *}[\tfrac{1}{p}])
	$ factoring as 
	$$
	H^n(\Fil_N^{\ge *}\RG_{\Prism^{(1)}}[\tfrac{1}{p}]) \surj H^n(\Fil_N^{\ge *}\RG_{\Prism^{(1)}}[\tfrac{1}{p}])/I \ra H^n(\Fil_\Hdg^{\ge *}\RG_\dR[\tfrac{1}{p}])
	$$
	with the second arrow provided by \Cref{lem: reduction of Nygaard filtration maps to Hodge}.
	It is enough to show that the induced map $$
	 H^n(\Fil_\Hdg^{j}\RG_\dR)[\tfrac{1}{p}]\xymatrix{\ar[r]&} H^n(\gr_\Hdg^{j}\RG_\dR[\tfrac{1}{p}])$$ is surjective (this will show that  $\Fil_{N \text{mod} E}^j = \Im \left(F_1^{\ge j}/I\right) \subset M/I[\frac{1}{p}]$ surjects on $\Fil_\Hdg^j H^n_{\dR}$). This is enough to show after base change $(\mf S,(E(u)))\ra (\Ainf, (\xi))$.
	
	Note that the map $H^n(\gr_N^{j}\RG_{\Prism^{(1)}}(\mstack X_{\mc O_{\mbb C_p}}/\Ainf)[\tfrac{1}{p}])\ra H^n(\gr_\Hdg^{j}\RG_\dR(\mstack X_{\mbb C_p}/\mbb C_p))$ is given by the natural projection $$H^n(\Fil^{\le j}_\HT \RG_{\Prism/I}(\mstack X_{\mc O_{\mbb C_p}}/ \Ainf)[\tfrac{1}{p}]\{j\}) \xymatrix{\ar[r]&} H^{n}(\mstack X_{\mbb C_p}, \wedge^j \mbb L_{\mstack X_{\mbb C_p}/\mbb C_p}[-j])$$ to $j$-th Hodge-Tate graded component and that the latter splits off as the Hodge-Tate weight $0$ component of $H^n(\Fil^{\le j}_\HT \RG_{\Prism/I}(\mstack X_{\mc O_{\mbb C_p}}/ \Ainf)[\tfrac{1}{p}]\{j\})$ by \Cref{rem: HT-decomposition for HT-truncations}. Thus it is enough to show that the corresponding direct summand lies in the image of $$\gr_1^j\otimes_K\mbb C_p\simeq H^n(\Fil_N^{\ge j}\RG_{\Prism^{(1)}}[\tfrac{1}{p}])/H^n(\Fil_N^{\ge j+1}\RG_{\Prism^{(1)}}[\tfrac{1}{p}])\subset H^n(\gr_N^{j}\RG_{\Prism^{(1)}}[\tfrac{1}{p}]).$$
	However, if the image of $\gr_1^j\otimes_K \mbb C_p$ does not contain $H^{n}(\mstack X_{\mbb C_p}, \wedge^j \mbb L_{\mstack X_{\mbb C_p}/\mbb C_p}[-j])$ then the map $H^n_{\Prism^{(1)}}(\mstack X_{\mc O_{\mbb C_p}}/\Ainf)[\tfrac{1}{p}]\ra H^n_\dR(\mstack X_{\mbb C_p}/\mbb C_p)$ can not be surjective (and it should be, since $H^n_{\Prism^{(1)}}(\mstack X_{\mc O_{\mbb C_p}}/\Ainf)[\tfrac{1}{p}]$ is free over $\Ainf[\frac{1}{p}]$ and so $H^n_\dR(\mstack X_{\mbb C_p}/\mbb C_p)\simeq H^n_{\Prism^{(1)}}(\mstack X_{\mc O_{\mbb C_p}}/\Ainf)[\tfrac{1}{p}])/\xi$). Indeed, note that this map is $G_K$-equivariant and $H^n_\dR(\mstack X_{\mbb C_p}/\mbb C_p)\simeq H^n_\dR(\mstack X_{K}/K)\otimes_K \mbb C_p$ is pure of Hodge-Tate weight $0$. By \Cref{rem: HT-decomposition for HT-truncations} we have a decomposition 
	$$\oplus_j H^n(\gr_N^{j}\RG_{\Prism^{(1)}}(\mstack X_{\mc O_{\mbb C_p}})[\tfrac{1}{p}])\simeq \oplus_j H^n(\Fil_\HT^{\le j}\RG_{\Prism/I}(\mstack X_{\mc O_{\mbb C_p}})[\tfrac{1}{p}])\{j\}\simeq \oplus_j \oplus_{i=0}^{\min(j,n)} H^{n-i}(\mstack X_{K}, \wedge^{i} \mbb L_{\mstack X_{K}/K})\otimes_K \mbb C_p(j-i)$$ as $G_K$-modules. The weight 0 component in this sum is exactly given by the associated graded of the Hodge filtration $\oplus_{j} H^{n-j}(\mstack X_{K}, \wedge^j \mbb L_{\mstack X_{K}/K})\otimes_K \mbb C_p$ and note that due to the Hodge-to-de Rham degeneration its dimension is equal to $\dim H^n_\dR(\mstack X_K/\mc O_K)$. Note that the filtration $F_1^{\ge *}$ on $H^n_{\Prism^{(1)}}[\frac{1}{p}]$ is complete; indeed $F_1^{\ge *}=F_2^{\ge *}$ and from the definition of $F_2^{\ge *}$ and the fact that $H^n_{\Prism^{(1)}}[\frac{1}{p}]$ is finitely generated (and $\phi_{\Prism}[\frac{1}{E}]$ is an isomorphism) it is not hard to see that $F_2^{\ge j+1}=E\cdot F_2^{\ge j}$ for $j\gg 0$.  As we saw above the associated graded $\oplus_j \gr^j_1(H^n_{\Prism^{(1)}}(\mstack X_{\mc O_{\mbb C_p}}/\Ainf)[\tfrac{1}{p}])$ embeds into $\oplus_j H^n(\gr_N^{j}\RG_{\Prism^{(1)}}(\mstack X_{\mc O_{\mbb C_p}}/\Ainf)[\tfrac{1}{p}])$ and by completeness should still surject on $\oplus_j H^n(\gr_\Hdg^{j}\RG_\dR(\mstack X_{\mbb C_p}/\mbb C_p))$ which is pure of Hodge-Tate weight 0. The map is $G_K$-equivariant and looking at the dimensions we see that the only way it is surjective is if $\gr^j_1(H^n_{\Prism^{(1)}}(\mstack X_{\mc O_{\mbb C_p}}/\Ainf)[\tfrac{1}{p}])\subset H^n(\gr_N^{j}\RG_{\Prism^{(1)}}(\mstack X_{\mc O_{\mbb C_p}}/\Ainf)[\tfrac{1}{p}])$ contains the Hodge-Tate weight 0 component. Thus we are done and $\Fil^*_{N\text{mod}E}=\Fil_\Hdg^*$.
\end{proof}
\end{prop}
In the course of the proof we used the following lemma:
\begin{lem}\label{lem: reduction of Nygaard filtration maps to Hodge}
	Let $(A,I)$ be a {bounded} prism. The composite map $\Fil_{N}^j \RG_{\Prism}(-/A)^{(1)} \ra \RG_{\Prism}(-/A)^{(1)} \ra \RG_{\Prism}(-/A)^{(1)}\otimes_A A/I \simeq (\Omega_{-/A/I,\dR}^\bullet)^\wedge_{p}$ (as a transformation of functors from $\Aff^\sm_{A/I}$ to $\DMod{A/I}$) factors naturally through $j$-th term of the Hodge filtration $(\Omega_{-/A/I,\dR}^{\ge j})^\wedge_{p}$.
\end{lem}
\begin{proof}
	Recall the proof of de Rham comparison \cite[Corollary 15.4]{BS_prisms}: namely by the definition of Nygaard filtration $\phi_\Prism$ gives a filtered map $\Fil_N^* \RG_{\Prism}(-/A)^{(1)} \ra I^* \RG_\Prism(-/A)$ which induces an equivalence $\widetilde\phi_\Prism\colon \RG_{\Prism}(-/A)^{(1)} \ra L\eta_I \RG_\Prism(-/A)$ \cite[Theorem 15.3]{BS_prisms}; here one uses the identification $L\eta_I K\simeq \tau_B^{\le 0}(I^* \otimes_A K)$ \cite[Proposition 5.8]{BMS2} for any $K\in\DMod{A}$. The de Rham comparison then is obtained by taking reduction mod $I$ and identifying $(L\eta_I \RG_\Prism(-/A))\otimes_A A/I$ with the de Rham complex (via the Hodge-Tate comparison and \cite[Proposition 6.12]{BMS1}). Let $K\coloneqq \RG_{\Prism}(-/A)^{(1)}$. The complex $\Fil_N^j \RG_{\Prism}(-/A)^{(1)}$ with filtration $F^i$ given by $F^i\coloneqq\Fil_N^j \RG_{\Prism}(-/A)^{(1)}$ if $i\le j$ and $F^i\coloneqq\Fil_N^i \RG_{\Prism}(-/A)^{(1)}$ otherwise, is identified with $\tau_B^{\le 0}(I^{\max{(*,j)}} \otimes_A K)$; the natural map $\Fil_N^j \RG_{\Prism}(-/A)^{(1)}\ra \RG_{\Prism}(-/A)^{(1)}$ is just induced by $I^{\max{(*,j)}} \otimes_A K\ra I^{*} \otimes_A K$. Note that in fact $H^0_B(I^* \otimes_A K) \xra{\sim} \tau_B^{\le 0}(I^* \otimes_A K)\otimes_A A/I$. Thus the natural map $\tau_B^{\le 0}(I^{\max{(*,j)}} \otimes_A K)\ra \tau_B^{\le 0}(I^* \otimes_A K)\otimes_A A/I$ quotients through $H^0_B(I^{\max{(*,j)}} \otimes_A K)$ which (see the description of $H^0_B(-)$ in \cite[Theorem 5.4(3)]{BMS2}) is exactly identified with $(\Omega_{-/A/I,\dR}^{\ge j})^\wedge_{p}$ (under the identification $H^0_B(I^* \otimes_A K)\simeq (\Omega_{-/A/I,\dR}^\bullet)^\wedge_{p}$).
\end{proof}

As a corollary we get the Hodge-Tate decomposition for $H^n_{\et}(\widehat{\mstack X}_{\mbb C_p},\mbb Q_p)\otimes_{\mbb Q_p}\mbb C_p$:
\begin{thm}[{Hodge-Tate decomposition}]\label{thm: hodge-tate decomposition}
	Let $\mstack X$ be a smooth Hodge-proper stack over $\mc O_K$. Then one has the Hodge-Tate decomposition:	
	$$
	H^n_{\et}(\widehat{\mstack X}_{\mbb C_p},\mbb Q_p)\otimes_{\mbb Q_p}\mbb C_p \simeq \bigoplus_{i+j=n} H^j(\mstack X_{K},\wedge^i\mbb L_{\mstack X_K/K})\otimes_{K}\mbb C_p(-i).
	$$
\end{thm}
\begin{proof} From \Cref{prop: D_dR for the etale cohomology of a smooth Hodge proper stack} it follows that one has a filtered $G_K$-equivariant isomorphism 
	$$
	H^n_{\et}(\widehat{\mstack X}_{\mbb C_p},\mbb Q_p)\otimes_{\mbb Q_p}B_{\dR}\xymatrix{\ar[r]^\sim&}H^n_\dR(\mstack X_K/K)\otimes_K B_\dR.
	$$
	Using that $\gr^iB_{\dR}=\mbb C_p(i)$ the Hodge-Tate decomposition is obtained by taking $\gr^0$ of both sides.\qedhere
\end{proof}

\subsubsection{Inequality on dimensions}\label{sssect:inequality of dimension}
In this section as another application of the \'etale comparison we establish an inequality between the lengths of crystalline and \'etale cohomology modulo $p^n$. As before $\mc O_K$ denotes the ring of integers of a finite extension $K/\mathbb Q_p$. We fix a Breuil-Kisin prism $(\mf S, (E(u)))$ with $\mf S/E(u)\simeq \mc O_K$.

Note that $\mf S/p^n\simeq W_n(k)[[u]]$. We have  $\mf S/p^n\otimes_{\mf S} \mf S/u\simeq \mf S/(p^n,u)\simeq W_n(k)$ and $\mf S/p^n[\tfrac{1}{E}]\simeq\mf S/p^n[\tfrac{1}{u}]\simeq  W_n(k)((u))$ (since $E(u)\in W(k)[u]$ is Eisenstein and thus is congruent to a power of $u$ mod $p$). Note that both $W_n(k)$ and $W_n(k)((u))$ are Noetherian and have finite lengths as modules over themselves (the only non-trivial submodules are given by the principal ideals $(p^i)$ with $i\le n$). Consequently any finitely generated module also has finite length.
\begin{lem}[Semicontinuity]\label{lem: semicontinuity of stalks}
	Let $M\in \Coh^+(\mf S/p^n)$ be a bounded below complex. Then for any $i$ one has
	$$
	\length_{W_n(k)} H^i(M\otimes_{\mf S/p^n} W_n(k)) \ge \length_{W_n(k)((u))} H^i(M[\tfrac{1}{E}]).
	$$

\begin{proof}
	By the universal coefficient formula we have a short exact sequence:
	$$
	\xymatrix{0 \ar[r] & H^i(M)/u \ar[r] &  H^i(M\otimes_{\mf S/p^n} W_n(k))\ar[r] & \Tor^1_{\mf S/p^n}(W_n(k), H^{i+1}(M))\ar[r] & 0.}
	$$ 
	Since the length is additive in short exact sequences we get an inequality
	$$
	\length_{W_n(k)}\left(H^i(M\otimes_{\mf S/p^n} W_n(k))\right)\ge \length_{W_n(k)}(H^i(M)/u).
	$$
	We also have $H^i(M[\tfrac{1}{E}])\simeq H^i(M)[\tfrac{1}{E}]\simeq H^i(M)[\tfrac{1}{u}]$ and thus (replacing $M$ by $H^i(M)$) it is enough to show that for a classical $M$ we have $\length_{W_n(k)}(H^i(M)/u)\ge \length_{W_n(k)((u))}(H^i(M)[\tfrac{1}{u}])$. Filtering $M$ by $p^kM$ and taking associated graded we reduce to the case $M$ is $p$-torsion. There the statement reduces further to the inequality on dimensions of stalks of a finitely generated $k[[u]]$-module which is an easy consequence of (classical) Nakayama lemma.
\end{proof}
\end{lem}

We use this lemma to prove the following:
\begin{prop}
	Let $\mstack X$ be a smooth Hodge-proper Artin stack over $\mc O_K$. Then
	$$\length_{\mathbb Z/p^n} H^i_\et(\widehat{\mstack X}_{\mbb C_p}, \mathbb Z/p^n) \le \length_{W_n(k)} H^i_\crys(\mstack X_k/W_n(k)).$$

	\begin{proof}
		Consider the mod $p^n$ twisted prismatic cohomology $\RG_{\Prism^{(1)}/p^n}(\mstack X/\mf S)\coloneqq \RG_{\Prism^{(1)}}(\mstack X/\mf S)/p^n$; by Hodge-properness of $\mstack X$ we have $\RG_{\Prism^{(1)}/p^n}(\mstack X/\mf S)\in \Coh^+(\mf S/p^n)$. We will use \Cref{lem: semicontinuity of stalks}. By base change for the mod $u$ reduction $(\mf S, (E))\ra (W(k),(p))$ we have
		$$ H^i_\crys(\mstack X_k/W_n(k))\simeq H^i(\RG_{\Prism^{(1)}/p^n}(\mstack X/\mf S)\otimes_{\mf S/p^n} W_n(k)).$$
		It remains to relate $\RG_{\Prism^{(1)}/p^n}(\mstack X/\mf S)[E^{-1}]$ to \'etale cohomology.  
		Note that for any $\mbb Z_p$-module $M$ one has $\length_{\mathbb Z_p}M = \length_{W(\mbb C_p^\flat)}\left( M\otimes _{\mbb Z_p}W(\mbb C_p^\flat)\right)$. By \Cref{cor: prismatic_vs_etale_torsion} (and \Cref{rem: comparison for the twisted prismatic cohomology}) reducing modulo $p^n$ we get an isomorphism $H^i_\et(\mstack X_{\mbb C_p},\mbb Z/p^n)\otimes_{\mbb Z_p} \Ainf[\frac{1}{\xi}]\simeq H^i_{\Prism^{(1)}/p^n}(\mstack X_{\mc O_{\mbb C_p}}/\Ainf)[\frac{1}{\xi}]$. 
		Note that $\Ainf[\frac{1}{\xi}]/p^n\simeq W_n(\mbb C_p^\flat)$ and we have
		$$\length_{\mathbb Z_p} H^i_\et(\mstack X_{\mbb C_p},\mbb Z/p^n) = \length_{W(\mbb C_p^\flat)} H^i_{\Prism^{(1)}/p^n}(\mstack X_{\mc O_{\mbb C_p}}/\Ainf)[\tfrac{1}{\xi}].$$
		Next, by base change for $(\mf S,(E))\ra (\Ainf,(\xi))$ one has
		$$H^i_{\Prism^{(1)}/p^n}(\mstack X_{\mc O_{\mbb C_p}}/\Ainf)[\tfrac{1}{\xi}]\simeq H^i_{\Prism^{(1)}/p^n}(\mstack X/\mf S)[\tfrac{1}{E}]\otimes_{\mf S/p^n[\tfrac{1}{E}]}W_n(\mbb C_p^\flat).$$
		Note that every finite length $\mf S/p^n[\tfrac{1}{E}]$-module is a direct sum of $\mf S/p^{s_i}[\tfrac{1}{E}]\simeq \mf S/p^n[\tfrac{1}{E}]\otimes_{\mbb Z/p^n} \mbb Z/p^{s_i}$ for $s_i\le n$. We have $\mf S/p^{s_i}[\tfrac{1}{E}]\otimes_{\mf S} W(\mbb C_p^\flat)\simeq W_{{s_i}}(\mbb C_p^\flat)$; we see that the tensor product $-\otimes_{\mf S} W(\mbb C_p^\flat) $ preserves the length. In particular,
		$$\length_{W_n(\mbb C_p^\flat)}\left( H^i_{\Prism^{(1)}/p^n}(\mstack X_{\mc O_{\mbb C_p}}/\Ainf)[\tfrac{1}{\xi}]\right) = \length_{\mathfrak S[E^{-1}]} \left(H^i_{\Prism^{(1)}/p^n}(\mstack X/\mf S)[\tfrac{1}{E}]\right).\qedhere$$
	\end{proof}
\end{prop}

Specializing to the case $n=1$ we get:
\begin{cor}\label{inequality}
	Let $\mstack X$ be a smooth Hodge-proper stack over $\mc O_K$. Then for any $i\ge 0$		
	$$\dim_{\mbb F_p}H^i_\et(\widehat{\mstack X}_{\mbb C_p},\mbb F_p) \le \dim_{k}H^i_\dR(\mstack X_k/k).$$
\end{cor}

\section{Local acyclicity of some quotient stacks} \label{sec: diff_etale}
Recall that in \Cref{sect:etale_cohomology_on_stacks} we introduced two versions of "\'etale cohomology of the generic fiber" for Artin stacks over $\mathcal O_{\mathbb C_p}$: namely \'etale cohomology of the algebraic generic fiber $\mstack X_{\mathbb C_p}$ and the Raynaud's generic fiber $\widehat{\mstack X}_{\mathbb C_p}$ correspondingly. Both versions have their own advantages: by Artin's comparison (see \Cref{Artins_comparison_stacks}) the algebraic version $R\Gamma_\et(\mstack X_{\mathbb C_p}, \mathbb Z_p)$ is related to the usual singular cohomology of $\mstack X(\mathbb C)$, while the rigid-analytic version $R\Gamma_\et(\widehat{\mstack X}_{\mathbb C_p}, \mathbb Z_p)$ is related to prismatic cohomology of $\mstack X$ via the \'etale comparison (\Cref{cor: Z_p etale cohomology of rigid fiber}).

We constructed a natural map $\Upsilon_{\mstack X}\colon R\Gamma_\et(\mstack X_{\mathbb C_p}, \mathbb Z_p) \to R\Gamma_\et(\widehat{\mstack X}_{\mathbb C_p}, \mathbb Z_p)$ and stated \Cref{etale_conjecture} that formulates when we expect this map to be an equivalence. In this section we show that $\Upsilon_{\mstack X}$ is an equivalence in the case of  the global quotient $[X/P]$ of a smooth proper scheme $X$ by a parabolic subgroup $P$ in a reductive group $G$. Our proof is geometric in nature and relies on the basic structure theory of reductive groups. 

Here is the plan of this section. In the first two subsections we compute de Rham and prismatic cohomology of a classifying stack of a torus $T$ by approximating it by (products of) projective spaces $\mathbb P^n$. In subsection \ref{subsect:prismatic_coh_of_An_quot_T} we show that prismatic and rigid-analytic \'etale cohomology of the quotient stacks $[\mbb A^n/T]$ and $BT$ are the same, provided the $T$-action on $\mbb A^n$ is contracting. We then use this in subsection \ref{subsect:case_of_BB} to deduce that $\Upsilon_{BB}$ is an equivalence for the classifying stack $BB$ of a Borel subgroup $B\subset G$. Finally, in subsection \ref{sec: comparison for [X/P]}, using the \'etale sheaf machinery developed in \Cref{sect:etale_sheaves}, we deduce from this the general case $[X/P]$.

\subsection{De Rham cohomology of $B\mathbb G_m$}\label{sec:de Rham cohomology of BG_m}
Recall that in \Cref{sec: Hodge and de Rham cohomology} for any Artin stack $\mstack X$ over any base ring $R$ we have defined its Hodge cohomology $R\Gamma_\Hdg(\mstack X/R)$ and de Rham cohomology $R\Gamma_\dR(\mstack X/R)$. In this section we compute them for $X = B\mathbb G_m$ over $R=\mathbb Z$. We also show that pullback under the natural map $\mathbb P^\infty \to B\mathbb G_m$ induces an equivalence on de Rham cohomology.

We denote by $\mbb Z(n)$ the character of $\mbb G_m$ given by $\mbb G_m\xra{t\ra t^n} \mbb G_m$ together with the corresponding quasi-coherent sheaf on $B\mbb G_m$. By \Cref{ex:contangent_for_BG} the cotangent complex of $B\mathbb G_m$ is given by $\mbb Z(0)[-1]$ and more generally $\wedge^i \mathbb L_{B\mathbb G_m/\mathbb Z} \simeq \mbb Z(0)[-i]$. This way the algebra $\oplus_{i=0}^\infty \wedge^i(\mathbb L_{B\mathbb G_m/\mathbb Z})[-i]$ is equivalent to $\oplus_{i=0}^\infty \mathbb Z(0)[-2i]$ with the algebra structure given by the natural isomorphisms $\mathbb Z(0)[-2i]\otimes\mathbb Z(0)[-2j] \simeq \mathbb Z(0)[-2(i+j)]$. We have $R\Gamma(B\mathbb G_m,\mathbb Z(0))\simeq \mathbb Z(0)^{\mathbb G_m} \simeq \mathbb Z$, since the functor of $\mathbb G_m$-invariants is $t$-exact over an arbitrary base. Consequently, $R\Gamma(B\mathbb G_m,\oplus_{i=0}^\infty \wedge^i(\mathbb L_{B\mathbb G_m/\mathbb Z})[-i])\simeq \oplus_{i=0}^\infty \mbb Z[-2i]$. Taking cohomology and picking a generator $x\in H^2_{\Hdg}(B\mathbb G_m/\mathbb Z)\simeq \mbb Z$ we get 
\begin{lem}
The algebra of Hodge cohomology $H^\bullet_{\Hdg}(B\mathbb G_m/\mathbb Z)$ of $B\mathbb G_m$ over $\mathbb Z$ is isomorphic to $\mathbb Z[x]$, $\deg(x)=2$.
\end{lem}
\begin{cor}
The Hodge-de Rham spectral sequence for $B\mathbb G_m$ over $\mathbb Z$ degenerates and $H^\bullet_\dR(B\mathbb G_m/\mathbb Z)\simeq \mathbb Z[x]$, $\deg(x)=2$.

\begin{proof}
The first statement holds for every smooth linearly reductive group $G$ over a base ring $R$, since $E_1^{p,q}\simeq H^p(BG,\wedge^q\mbb L_{G/R})\simeq H^p(G, \Sym^q(\mf g^*)[-q]) \simeq 0$ for $p\ne q$ by semi-simplicity of the category of $G$-representations. So $\gr^\bullet H^\bullet_\dR(B\mathbb G_m/\mathbb Z) \simeq 
H^\bullet_\Hdg(B\mathbb G_m/\mathbb Z) \simeq \mathbb Z[x]$. By choosing a homogeneous lift of $x$ in $H^2_\dR(B\mathbb G_m/\mathbb Z)$ we obtain a filtered map $\mathbb Z[x] \to H^\bullet_\dR(B\mathbb G_m/\mathbb Z)$, which induces an isomorphism on the associated graded pieces, thus is an isomorphism itself.
\end{proof}
\end{cor}

\begin{rem}
Since $B\mathbb G_m$ is smooth over $\mathbb Z$ and $\mbb Z$ has finite $\Tor$-dimension, by base change the same result holds for any $R$:
$$H^\bullet_\dR(B\mathbb G_m/R)\simeq H^\bullet_\Hdg(B\mathbb G_m/R)\simeq R[x],\ \deg(x) = 2.$$
\end{rem}

\smallskip Let $\mathbb P^n$ be the projective space over $\Spec \mathbb Z$. Then by the standard computation of the Hodge cohomology of ${\mathbb P^n}$ (e.g. see \cite[Lemma 50.11.3]{StacksProject}) we have
$$H^\bullet_\Hdg(\mathbb P^n/\mathbb Z) \simeq \mathbb Z[y]/y^{n+1}, \ \deg(y)=2.$$
The Hodge-de Rham spectral sequence degenerates since the cohomology of the graded pieces $R\Gamma(\mathbb P^n, \Omega^i_{\mathbb P^n/\mbb Z}[-i])$ are concentrated in even total degrees and so we have the same answer for the de Rham cohomology
$$H^\bullet_\dR(\mathbb P^n/\mathbb Z) \simeq \mathbb Z[y]/y^{n+1}.$$
Note that any subspace inclusion $i\colon \mathbb P^k \inj \mathbb P^n$ induces a map $i^*\colon R\Gamma_\dR(\mathbb P^n/\mathbb Z)\ra R\Gamma_\dR(\mathbb P^n/\mathbb Z)$, which on the cohomology can be identified with the quotient by $y^{k+1}$: $\mathbb Z[y]/(y^{n+1})\surj \mathbb Z[y]/(y^{k+1})$.

The line bundle $\mathcal O_{\mathbb P^n}(1)$ on $\mathbb P^n$ is classified by a map $\mathcal O_{\mathbb P^n}(1)\colon \mathbb P^n \ra B\mathbb G_m$ and we can consider the corresponding pull-back map $\mathcal O_{\mathbb P^n}(1)^*\colon R\Gamma_\dR(B\mathbb G_m/\mathbb Z)\ra R\Gamma_\dR(\mathbb P^n/\mathbb Z)$ on the de Rham cohomology.
\begin{prop}\label{de_Rham_fur_P_n}
Under the identifications $H^\bullet_\dR(B\mathbb G_m/\mathbb Z)\simeq \mathbb Z[x]$  and $H^\bullet_\dR(\mathbb P^n/\mathbb Z)\simeq \mathbb Z[y]/y^{n+1}$ we have
$$\mathcal O_{\mathbb P^n}(1)^*(x)=\pm y \qquad\text{and}\qquad H^2_\dR(B\mathbb G_m/\mathbb Z)\simeq H^2_\dR(\mathbb P^n /\mathbb Z).$$
As a consequence, the fiber of the pullback map $\mathcal O_{\mathbb P^n}(1)^*\colon R\Gamma_\dR(B\mathbb G_m/\mathbb Z) \to R\Gamma_\dR(\mathbb P^n/\mathbb Z)$ is $(2n+1)$-coconnected.

\begin{proof}
Consider the map of presheaves $d\log\colon \mathcal O^\times_- \to \Omega^1_-$ on $\Aff^{\op,\sm}_{/B\mathbb G_m}$ given on each $X\to B\mathbb G_m$ by the map $d\log\colon \mathcal O_X^\times \to \Omega_X^1$ that sends an invertible function $f\in \mathcal O^\times_X(X)$ to the differential form $d\log(f) \coloneqq df/f$. Passing to the limit over $\Aff^{\op,\sm}_{/B\mathbb G_m}$ and applying $H^1$, we obtain a map
$$d\log \colon \Pic(B\mathbb G_m) \xymatrix{\ar[r] &} H^1(B\mathbb G_m, \Omega^1_-) \stackrel{(*)}{\simeq} H^1(B\mathbb G_m, \mathbb L_{B\mathbb G_m/\mathbb Z}) \simeq H^2_\dR(B\mathbb G_m)\simeq \mathbb Z,$$
where $\Pic(\mstack X) \coloneqq \pi_0\Hom_{\Stk}(\mstack X, B\mathbb G_m)$ and the equivalence $(*)$ follows from the smooth descent for cotangent complex (\Cref{Hodge_as_limit}). On the other hand
$$\Pic(B\mathbb G_m) = \pi_0 \Hom_{\Stk}(B\mathbb G_m, B\mathbb G_m) \simeq \Hom_{\Grp}(\mathbb G_m, \mathbb G_m) \simeq \mathbb Z.$$
So the map $d\log\colon \mr{Pic}(B\mathbb G_m)\to H^2_\dR(B\mathbb G_m)$ is some map from $\mathbb Z$ to $\mathbb Z$ and we need to show it is an isomorphism.

To this end consider the pull-back to $\mathbb P^n$ of the transformation $d\log\colon R\Gamma_{\et}(-, \mathcal O^\times)\to \Omega^1$ under $\mathcal O_{\mathbb P^n}(1)\colon \mathbb P^n \ra B\mathbb G_m$. This gives a commutative square:
$$
\xymatrix{
\mr{Pic}(B\mathbb G_m/\mathbb Z) \ar[d]_{d\log}\ar[r]^{\mathcal O(1)^*}& \mr{Pic}(\mathbb P^n/\mathbb Z) \ar[d]_{d\log}\\
H^2_\dR(B\mathbb G_m/\mathbb Z)\ar[r]^{\mathcal O(1)^*}& H^2_\dR(\mathbb P^n/\mathbb Z).
}
$$
Since the line bundle $\mathcal O_{\mathbb P^n}(1)$ generates $\mr{Pic}(\mathbb P^n/\mathbb Z)$ the top horizontal arrow is an isomorphism. Since all groups in the diagram are isomorphic to $\mathbb Z$, if we show that the right vertical arrow is an isomorphism it will follow that all remaining arrows are isomorphisms as well. Consider an embedding of a line $i\colon \mathbb P^1\hookrightarrow \mathbb P^n$, then $i^*$ induces isomorphisms on both the Picard groups and $H^2_\dR$ and so it is enough to check that $d\log$ is an isomorphism for $\mathbb P^1$, which is is easy. We get that a generator $x$ in $H^2_\dR(B\mathbb G_m/\mathbb Z)$ is getting mapped to a generator of $H^2_\dR(\mathbb P^n/\mathbb Z)$, two options given by $\pm y$.

The pullback $\mathcal O_{\mathbb P^n}(1)^*\colon H^\bullet_\dR(B\mathbb G_m)\ra H^\bullet_\dR(\mathbb P^n)$ is a homomorphism of algebras and since $\mathcal O_{\mathbb P^n}(1)^*(x)=\pm y$ it can be identified with the projection $\mathbb Z[y]\surj \mathbb Z[y]/(y^{n+1})$ and we have $H^\bullet(\fib(\mathcal O_{\mathbb P^n}(1)^*))\simeq y^{n+1}\cdot \mathbb Z[y]$ which lives in cohomological degrees $\ge 2n+2$, so the fiber is $(2n+1)$-coconnected.
\end{proof}
\end{prop}

Now consider an infinite system of subspaces $\mathbb P^0\overset{i_0}{\hookrightarrow} \mathbb P^1\overset{i_1}{\hookrightarrow} \cdots \overset{i_{n-1}} {\hookrightarrow} \mathbb P^n \overset{i_n}{\hookrightarrow} \cdots$. The collection of line bundles $\mathcal O_{\mathbb P^n}(1)$ on $\mathbb P^n$ produces a compatible system of maps $\mathcal O_{\mathbb P^n}(1)\colon \mathbb P^n\ra B\mathbb G_m$. Let $\mathcal{O}_{\mathbb P^\infty}(1)^*\colon R\Gamma_\dR(B\mathbb G_m/\mathbb Z)\ra \prolim R\Gamma_\dR(\mathbb P^n/\mathbb Z)$ be the natural map to the limit.
\begin{cor}\label{BG_m_limit_dR}
The pullback map
$$\mathcal{O}_{\mathbb P^\infty}(1)^* \colon R\Gamma_\dR(B\mathbb G_m / \mathbb Z) \xymatrix{\ar[r] &} \prolim R\Gamma_\dR(\mathbb P^n / \mathbb Z)$$
is an equivalence.

\begin{proof}\label{cor:de Rham cohomology of BG_m as a limit}
By \Cref{de_Rham_fur_P_n}, the coconnectivity of the fiber of the map $\mathcal{O}_{\mathbb P^n}(1)^*$ tends to $\infty$ when $n\to \infty$.
\end{proof}
\end{cor}

\subsection{Prismatic and \'etale cohomology of $B\mathbb G_m$}\label{prismatoc_coh_of_Btori}
Fix a Breuil-Kisin prism $(\mf S,(E))$ with $\mf S/E\simeq \mc O_K$. Again we consider a system $\mathbb P^0\overset{i_0}{\hookrightarrow} \mathbb P^1\overset{i_1}{\hookrightarrow} \cdots \overset{i_{n-1}} {\hookrightarrow} \mathbb P^n \overset{i_n}{\hookrightarrow} \cdots$, but now over $\mathcal O_K$. We have compatible maps $\mathcal O_{\mathbb P^n}(1)\colon \mathbb P^n\ra B\mathbb G_m$ which produce a map $\mathcal{O}_{\mathbb P^\infty}(1)^*\colon R\Gamma_\Prism(B\mathbb G_m/\mf S)\to \prolim R\Gamma_\Prism(\mathbb P^n/ \mf S)$. To identify the right hand side recall that the \'etale cohomology of $\mathbb P^n_C$ are given by
$$H^*_{\et}(\mathbb P^n_{C}, \mathbb Z_p)\simeq \bigoplus_{i=0}^n \mathbb Z_p(-i)[-2i].$$ 
Recall that by \cite[Theorem 1.2.1]{Kisin_BKModules} the functor $\BK\colon \Rep_{G_K}^{\#, \crys} \to \Mod_{\mathfrak S}^\phi$ is a symmetric monoidal equivalence and $\BK(\mathbb Z_p(n)) \simeq \mathfrak S\{n\}$. Moreover, since the crystalline cohomology of $\mathbb P^n$ are torsion free, by \cite[Theorem 14.6(iii)]{BMS1} there is a natural isomorphism of Breuil-Kisin modules $H^i_\Prism(\mathbb P^n/\mathfrak S) \simeq \BK(H^i_\et(\mathbb P^n_K, \mathbb Z_p))$. So
$$
H^*_\Prism(\mathbb P^n/\mf S)\simeq \mr{BK}(H^*_\et(\mathbb P^n_{{\mathbb C_p}}, \mathbb Z_p))\simeq \bigoplus_{i=0}^n \mr{BK}(\mathbb Z_p(-i))[-2i]\simeq \bigoplus_{i=0}^n \mf S\{-i\}[-2i]
$$
as a ring object in the category of graded Breuil-Kisin modules. Here the ring structure on the right hand side is induced by isomorphisms $\mf S\{-i\}\otimes_{\mf S} \mf S\{-j\}\simeq \mf S\{-i-j\}$ if $i+j\le n$ and zero otherwise. In particular the underlying $\mf S$-module is just $\mf S[x]/(x^{n+1})$.
\begin{lem}\label{lem:prismatic_coh_of_BGm_vs_Pn}
The map $\mathcal{O}_{\mathbb P^\infty}(1)^*: R\Gamma_\Prism(B\mathbb G_m/\mf S)\to \prolim R\Gamma_\Prism(\mathbb P^n/\mf S)$ is an equivalence.

\begin{proof}
We will use the de Rham comparison and the derived Nakayama lemma. Namely, since $\varphi_{\mf S}$ is faithfully flat, it is enough to prove that 
$$
\varphi_{\mf S}(\mathcal{O}_{\mathbb P^\infty}(1)^*)\colon \varphi_{\mf S}^* R\Gamma_\Prism(B\mathbb G_m/\mf S)\xymatrix{\ar[r] &} \varphi_{\mf S}^*(\prolim R\Gamma_\Prism(\mathbb P^n/\mf S))
$$
is an equivalence. Moreover, since both sides are $E$-adically complete, it is enough to check this modulo $E\in \mf S$. The $\mf S$-algebra $\phi_{\mf S*}\mf S$ is perfect as an $\mf S$-module and so
$$
\varphi_{\mf S}^*(\prolim R\Gamma_\Prism(\mathbb P^n/\mf S))\xymatrix{\ar[r] &} \prolim \varphi_{\mf S}^* R\Gamma_\Prism(\mathbb P^n/\mf S).
$$
is an equivalence. Since $\mathcal O_K$ is a perfect $\mathfrak S$-module as well, the natural map 
$$
(\prolim \varphi_{\mf S}^* R\Gamma_\Prism(\mathbb P^n/\mf S))\otimes_{\mf S} \mathcal O_K \xymatrix{\ar[r] &} \prolim \left(\varphi_{\mf S}^* R\Gamma_\Prism(\mathbb P^n/\mf S)\otimes_{\mf S} \mathcal O_K\right)
$$
is also an equivalence. So it is enough to check that 
$$
\varphi_{\mf S}(\mathcal{O}_{\mathbb P^\infty}(1)^*)/E \colon \varphi_{\mf S}^*(R\Gamma_\Prism(B\mathbb G_m/\mf S)) \otimes_{\mf S} \mathcal O_K \xymatrix{\ar[r] &} \prolim \left(\varphi_{\mf S}^* R\Gamma_\Prism(\mathbb P^n/\mf S)\otimes_{\mf S}\mathcal O_K\right)
$$
is an equivalence. But by the de Rham comparison (\Cref{prop: de Rham comparison}) this map identifies with
$$\mathcal{O}_{\mathbb P^\infty}(1)^*\colon R\Gamma_\dR(B\mathbb G_m/\mc O_K)\xymatrix{\ar[r] &} \prolim R\Gamma_\dR(\mathbb P^n/\mc O_K),$$
which is an equivalence 
 by \Cref{cor:de Rham cohomology of BG_m as a limit}. Indeed we have $\prolim R\Gamma_\dR(\mathbb P^n/\mc O_K)\simeq (\prolim R\Gamma_\dR(\mathbb P^n/\mbb Z))\otimes_{\mbb Z}\mc O_K$ by the increasing coconnectivity of $\mc O_{P^n}(1)^*$ and finite Tor-dimension of $\mbb Z$, and thus we are done by base change for de Rham cohomology.
\end{proof}
\end{lem}
\begin{rem}\label{rem: the pull-back to P^n is 2n+1 connected }
In particular we get that 
$$H^*_\Prism(B\mathbb G_m/\mf S)\simeq \Sym_{\mf S}(\mathfrak S\{-1\}[-2]) \simeq \bigoplus_{i\ge 0} \mf S\{-i\}[-2i]$$
as a ring object in the category of graded Breuil-Kisin modules (with the underlying $\mf S$-algebra isomorphic to $\mf S[x]$). The natural map $\RG_\Prism(B\mbb G_m/\mf S) \ra \RG_\Prism(\mbb P^n/\mf S)$ is given on cohomology by the  surjection $\mf S[x]\surj \mf S[x]/x^{n+1}$ and thus is $(2n+1)$-connected. 
\end{rem}

\begin{rem}
	Note that by base change (for prismatic cohomology) we also have 
	$$
	H^*_\Prism(B\mbb G_m/\Ainf)\simeq \Sym_{\Ainf}(\Ainf\{-1\}[-2]) \simeq \bigoplus_{i\ge 0} \Ainf\{-i\}[-2i].
	$$
\end{rem}
A similar argument can be given to compute prismatic cohomology of a split torus. We will deduce it from the case of $\mathbb G_m$ instead:
\begin{prop}\label{prop: prismatic cohomology of BT}
Let $T$ be a split torus over $\mathcal O_K$. Then there is a canonical isomorphism
$$\Sym_{\Mod_{\mathfrak S}^\phi}(X^*(T)\otimes_{\mathbb Z} \mathfrak S\{-1\}) \xymatrix{\ar[r]^-\sim &} H^*_\Prism(BT/\mathfrak S),$$
where $X^*(T)$ denotes the character lattice of $T$ and $X^*(T)\otimes_{\mathbb Z} \mathfrak S\{-1\}$ is placed in cohomological degree $2$. Non-canonically $H^*_\Prism(BT/\mf S)\simeq \mf S[x_1,\ldots,x_k]$ as graded $\mf S$-algebras, where $k=\dim T$ and $\deg x_i=2$.

\begin{proof}
First note that there is a natural map of sets $X^*(T)\times\mathfrak S\{-1\} \to H^2_\Prism(BT/\mathfrak S)$ defined as follows: for a fixed character $\chi \colon T \to \mathbb G_m$ we take pullback of $H^2(B\mathbb G_m/\mathfrak S) \simeq \mathfrak S\{-1\}$ under the induced map $B(\chi)\colon BT\to B\mathbb G_m$. This map is in fact linear in $\chi$; indeed since $B(\chi_1+\chi_2)\colon BT\ra B\mbb G_m$ is given by the composition 
$$
\xymatrix{BT\ar[r]^{B(\Delta)}& BT\times BT \ar[rr]^{B(\chi_1)\times B(\chi_2)} && B\mbb G_m\times B\mbb G_m \ar[r]^{B(m)}& B\mbb G_m}
$$ 
where $m\colon \mbb G_m\times \mbb G_m \ra \mbb G_m$ is the multiplication map, thus it is enough to see that for any $x\in H^2_\Prism(B\mbb G_m/\mf S)$ we have $B(m)^*(x)= x\otimes 1 + 1\otimes x \in H^2_\Prism(B\mbb G_m \times B\mbb G_m/\mf S)$ (where we identify $H^2_\Prism(B\mbb G_m \times B\mbb G_m/\mf S)$ with $H^2(B\mbb G_m/\mf S)\otimes_{\mf S} H^0(B\mbb G_m/\mf S)\oplus H^0(B\mbb G_m/\mf S)\otimes_{\mf S}H^2(B\mbb G_m/\mf S)$ via the K\"unneth formula). This is seen immediately by restricting further to $B\mbb G_m\times Be$ and $Be\times B\mbb G_m$ (where $e\hookrightarrow \mbb G_m$ is the identity).

 Extending by multiplicativity we obtain a comparison homomorphism
$$\Sym_{\Mod_{\mathfrak S}^\phi}(X^*(T)\otimes_{\mathbb Z} \mathfrak S\{-1\}) \xymatrix{\ar[r] &} H^*_\Prism(BT/\mathfrak S),$$
which can be seen to be an isomorphism by the K\"unneth formula (and the case of a $1$-dimensional torus). The formula $H^*_\Prism(BT/\mf S)\simeq \mf S[x_1,\ldots,x_k]$ follows by fixing a splitting $T\simeq (\mbb G_m)^k$ and using the K\"unneth formula.
\end{proof}
\end{prop}
\begin{rem}\label{rem: prismatic cohomology of BT over O_C}
Note that (e.g. by \Cref{prop:Conrad}) any torus $T$ over $\mc O_{\mbb C_p}$ is split. By an analogous argument one gets a functorial isomorphism
$$
\Sym_{\Mod_{\Ainf}^\phi}(X^*(T)\otimes_{\mathbb Z} \Ainf \{-1\}) \xymatrix{\ar[r]^-\sim &} H^*_\Prism(BT/\Ainf).
$$
If $T$ was defined (but not necessarily split) over $\mc O_K$, then by functoriality this is an equivalence of $G_K$-modules (with the natural $G_K$-action on $X^*(T_{\mc O_{\mbb C_p}})$).
\end{rem}

\begin{cor} \label{BG_m_lim_et}
The pullback map on the \'etale cohomology of the Raynaud generic fibers
$$\mathcal{O}_{\mathbb P^\infty}(1)^*\colon R\Gamma_\et((\widehat{B\mathbb G_m})_{\mbb C_p},\mathbb Z_p)\xymatrix{\ar[r] &} \prolim R\Gamma_\et((\widehat{\mathbb P^n})_{\mbb C_p},\mathbb Z_p)$$
is an equivalence.

\begin{proof}
By the \'etale comparison \Cref{etale_comparison} we have
$$R\Gamma_\et((\widehat{B\mathbb G_m})_{\mbb C_p}, \mathbb Z_p) \simeq \left(R\Gamma_\Prism(B\mathbb G_m/ \mathfrak S)\otimes_{\mathfrak S} W(\mbb C_p^\flat) \right)^{\phi = 1} \quad\text{and}\quad R\Gamma_\et((\widehat{\mathbb P^n})_{\mbb C_p}, \mathbb Z_p) \simeq \left(R\Gamma_\Prism(\mathbb P^n/ \mathfrak S)\otimes_{\mathfrak S} W(\mbb C_p^\flat) \right)^{\phi = 1}.$$
Moreover, by flatness of $W(\mbb C_p^\flat)$ over $\mathfrak S$ the natural map
$$\left((\prolim R\Gamma_\Prism(\mathbb P^n /\mathfrak S))\otimes_{\mathfrak S} W(\mbb C_p^\flat)\right)^{\phi=1} \xymatrix{\ar[r] & } \prolim \left(R\Gamma_\Prism(\mathbb P^n /\mathfrak S)\otimes_{\mathfrak S} W(\mbb C_p^\flat)\right)^{\phi=1}$$
is an equivalence. So we conclude by the previous lemma.
\end{proof}
\end{cor}

The map $\mathcal{O}_{\mathbb P^\infty}(1)^*\colon R\Gamma_\et(({B\mathbb G_m})_{\mbb C_p},\mathbb Z_p)\ra\prolim R\Gamma_\et(({\mathbb P^n})_{\mbb C_p},\mathbb Z_p)$ is an equivalence by the $\mbb A^1$-homotopy invariance of the (algebraic) \'etale cohomology over a field of $\mr{char} \ \! 0$:
\begin{lem}\label{BG_m_limalg}
The map $\mathcal{O}_{\mathbb P^\infty}(1)^*\colon R\Gamma_\et(({B\mathbb G_m})_{\mbb C_p},\mathbb Z_p)\ra\prolim R\Gamma_\et(({\mathbb P^n})_{\mbb C_p},\mathbb Z_p)$ is an equivalence.

\begin{proof}
First note that for each $n$ we have an open embedding $j_n\colon \mathbb P^{n}\simeq [(\mathbb A^{n+1}\setminus 0)/\mathbb G_m]\ra [\mathbb A^{n+1}/\mathbb G_m]$ and a projection $p_n\colon[\mathbb A^{n+1}/\mathbb G_m]\ra [\mr{pt}/\mathbb G_m]=B\mathbb G_m$. The map $\mathcal O_{\mathbb P^n}(1)\colon\mathbb P^n\ra \mathbb B\mathbb G_m$ factors as $p_n\circ j_n$. We also have a section $s_n\colon B\mathbb G_m \simeq [\{0\}/\mathbb G_m] \hookrightarrow [\mathbb A^{n+1}/\mathbb G_m]$ of $p_n$. Since (algebraic) \'etale $\mathbb Z_p$-cohomology over $C$ are $\mathbb A^1$-homotopy invariant, the maps 
$$
p_n^*\colon R\Gamma_\et(({B\mathbb G_m})_{\mbb C_p},\mathbb Z_p)\leftrightarrows R\Gamma_\et([\mathbb A^n/\mathbb G_m]_{\mbb C_p},\mathbb Z_p)\ :s_n^*
$$
are inverse equivalences. In particular, if we choose a sequence of linear embeddings (which are then automatically $\mathbb G_m$-equivariant) $\mathbb A^1\overset{i_0}{\hookrightarrow}\mathbb A^2\overset{i_1}{\hookrightarrow}\mathbb A^3\overset{i_2}{\hookrightarrow}\ldots $ the natural map $p_\infty^*\colon R\Gamma_\et(({B\mathbb G_m})_{\mbb C_p},\mathbb Z_p)\ra \prolim R\Gamma_\et([\mathbb A^n/\mathbb G_m]_{\mbb C_p},\mathbb Z_p)$ is an equivalence. So it is enough to prove,for each $n$, that the fiber of the map $j_n^*\colon R\Gamma_\et([\mathbb A^{n+1}/\mathbb G_m]_{\mbb C_p},\mathbb Z_p)\to R\Gamma_\et(\mathbb P^n_{\mbb C_p},\mathbb Z_p)$ induced by the open embedding $j_n\colon \mathbb P^n \simeq (\mathbb A^n\setminus\{0\})/\mathbb G_m \to [\mathbb A^n/\mathbb G_m]$ is $2n+1$-coconnected.

To see this we will use \v Cech objects for covers $u_1\colon \mathbb A^{n+1} \to [\mathbb A^{n+1}/\mathbb G_m]$ and $u_2\colon \mathbb A^{n+1}\setminus \{0\} \to \mathbb P^n$. The open embedding $j_n\colon \mathbb A^{n+1}\setminus \{0\} \to \mathbb A^{n+1}$ is $\mathbb G_m$-equivariant, inducing exactly $j_n\colon\mathbb P^n\to [\mathbb A^{n+1}/\mathbb G_m]$ on the quotients, and the corresponding \v Cech objects naturally map to each other:
$$
\xymatrix{
\ldots \ar[d] \ar[r] & \ldots \ar[d]\\
\mathbb G_m^k\times\mathbb A^{n+1}\setminus \{0\} \ar[d]\ar[r]^-{j_{n,k}}& \mathbb G_m^k\times\mathbb A^{n+1}\ar[d]\\
\ldots \ar[d] \ar[r] & \ldots \ar[d]\\
\mathbb G_m\times\mathbb A^{n+1}\setminus \{0\} \ar[d]\ar[r]^-{j_{n,1}}& \mathbb G_m\times\mathbb A^{n+1}\ar[d]\\
\mathbb A^{n+1}\setminus \{0\} \ar[r]^-{j_{n}=j_{n,0}}& \mathbb A^{n+1}.
}
$$

For each $j_{n,k}$ we have the map $\underline{\mathbb Z}_p \to (j_{n,k})_*j_{n,k}^*\underline{\mathbb Z}_p \simeq (j_{n,k})_*\mathbb Z_p$ of (pro)-\'etale sheaves on $(\mathbb G_m^k\times\mathbb A^{n+1})_{\mbb C_p}$ and the fiber is identified with $(i_{j,n})_*i_{j,n}^!\underline{\mathbb Z}_p$ (where $i_{j,n}$ is the embedding $\mathbb G_m^k\times \{0\}\hookrightarrow \mathbb G_m^k\times \mathbb A^{n+1}$) which is equivalent to $(i_{j,n})_*\underline{\mathbb Z}_p[2n+1]$ by the local Poincare duality. It follows then that the corresponding map on global sections $j_{n,k}^*\colon R\Gamma_\et((\mathbb G_m^k\times \mathbb A^{n+1})_{\mbb C_p},\mathbb Z_p)\ra R\Gamma_\et((\mathbb G_m^k\times \mathbb A^{n+1}\setminus \{0\})_{\mbb C_p},\mathbb Z_p)$ has $2n+1$-coconnected fiber. Finally, the fiber of the map $j_n^*\colon R\Gamma_\et([\mathbb A^{n+1}/\mathbb G_m]_{\mbb C_p},\mathbb Z_p)\ra R\Gamma_\et(\mathbb P^n_{\mbb C_p},\mathbb Z_p)$ is given by the totalization of fibers of $j_{n,k}^*\colon R\Gamma_\et((\mathbb G_m^k\times \mathbb A^{n+1})_{\mbb C_p},\mathbb Z_p)\ra R\Gamma_\et((\mathbb G_m^k\times \mathbb A^{n+1}\setminus \{0\})_{\mbb C_p},\mathbb Z_p)$, each of them is $2n+1$-coconnected, so the totalization is as well.
\end{proof}
\end{lem}

We finish this subsection by showing that the natural map from algebraic to adic \'etale $\mathbb Z_p$-cohomology of $B\mathbb G_m$ is an equivalence:
\begin{prop}\label{two_etale_coh_comp_BG_m}
The map $\Upsilon_{B\mathbb G_m}\colon R\Gamma_\et({(B\mathbb G_m)}_{\mbb C_p}, \mathbb Z_p)\ra R\Gamma_\et({(\widehat{B\mathbb G_m})}_{\mbb C_p}, \mathbb Z_p)$ from \Cref{constr:alg_to_adic_et_comp} is an equivalence.

\begin{proof}
By \Cref{BG_m_lim_et} and \Cref{BG_m_limalg} we have natural equivalences
$$R\Gamma_\et((\widehat{B\mathbb G_m})_{\mbb C_p},\mathbb Z_p)\simeq\prolim R\Gamma_{\et}((\widehat{\mathbb P^n})_{\mbb C_p},\mathbb Z_p) \qquad \text{and} \qquad R\Gamma_\et(({B\mathbb G_m})_{\mbb C_p},\mathbb Z_p)\simeq\prolim R\Gamma_\et(({\mathbb P^n})_{\mbb C_p},\mathbb Z_p)$$
under which the map $\Upsilon_{B\mathbb G_m}$ is identified with the limit $\prolim \Upsilon_{\mathbb P^n}$. But each $\Upsilon_{\mathbb P^n}$ is an equivalence (since $\mathbb P^n$ is proper), so we are done.
\end{proof}
\end{prop}

\begin{rem}\label{rem:torus_same_arg}
A similar argument can be applied to any split torus $T\simeq \mathbb G_m^k$ of higher dimension. More precisely we can approximate it by $(\mbb P^n)^k$ (mapping to $BT$ via $(\mc O_{\mbb P^n}(1)^*)^k$); indeed from K\"unneth formula and \Cref{rem: the pull-back to P^n is 2n+1 connected } it follows that the corresponding map $\RG_\Prism(B\mbb G_m^k/\mf S)\ra \RG_\Prism((\mbb P^n)^k/\mf S)$ is also at least $(2n+1)$-coconnected and so the map to the limit
$$
\RG_\Prism(B\mbb G_m^k/\mf S)\xymatrix{\ar[r]&} \lim_n \RG_\Prism((\mbb P^n)^k/\mf S)
$$ is an equivalence. Then arguing as in \Cref{BG_m_lim_et} we get that 
$$R\Gamma_\et((\widehat{B\mathbb G_m})_{\mbb C_p}^k,\mathbb Z_p)\xymatrix{\ar[r] &} \lim_n R\Gamma_\et((\widehat{\mathbb P^n})_{\mbb C_p}^k,\mathbb Z_p)$$
is an equivalence, while $$R\Gamma_\et(({B\mathbb G_m})_{\mbb C_p}^k,\mathbb Z_p)\xymatrix{\ar[r] &} \lim_n R\Gamma_\et(({\mathbb P^n})_{\mbb C_p}^k,\mathbb Z_p)$$ is an equivalence by the K\"unneth formula for the algebraic \'etale cohomology. It then follows as in \Cref{two_etale_coh_comp_BG_m} that $\Upsilon_T\colon R\Gamma_\et({(BT)}_{\mbb C_p}, \mathbb Z_p)\ra R\Gamma_\et({(\widehat{BT})}_{\mbb C_p}, \mathbb Z_p)$ is an equivalence.
\end{rem}   
\begin{rem}
Let $T$ be a (not necessarily split) torus over $\mc O_K$. From \Cref{rem: prismatic cohomology of BT over O_C} and the \'etale comparison, by computing the $\phi$ invariants on $\RG_{\Prism}(BT/\mf S)\otimes_{\mf S} W(\mbb C_p^\flat)\simeq \RG_{\Prism}(BT/\Ainf)\otimes_\Ainf W(\mbb C_p^\flat)$ one gets a natural isomorphism of graded rings with $G_K$-action:
$$
H^*_\et(BT_{\mbb C_p},\mathbb Z_p) \areq H^*_\et(\widehat{BT}_{\mbb C_p},\mathbb Z_p) \simeq \Sym^*_{\mbb Z_p}(X^*(T_{\mc O_{\mbb C_p}})\otimes_{\mbb Z} \mbb Z_p(-1)[-2]).
$$
Also note that $X^*(T)\otimes_{\mbb Z} \mbb Z_p(-1) \simeq T_p(T)^\vee$ as $G_K$-modules, where $T_p(T)^\vee$ is the $\mbb Z_p$-dual of the Tate module $T_p(T)$ of $T$. Thus we get 
$$
H^*_\et(BT_{\mbb C_p},\mathbb Z_p) \areq H^*_\et(\widehat{BT}_{\mbb C_p},\mathbb Z_p) \simeq \Sym^*_{\mbb Z_p}(T_p(T)^\vee[-2]).
$$
\end{rem}

\subsection{The case of $[\mathbb A^n/T]$}\label{subsect:prismatic_coh_of_An_quot_T}
Let $T\simeq \mathbb G_m^r$ be a split algebraic torus over $\mathcal O_K$ and let $\chi\colon T \ra \mathbb G_m$ be a character. Consider a $T$-representation $V$ which is free as an $\mathcal O_K$-module. $V$ is a sum of characters $V\simeq \chi_1\oplus\cdots\oplus\chi_n$ and we denote by $R(V)\coloneqq \mbb N\cdot \chi_1 + \ldots + \mbb N\cdot \chi_n\subset X^*(T)$ the cone spanned by $\chi_i$. For the rest of this subsection we make the following assumption:

\begin{assumption}\label{ass_on_T_action}
$V$ is conical (see \ref{defn: conical action}). Let $R(V)\subset X^*(T)$ be the set of $T$-weights appearing in $A$. Note that then for any non-zero $\chi\in R(V)$ we have $0\notin \chi+ R(V)$. 
\end{assumption}

Let $A=\Sym_{\mathcal O_K}(V^\vee)$ and $\mathbb A^n \coloneqq \Spec A$ be the affine space corresponding to $V$. Surely, this $\mathbb A^n$ has a natural $T$-action. Let's pick coordinates $x_1,\ldots, x_n\in V^*\subset A$ which agree with the direct sum decomposition $V\simeq\chi_1\oplus\cdots\oplus\chi_n$. The action of $T$ makes $A$ a $X^*(T)$-graded algebra which we denote by $A_{\bullet}$. 
 Note that the non-zero graded parts in $A$ are given by $-R(V)\subset X^*(T)$ . Recall that the category $\QCoh^+([\mathbb A^n/T])$ is identified with the bounded below derived category of graded $A_\bullet$-modules. Recall also that $\QCoh(BT)$ is identified with the derived category of $X^*(T)$-graded $\mathcal O_K$-modules (with respect to the trivial grading on $\mathcal O_K$). In this terms the functor of global sections $R\Gamma$ on both categories is given by taking the degree $0$ part of the corresponding complex with respect to $X^*(T)$-grading. 

Given a character $\chi\in X^*(T)$, for a $T$-module $M$ we will denote by $M(\chi)$ the corresponding shift of the grading, $M(\chi)_\mu \coloneqq M_{\mu-\chi}$. We consider the natural projection $i\colon [\mathbb A^n/T]\surj  BT$.
\begin{lem}\label{lem:Prisms A^n/T vs BT}
The pullback map
$$p^*\colon  R\Gamma_\Prism(BT / \mathfrak S)\tto  R\Gamma_\Prism([\mathbb A^n/T] / \mathfrak S) $$
is an equivalence.

\begin{proof}
By derived $E(u)$-completeness and exhaustiveness of Hodge-Tate filtration (\Cref{prop: HT_filtration_is_exaustive}) it is enough to prove that $i_0^*$ induces an equivalence on the  Hodge cohomology. Let $\mathfrak t\coloneqq\mathrm{Lie}(T)$ be the Lie algebra of $T$. As a complex of graded $A$-modules the cotangent complex $\mathbb L_{[\mathbb A^n/T]}$ (relative to $\mathcal O_K$) is equivalent to the $2$-term complex $\Omega^1_{\mathbb A^n} \xra{a^*} A_\bullet\otimes_{\mathcal O_K}\mathfrak{t}^\vee$, where $a^*$ is the coderivative of the $T$-action (see \Cref{ex: cotangent complex of X/G}). We have a $T$-equivariant isomorphism
$$\Omega^1_{\mathbb A^n} \simeq \bigoplus_{i=1}^n A_\bullet\cdot dx_i \simeq \bigoplus_{i=1}^n A_{\bullet}(\chi_i)$$
and so the cotangent complex is equivalent to a $2$-term complex $\left(\bigoplus_{i=1}^n A_{\bullet}(-\chi_i)\right)\to A_{\bullet}\otimes_{\mathcal O_K}\mf t^\vee$. Similarly $\wedge^i\ \! \mathbb L_{[\mathbb A^n/T]}$ is equivalent to the corresponding wedge power $\wedge^i_{A}\ \!\left((\bigoplus_{i=1}^n A_{\bullet}(-\chi_i))\xra{p} A_{\bullet}\right)$.

On the other hand, the cotangent complex $\mathbb L_{BT}$ is equivalent to $\mathfrak t^\vee[-1]$. The pullback functor $p^*$ is given by tensoring up with $A$, so $p^*\mathbb L_{BT}$ is equivalent to $(A_{\bullet}\otimes_{\mathcal O_K}\mf t^\vee)[-1]$. The map $R\Gamma(BT, \mathbb L_{BT})\to R\Gamma([\mathbb A^n/T], p^*\mathbb L_{BT})$ is then given by the restriction of the embedding
$$
\mf t^\vee[-1]\to A_\bullet\otimes_{\mathcal O_K}\mf t^\vee[-1]
$$
to the degree $0$ part. Since $A_0 \simeq \mathcal O_K$ we see that it is an equivalence.

Finally, the map $p^*\mathbb L_{BT}\ra \mathbb L_{[\mathbb A^n/T]}$ is given by sending $p^*\mathbb L_{BT}\simeq A_\bullet\otimes_{\mathcal O_K}\mf t^\vee[-1]$ to the similar term in $\mathbb L_{[\mathbb A^n/T]}\simeq (\bigoplus_{i=1}^n A_{\bullet}(-\chi_i))\to A_{\bullet}\otimes_{\mathcal O_K}\mf t^\vee$:
$$\xymatrix{
\cdots \ar[r] & 0\ar[d]^0\ar[r] &A_\bullet\otimes_{\mathcal O_K}\mf t^\vee\ar[d]^{\sim} \ar[r] & \cdots\\ 
\cdots \ar[r] & \left(\bigoplus_{i=1}^n A_{\bullet}(-\chi_i)\right)\ar[r] & A_\bullet\otimes_{\mathcal O_K}\mf t^\vee \ar[r] & \cdots .
}$$
From the assumptions on $V$ we know that the degree $0$ component of $\bigoplus_{i=1}^n A_{\bullet}(-\chi_i)$ is $0$. So
this map is an equivalence when restricted to the part of degree $0$. Thus $
p^*\colon R\Gamma(BT, \mathbb L_{BT}) \ra R\Gamma([\mathbb A^n/T], \mathbb L_{[\mathbb A^n/T]})$ is an equivalence. A similar argument works for $\wedge^i\ \mathbb L_{[\mathbb A^n/T]}$.
\end{proof}
\end{lem}

\begin{rem}\label{rem: non-split case A^n/T vs BT} In fact in \Cref{lem:Prisms A^n/T vs BT} we do not need to assume that $T$ is split over $\mc O_K$. Indeed by {\Cref{prop:Conrad}} any $T$ over $\mc O_K$ splits after an unramified extension $\mc O_K\ra \mc O_{K'}$. A uniformizer $\pi\in \mc O_K$ then gives a uniformizer of $\mc O_{K'}$ and we have a map of the corresponding Breuil-Kisin prisms $(W(k)[[u]],E(u))\ra (W(k')[[u]],E(u))$ which is faithfully flat. Since the map $p^*$ in \Cref{lem:Prisms A^n/T vs BT} is an equivalence for $BT_{\mc O_{K'}}$ and $[\mbb A^n/T]_{\mc O_{K'}}$, it is also true for $BT$ and $[\mbb A^n/T]$ by base change for prismatic cohomology and faithfully flat descent (applied to $W(k)[[u]]\ra W(k')[[u]]$).
\end{rem}

From the \'etale comparison (\Cref{etale_comparison}) we deduce:
\begin{cor}\label{cor:Rigid_A^1_vs_BG_m}
Under the assumptions from \ref{ass_on_T_action} on $V$ the pullback map
$$p^*\colon R\Gamma_\et(\widehat{BT}_{\mbb C_p}, \mathbb Z_p) \xymatrix{\ar[r] &} R\Gamma_\et(\widehat{[\mathbb A^n/T]}_{\mbb C_p}, \mathbb Z_p)$$
is an equivalence.
\end{cor}

By the $\mathbb A^1$-homotopy invariance of the algebraic \'etale cohomology over characteristic zero fields we also have:
\begin{lem}\label{Algebraic_A^1_vs_BG_m}
For any $T$-representation $V$ the map
$$p^*\colon R\Gamma_\et(BT_{\mbb C_p}, \mathbb Z_p) \xymatrix{\ar[r] &} R\Gamma_\et({[\mathbb A^n/T]}_{\mbb C_p}, \mathbb Z_p)$$
is an equivalence.
\end{lem}

\begin{prop}\label{prop:alg vs rigid etale coh for A^n/T}
Under the assumptions on $V$ from \ref{ass_on_T_action} the comparison map 
$$
\Upsilon_{[\mathbb A^n/T]}\colon R\Gamma_\et({[\mathbb A^n/T]}_{\mbb C_p},\mathbb Z_p)\xymatrix{\ar[r] &} R\Gamma_\et(\widehat{[\mathbb A^n/T]}_{\mbb C_p}, \mathbb Z_p)
$$
is an equivalence.

\begin{proof}
This follows from the existence of the commutative square
$$
\xymatrix{
R\Gamma_\et(BT_{\mbb C_p},\mathbb Z_p) \ar[d]^{p^*}_\sim \ar[rr]^-{\Upsilon_{BT}}_-\sim && R\Gamma_\et(\widehat{BT}_{\mbb C_p}, \mathbb Z_p)\ar[d]^{p^*}_\sim\\
R\Gamma_\et({[\mathbb A^n/T]}_{\mbb C_p},\mathbb Z_p) \ar[rr]^-{\Upsilon_{[\mathbb A^n/T]}} && R\Gamma_\et(\widehat{[\mathbb A^n/T]}_{\mbb C_p}, \mathbb Z_p),
}
$$
where we already know that all arrows except $\Upsilon_{[\mathbb A^n/T]}$ are equivalences by \Cref{two_etale_coh_comp_BG_m}, \Cref{cor:Rigid_A^1_vs_BG_m} and \Cref{Algebraic_A^1_vs_BG_m}.
\end{proof}
\end{prop}

\begin{rem}\label{rem: comparison for quotients by conical action}
	By a similar argument it is not hard to show that if one has a smooth $\mc O_K$-scheme $X$ with a locally linear $\mbb G_m$-action with $X=X^+$ (where $X^+$ is the attractor scheme, see e.g. \cite[Section 3.2.2]{KubrakPrikhodko_HdR} for the definition in relative setting) and $X^0\coloneqq X^{\mbb G_m}$ being proper, then $\Upsilon_{[X/\mbb G_m]}$ is an equivalence. Indeed, one has an affine map $p\colon [X^+/\mbb G_m] \surj [X^0/\mbb G_m]$ and $p_*$ induces an equivalence of categories $\QCoh([X^+/\mbb G_m])\xra{\sim}\Mod_{\mc A}(\QCoh([X^0/\mbb G_m]))$ with $\mc A\coloneqq p_*\mc O_{[X^+/\mbb G_m]}$; moreover, since the action of $\mbb G_m$ on $X^0$ is trivial $\QCoh([X^0/\mbb G_m])$ is identified with $\mbb Z$-graded objects in $\QCoh(X^0)$ and by the definition of $X^+$ as the attractor, the graded algebra $\mc A_\bullet$ is non-positively graded with $\mc A_0\simeq \mc O_{X^0}$. Considering the map $p^*\mbb L_{[X^0/\mbb G_m]}\ra \mbb L_{[X^+/\mbb G_m]}$ and taking pushforward one also sees that the cofiber of $p_*p^*\mbb L_{[X^0/\mbb G_m]}\ra p_*\mbb L_{[X^+/\mbb G_m]}$ has strictly negative degrees and thus (considering also $\lambda^i\mbb L$) one gets that $p^*$ induces an equivalence on Hodge cohomology. This then shows that $p^*$ is also an equivalence for the prismatic cohomology, as well as the \'etale cohomology of Raynaud generic fiber. Moreover, by \cite{bialynicki1973} the map of the generic fibers $p\colon {[X^+/\mbb G_m]}_{\mbb C_p} \ra {[X^0/\mbb G_m]}_{\mbb C_p}$ is in fact a fibration in affine spaces and thus $p^*$ is an equivalence for the \'etale cohomology of the algebraic generic fiber as well. This way we can reduce to $[X^0/\mbb G_m]$, where the statement follows from \Cref{thm:main_result} since $X^0$ is automatically smooth, and proper by the assumption.
\end{rem}

\subsection{The case of $BB$}\label{subsect:case_of_BB}
Let $G$ be a quasi-split reductive group over $\mathcal O_K$ and let $B\subset G$ be a Borel subgroup. In this subsection using the results of the previous two subsections we prove that the comparison map from \Cref{constr:alg_to_adic_et_comp} is an equivalence for the classifying stack $BB$.

Let $T\subset B$ be the maximal torus and consider the natural cover $\pi\colon BT\to BB$. Then the $n$-th term $\check{C}(\pi)_n$ of the corresponding \v{C}ech diagram is equivalent to 
$$
[T\backslash \underbrace{B\times_T B\times_T\cdots\times_T B}_n/T]\simeq [(\underbrace{U\times U\times \cdots \times U}_n)/T]
$$
with the $T$-action on $U\times\cdots\times U$ given by the simultaneous adjoint action on all factors $t\circ(u_1,\ldots,u_n)=(tu_1t^{-1}, \ldots, tu_nt^{-1})$ (see the proof of \Cref{prop:BB is coh proper}).

\begin{prop}\label{prisms_BT_BB}
The pullback map $\pi^*\colon R\Gamma_\Prism(BB /\mathfrak S) \tto R\Gamma_\Prism(BT /\mathfrak S)$ is an equivalence.

\begin{proof}
By the smooth descent (\Cref{rem:on_Gamma_prism}) we know that
$$R\Gamma_\Prism(BB /\mathfrak S)\simeq \Tot R\Gamma_\Prism([U^\bullet/T] / \mathfrak S).$$
On the other hand, as a scheme $U$ is $T$-equivariantly isomorphic to an affine space $\mf u\coloneqq \Spec(\Sym(\mathfrak u^\vee))$ (see \Cref{U as an example}). The representation $\mf u$ is conical and more generally this applies to each $U^n$ with the diagonal $T$-action. Hence by \Cref{lem:Prisms A^n/T vs BT} (and \Cref{rem: non-split case A^n/T vs BT}) all pullback maps $R\Gamma_\Prism(BT / \mathfrak S) \to R\Gamma_\Prism([U^\bullet/T] / \mathfrak S)$ are equivalences. We get that $p^*\colon R\Gamma_\Prism(BT / \mathfrak S) \ra R\Gamma_\Prism(BB /\mathfrak S)$ for the projection $p\colon BB\ra BT$ is an equivalence. But since the composition $p\circ \pi\colon BT\ra BT$ is the identity it also follows that $\pi^* \colon R\Gamma_\Prism(BB /\mathfrak S) \ra R\Gamma_\Prism(BT /\mathfrak S)$ is an equivalence.
\end{proof}
\end{prop}

\begin{rem}
By base change the analogues of \Cref{prisms_BT_BB} and \Cref{lem:Prisms A^n/T vs BT} also hold for $\RG_\Prism(-/\Ainf)$. 	
\end{rem}

From this we deduce that $\Upsilon_{BB}$ is an equivalence:
\begin{prop}\label{etale_alg_rig_BB}
Let $G$ be a split reductive group, $T\subset B \subset G$ be a maximal torus and a Borel subgroup correspondingly. Then we have a diagram of quasi-isomorphisms
$$\xymatrix{
R\Gamma_\et({BB}_{\mbb C_p}, \mathbb Z_p)\ar[d]_{\pi^*}^\sim\ar[rr]^-{\Upsilon_{BB}}_-\sim && R\Gamma_\et({\widehat{BB}}_{\mbb C_p},\mathbb Z_p)\ar[d]_{\pi^*}^\sim \\
R\Gamma_\et({BT}_{\mbb C_p}, \mathbb Z_p)\ar[rr]^-{\Upsilon_{BT}}_-\sim && R\Gamma_\et({\widehat{BT}}_{\mbb C_p},\mathbb Z_p).
}$$ 
In particular the comparison map $\Upsilon_{BB}\colon R\Gamma_\et({BB}_{\mbb C_p}, \mathbb Z_p)\to R\Gamma_\et({\widehat{BB}}_{\mbb C_p},\mathbb Z_p)$ is an equivalence.

\begin{proof}
The bottom horizontal arrow is an equivalence by \Cref{rem:torus_same_arg} and the right vertical arrow is an equivalence by \Cref{prisms_BT_BB} and the \'etale comparison (\Cref{etale_comparison}). So it is enough to prove that the left vertical map is an equivalence.

To see this, note that since $B \simeq T\rtimes U$, we have an isomorphism of schemes $B \simeq T\times U$. Moreover, as a scheme $U$ is isomorphic to an affine space. So from homotopy $\mathbb A^1$-invariance we get that the pullback map
$$(\pi \times\ldots\times \pi)^*\colon R\Gamma_\et(B_{\mbb C_p}\times\ldots\times B_{\mbb C_p}, \mathbb Z_p)\xymatrix{\ar[r] &} R\Gamma_\et(T_{\mbb C_p}\times \ldots\times T_{\mbb C_p}, \mathbb Z_p)$$
is an equivalence. The statement for $BT_{\mbb C_p}$ and $BB_{\mbb C_p}$ now follows from considering the simplicial schemes $(BT)_\bullet$ and $(BB)_\bullet$ and identifying $R\Gamma_\et({BT}_{\mbb C_p}, \mathbb Z_p)$ and $R\Gamma_\et({BB}_{\mbb C_p}, \mathbb Z_p)$ with the corresponding totalizations.  
\end{proof}
\end{prop}
\subsection{$BP$ versus $BL$}\label{ssec:BPvsBL }
In this small subsection we briefly discuss how to generalize \Cref{prisms_BT_BB} to the case of any parabolic using a result of Totaro from \cite{Totaro_deRhamBG}. 

\begin{prop}\label{prisms_BP_BL}
	Let $P\subset G$ be a parabolic in a reductive group over $\mc O_K$ and let $L\subset P$ be a Levi. Then the natural map $\pi\colon BL\ra BP$ induces an equivalence 
	$$
	\pi^*\colon R\Gamma_\Prism(BP /\mathfrak S) \tto R\Gamma_\Prism(BL /\mathfrak S).
	$$
\end{prop} 
\begin{proof}
	By derived $E$-completeness, considering reductions modulo $E$ and using the exhaustiveness of Hodge-Tate filtration (\Cref{prop: HT_filtration_is_exaustive}) we reduce to proving that $\pi^*$ induces an equivalence on Hodge cohomology over $\mc O_K$. By Hodge-properness both $R\Gamma_\Hdg(BP/\mc O_K)$ and $R\Gamma_\Hdg(BL /\mc O_K)$ are $p$- and so also $\pi$-complete and thus we can reduce further to showing that $\pi^*\colon R\Gamma_\Hdg(BP_k/k) \ra R\Gamma_\Hdg(BL_k/k)$ is an equivalence. This is now covered by \cite[Theorem 6.1]{Totaro_deRhamBG}.
\end{proof}
\begin{rem}\label{rem: prisms_BP_BL over O_C}
	By base change we also get an equivalence
	$$
	\pi^*\colon R\Gamma_\Prism(BP_{\mc O_{\mbb C_p}} /\Ainf) \xymatrix{\ar[r]^\sim&} R\Gamma_\Prism(BL_{\mc O_{\mbb C_p}}/\Ainf).
	$$
\end{rem}

\subsection{The general case}\label{sec: comparison for [X/P]}
In this subsection using the results of previous subsections and the machinery of \'etale sheaves on stacks developed in \Cref{sect:etale_sheaves} we will deduce that for a smooth proper scheme $X$ over $\mathcal O_K$ and a parabolic subgroup $P$ of a reductive group $G$ acting on $X$ the comparison map (\Cref{constr:alg_to_adic_et_comp}) between two versions of \'etale cohomology of $[X/P]$ is an equivalence. We first establish the following particular case:
\begin{prop}\label{X/B} Let $B\subset G$ be a Borel in a split reductive group over $\mc O_K$.
Let $\mstack X\ra BB$ be a smooth proper schematic map. Then the comparison map for $\mstack X$
$$\Upsilon_{\mstack X}\colon R\Gamma_\et(\mstack X_{\mbb C_p}, \mathbb Z_p) \xymatrix{\ar[r] &} R\Gamma_\et(\widehat{\mstack X}_{\mbb C_p}, \mathbb Z_p)$$
is an equivalence.
\begin{proof}

Since by construction both parts are $p$-adically derived complete it is enough to prove that the natural map
$$\Upsilon_{\mstack X, \mathbb F_p} \colon R\Gamma_\et(\mstack X_{\mbb C_p}, \mathbb F_p) \xymatrix{\ar[r] &} R\Gamma_\et(\widehat{\mstack X}_{\mbb C_p}, \mathbb F_p)$$
is an equivalence. Note that the algebraic stack $BB_{\mbb C_p}$ is \'etale 1-connected by \Cref{ex: BG is etale simply-connected if G is connected} and $\mbb F_p$-locally acyclic by \Cref{etale_alg_rig_BB}. Applying \Cref{prop: Lambda-acyclicity translates via proper smooth maps} to $\mstack X_{\mc O_{C}}\ra BB_{\mc O_{C}}$ we are done. 
\end{proof}
\end{prop} 

Now we can deduce the main result of this section:
\begin{thm}\label{thm:main_result}
Let $G$ be a reductive $\mathcal O_K$-group scheme and let $P$ be a parabolic subgroup of $G$. Let $X$ be a smooth proper $\mathcal O_K$-scheme equipped with an action of $P$. Then the comparison map is an equivalence
$$\Upsilon_{[X/P]}\colon R\Gamma_{\et}([X/P]_{\mbb C_p}, \mathbb Z_p) \xymatrix{\ar[r] &} R\Gamma_{\et}({\widehat{[X/P]}}_{\mbb C_p}, \mathbb Z_p)$$
is an equivalence.

\begin{proof} Note that the statement depends only on the base change of $X$, $P$ and $G$ to $\mc O_{\mbb C_p}$. In particular, replacing $\mc O_K$ by a suitable finite (and even unramified) extension $\mc O_K\ra \mc O_{K'}$ we can assume that $G$ is split.
Then let $B\subseteq P$ be a Borel subgroup. The natural cover $[X/B] \to [X/P]$ is smooth, proper and schematic. Consequently all elements of the corresponding \v Cech nerve are smooth, proper and schematic over $BB$. Since both parts of the comparison in question satisfy smooth descent, the result follows by the previous proposition.
\end{proof}
\end{thm}

By combining the results of this and the previous section we deduce the main application of this work:
\begin{thm}\label{thm:main_application}
Let $X$ be a smooth proper scheme over $\mathcal O_K$ equipped with an action of a parabolic subgroup $P$ in a connected reductive group scheme $G$ and choose an isomorphism $\iota\colon  \mbb C_p \xra{\sim} \mathbb C$. Then for any choice of an isomorphism $\iota\colon  \mbb C_p \xra{\sim} \mathbb C$ and all $n\in \mathbb Z_{\ge 0}$ we have an inequality
$$\dim_{\mathbb F_p} H_{P(\mathbb C)}^n(X(\mathbb C), \mathbb F_p) \le H_{\dR}^n([X/P]/k),$$
where $H^n_{P(\mathbb C)}(X(\mathbb C), \mathbb F_p)$ denotes the $P(\mathbb C)$-equivariant singular cohomology of $X(\mathbb C)$.

\begin{proof}
This follows by combining the isomorphism from the previous theorem with \Cref{inequality} and Artin's comparison (\Cref{Artins_comparison_stacks}, see also \Cref{ex: BG anlytification}).
\end{proof}
\end{thm}

\section{Applications}\label{sect_applications}
\subsection{De Rham and prismatic cohomology of $BG$}\label{sec: de Rham and prismatic cohomology of BG}
In this section we recall some of the results of \cite{Totaro_deRhamBG} about the de Rham cohomology of $BG$ over non-torsion primes and deduce the analogous statements for the prismatic cohomology. Let $G$ be a split reductive group over $\mbb Z$.
\begin{defn}\label{defn: torsion primes}
A prime $p$ is called torsion for $G$ if the $p$-torsion in $H^*_{\mr{sing}}(G(\mbb C),\mbb Z)$ is nontrivial. 	 Since $H^*_{\mr{sing}}(G(\mbb C),\mbb Z)$ is a finitely-generated abelian group there is only a finite set of torsion primes for a given $G$.
\end{defn}

\begin{rem}
People often use other definitions of torsion primes (see e.g. \cite[Section 10]{Totaro_deRhamBG} or \cite[Definition 2.43]{juteau2014parity}) that are more combinatorial in nature. It is true that all these definitions are equivalent (this can be deduced from \cite[Th\`eor\'emes A,B,C]{ArmandBorel}). We would like to note that unfortunately this equivalence is hugely based on a case-by-case consideration and relatively explicit computations.
\end{rem}

Borel also showed (see \cite[Th\`eor\'eme B]{ArmandBorel}) that $p$ is a torsion prime for $G$ if and only if the singular cohomology $H^*_{\text{sing}}(BG(\mathbb C),\mathbb Z)$ of the classifying space $BG(\mbb C)$ (in place of $G(\mbb C)$) has non-trivial $p$-torsion.
Quite remarkably, in \cite{Totaro_deRhamBG} Totaro was able to give another cohomological description of torsion primes, but now in terms of the classifying stack $BG$, more precisely the Hodge theory of the corresponding reduction $BG_{\mbb F_p}$. Let $\mf g$ be the affine space associated to the Lie algebra of $G$ and consider the quotient stack $[\mf g/G]$. We have an affine map $\pi\colon [\mf g/G] \ra BG$.
\begin{thm}[{\cite[Therorem 10.1]{Totaro_deRhamBG}}]\label{thm:description of torsion primes}
A prime $p$ is non-torsion for $G$ if and only if $H^{>0}(BG_{\mbb F_p},\pi_*\mathcal O_{{[\mf g/G]}_{\mbb F_p}}) \simeq 0$. This happens if and only if $H^{2n}_{\dR}(BG_{\mbb F_p}/\mbb F_p)\simeq H^{n,n}(BG_{\mbb F_p}/\mbb F_p)$ and $H^{2n+1}_{\dR}(BG_{\mbb F_p}/\mbb F_p)\simeq 0$ for any $n\ge 0$.
\end{thm}

\begin{ex} 
	$H^{*}(BG_{\mbb F_p},\pi_*\mathcal O_{{({\mf g}/G)}_{\mbb F_p}})$ can be equivalently described as the group cohomology $H^{*}(G_{\mbb F_p},\mc O(\mf g)_{\mbb F_p})$. By \cite[Proposition and proof of Theorem in 2.2]{donkin1988conjugating} and \cite[II.4.22]{Jantzen} if $p$ is \textit{good} the algebra $\mc O(\mf g)_{\mbb F_p}$ has a good filtration and so $H^{>0}(G_{\mbb F_p},\mc O(\mf g)_{\mbb F_p}) \simeq 0$. However, not every bad (= not good) prime is torsion (e.g. $p=2$ is bad for $\mr{Sp}_n$ but is non-torsion).
\end{ex}	
\begin{ex} Following \cite[Section 2.6]{juteau2014parity} given $G$ it is possible to completely describe the corresponding set of torsion primes. First, a reductive group has the same torsion primes as its derived subgroup. Then, if $G$ is semi-simple, torsion primes are those of its simply connected cover, together with the primes dividing
	the order of $\pi_1(G)$. The set of torsion primes of a semi-simple and
	simply connected group coincides with the union of torsion primes of its simple factors.
		Finally, if $G$ is simple and simply connected the torsion primes are completely described as follows:
		\begin{itemize}
			\item there are no torsion primes in type $A_n$ and $C_n$;
			\item $p=2$ if $G$ is of type $B_n(n\ge 3), D_n,G_2$;
			\item  $p=2,3$ if $G$ is of type $E_6$, $E_7$, or $F_4$;
			\item $p=2,3,5$ if $G$ is of type $E_8$.
		\end{itemize}  
\end{ex}	

\begin{ex}
	Any prime $p$ is torsion for some $G$; indeed, for example, $H^2(\mr{PGL}_p(\mbb C),\mbb Z)\simeq \mbb F_p$ and thus $p$ is a torsion prime for $\mr{PGL}_p$.
\end{ex}

Let $\mbb Z_{(p)}$ be the localization of $\mbb Z$ at $p$. \Cref{thm:description of torsion primes} was used by Totaro to compute the de Rham cohomology of $BG$ locally at any non-torsion prime:
\begin{thm}[{\cite[Theorem 10.2]{Totaro_deRhamBG}}]\label{thm:de rham cohomology for non-singular primes}
Let $G$ be a split connected reductive group over $\mathbb Z$ and let $p$ be a non-torsion prime for $G$. Then the de Rham cohomology $H^*_\dR(BG/\mathbb Z)_{(p)}\simeq H^*_\dR(BG_{\mbb Z_{(p)}}/\mathbb Z_{(p)})$, is a polynomial ring on generators of degrees equal to 2 times the fundamental degrees of $G$. This graded ring is (non-canonically) isomorphic to the singular cohomology $H^*_\sing(BG(\mathbb C),\mathbb Z_{(p)})$.
\end{thm}

\begin{ex}
We recall that the \emdef{fundamental degrees} of a split connected reductive group $G$ are the degrees of the generators of the ring of $G_{\mbb Q}$-invariants $(\Sym \mf g_{\mbb Q}^*)^{G_{\mbb Q}}$ which is itself a polynomial ring. Note that $(\Sym \mf g_{\mbb Q}^*)^{G_{\mbb Q}}$ only depends on the Lie algebra $\mf g$ and not the group $G$ itself. For the completeness of the exposition we give the list of fundamental degrees for simple groups:
\begin{center}
	\begin{tabular}{c l}
		$\mbb G_m$ & $1,$\\
		$A_n$ & $2,3,\ldots , n+1,$\\
		$B_n$ & $2,4,6,\ldots, 2n,$\\
		$C_n$ & $2,4,6,\ldots, 2n,$\\
		$D_n$ & $2,4,6,\ldots, 2n-2,n,$\\
		$G_2$ & $2,6,$\\
		$F_4$ & $2,6,8,12$,\\
		$E_6$ & $2,5,6,8,9,12$,\\
		$E_7$ & $2,6,8,10,12,14,18$,\\
		$E_8$ & $2,8,12,14,18,20,24,30$.\\
	\end{tabular}
\end{center}
For a general $G$ the set of fundamental degrees (with multiplicities) is given by the union of fundamental degrees of the simple components of $\mf g$.	
\end{ex}	

\begin{ex}
	If we discard the non-torsion condition on $p$ the structure of the singular cohomology ring $H^*_\sing(BG(\mathbb C),\mathbb Z_{(p)})$ is usually much more complicated. As an example of a particularly tricky computation one can see \cite{Vistoli} where the additive structure of $H^*_\sing(B\mr{PGL}_p(\mathbb C),\mathbb Z)$ is understood.
\end{ex}

\begin{rem}\label{rem:de Rham as deformation of \'etale}
Below (see \Cref{cor:prismatic cohomology of BG as a ring}) we deduce from \cite[Theorem 10.2]{Totaro_deRhamBG} an analogous result for the prismatic cohomology of $BG$ over $\mathbb Z_p$. This gives a more precise relation between  $H^*_\dR(BG/\mathbb Z_p)$ and $H^*_{\text{sing}}(BG(\mathbb C),\mathbb Z_p)$: namely after a choice of a uniformizer $\pi\in \mbb Z_p$ one can be seen as a deformation of the other.
\end{rem}

Let's now fix $\mc O_K$ and a Breuil-Kisin prism prism $(\mf S,E(u))$ with $\mf S/E(u)\simeq \mc O_K$.
We will use \Cref{thm:de rham cohomology for non-singular primes} to compute the prismatic cohomology $H^*_\Prism(BG/\mf S)$ under the assumption that $p$ is non-torsion. 
\begin{thm}
Let $G$ be a split reductive group over $\mathbb Z$ and let $p$ be a non-torsion prime for $G$. Then the prismatic cohomology $H^*_\Prism(BG/\mathfrak S)$ is a free graded polynomial algebra over $\mf S$ on generators of degrees equal to $2$ times the fundamental degrees of $G$.

\begin{proof}
Recall that the proof of \Cref{thm:de rham cohomology for non-singular primes} in \cite{Totaro_deRhamBG} is based on the analogous statement for the Hodge cohomology. Namely under the non-torsion assumption on $p$ one has $H^{i,j}(BG/\mbb Z_p)=0$ if $i\neq j$ and $H^*_\Hdg(BG/\mbb Z_p)\simeq \oplus_n H^{n,n}(BG/\mbb Z_p)$ is a polynomial algebra on generators of needed degrees. By base change \Cref{Hodge_and_deRham_basechange} we have $H^{i,j}(BG/\mc O_K)=0$ for $i\neq j$ and $H^*_\Hdg(BG/\mc O_K)$ is also a polynomial algebra; now looking at the Hodge-Tate spectral sequence for $H^*_{\Prism/I}(BG/\mf S)$ we see that it degenerates. Using that $\gr^*_\HT H^*_{\Prism/I}(BG/\mf S)\simeq H^*_\Hdg(BG/\mc O_K)$ and lifting the polynomial generators we get that the Hodge-Tate cohomology $H^*_{\Prism/I}(BG/\mf S)$ is a polynomial algebra over $\mc O_K$ on generators of degrees equal to $2$ times the fundamental degrees of $G$.

 By the universal coefficients formula we have short exact sequences:
$$\xymatrix @R=.5pc{
0 \ar[r] & H^{2i}_\Prism(BG/\mf S)\otimes_{\mf S} \!\mc O_K\ar[r] & H^{2i}_{\Prism/I}(BG/\mf S) \ar[r] &\Tor^1_{\mf S}(H^{2i+1}_\Prism(BG/\mf S), \mc O_K)  \ar[r] & 0, \\
0 \ar[r] & H^{2i+1}_\Prism(BG/\mf S)\otimes_{\mf S} \!\mc O_K  \ar[r] & H^{2i+1}_{\Prism/I}(BG/\mf S) \ar[r] &  \Tor^1_{\mf S}(H^{2i+2}_\Prism(BG/\mf S), \mc O_K) \ar[r] & 0.
}$$
From the above we know that $H^{2i+1}_{\Prism/I}(BG/\mf S) \simeq 0$ and so $H^{2i+1}_\Prism(BG/\mf S)\otimes_{\mf S}\! \mc O_K \simeq 0$. Since $H^{2i+1}_\Prism(BG/\mf S)$ is finitely generated over $\mf S$ (by \Cref{cor: prismatic cohomology are bounded below coherent}) it follows that $H^{2i+1}_\Prism(BG/\mf S)\simeq 0$ for $i$. Then tautologically $\Tor^1_{\mf S}(H^{2i+1}_\Prism(BG/\mf S), \mc O_K) \simeq 0$ and thus $H^{2i}_{\Prism/I}(BG/\mc O_K)\xra{\sim} H^{2i}_\Prism(BG/\mf S)\otimes_{\mf S} \mc O_K.$
Summarizing, we get an isomorphism of graded $\mc O_K$-algebras
$$
H^*_\Prism(BG/\mf S)\otimes_{\mf S} \! \mc O_K\xra{\sim} H^*_{\Prism/I}(BG/\mf S).
$$
Note that identifying $\Tor^1_{\mf S}(H^{2i}_\Prism(BG/\mf S), \mc O_K)$ with $E(u)$-torsion in $H^{2i}_\Prism(BG/\mf S)$ from the short exact sequences above we also get that 
$H^{2i}_\Prism(BG/\mf S)$ is $E(u)$-torsion free.

Let $x_{e_1},\ldots, x_{e_n}$ be the polynomial generators of $H^*_{\Prism/I}(BG/\mc O_K)$ of degree $2e_i$, where $e_i$ are the fundamental degrees and let $\tilde{x}_{e_1},\ldots, \tilde{x}_{e_n}\in H^*_\Prism(BG/\mf S)$ be any lifts of $x_{e_1},\ldots, x_{e_n}$ under the projection
$$
H^*_\Prism(BG/\mf S)\xymatrix{\ar@{->>}[r] &} H^*_\Prism(BG/\mf S)\otimes_{\mf S} \!\mc O_K \simeq H^*_{\Prism/I}(BG/\mc O_K).
$$
The choice of lifts gives a map from the polynomial ring
$$\tilde f\colon \mf{S}[y_{e_1},\ldots, y_{e_n}] \xymatrix{\ar[r] &} H^*_\Prism(BG/\mf S),$$
which by the above is an isomorphism when reduced modulo $E(u)$. The graded components on both sides are finitely generated $\mf S$-modules that are $E(u)$-torsion free, so (derived) Nakayama lemma can be applied and $\tilde f$ is an isomorphism.
\end{proof}
\end{thm}

In most cases one can also describe $H^*_\Prism(BG/\mf S)$ by looking at its image in $H^*_\Prism(BT/\mf S)$ under the natural map $q\colon BT\to BG$:
\begin{prop}\label{prop:char classes as W-invariants}
Let $G$ be a split reductive group and let $p$ be a non-torsion prime or $p\neq 2$ if $G$ has a simple factor of type $C_n$. Then $q^*\colon H^*_\Prism(BG/\mf S)\to H^*_\Prism(BT/\mf S)$ maps isomorphically onto the subring of (non-derived) $W$-invariants:
$$
H^*_\Prism(BG/\mf S) \xymatrix{\ar[r]^-\sim &} H^*_\Prism(BT/\mf S)^W.
$$

\begin{proof}
Since the cohomology groups $H^*_\Prism(BG/\mf S)$ are free as $\mf S$-modules we have an embedding 
$$
H^*_\Prism(BG/\mf S)\hookrightarrow H^*_\Prism(BG/\mf S)\otimes_{\mf S} W(\mathbb C_p^\flat) \simeq H^*_\et(BG_{{\mathbb C_p}},\mathbb Z_p)\otimes_{\mathbb Z_p} W(\mathbb C_p^\flat),
$$
where we use \Cref{thm:main_result} and \Cref{cor: adic_etale_comparison2} for the isomorphism on the right. Considering the same diagram for $T$ we obtain a commutative square 
$$
\xymatrix{H^*_\Prism(BG/\mf S)\ar[r]\ar[d]& H^*_\et(BG_{{\mathbb C_p}},\mathbb Z_p)\otimes_{\mathbb Z_p} W(\mathbb C_p^\flat) \ar[d]\\
H^*_\Prism(BT/\mf S)\ar[r]& H^*_\et(BT_{{\mathbb C_p}},\mathbb Z_p)\otimes_{\mathbb Z_p}W(\mathbb C_p^\flat),
}
$$
where the bottom horizontal map is $W$-equivariant. The $W$-action on $H^*_\et(BT_{{\mathbb C_p}},\mathbb Z_p)$ is induced by the adjoint action of elements of $G_{{\mathbb C_p}}$ lying in the connected component of the identity and so this action is trivial when restricted to the image of $q^*\colon H^*_\et(BG_{{\mathbb C_p}}, \mathbb Z_p)\to H^*_\et(BT_{{\mathbb C_p}}, \mathbb Z_p)$. It follows that the image of $q^*\colon H^*_\Prism(BG/\mf S)\ra H^*_\Prism(BT/\mf S)$ lands in $W$-invariants $H^*_\Prism(BT/\mf S)^W$. 

It remains to check that 
$$q^*\colon H^*_\Prism(BG/\mf S)\xymatrix{\ar[r]&} H^*_\Prism(BT/\mf S)^W$$
is an isomorphism. Using \Cref{prop: prismatic cohomology of BT} together with an isomorphism $\Sym_{\mf S} (X^*(T)\otimes_{\mbb Z}\mf S)\simeq \mc O(\mf t_{\mf S})$ we can identify the right hand side with the ring of polynomial functions $\mathcal O(\mf t_{\mf S})$ on $\mf t_{\mf S}$, moreover this identification is $W$-equivariant. So $H^*_\Prism(BT/\mf S)^W \simeq \mathcal O(\mf t_{\mf S})^W$. By \cite[Th\`eor\'eme]{Demazure_WeylInvariants} under a milder condition on $p$ the ring $\mathcal O(\mf t_{\mf S})^W$ is a polynomial ring in variables of degrees equal to the fundamental degrees of $G$, in particular each graded component is free as an $\mf S$-module. So by Nakayama it is enough to check that the map $q^*$ is an isomorphism modulo the ideal $(p,u)\subset \mf S$. By the crystalline comparison\footnote{Here $M^{(-1)}$ denotes the Frobenius untwist $M\otimes_{k,\phi^{-1}} k$.} $H^*_\Prism(BG/\mf S)\otimes_{\mf S}\mf S/(p,u)\simeq H^*_{\dR}(BG_k/k)^{(-1)}$ and similarly for $BT$. Moreover, by \cite[Th\`eor\'eme]{Demazure_WeylInvariants} we have $\mathcal O(\mf t_{\mf S})^W/(p,u)\simeq \mathcal O(\mf t_{k})^W$. And, finally, the fact that $q^*$ induces an isomorphism $H^*_{\dR}(BG_k/k)^{(-1)}\xra{\sim } \mathcal O(\mf t_{k})^W$ is checked during the proof of Theorem 10.2 in \cite{Totaro_deRhamBG}.
\end{proof}
\end{prop}
\begin{cor}\label{cor:prismatic cohomology of BG as a ring}
Under the same assumption on $G$ and $p$ as in \Cref{prop:char classes as W-invariants} we have 
$$
H^*_\Prism(BG/\mf S)\simeq \Sym_{\mf S}(\mf S\{-e_1\}[-2e_1]\oplus\cdots \oplus \mf S\{-e_n\}[-2e_n]) 
$$
as a graded ring in the category of Breuil-Kisin modules. Consequently,
$$
H^*_\Prism(BG/\Ainf)\simeq \Sym_{\Ainf}(\Ainf\{-e_1\}[-2e_1]\oplus\cdots \oplus \Ainf\{-e_n\}[-2e_n]) 
$$
as a graded ring in the category of Breuil-Kisin-Fargues modules.
\end{cor}
\begin{proof}
	The first part follows from \Cref{prop:char classes as W-invariants} and the description of $H^*_\Prism(BT/\mf S)$ in \Cref{prop: prismatic cohomology of BT} (see also \Cref{rem: prismatic cohomology of BT over O_C}). The second part follows from the first by base change.
\end{proof}


\subsection{$\Ainf$-valued characteristic classes}\label{sec:characteristic classes}
As another application in this section we begin to set up the theory of $\Ainf$-valued characteristic classes, which should (roughly speaking) interpolate between the \'etale, de Rham, and crystalline Chern classes. In fact the characteristic classes that we are going to define will live rather in the $\Ainf$-cohomology (that in the case of schemes was constructed in \cite{BMS1}) $\RG_{\Prism^{(1)}}(\mstack X/\Ainf)$ than the prismatic cohomology itself (with extra BK/Tate-twists needed if we want to take the Frobenius/Galois-action into account). This is more natural if we want to lift the characteristic classes to the corresponding terms in the Nygaard filtration (which we are planning to do in a sequel) since this is where the Nygaard filtration is actually defined.

\begin{construction}
Recall that by \Cref{thm:main_result} the natural map 
$$
\Upsilon_{B\GL_r}\colon H^{*}_\et(B\GL_{r, \mathbb C_p}, \mathbb Z_p)\xymatrix{\ar[r]^\sim&} H^{*}_\et(\widehat{B\GL}_{{r, \mathbb C_p}}, \mathbb Z_p)
$$
is an equivalence. The algebra $H^{*}_\et(B\GL_{r, \mathbb C_p}, \mathbb Z_p)\simeq  H^{*}_\sing(B\GL_{r}(\mathbb C), \mathbb Z_p)$ is a polynomial algebra in \textit{topological} Chern classes $c_i$ and is concentrated in even degrees. Any prime $p$ is non-torsion for $\GL_r$ and thus by \Cref{cor:prismatic cohomology of BG as a ring} the prismatic cohomology ring $H^{*}_{\Prism^{(1)}}(B\GL_r/\Ainf)$ is also concentrated in even degrees. In particular, since $H^{2i-1}_{\Prism^{(1)}}(B\GL_r/\Ainf)\simeq 0$ via \'etale comparison (\Cref{cor: Z_p etale cohomology of rigid fiber}) $H^{2i}_\et(\widehat{B\GL}_{{r, \mathbb C_p}}, \mathbb Z_p)$ maps isomorphically to $\phi_{{\Prism^{(1)}}}$-invariants in $H^{2i}_{\Prism^{(1)}}(B\GL_r/\Ainf)\otimes_\Ainf W(\mbb C_p^\flat)$. Recall that by \Cref{ex: twists} we have a $G_K$-equivariant map $\mbb Z_p(i)\ra \Ainf \{i\}$. Twisting the \'etale comparison by this map then gives a well defined $G_K$-equivariant map 
\begin{equation}\label{kappa}
\kappa_{B\GL_r, i}\colon H^{2i}_\et(B\GL_{r, \mathbb C_p}, \mathbb Z_p)(i) \tto H^{2i}_{\Prism^{(1)}}(B\GL_r/\Ainf)\{i\}\otimes_{\Ainf} W(\mathbb C_p^\flat)
\end{equation}
and $H^{2i}_\et(B\GL_{r, \mathbb C_p}, \mathbb Z_p)(i)$ is still identified with the $\phi_{\Prism^{(1)}}$-invariants on the right. Indeed, due to the formula in \Cref{ex: twists} $\Ainf\{i\}\otimes_\Ainf W(\mbb C_p^\flat)$ is isomorphic to $W(\mbb C_p^\flat)(i)$ as a $(\phi,G_K)$-module; this isomorphism is induced by multiplication by $\mu^i$ which becomes a unit in $W(\mbb C_p^\flat)$. It is easy to see that the composition of $\kappa_{X,i}$ with this isomorphism gives us back the \'etale comparison (composed with $\Upsilon_{B\GL_r}$), where we return to the discussion above.

 Finally note that by \Cref{cor:prismatic cohomology of BG as a ring} each individual cohomology $H^{2i}_{\Prism^{(1)}}(B\GL_r/\Ainf)$ is isomorphic to a direct sum of the $(-i)$-th Breuil-Kisin twists $\Ainf\{-i\}$ and that the twists on the right cancel each other. Since $\mbb Z_p\simeq \Ainf^{\phi=1}\simeq W(\mbb C_p^\flat)^{\phi=1}$ and trivially $\mbb Z_p$ generates $\Ainf$ over $\Ainf$, the $\phi_{\Prism^{(1)}}$-invariants of $H^{2i}_{\Prism^{(1)}}(B\GL_r/\Ainf)\{i\}\otimes_{\Ainf} W(\mathbb C_p^\flat)$ in fact lie in $H^{2i}_{\Prism^{(1)}}(B\GL_r/\Ainf)\{i\} \subset H^{2i}_{\Prism^{(1)}}(B\GL_r/\Ainf)\{i\}\otimes_{\Ainf} W(\mathbb C_p^\flat)$ and generate it as $\Ainf$-module. This provides a well-defined $G_K$-equivariant map 
$$\kappa_{B\GL_r,i}\colon H^{2i}_\et(B\GL_{r, \mathbb C_p}, \mathbb Z_p)(i) \tto H^{2i}_{\Prism^{(1)}}(B\GL_r/\Ainf)\{i\}$$ which extends to an isomorphism $H^{2i}_\et(B\GL_{r, \mathbb C_p}, \mathbb Z_p)(i)\otimes_{\mbb Z_p}\Ainf \xra{\sim}  H^{2i}_{\Prism^{(1)}}(B\GL_r/\Ainf)\{i\}$.

\end{construction}

\begin{rem}\label{rem: chern classes etale vs sing}
	Fix an isomorphism $\iota\colon \mbb C_p\xra{
	\sim} \mbb C$ and consider the induced identification  $H^{*}_\et(B\GL_{r, \mathbb C_p}, \mathbb Z_p)\simeq H^{*}_\et(B\GL_{r, \mathbb C}, \mathbb Z_p)\simeq  H^{*}_\sing(B\GL_{r}(\mathbb C), \mathbb Z_p)$. Any two isomorphisms $\iota_1,\iota_2$ differ by an automorphism of $\mbb C$; however we have a canonical identification $H^{*}_\et(B\GL_{r,{\mathbb C}}, \mathbb Z_p)\simeq  H^{*}_\et(B\GL_{r, \ol{\mathbb Q}}, \mathbb Z_p) $ and thus the action of $\Aut(\mbb C)$ on $H^{*}_\et(B\GL_{r,{\mathbb C}}, \mathbb Z_p)$ factors through the absolute Galois group $G_{\mbb Q}$ of $\mbb Q$. The cohomology group $H^{2i}_\et(B\GL_{r, \ol{\mathbb Q}}, \mathbb Z_p)$ is given by a direct sum of $(-i)$-th Tate twists; thus if we tensor it up with $\mbb Z_p(i)$ the $G_{\mbb Q}$-action becomes trivial and we get a canonical identification 
$$
H^{2i}_\et(B\GL_{r, \mathbb C_p}, \mathbb Z_p)(i)\simeq H^{2i}_\et(B\GL_{r, \ol{\mathbb Q}}, \mathbb Z_p)(i) \simeq H^{2i}_\sing(B\GL_{r}(\mathbb C), \mathbb Z_p),
$$
which in particular does not depend on $\iota$.
\end{rem}

Using this we construct some canonical (polynomial) generators for the ring
$$
\mr{CC}_\Ainf\coloneqq \bigoplus_{i=0}^\infty H^{2i}_{\Prism^{(1)}}(B\GL_r/\Ainf)\{i\}\xymatrix{\ar[r]^\sim &} \Sym_{\Ainf}(\Ainf[-2]\oplus \Ainf[-4]\cdots \oplus \Ainf[-2n]). 
$$of $\Ainf$-characteristic classes.
\begin{construction}\label{constr: Chern classes}\begin{enumerate}\item We define 
		$$
		c_i^\et \in H^{2i}_\et(B\GL_{r, \mathbb C_p}, \mathbb Z_p)(i)
		$$ as the image of the $i$-th Chern class $c_i\in H^{2i}_\sing(B\GL_{r}(\mathbb C), \mathbb Z_p)$ under the identification in the remark above. For a vector bundle $E$ on a ${\mathbb C_p}$-scheme (or, more generally, a stack) $X$ classified by a map $f\colon X \to B\GL_r$ we define the \emdef{\'etale Chern classes} $c_i^{\et}(E) \in H^{2i}_\et(X, \mbb Z_p)(i)$ as the pullback of $c_i^{\Ainf}$ along $f$.
\item We then define
$$c_i^\Ainf \coloneqq \kappa_{B\GL_r,i}(c_i^\et) \in H^{2i}_{\Prism^{(1)}}(B\GL_r/\Ainf)\{i\}.$$ Similarly, for a vector bundle $E$ on an $\mathcal O_{\mathbb C_p}$-stack $\mstack X$ classified by a map $f\colon \mstack X \to B\GL_r$ we define the \emdef{$\Ainf$-valued Chern classes} $c_i^{\Ainf}(E) \in H^{2i}_{\Prism^{(1)}}(\mstack X/\Ainf)\{i\}$ as the pullback of $c_i^{\Ainf}$ along $f$. 

\end{enumerate}
\end{construction}
%

Since $c_i$ are the generators for $H^{2i}_\sing(B\GL_{r}(\mathbb C), \mathbb Z_p)$ we have an isomorphism 
$$
\mr{CC}_\Ainf\simeq \Ainf[c_1^\Ainf,\ldots,c_n^\Ainf].
$$

\begin{rem}\label{rem: etale Chern compare with topological} Our definition of \'etale Chern classes $c_i^\et(E)$ agrees with a standard one (see \cite[Expose VII]{SGA5}). Indeed by \cite[Proposition 3.8.5]{SGA5} the $i$-th etale Chern class of a bundle $E$ over $X_{\mbb C}$ in the sense of \textit{loc.cit.} lies in $H^{2i}_\et(X_{\mbb C},\mu_{p^k}^{\otimes i})$ and is given by the image of $c_i\in H^{2i}_\sing(X(\mbb C),\mbb Z/p^k \mbb Z)$ under the \'etale comparison. One then returns to our setting of $\mbb Z_p$-coefficients by passing to the limit over $k$.
\end{rem}	

\begin{notation}\label{not: map kappa}
	Let $\mstack X$ be a stack over $\mc O_{\mbb C_p}$. We denote by $\kappa_{\mstack X,i}$ the natural composition
	$$
	\xymatrix{\kappa_{\mstack X,i}\colon H^{2i}_\et(\mstack X_{\mbb C_p},\mbb Z_p)(i)\ar[r]^-{\Upsilon_{\mstack X}}& H^{2i}_\et(\widehat{\mstack X}_{\mbb C_p},\mbb Z_p)(i)\ar[r]& H^{2i}_{\Prism^{(1)}}(\widehat{\mstack X}_{\mc O_{\mbb C_p}},\Ainf))\{i\} \otimes_\Ainf W(\mbb C_p^\flat), 
	}
	$$
	where the second map is obtained from the \'etale comparison by tensoring with $\mbb Z_p(i)\ra \Ainf\{i\}$ (discarding the $(\phi,G_K)$-structure this is given by division of the usual \'etale comparison map by $\mu^i$).
	If $\mstack X$ is Hodge-proper then by \Cref{lem: can descend to A_inf[1/mu]} this descends to a map 
	$$
	\kappa_{\mstack X,i}\colon H^{2i}_\et(\mstack X_{\mbb C_p},\mbb Z_p)(i) \tto H^{2i}_{\Prism^{(1)}}(\widehat{\mstack X}_{\mc O_{\mbb C_p}}/\Ainf))\{i\}[\tfrac{1}{\mu}]
	$$
	which we continue to denote the same way.
\end{notation}

\begin{rem}\label{warn:mu_torsion_free} Let $X$ be a smooth proper $\mathcal O_{\mathbb C_p}$-scheme (or more generally, a Hodge-proper stack) and let $E$ be a vector bundle on $X$. One can be tempted to define $c_i^\Ainf(E)\in H_{\Prism^{(1)}}^{2i}(X/\Ainf)\{i\}$ using the \'etale comparison for $X$ as we did for $B\GL_r$. However, the problem is that the \'etale comparison only gives a map to the $\mu$-localization, but not $H^*_\Prism(X/\Ainf)$ itself. Passing to the localization makes us lose some information about $c_i^\Ainf(E)$. Indeed, let $f\colon X \to B\GL_r$ be the map defining $E$, then the \'etale comparison (together with \Cref{lem: can descend to A_inf[1/mu]}) induces a commutative diagram 
$$\xymatrix{
H_{\Prism^{(1)}}^{2i}(B\GL_r/\Ainf)\{i\} \ar[r]\ar[d]^{f^*} & H_{\Prism^{(1)}}^{2i}(B\GL_r/\Ainf)\{i\}[\tfrac{1}{\mu}] \ar[d]^{f^*[\tfrac{1}{\mu}]} & \ar[l]_-{\kappa_{B\GL_r,i}}^-\sim H^{2i}_\et(B\GL_{r,\mathbb C_p}, \mathbb Z_p)(i)\otimes_{\mathbb Z_p} \Ainf[\tfrac{1}{\mu}] \ar[d]^{f_{\et}^*}\\
H_{\Prism^{(1)}}^{2i}(X/\Ainf)\{i\} \ar[r] & H_{\Prism^{(1)}}^{2i}(X/\Ainf)\{i\}[\tfrac{1}{\mu}] & \ar[l]_(.55){\kappa_{X,i}}^(.55)\sim H^{2i}_\et(X_{\mathbb C_p}, \mathbb Z_p)(i)\otimes_{\mathbb Z_p} \Ainf[\tfrac{1}{\mu}].
}$$
It follows that the image of $c_i^\et(E)$ in $H_{\Prism^{(1)}}^{2i}(X/\Ainf)\{i\}[\frac{1}{\mu}]$ under $\kappa_{X,i}$ agrees with the image of $c_i^\Ainf(E)$. But if $H_{\Prism^{(1)}}^{2i}(X/\Ainf)$ has $\mu$-torsion, $H_{\Prism^{(1)}}^{2i}(X/\Ainf)\{i\}$ does not embed into the localization $H_{\Prism^{(1)}}^{2i}(X/\Ainf)\{i\}[\frac{1}{\mu}]]$, thus $\kappa_{X,i}(c_i^\et(E))$ does not define any element of $H_{\Prism^{(1)}}^{2i}(X/\Ainf)\{i\}$. In particular, if $c_i^\Ainf(E)\in H_{\Prism^{(1)}}^{2i}(X/\Ainf)\{i\}_{\mu\mr{-tors}}$ then  $\kappa_{X,i}(c_i^\et(E))$ is necessarily 0 even though $c_i^\Ainf(E)$ might not be.
\end{rem}

We would now like to show that the projective (or rather relative flag) bundle formula for the Chern classes still holds in our context. However to formulate it properly we need to take the Breuil-Kisin twists into account. The following is a more or less standard way to do that:
\begin{notation}\label{not: twisted Ainf cohomology}
	For an Artin stack $\mstack X$ over $\mc O_{\mbb C_p}$ we denote
	$$
	H^*_{\Prism,\mr{tw}}(\mstack X/\Ainf)\coloneqq \bigoplus_{i=0}^\infty H^i_{\Prism^{(1)}}(\mstack X/\Ainf)\{\floor{i/2}\}.
	$$
	Note that passage from $H^*_{\Prism}(\mstack X/\Ainf)$ to $H^*_{\Prism,\mr{tw}}(\mstack X/\Ainf)$ only changes the $(\phi,G_K)$-module structures. Note in particular that $H^*_{\Prism,\mr{tw}}(\mstack X/\Ainf)$ \textit{no more} has a natural algebra structure in Breuil-Kisin-Fargues modules, but its even part $H^{\mr{even}}_{\Prism,\mr{tw}}(\mstack X/\Ainf)\subset H^*_{\Prism,\mr{tw}}(\mstack X/\Ainf)$ still does. The whole $H^*_{\Prism,\mr{tw}}(\mstack X/\Ainf)$ has a natural module structure over $H^{\mr{even}}_{\Prism,\mr{tw}}(\mstack X/\Ainf)$. In the case $\mstack X=B\GL_r$ we have $H^*_{\Prism,\mr{tw}}(B\GL_r/\Ainf)\simeq H^{\mr{even}}_{\Prism,\mr{tw}}(B\GL_r/\Ainf)\simeq \mr{CC}_\Ainf$;
	in particular $H^*_{\Prism,\mr{tw}}(B\GL_r/\Ainf)$ is an algebra. 
\end{notation}	

\begin{defn}\label{rem:twisted Ainf-cohomology}
	For a vector bundle $E$ on an Artin stack $\mstack X$ the total $\Ainf$-Chern class $c^\Ainf(E)\in 
	H^{\mr{even}}_{\Prism,\mr{tw}}(\mstack X/\Ainf)
	$ is defined as $$c^\Ainf(E)\coloneqq \sum_{i=0}^{\rk(E)} c_i^\Ainf(E),$$
	where $c_i^\Ainf(E)\in H^{2i}_{\Prism,\mr{tw}}(\mstack X/\Ainf)$ is the pull-back of $c_i^\Ainf\in H^{2i}_{\Prism,\mr{tw}}(B\GL_{\rk(E)}/\Ainf)$ under the map classifying $E$.
\end{defn}


\begin{rem}\label{rem: Chern classes as symmetric polynomials}
	Let $B_r\subset \GL_r$ be a Borel and let $T_r\subset B_r$ be a maximal torus. Let's fix a trivialization $T_r\simeq \mbb G_m^r$ and let $\mbb G_{m,i}$ be the $i$-th component. From \Cref{prisms_BT_BB} it follows that $\RG_\Prism(BB_r/\Ainf)\simeq \RG_{\Prism}(BT_r/\Ainf)$ and by the K\"unneth formula \Cref{Kunneth_for_Prism_stacks} we get $H^*_{\Prism,\mr{tw}}(BT_r/\Ainf)\simeq \otimes_{i=0}^r H^*_{\Prism,\mr{tw}}(B\mbb G_{m,i}/\Ainf)$. Putting $t_i\coloneqq (c_1^\Ainf)_i\in H^2_{\Prism,\mr{tw}}(B\mbb G_{m,i}/\Ainf)$ we get an isomorphism $H^*_{\Prism,\mr{tw}}(BT_r/\Ainf)\simeq \Ainf[t_1,\ldots, t_r]$. Also note that by \Cref{prop:char classes as W-invariants} we get an identification $H^*_{\Prism,\mr{tw}}(B\GL_r/\Ainf) \xra{\sim} H^*_{\Prism,\mr{tw}}(BT_r/\Ainf)^{\Sigma_r}$ where the symmetric group $\Sigma_r$ acts by permuting the components $\{\mbb G_{m,i}\}_i$. In fact we have a commutative diagram
$$
\xymatrix{H^*_{\Prism,\mr{tw}}(B\GL_r/\Ainf)\ar[r]^\sim& H^*_{\Prism,\mr{tw}}(BT_r/\Ainf)^{\Sigma_r} & \ar[l]_(.53)\sim \big(\Ainf[(c_1^\Ainf)_1,\ldots, (c_1^\Ainf)_r]\big)^{\Sigma_r}\\
H^*_{\sing}(B\GL_r(\mbb C),\mbb Z_p)\ar[r]^\sim \ar[u]& H^*_{\sing}(BT_r(\mbb C),\mbb Z_p)^{\Sigma_r} \ar[u] & \ar[l]_(.48)\sim\big(\mbb Z_p[(c_1)_1,\ldots, (c_1)_r]\big)^{\Sigma_r}\ar[u]} 
$$
and since by the standard computation in topology $c_i\in H^{2i}_{\sing}(B\GL_r(\mbb C),\mbb Z_p)$ is equal to $\sigma_i((c_1)_1,\ldots, (c_1)_r)\in \mbb Z_p[(c_1)_1,\ldots, (c_1)_r]^{\Sigma_r}$ where $\sigma_i(x_1,\ldots,x_r) \in \mathbb Z[x_1, \ldots, x_r]^{\Sigma_r}$ is the $i$-th elementary symmetric polynomial, the same formula holds true in the $\Ainf$-setting, namely
$$
c_i^\Ainf=\sigma_i(t_1,\ldots, t_n)\in \Ainf[t_1,\ldots, t_r]^{\Sigma_r}\simeq H^*_{\Prism,\mr{tw}}(B\GL_r/\Ainf).
$$
\end{rem}
\begin{rem}\label{rem: partial flags}
	Let $\ul r=(r_1,r_2, \ldots, r_k), \  r_1\ge r_2 \ge \ldots r_k>0, \  r_1+\ldots +r_k=r$ be a partition of $r$ and let $P_{\ul r}\subset \GL_r$ be the corresponding parabolic (given up to isomorphism by the block-upper-triangular matrices with blocks of the size $r_1,r_2, \ldots, r_k$). In particular $P_r=\GL_r$ and $P_{(1,\ldots,1)}$ is a Borel. By \Cref{rem: prisms_BP_BL over O_C} the map $B\GL_{\ul r} \ra BP_{\ul r}$ (given by the natural embedding $\GL_{\ul r}\coloneqq \GL_{r_1}\times\ldots \times \GL_{r_k} \ra P_{\ul r}$) induces on equivalence on $\RG_\Prism(-/\Ainf)$. Then taking pull-back via the composition $ BT_r\ra B\GL_{\ul r} \ra BP_{\ul r}$ we can make an identification 
	$$
	H^*_{\Prism,\mr{tw}}(BP_{\ul r}/\Ainf) \xymatrix{\ar[r]^\sim&} \Ainf[t_1,\cdots,t_r]^{\Sigma_{r_1}\times\ldots\times \Sigma_{r_k}}; 
	$$
    moreover the pull-back map $q^*\colon H^*_{\Prism,\mr{tw}}(B\GL_r/\Ainf) \ra H^*_{\Prism,\mr{tw}}(BP_{\ul r}/\Ainf)$ for $q\colon  BP_{\ul r}\ra B\GL_r$ in this terms is identified with the natural embedding
    $$
    \Ainf[t_1,\cdots,t_r]^{\Sigma_{r}}\xymatrix{\ar[r]&} \Ainf[t_1,\cdots,t_r]^{\Sigma_{r_1}\times\ldots\times \Sigma_{r_k}}.
    $$
\end{rem}
\begin{prop}[Relative flag bundle formula]\label{splitting_principle}
Let $E$ be a rank $r$-vector bundle on a smooth $\mathcal O_{\mbb C_p}$-stack $\mstack X$ and let $p\colon \mathrm{Fl}_E\ra \mstack X$ be the full flag variety of $E$ relative to $\mstack X$. Note that the pullback of $E$ to $\Fl_E$ admits the tautological filtration (=flag) with associated graded $\bigoplus_{i=1}^r \mathcal L_i$, where $\mathcal L_1, \ldots, \mathcal L_r$ are some line bundles. Then
we have an isomorphism
$$H^*_{\Prism,\mr{tw}}(\Fl_E/\Ainf) \simeq H^*_{\Prism,\mr{tw}}(\mstack X/\Ainf)[\xi_1, \ldots, \xi_k]/\big(\sigma_i(\xi_1, \ldots, \xi_k) = c_i^\Ainf(E)\big)$$
as modules over $H^{\mr{even}}_{\Prism,\mr{tw}}(\mstack X/\Ainf)$ in $G_K$-equivariant Breuil-Kisin-Fargues modules, which sends $\xi_i$ to $c_1^\Ainf(\mathcal L_i)$, and $\sigma_i \in \mathbb Z[x_1, x_2, \ldots, x_k]$ is the $i$-th elementary symmetric polynomial.

\begin{proof}
There is a natural equivalence $\Fl_E \simeq \mstack X\times_{B\GL_r} BB_r$, where $B_r$ is the standard Borel subgroup of $\GL_r$. It follows that there is a natural map of $E_\infty$-algebras
\begin{align}\label{eq:cohomology_of_flags}
R\Gamma_\Prism(\mstack X/\Ainf) \otimes_{R\Gamma_\Prism(B\GL_r/\Ainf)} R\Gamma_\Prism(BB_r/\Ainf) \xymatrix{\ar[r] &} R\Gamma_\Prism(\Fl_E/\Ainf).
\end{align}
We claim that \eqref{eq:cohomology_of_flags} is an equivalence. Since the homomorphism $$\Ainf[t_1, \ldots, t_k]^{\Sigma_k} \simeq H^*_\Prism(B\GL_k/\Ainf) \xymatrix{\ar@{^(->}[r]&} H^*_\Prism(BB_k/\Ainf) \simeq \Ainf[t_1, \ldots , t_k]$$
is finite free, $R\Gamma_\Prism(BB_r/\Ainf)$ is a finite free $R\Gamma_\Prism(B\GL_r/\Ainf)$-module, so the functor
$$M \xymatrix{\ar@{|->}[r] &} M \otimes_{R\Gamma_\Prism(B\GL_r/\Ainf)} R\Gamma_\Prism(BB_r/\Ainf)$$
preserves limits, and hence the left hand side of \eqref{eq:cohomology_of_flags} satisfies smooth descent with respect to $\mstack X$. So does the right hand side.
Thus to show that the map \eqref{eq:cohomology_of_flags} is an equivalence we can assume that $E$ is a trivial vector bundle, in which case $\Fl_E \simeq \mstack X\times \big(\GL_r/B_r\big)$ and the result follows by the K\"unneth formula (see \Cref{Kunneth_for_Prism_stacks}).

Next, since $H^*_\Prism(BB_r/\Ainf)$ is flat over $H^*_\Prism(B\GL_r/\Ainf)$, there is a natural isomorphism
$$H^*\left(R\Gamma_\Prism(\mstack X/\Ainf) \otimes_{R\Gamma_\Prism(B\GL_r/\Ainf)} R\Gamma_\Prism(BB_r/\Ainf)\right) \simeq H^*_\Prism(\mstack X/\Ainf) \otimes_{H^*_\Prism(B\GL_r/\Ainf)} H^*_\Prism(BB_r/\Ainf)$$ 
which then also gives an isomorphism between the twisted versions:
$$
H^*_{\Prism,\mr{tw}}(\Fl_E/\Ainf)\simeq H^*_{\Prism,\mr{tw}}(\mstack X/\Ainf) \otimes_{H^*_{\Prism,\mr{tw}}(B\GL_r/\Ainf)} H^*_{\Prism,\mr{tw}}(BB_r/\Ainf)
$$
(here we used the algebra structure on $H^*_{\Prism,\mr{tw}}(B\GL_r/\Ainf)$ (see \Cref{not: twisted Ainf cohomology}) and the fact that  $H^*_{\Prism,\mr{tw}}(BB_r/\Ainf)$ is concentrated in even degrees).
Finally, since the pull-back of $c_i^\Ainf\in H^{2i}_{\Prism,\mr{tw}}(B\GL_r/\Ainf)$ in $H^{2i}_{\Prism,\mr{tw}}(BB_r/\Ainf)\simeq H^{2i}_{\Prism,\mr{tw}}(BT_r/\Ainf)$ is expressed as elementary symmetric polynomial in first $\Ainf$-Chern classes, the ring on the right hand side is exactly given by $H^*_{\Prism,\mr{tw}}(X/\Ainf)[\xi_1, \ldots, \xi_r]/\big(\sigma_i(\xi_1, \ldots, \xi_k) = c_i^\Ainf(E)\big)$.
\end{proof}
\end{prop}

\begin{rem}\label{rem: projective bundle formula}
	Replacing $B_r$ in the proof by a parabolic $P_{\ul r}$ one also gets similar formulas for the cohomology of the corresponding relative partial flag varieties. In particular taking $\ul r=(1,r-1)$ one gets the projective bundle formula 
	$$H^*_{\Prism,\mr{tw}}(\mbb P_{\mstack X}(E)/\Ainf) \simeq H^*_{\Prism,\mr{tw}}(\mstack X/\Ainf)[\xi]/\big(c(\xi)\big),$$
	where $c(\xi)\coloneqq \sum_{i=0}^{\rk(E)} c_{i}^\Ainf(E) \cdot \xi^{\rk(E)-i}$ and $\xi$ maps to $c_1^\Ainf(\mc O_{\mbb P_{\mstack X}(E)}(1))$.
\end{rem}
\begin{rem}\label{rem: trivialized projective bundle formula}
	If one chooses a compatible system $\epsilon\coloneqq (1,\zeta_p, \zeta_{p^2},\ldots)$ of $p^n$-th roots of unity, thus trivializing the Tate twist $\mbb Z_p(i)$ and the Breuil-Kisin twist $\Ainf\{1\}$ as an $\Ainf$-module (by giving a particular choice of an element $\mu=[\epsilon]-1 \in \Ainf$), one can define non-twisted Chern classes $c_i^\Ainf\in H^{2i}_\Prism(B\GL_r/\Ainf)$. Then one gets a similar isomorphism for the untwisted prismatic cohomology: $$H^*_{\Prism}(\Fl_E/\Ainf) \simeq H^*_{\Prism}(X/\Ainf)[\xi_1, \ldots, \xi_k]/\big(\sigma_i(\xi_1, \ldots, \xi_k) = c_i^\Ainf(E)\big)$$
\end{rem}

 We deduce the following Grothendieck-style axiomatic characterization of our $\Ainf$-valued Chern classes:
\begin{cor}\label{cor: prism_Cherns_Grothendieck}
The $\Ainf$-Chern classes $c_i^{\Ainf}$ are the unique $H^{2*}_{\Prism,\mr{tw}}(- /\Ainf)$-valued characteristic classes satisfying the following set of axioms:
\begin{enumerate}
\item (Naturality) $\Ainf$-Chern classes commute with pullbacks.

\item (Multiplicativity) For any short exact sequence of vector bundles $0 \to E_1 \to E_2 \to  E_3 \to 0$ we have
$$c^\Ainf(E_2) = c^\Ainf(E_1) \cdot c^\Ainf(E_3).$$

\item (Normalization) For the line bundle $\mathcal O_{\mathbb P^1}(1)$ we have $c_1^\Ainf(\mathcal O_{\mathbb P^1}(1))=\kappa(c_1^\et(\mathcal O_{\mathbb P^1}(1))$, where $$\kappa=\kappa_{\mbb P^1,1}\colon H^2_\et(\mathbb P^1_{\mbb C_p},\mbb Z_p)(1)\tto H^2_{\Prism,\mr{tw}}(\mathbb P^1/\Ainf)$$ is induced by the \'etale comparison map.
\end{enumerate}

\begin{proof}
The first property is automatic because of our definition of Chern classes. The last follows by considering the commutative square 
$$
\xymatrix{H^2_\et(\mathbb P^1_{\mbb C_p},\mbb Z_p)(1)\ar[rr]^-{\kappa_{\mbb P^1,1}} && H^2_{\Prism,\mr{tw}}(\mathbb P^1/\Ainf)\{1\}\\
H^2_\et(B\mbb G_{m,\mbb C_p},\mbb Z_p)(1)\ar[u]_\sim\ar[rr]^-{\kappa_{B \mbb G_m,1}} && H^2_{\Prism,\mr{tw}}(B\mathbb G_m/\Ainf)\{1\}\ar[u]^\sim
}
$$ induced by $\mc O_{\mbb P^1}(1)\colon \mbb P^1\ra B\mbb G_m$ and noting that $c_1^\Ainf=\kappa_{B \mbb G_m,1}(c_1^\et)\in H^2_\Prism(B\mathbb G_m/\Ainf)\{1\}$. The multiplicativity follows from the multiplicativity in the universal case, which we now show. Namely, a short exact sequence $0 \to E_1 \to E \to  E_2 \to 0$  with $\dim E_i=r_i$, gives a map $X\ra BP_{(r_1,r_2)}$. It is then enough to show that for the natural maps $q\colon BP_{(r_1,r_2)} \ra B\GL_{r_1+r_2}$, $q_i\colon BP_{(r_1,r_2)} \ra B\GL_{r_i}$ sending (the universal) short exact sequence $0 \to E_1 \to E \to  E_2 \to 0$ to $E$ and $E_i$ correspondingly, one has $q^*(c^\Ainf)=q_1^*(c^\Ainf)q_2^*(c^\Ainf)$. This can be seen either using the analogous equality for Chern classes in the singular cohomology or directly, using Remarks \ref{rem: Chern classes as symmetric polynomials}, \ref{rem: partial flags} and the fact that for the "total" elementary symmetric polynomial
$$S_r(x) \coloneqq \sum_{i=0}^{r} \sigma_i(t_1, t_2, \ldots, t_r)x^i = \prod_{i=0}^r (1 + t_ix)$$
we have $S_{r_1+r_2}(x) = S_{r_1}(x) \cdot S_{r_2}(x)$. 

The uniqueness then follows from the "normalization" (this defines first Chern classes of line bundles, since they are pulled back from $B\mbb G_m$ and $H^{2}_{\Prism,\mr{tw}}(B\mbb G_m/\Ainf)\xra{\sim} H^{2}_{\Prism,\mr{tw}}(\mbb P^1/\Ainf)$ via the map $\mbb P^1\ra B\mbb G_m$ classifying $\mathcal O_{\mathbb P^1}(1)$) and the flag bundle formula (which then uniquely determines Chern classes for bundles of higher rank). 
\end{proof}
\end{cor}
We now would like to compare $c_i^\Ainf(E)$ with the Chern classes $c_i^\dR(E_{\mc O_{\mbb C_p}})$ and $c_i^\crys(E_{\ol k})$ in de Rham and crystalline cohomology via the corresponding comparisons. 
For this it is good to understand better the reductions of the twist $\Ainf\{n\}$ modulo $\ker \theta $ and $W(\mf m_{\mbb C_p^\flat})$ correspondingly. We denote by $W(\ol k)\{n\}$ the crystalline twist, namely it is a ${W(\ol k)}$-module of rank 1 on which $\phi_{W(\ol k)\{n\}}\colon W(\ol k)\{n\}[\frac{1}{p}]\xra{\sim} W(\ol k)\{n\}[\frac{1}{p}]$ acts by multiplication by $p^n$ times Frobenius $\phi$.

\begin{lem}\label{lem: reductions of Ainf-twist}
In the above notations:
\begin{enumerate}[label=(\arabic*)]
	\item The reduction of $\Ainf\{n\}$ modulo $\ker \theta$ (where $\theta\colon \Ainf \surj \mc O_{\mbb C_p}$ is the Fontaine's map) is canonically identified with $\mc O_{\mbb C_p}$. In particular this identification is $G_K$-equivariant for any finite extension $K/\mbb Q_p$.
	\item The reduction of $\Ainf\{n\}$ modulo $W(\mf m_{\mbb C_p^\flat})$ (via the map $\Ainf\surj W(\ol k)$) is canonically identified with $W(\ol k)\{n\}$. This identification is $(\phi, G_K)$-equivariant.
	\end{enumerate}
\end{lem}
\begin{proof} Note that $\Ainf$, $\mathcal O_{\mathbb C_p}$, and $W(\overline k)$ are algebra objects in the (abelian) symmetric monoidal category of continuous $G_K$-modules and the reduction functors are symmetric monoidal; since $\Ainf\{n\} = \Ainf\{1\}^{\otimes n}$, it is enough to prove the assertion for $n=1$.
We use the description $\Ainf\{1\}\simeq \frac{1}{\mu}(\Ainf\otimes_{\mbb Z_p}\mbb Z_p(1))$ (for some choice of $\mu$) which is $(\phi, G_K)$-equivariant (see \cite[Example 4.24]{BMS1}). First we identify the $G_K$-action on the ideal $(\mu)\subset \Ainf$ (and the corresponding inverse $\frac{1}{\mu}(\Ainf)$). We have $\mu=[\varepsilon]-1$, so $\ker(\chi_{p^\infty})\subset G_K$ stabilizes $\mu$ and the action on it factors through $\mbb Z_p^\times$. Let $n\in \mbb N$, $(n,p)=1$ be an integer, then $n\in \mbb Z_p^\times$ and $$n^*\mu=[\varepsilon]^n-1=\mu\cdot(1+[\varepsilon]+\cdots+[\varepsilon]^{n-1}).$$  
The element $\sigma_n\coloneqq (1+[\varepsilon]+\cdots+[\varepsilon]^{n-1})$ is equal to $n$ both modulo $\ker \theta$ and modulo $W(\mf m_{\mbb C_p^\flat})$: indeed this reduces to showing that $[\varepsilon]$ reduces to 1. This is obvious for $\ker \theta$ from the definition of $\theta$ since $\varepsilon=(1,\zeta_p,\ldots)$. For $W(\mf m_{\mbb C_p^\flat})$ we have $[\varepsilon] \pmod{W(\mf m_{\mbb C_p^\flat})} =[\ol\varepsilon]\in W(\ol k)$ where $\ol\varepsilon \in \ol k$ is the reduction of $\varepsilon$ modulo $\mf m_{\mbb C_p^\flat}$; then since $\zeta_{p^i}-1 \mf m_{\mbb C_p}$ we have $\ol\varepsilon=1 \mod \mf m_{\mbb C_p^\flat}$ and so $[\varepsilon] \pmod{W(\mf m_{\mbb C_p^\flat})} =[1]=1\in W(\ol k)$. Since this holds for any $n\in\mbb Z\cup \mbb Z_p^\times$ which is dense in $\mbb Z^\times$ we get that modulo these ideals $G_K$ acts on $\mu$ as on a Tate twist. In particular the corresponding  reductions of $\frac{1}{\mu}\Ainf$ are $G_K$-equivariantly identified with $\mc O_{\mbb C_p}(-1)$ and $W(\ol k)(-1)$. Then for $\Ainf\{1\}=\frac{1}{\mu}(\Ainf(1))$ the twists cancel and we get that $\Ainf\{n\}\otimes_\Ainf \mc O_{\mbb C_p}\simeq \mc O_{\mbb C_p}$ and $\Ainf\{n\}\otimes_\Ainf W(\ol k)\simeq W(\ol k)$ as $G_K$-modules. 

Finally, $\phi(\mu)=\widetilde\xi\cdot \mu$ where $\widetilde\xi=1+[\varepsilon]+\ldots+[\varepsilon]^{p-1}$ which by the computation above reduces to $p$ in $W(\ol k)$. This identifies $\Ainf\{n\}\otimes_\Ainf W(\ol k)\simeq W(\ol k)\{n\}$.
\end{proof}

\begin{rem}
Recall that one has well-defined theories of de Rham \cite[Lemma 50.21.1]{StacksProject} and crystalline \cite{Berthelot-Illusie} Chern classes for vector bundles on smooth schemes that satisfy Grothendieck-style axioms. Namely they are natural and multiplicative (similar to \Cref{cor: prism_Cherns_Grothendieck}) with the following normalization: for a line bundle $\mc L$ on $X$ over $\mc O_{\mbb C_p}$ we have $c_1^\dR(\mc L)=d\log^\dR(\mc L)\in H^2_\dR(X/\mc O_{\mbb C_p})$ where the map $$
d\log^\dR \colon \Pic^0(X) \tto H^2(X,\Omega^{\ge 1}_{X,\dR/\mc O_{\mbb C_p}})\tto H^2_\dR(X/\mc O_{\mbb C_p})
$$
is obtained by applying $H^2$ to the map of sheaves $d\log\colon \mc O^\times_X[-1] \ra \Omega^{\ge 1}_{X,\dR/\mc O_{\mbb C_p}}$ sending $f$ to the differential $1$-form $d\log f \coloneqq df/f$. Similarly, for a line bundle $\mc L$ on $X$ over $\ol k$ one has $c_1^\crys(\mc L)=d\log^\crys(\mc L)\in H^2_\crys(X/W(\ol k))$ where
$$d\log^\crys \colon \Pic^0(X) \tto H^2_\crys(X/W(\ol k))$$
is composition of the map $\Pic^0(X) \ra H^2_\crys(X,J)$, obtained by applying $H^2_\crys$ to the composition of the boundary map\footnote{Here $J$ is the kernel of the surjection $\mc O_{X_\crys}\surj \ol{\mc O}_{X_\crys}$ of sheaves on $X_\crys$. One then has a short exact sequence $0\ra (1+J) \ra \mc O_{X_\crys}^\times\ra \ol{\mc O}_{X_\crys}^\times\ra 0$ where one can identify the kernel $(1+J)$ with $J$ via (pd)-logarithm. This provides a map $\ol{\mc O}_{X_\crys}^\times \ra J[1]$ in the derived category.} $\ol{\mc O}_{X_\crys}^\times[-1] \ra J$ and the map $J\ra \mc O_{X_\crys}$ (we have $H^2(X_\crys, \ol{\mc O}_{X_\crys}^\times[-1])\simeq H^1(X_\crys, \ol{\mc O}_{X_\crys}^\times) \simeq \Pic^0(X)$ and  $H^2_\crys(X/W(\ol k))\coloneqq H^2(X_\crys,\mc O_{X_\crys})$). Taking the right Kan extension of the corresponding maps of sheaves (or rather the induced maps on derived global sections) we can extend these maps functorially to obtain
$$
d\log^\dR \colon \Pic^0(\mstack X)\ra H^2_\dR(\mstack X/\mc O_{\mbb C_p}) \ \ \ \text{ and } \ \ \ d\log^\crys \colon \Pic^0(\mstack X_{\ol k}) \tto H^2_\crys(\mstack X_{\ol k}/W(\ol k))
$$
for any smooth $\mc O_{\mbb C_p}$-stack $\mstack X$.
In particular, taking the tautological line bundle $\mc L_{\mr{taut}}$ on $B\mbb G_m$ we get classes $c_1^\dR\coloneqq d\log^\dR(\mc L_{\mr{taut}}) \in H^2_\dR(B\mbb G_m/\mc O_{\mbb C_p})$ and $c_1^\crys\coloneqq d\log^\crys(\mc L_{\mr{taut}}) \in H^2_\crys(B\mbb G_m/W(\ol k))$; by functoriality for any line bundle $\mc L$ on a smooth scheme $X$ over $\mc O_{\mbb C_p}$ we have $c_1^\dR(\mc L)=f^*c_1^\dR$ (resp. $c_1^\crys(\mc L_{\ol k})=f^*c_1^\crys$) where $f\colon X\ra B\mbb G_m$ is the map classifying $\mc L$. 
\end{rem}	

\begin{rem}
	Recall that in course of the proof of \Cref{de_Rham_fur_P_n} we showed in particular that the map $d\log^\dR_{\mbb Z}\colon \Pic^0(B\mbb G_m)\ra H^2_\dR(B\mbb G_m/\mbb Z)$ is an isomorphism. By base change it follows that $c_1^\dR\in H^2_\dR(B\mbb G_m/\mc O_{\mbb C_p})\simeq \mc O_{\mbb C_p}$ is a generator. 
	Moreover, we can consider $d\log^\dR_{W(\ol k)}\colon \Pic^0(B\mbb G_m)\ra H^2_\dR(B\mbb G_m/W(\ol k))$ which agrees with $d\log^\crys$ via the chains of isomorphisms
	$$\xymatrix{
	H^2_\dR(B\mbb G_m/W(\ol k)) \ar[r]^{\mc O(1)^*}_\sim & H^2_\dR(\mbb P^1/W(\ol k))\ar[r]^{\sim} & H^2_\crys(\mbb P^1/W(\ol k)) & \ar[l]^{\sim}_{\mc O(1)^*} H^2_\crys(B\mbb G_m/W(\ol k))\\ 
	\Pic^0(B\mbb G_m) \ar[u]^{d\log^\dR_{W(k)}}\ar[r]^{\mc O(1)^*}_\sim & \Pic^0(\mbb P^1) \ar[u]^{d\log^\dR_{W(k)}} \ar@{=}[r] & \Pic^0(\mbb P^1)\ar[u]^{d\log^\crys} &  \Pic^0(B\mbb G_m)\ar[l]^{\sim}_{\mc O(1)^*}\ar[u]^{d\log^\crys}.}
	$$
	Here commutativity in the middle is classical (see e.g. \cite[Proposition 2.3]{Berthelot-Illusie}) and the top left horizontal isomorphism follows from \Cref{lem:prismatic_coh_of_BGm_vs_Pn} by base change with respect to $\mf S\ra W(\ol k)$. This shows that $c_1^\crys \in H^2_\crys(B\mbb G_m/W(\ol k))\simeq W(\ol k)$ is a generator as well.  
\end{rem}

Since the reduction of $A\{n\}$ to $W(\ol k)$ still carries a $\phi$-action it makes sense to define the twisted version of the crystalline cohomology as well.
\begin{defn}
	For a smooth Artin stack $\mstack X$ over a field $k$ of characteristic $p$ we define 
	$$
	H^*_{\crys,\mr{tw}}(\mstack X/W(\ol k))\coloneqq \bigoplus_{i=0}^\infty H^i_{\crys}(\mstack X/W(\ol k))\{\floor{i/2}\}.
	$$

\end{defn}

Since $H^*_{\Prism}(B\GL_r/\Ainf)$ is free as $\Ainf$-module, from the crystalline comparison and \Cref{lem: reductions of Ainf-twist} we get an isomorphism 
$$
H^*_{\Prism,\mr{tw}}(B\GL_r/\Ainf)\otimes_\Ainf W(\ol k)\simeq H^*_{\crys,\mr{tw}}(B\GL_r/W(\ol k))
$$
of algebras in $(\phi,G_K)$-modules. We note that classical crystalline Chern classes $c_i^\crys(E)$ for a bundle $E$ on a smooth $\ol k$-scheme $X$ naturally lie in $H^{2i}_{\crys,\mr{tw}}(X/W(\ol k))$ and moreover (taking into account the twist) are $\phi$-invariant.

\begin{rem}
Note that by the de Rham and crystalline comparisons (and base change for prismatic cohomology) from \Cref{prop:char classes as W-invariants} one deduces isomorphisms
$$
H^*_{\dR}(B\GL_r/\mc O_{\mbb C_p})\simeq H^*_{\dR}(BT_r/\mc O_{\mbb C_p})^{\Sigma_r} \ \ \ \text{ and } \ \ \ H^*_{\crys,\mr{tw}}(B\GL_r/W(\ol k))\simeq H^*_{\crys,\mr{tw}}(BT_r/W(\ol k))^{\Sigma_r}
$$	
analogous to \Cref{rem: Chern classes as symmetric polynomials}. Trivializing $T_r\simeq \mbb G_m^r$ we can identify $H^*_{\dR}(BT_r/\mc O_{\mbb C_p})\simeq \mc O_{\mbb C_p}[(c_1^\dR)_1,\ldots, (c_1^\dR)_r]$ and consider $c_i^\dR\coloneqq \sigma_i((c_1^\dR)_1,\ldots, (c_1^\dR)_r)$ as an element of $H^*_{\dR}(B\GL_r/\mc O_{\mbb C_p})$. From multiplicativity and the "relative flag bundle formula" for the classically defined $c_i^\dR(E)$ it is easy to see that in fact $c_i^\dR(E)$ is equal to $f^*c_i^\dR$ where $f\colon X\ra B\GL_{\rk(E)}$ classifies $E$. Similarly for $c_i^\crys\coloneqq \sigma_i((c_1^\crys)_1,\ldots, (c_1^\crys)_r)$ we have $c_i^\crys(E)=f^*c_i^\crys$.
\end{rem}

Recall that due to freeness of $H^{*}_{\Prism}(B\GL_r/\Ainf)$ over $\Ainf$ we have isomorphisms $H^{2i}_{\Prism,\mr{tw}}(B\GL_r/\Ainf)\otimes_\Ainf W(\ol k)\simeq H^{2i}_{\dR}(B\GL_r/\mc O_{\mbb C_p})$ and  $H^{2i}_{\Prism,\mr{tw}}(B\GL_r/\Ainf)\otimes_\Ainf W(\ol k)\simeq H^{2i}_{\crys,\mr{tw}}(B\GL_r/W(\ol k))$. By \Cref{lem: reductions of Ainf-twist} these isomorphisms are $G_{\mbb Q_p}$ and $\phi$-equivariant\footnote{Here we consider $B\GL_r$ as a stack over $\mbb Z_p$ and this way obtain a $G_{\mbb Q_p}$-action on $H^{2i}_{\Prism}(B\GL_r/\Ainf)$ and $H^{2i}_{\Prism,\mr{tw}}(B\GL_r/\Ainf)$.}. We would like now to compare the reductions of $c_i^\Ainf$ modulo $\ker \theta$ and $W(\mf m_{\mbb C_p^\flat})$ with $c_i^\dR$ and $c_i^\crys$ correspondingly. While in a sequel we aim to show that they actually agree, here we only prove a weaker statement:
\begin{prop}\label{prop: Chern classes differ by units}
	For any $i\ge 0$ the images of $c_i^\Ainf$ in $H^{2i}_{\dR}(B\GL_r/\mc O_{\mbb C_p})$ and $H^{2i}_{\crys,\mr{tw}}(B\GL_r/W(\ol k))$ via the corresponding reductions differ from $c_i^\dR$ and $c_i^\crys$ by some constants in $\mbb Z_p^\times$.
\end{prop}

\begin{proof}
	Recall that $c_1^{\Ainf}$ is a generator of $H^{2}_{\Prism,\mr{tw}}(B\mbb G_m/\Ainf)$ as an $\Ainf$-module and that it is $G_{\mbb Q_p}$ and $\phi$-invariant. Its image in $H^2_\dR(B\mbb G_m/\mathcal O_{\mathbb C_p})$ is non-zero and $G_{\mbb Q_p}$-invariant and so is $c_1^\dR$ (by naturality). Since $H^2_\dR(B\mbb G_m/\mathcal O_{\mathbb C_p})\simeq H^2_\dR(\mbb P^1/\mathcal O_{\mathbb C_p})$ and $H^2_\dR(\mbb P^1/\mathcal O_{\mathbb C_p})^{G_{\mbb Q_p}}\simeq H^2_\dR(\mbb P^1/\mbb Z_p)\simeq \mbb Z_p$ we get that the image of $c_1^{\Ainf}$ in $H^2_\dR(B\mbb G_m/\mathcal O_{\mathbb C_p})$ differs from $c_1^\dR$ by a unit in $\mbb Z_p$. Similarly the image in $H^2_{\crys,\mr{tw}}(B\mbb G_m/W(\overline k))$ is a generator of $H^2_{\crys,\mr{tw}}(B\mbb G_m/W(\overline k))^{G_{\mbb Q_p}}$, the latter is identified with $H^2_{\crys}(B\mbb G_m/W(\mbb F_p))\simeq \mbb Z_p$ and we get that the image of $c_1^{\Ainf}$ in $H^2_{\crys,\mr{tw}}(B\mbb G_m/W(\overline k))$ differs from $c_1^\crys$ by a unit in $\mbb Z_p$ as well. The statement for $c_i^\Ainf$ then follows automatically since it is given as elementary symmetric polynomials in first Chern classes (after pull-back under $BT_r\ra B\GL_r$) and the same is true for $c_i^\dR$ and $c_i^\crys$.
\end{proof}

\begin{rem}\label{rem: polynomials in Chern classes also differ by units!}
More generally for any homogeneous polynomial expression $f(x_1,\ldots ,x_r)\in \mbb Z_p[x_1,\ldots,x_r]$ of degree $2n$ where we put $\deg x_i=2i$ we have that the images of $f(c_1^\Ainf,\ldots ,c_r^\Ainf)\in H^{2n}_{\Prism,\mr{tw}}(B\GL_r/\Ainf)$ in $H^{2n}_{\dR}(B\GL_r/\mc O_{\mbb C_p})$ and $H^{2n}_{\crys,\mr{tw}}(B\GL_r/W(\ol k))$ differ from $f(c_1^\dR,\ldots ,c_r^\dR)$ and $f(c_1^\crys,\ldots ,c_r^\crys)$ by units in $\mbb Z_p$ as well. Indeed following the proof of \Cref{prop: Chern classes differ by units} it is easy to see that if $c_1^\dR =\alpha\cdot \ol{c^\Ainf_1}$ then $c_i^\dR =\alpha^i\cdot \ol{c^\Ainf_i}$; thus we also have $f(c_1^\dR,\ldots ,c_r^\dR)=\alpha^n \cdot \ol{f(c_1^\Ainf,\ldots ,c_r^\Ainf)}$ (here we denoted by $\ol{x}$ the image of a class $x\in H^*_{\Prism,\mr{tw}}(B\GL_r/\Ainf)$ in $H^{*}_{\dR}(B\GL_r/\mc O_{\mbb C_p})$). Same argument works in the crystalline case.
\end{rem}	
This leads to the following implication, showing in particular that there are some positive characteristic algebraically-geometric obstructions for the topological triviality of vector bundles (see \Cref{rem: etale Chern are 0 is the same as topological} below): 
\begin{cor}\label{cor: top Chern classes are  0 => de Rham are}
	Let $X$ be a smooth proper scheme over $\mathcal O_{\mathbb C_p}$ and let $E$ be a vector bundle over $X$. Assume that $H^{2i}_\Prism(X/\Ainf)$ is $\mu$-torsion free. Then if $c_i^\et(E_{|X_{{\mathbb C_p}}})$ vanishes, then $c_i^\dR(E) = 0$ and $c_i^\crys(E_{|X_{\overline k}}) = 0$.  
	
	\begin{proof}
		Let $f\colon X\ra B\GL_r$ be the map classifying $E$. By \Cref{prop: Chern classes differ by units} and functoriality of de Rham comparison $c_\dR^i(E)=f^*c_\dR^i$ is equal up to a unit to the image of $c_i^\Ainf(E)=f^*c_i^\Ainf$ under the natural map $H^{2i}_{\Prism,\mr{tw}}(X/\Ainf)\otimes_\Ainf \mc O_{\mbb C_p} \ra H^{2i}_\dR(X/\mc O_{\mbb C_p})$. Note that by definition the image of $c_i^\Ainf(E)$ in the $\mu$-localization $H^{2i}_{\Prism,\mr{tw}}(X/\Ainf)[\frac{1}{\mu}]$ agrees with the image of $c_i^\et(E_{|X_{{\mathbb C_p}}})$ under the \'etale comparison map $\kappa_{X,i}$ (see \ref{warn:mu_torsion_free} for details). Since $H^{2i}_\Prism(X/\Ainf)$ is $\mu$-torsion free $H^{2i}_{\Prism,\mr{tw}}(X/\Ainf)$ embeds in $H^{2i}_{\Prism,\mr{tw}}(X/\Ainf)[\frac{1}{\mu}]$ and thus $c_i^\et(E_{|X_{{\mathbb C_p}}})=0 \Rightarrow c_i^\Ainf(E)=0$. It follows that $c_\dR^i(E)=0$. The proof for $c_i^\crys(E_{|X_{\overline k}})$ is analogous.\qedhere
		
	\end{proof}
\end{cor}
\begin{rem}\label{rem: etale Chern are 0 is the same as topological}
Note that by the definition of \'etale Chern classes and Artin's comparison, for any isomorphism $\iota\colon \mbb C_p\xra{\sim}\mbb C$ the condition $c_i^\et(E_{|X_{{\mathbb C_p}}})=0$ is equivalent to $c_i(E(\mbb C))=0$. Here, since $c_i\in H^{2i}_\sing(B\GL_r(\mbb C),\mbb Z_p)$ lies in the image of the natural embedding $H^{2i}_\sing(B\GL_r(\mbb C),\mbb Z)\subset H^{2i}_\sing(B\GL_r(\mbb C),\mbb Z_p)$ it does not matter whether we consider $c_i(E(\mbb C))$ as an element of $\mbb Z$ or $\mbb Z_p$-valued singular cohomology. Note also that the reduction $\ol {c_i^\dR(E)}\in H^{2i}_\dR(X_k/k)$ of $c_i^\dR(E)$ modulo $\mf m_{\mbb C_p}$ is equal to $c_i^\dR(E_{|X_k})$ (since Chern classes in de Rham cohomology of smooth schemes commute with base change). In particular from \Cref{cor: top Chern classes are  0 => de Rham are} we get that if $c_i^\dR(E_{|X_k})\neq 0$ for some $i>0$ then the corresponding topological vector bundle $E(\mbb C)$ is non-trivial.
\end{rem}	

\begin{rem}
	Using \Cref{rem: polynomials in Chern classes also differ by units!} it is easy to show a similar vanishing statement for any homogeneous polynomial expression $f(c_1,\ldots,c_r)$ of degree $2i$ in Chern classes in place of a single Chern class $c_i$. One needs the same "$\mu$-torsion free" assumption on $H^{2i}_\Prism(X/\Ainf)$.
\end{rem}

\begin{rem}
	The topological application is mostly interesting in the case when $c_i^\dR(E)$ (or equivalently $c_i^\crys(E)$) are killed by some power of $p$. Otherwise we also have $c_i^\dR(E_{|X_{\mbb C}})\neq 0$ and the non-vanishing of $c_i(E(\mbb C))$ follows by complex-analytic methods. This also applies to any polynomial expression $p(c_1,\ldots,c_r)$ as above.
\end{rem}

We end the subsection with a final remark on how to descend the $\Ainf$-characteristic classes to $H^*_{\Prism,\mr{tw}}(-/\mf S)$ in case we are only interested in schemes over $\mc O_K$:
\begin{rem}\label{rem: BK-Chern classes} Choose a finite extension $K/\mbb Q_p$ and a Breuil-Kisin prism $(\mf S,E(u))$ with $\mf S/E(u)\simeq \mc O_K$. Note that for any $i\ge 0$ one has an isomorphism $(\mf S\{i\})^{(1)}\otimes_{\mf S}\footnote{Note the presence of the Frobenius twist on the left hand side, as well as the order in which the Frobenius and Breuil-Kisin twist are applied. This accounts for the fact that the Breuil-Kisin-Fargues twist $\Ainf\{n\}$ is obtained as a base change of the Breuil-Kisin twist $\mf S\{n\}$ via the Frobenius-twisted map $\mf S\ra \Ainf$ sending $u$ to $[\pi^\flat]^p$ and not to $[\pi^\flat]$.} \Ainf\xra{\sim} \Ainf\{i\}$. Moreover, for the trivial twist, the $\phi$-invariants in $\Ainf$ (that are just identified with $\mbb Z_p\subset \Ainf$) lie in the image of $\mf S \ra \Ainf$ (and, moreover, generate it over $\mf S$). Recall that by \Cref{cor:prismatic cohomology of BG as a ring}, the module $H^{2i}_{\Prism}(B\GL_r/\mf S)\{i\}$ is a sum of trivial Breuil-Kisin twists. Using this, one sees that the class $c_i^\Ainf$ lies in the image of the natural embedding
	$$
	H^{2i}_{\Prism,\mr{tw}}(B\GL_r/\mf S)\coloneqq \left(H^{2i}_{\Prism}(B\GL_r/\mf S)\{i\}\right)^{(1)}\tto H^{2i}_{\Prism,\mr{tw}}(B\GL_r/\Ainf),
	$$
induced by the map $\mf S\ra \Ainf$ (sending $u$ to $[\pi^\flat]$). 	
Since the classes $c_i^\Ainf$ are by definition $\phi$-invariant, they come as images of some classes $c_i^{\BK}\in H^{2i}_{\Prism,\mr{tw}}(B\GL_r/\mf S)$. In particular if we want to consider smooth schemes (or stacks) with vector bundles over $\mc O_K$ by considering pull-backs of $c_i^{\BK}$ we get a well defined $H^{*}_{\Prism,\mr{tw}}(-/\mf S)$-valued theory of characteristic classes, which we denote by $c_i^{\BK}(-)$.	
\end{rem}	

\begin{rem}\label{rem: realizing de Rham as a deformation}
A similar construction can be applied to any reductive group $G$, provided $p$ is non-torsion. Similarly to the case of $\GL_n$, by \Cref{cor:prismatic cohomology of BG as a ring}, the module $H^{2i}_{\Prism}(B\GL_r/\mf S)\{i\}$ is a sum of trivial Breuil-Kisin twists. Then, given a set of polynomial generators $x_1,\ldots,x_n$	of $H^*_\sing(BG(\mbb C),\mathbb Z_p)$, using the \'etale comparison one gets polynomial generators  $x_1^\Ainf,\ldots,x_n^\Ainf$ of $H^*_{\Prism,\mr{tw}}(BG/\Ainf)$ that are $\phi$-invariant, and this way (arguing as in the above remark) also provide a set of polynomial generators $x_1^\BK,\ldots,x_n^\BK$ for $H^*_{\Prism,\mr{tw}}(BG/\mf S)$. One can see that the (composite) twist doesn't affect the de Rham comparison even over $\mf S$: namely, using the definition of $\mf S\{1\}$ as a limit of $E_r\mf S/E_{r}^2\mf S$ from \cite[Example 4.2]{BMS1} one can see that $(\mf S\{1\})^{(1)}\simeq \lim_r \phi(E_r)\mf S/\phi(E_{r}^2)\mf S$. Now $(E,\phi(E_r))$ is a regular sequence for any $r$, with $\phi(E_r) \equiv p^r \bmod E$ so $(\mf S\{1\})^{(1)}\otimes_{\mf S} \mc O_K \simeq \lim_r p^r\mc O_K/p^{2r}\mc O_K$ with the maps $p^{r+1}\mc O_K/p^{2r+2}\mc O_K \ra p^r\mc O_K/p^{2r}\mc O_K$ induced by first dividing by $p$ and then reducing modulo $p^{2r}$. This naturally identifies with $\lim_r \mc O_K/p^r\simeq \mc O_K$. Note, however, that natural action of $\mf S$ on $\mc O_K\simeq (\mf S\{1\})^{(1)}\otimes_{\mf S} \mc O_K$ is through $\phi_{\mf S}$. In particular, we do have an an isomorphism 
$$
H^*_{\Prism,\mr{tw}}(BG/\mf S)\otimes_{\mf S} \mc O_K \simeq H^*_{\Prism}(BG/\mf S)\otimes_{\mf S,\phi_{\mf S}} \mc O_K \simeq H^*_\dR(BG/\mc O_K)
$$
and can define a set of polynomial generators $x_1^\dR,\ldots,x_n^\dR$ of $H^*_\dR(BG/\mc O_K)$ by taking reductions of $x_1^\BK,\ldots,x_n^\BK$ modulo $E$. Returning to \Cref{rem:de Rham as deformation of \'etale} we get that for any given choice of the uniformizer $\pi\in \mc O_K$ there is a deformation from $H^*_\sing(BG(\mbb C),\mathbb Z_p)$ to $H^*_\dR(BG/\mc O_K)$ over $\mf S$ given by $H^*_{\Prism,\mr{tw}}(BG/\mf S)\otimes_{\mf S} \mc O_K$ which in particular gives a way of associating to a set of polynomial generators $x_1,\ldots,x_n$	of $H^*_\sing(BG(\mbb C),\mathbb Z_p)$ a set of polynomial generators $x_1^\dR,\ldots,x_n^\dR$ of $H^*_\dR(BG/\mc O_K)$.
\end{rem}

\subsection{Integrals and Chern numbers}\label{sec:Integrals and Chern numbers}
As another application in this subsection we prove a comparison result on characteristic numbers in different theories that is valid without any assumptions on $\Ainf$-cohomology of $X$.

Let $X$ be a connected smooth proper $\mathcal O_K$-scheme of relative dimension $d$. 
\begin{notation}\label{not:de Rham/crys canonical classes}
\begin{enumerate}
	\item 
The $\mc O_K$-module $H^{2d}_\dR(X/\mathcal O_K)$ is free of rank 1 and one has the canonical generator $\omega^\dR_X\in H^{2d}_\dR(X/\mathcal O_K)$ which we call the \textit{(de Rham) canonical volume form}. One easy way to define it is to use the Hodge filtration and the natural identification $H^{2d}_\dR(X/\mathcal O_K)\simeq H^d(X,\Omega^d_{X/\mc O_K})$, where the right hand side is canonically identified with $\mc O_K$ via the "trace map":	
$$
H^d(X,\Omega^d_{X/\mc O_K})\simeq \Hom_X(\mc O_X,\Omega^d_{X/\mc O_K}[-d])\xymatrix{\ar[r]^{\tr_X}_\sim &} \mc O_K.
$$ 
We then put $\omega^\dR_X\coloneqq \tr^{-1}_X(1)$ (where the identification $H^{2d}_\dR(X/\mathcal O_K)\simeq H^d(X,\Omega^d_{X/\mc O_K})$ is implicit).

	\item In \cite[Chapitre VII, 1.4]{berthelot2006cohomologie} Berthelot constructed the "trace map" for the crystalline cohomology. In particular, in our situation, one gets a canonical isomorphism 
	$$
	H^{2d}_\crys(X_k/W(k))\xymatrix{\ar[r]^{\tr_X^\crys}_\sim &} W(k).
	$$ 
	Moreover, it is not hard to see that if one puts the twist $H^{2d}_\crys(X_k/W(k))\{d\}$ on the left hand side instead, this map becomes $\phi$-equivariant. Similarly, we define the \emdef{(crystalline) canonical volume form} $\omega_X^\crys\in H^{2d}_\crys(X_k/W(k))\{d\}$ as $(\tr_X^\crys)^{-1}(1)$.  Note that $\omega_X^\crys$ is $\phi$-invariant and that $\big(H^{2d}_\crys(X_k/W(k))\{d\}\big)^{\phi=1}\simeq \mbb Z_p$.
	
	\item We will also call any $\mc O_K$-(resp. $\mbb Z_p$-)generator of $H^{2d}_\dR(X/\mathcal O_K)$ (resp. $\big(H^{2d}_\crys(X_k/W(k))\{d\}\big)^{\phi=1}$) a \emdef{de Rham (resp. crystalline) volume form}.
\end{enumerate}
\end{notation}

We are now going to provide another natural choice of de Rham and crystalline volume forms using the bridge from the topological data of $X(\mbb C)$ to the integral de Rham data of $X$ that is provided by the $\Ainf$-cohomology and the \'etale comparison. 	
\begin{construction}
	Note that by the Poincar\'e duality for \'etale cohomology $H^{2d}_\et(X_{\mathbb C_p}, \mathbb Z_p) \simeq \mathbb Z_p(-d)$, thus  $$H^{2d}_\Prism(X/\mf S)\simeq \BK(H^{2d}_\et(X_{\mathbb C_p}, \mathbb Z_p))\simeq \mf S\{-d\}$$ and so $H^{2d}_{\Prism^{(1)}}(X_{\mc O_{\mbb C_p}}/\Ainf)\simeq \Ainf\{-d\}$. We get that $H^{2d}_{\Prism,\mr{tw}}(X_{\mc O_{\mbb C_p}}/\Ainf) \coloneqq H^{2d}_{\Prism^{(1)}}(X_{\mc O_{\mbb C_p}}/\Ainf)\{d\}\simeq \Ainf(d)$ as a $(G_K,\phi)$-module. Note that $H^{2d}_{\Prism}(X_{\mc O_{\mbb C_p}}/\Ainf)$ is the top prismatic cohomology group and since it is free over $\Ainf$, the de Rham and crystalline comparisons hold in the non-derived form, namely
	$$
	H^{2d}_{\Prism,\mr{tw}}(X_{\mc O_{\mbb C_p}}/\Ainf)\otimes_{\Ainf} \mc O_{\mbb C_p}\simeq H^{2d}_\dR(X_{\mc O_{\mbb C_p}}/\mc O_{\mbb C_p}) \ \ \ \text{ and } \ \ \  H^{2d}_{\Prism,\mr{tw}}(X_{\mc O_{\mbb C_p}}/\Ainf)\otimes_{\Ainf} W(\ol k)\simeq H^{2d}_{\crys,\mr{tw}}(X_{\ol k}/W(\ol k)).
	$$
	
	Recall that we defined a map (see \Cref{not: map kappa})
	$$\kappa_{X,d}\colon H^{2d}_\et(X_{\mathbb C_p},\mbb Z_p)(d)\xra{\sim}\footnote{Here we use that $X$ is proper.} H^{2d}_\et(\widehat X_{\mathbb C_p},\mbb Z_p)(d)\tto H^{2d}_{\Prism,\mr{tw}}(X_{\mc O_{\mbb C_p}}/\Ainf)\otimes_\Ainf W(\mbb C_p^\flat)$$,
	which in our case factors through $H^{2d}_{\Prism,\mr{tw}}(X_{\mc O_{\mbb C_p}}/\Ainf)$
	since the $\phi$-invariants in $H^{2d}_{\Prism,\mr{tw}}(X_{\mc O_{\mbb C_p}}/\Ainf)\otimes_\Ainf W(\mbb C_p^\flat)\simeq W(\mbb C_p^\flat)$ are given by $\mbb Z_p$ and lie in $\Ainf\subset W(\mbb C_p^\flat)$ (which is identified with the image of $H^{2d}_{\Prism,\mr{tw}}(X_{\mc O_{\mbb C_p}}/\Ainf)$). Moreover we see that $H^{2d}_\et(X_{\mathbb C_p},\mbb Z_p)(d)$ generates $H^{2d}_{\Prism,\mr{tw}}(X_{\mc O_{\mbb C_p}}/\Ainf)$ over $\Ainf$.
	
 Analogously to \Cref{rem: chern classes etale vs sing}, from Artin's comparison we get a canonical isomorphism 
	$$
	H^{2d}_\sing(X(\mbb C),\mbb Z_p)\simeq H^{2d}_\et(X_{\mathbb C_p},\mbb Z_p)(d).
	$$	
Note that $X(\mbb C)$ is a real $2d$-dimensional manifold and the complex structure on $X(\mbb C)$ endows it with a natural orientation. This gives a $\mathbb Z$-generator $[X(\mbb C)]$ of $H^{2d}_\sing(X(\mbb C),\mbb Z)$ and, then, a $\mathbb Z_p$-generator $[X(\mbb C)]^\et\in H^{2d}_\et(X_{\mathbb C_p},\mbb Z_p)(d)$. By the discussion above the class $[X(\mbb C)]^\Ainf\coloneqq \kappa_{X,d}([X(\mbb C)])\in H^{2d}_{\Prism,\mr{tw}}(X_{\mathcal O_{\mathbb C_p}}/\Ainf)$ then is a generator as an $\Ainf$-module. Its reduction $[X(\mbb C)]^\dR\in H^{2d}_\dR(X_{\mc O_{\mbb C_p}}/{\mc O_{\mbb C_p}})\simeq \mc O_{\mbb C_p}$ modulo $\ker \theta$ is also a generator and is $G_K$-invariant (since the $G_K$-action on $H^{2d}_\et(X_{\mathbb C_p},\mbb Z_p)(d)$ was trivial), it thus descends to an $\mc O_K$-generator $[X(\mbb C)]^\dR\in H^{2d}_\dR(X/{\mc O_{K}})$. Analogously, reducing $[X(\mbb C)]^\Ainf$ modulo $W(\mf m_{\mbb C_p^\flat})$ we get a $(G_K,\phi)$-invariant $W(\ol k)$-generator $[X(\mbb C)]^\crys\in H^{2d}_\crys(X_{\ol k}/W(\ol k))\{d\}$ which then descends to a $\phi$-invariant $W(k)$-generator of $[X(\mbb C)]^\crys\in H^{2d}_\crys(X_{k}/W(k))\{d\}$.
\end{construction}
\begin{rem}\label{rem:compatibility of volume forms}
	We expect that $[X(\mbb C)]^\dR=\omega_X^\dR$ and $[X(\mbb C)]^\crys=\omega_X^\crys$, however, we do not show this here and plan to return to that question in the future. 
\end{rem}	

\begin{notation} Let $f(x_1, \ldots, x_r) \in \mbb Z_p[x_1,\ldots, x_r]$ be a polynomial. Recall that for an oriented topological closed manifold $Y$ of dimension $2d$ with complex bundle $E$ of rank $r$ the corresponding \textit{Chern number} $c(f,E)$ is defined as the image of $f(c_1,\ldots, c_n)\in H^*_\sing(Y,\mbb Z_p)$ under the natural composition 
$$\xymatrix{H^*_\sing(Y,\mbb Z_p)\ar[r] & H^{2d}_\sing(Y,\mbb Z_p) \ar[rr]^-{\lambda\cdot [Y]\mapsto \lambda}_-\sim && \mbb Z_p}.$$
\end{notation}

\begin{construction}\begin{enumerate}
		\item For a class $\alpha \in H^{*}_\dR(X/\mathcal O_K)$ we define the \textit{"integral"} $\int_{X} \alpha \in \mathcal O_K$ as the image of the composition
		$$\int_{X} \colon\xymatrix{ H^*_\dR(X/\mathcal O_K) \ar[r] & H^{2d}_\dR(X/\mathcal O_K) \ar[rr]^-{\lambda\cdot \omega^\dR_X \mapsto \lambda}_-\sim && \mathcal O_K,}$$
		where the first arrow is the projection on the top-degree component and the second trivialization is given by the choice of the canonical volume form $\omega^\dR_X$ as a generator. Similarly, for a class $\alpha \in H^{*}_\crys(X_k/W(k))$ we can define $\int_X \alpha \in W(k)$ as image of the composition
		$$\int_{X} \colon\xymatrix{ H^*_\crys(X_k/W(k))\{d\} \ar[r] & H^{2d}_{\crys,\mr{tw}}(X_k/W(k)) \ar[rr]^-{\lambda\cdot \omega_X^\crys \!\mapsto \lambda}_-\sim && W(k)}.$$
		Note that since $[X(\mbb C)]^\dR$ and $[X(\mbb C)]^\crys$ are also volume forms we have 
		$$
		\int_{X} [X(\mbb C)]^\dR \in \mc O_K^\times \ \ \ \ \text{ and } \ \ \ \ \! \int_{X}  [X(\mbb C)]^\crys\in \mbb Z_p^\times.
		$$
\item Let $X$ be a connected smooth proper $\mathcal O_K$-scheme and let $E$ be vector bundle on $X$ of rank $r$. For a polynomial $f(x_1, \ldots, x_r) \in \mathcal O_K[x_1,\ldots, x_r]$ we define the corresponding \emdef{de Rham Chern number} $c^\dR(f,E)$ of $E$ as
$$c^\dR(f,E)\coloneqq \int_{X} f(c_1^\dR(E), c_2^\dR(E), \ldots, c_r^\dR(E)).$$
Similarly, for $f(x_1, \ldots, x_r) \in W(k)[x_1,\ldots, x_r]$ the \emdef{crystalline Chern number} $c^\dR(f,E|_{X_k})$ is defined as 
$$
c^\crys(f,E|_{X_k})\coloneqq \int_{X} f(c_1^\crys(E|_{X_k}), c_2^\dR(E|_{X_k}), \ldots, c_r^\crys(E|_{X_k})).
$$

\end{enumerate}
\end{construction}

\begin{prop}\label{prop: p-adic valuations of Chern numbers}
Let $X$ be a smooth proper $\mathcal O_K$-scheme of relative dimension $d$ and $E$ a rank $r$ vector bundle on $X$. Then for any polynomial $f(x_1, \ldots, x_r) \in \mathbb Z_p[x_1, \ldots, x_r]$ the corresponding de Rham Chern number $c^\dR(f,E)$ (resp. crystalline Chern number  $\int_{X} f(c_1^\crys(E|_{X_k}), \ldots, c_r^\crys(E|_{X_k}))$) coincides with the topological Chern number $c(f,E(\mbb C))$ up to an invertible constant from $\mathcal O_K$ (resp. $\mbb Z_p$). In particular, $p$-adic valuations of these characteristic numbers (for any choice of $\omega$) are the same.

\begin{proof}
	 Without loss of generality we can assume $f$ is homogeneous of degree $2d$ (where we count $\deg x_i=2i$). By Artin's comparison and compatibility of topological and \'etale Chern classes (see \Cref{rem: etale Chern compare with topological}) the Chern number $c(f,E(\mbb C))$ can be computed as the image of $f(c_1^\et(E|_{X_{\mbb C_p}}), \ldots, c_r^\et(E|_{X_{\mbb C_p}}))\in H^{2d}_{\et}(X_{\mbb C_p},\mbb Z_p)(d)$ under the composite map 
	$$\xymatrix{ H^{2d}_{\et}(X_{\mbb C_p},\mbb Z_p)(d) \ar[rr]^-{\lambda\cdot [X(\mbb C)]^\et\mapsto \lambda}_-\sim && \mbb Z_p},$$	where $[X(\mbb C)]^\et$ is the image of the fundamental class $[X(\mbb C)]$ under the $H^{2d}_\sing(X(\mbb C),\mbb Z_p)\xra{\sim} H^{2d}_{\et}(X_{\mbb C_p},\mbb Z_p)(d)$. Note that $[X(\mbb C)]^\dR$ and $[X(\mbb C)]^\crys$ are defined as reductions of $[X(\mbb C)]^\Ainf=\kappa_{X,d}([X(\mbb C)]^\et)$. By construction the image of $f(c_1^\et(E|_{X_{\mbb C_p}}), \ldots, c_r^\et(E|_{X_{\mbb C_p}}))$ in $H^{2d}_{\Prism, \mr{tw}}(X_{\mathcal O_{\mathbb C_p}}/\Ainf)$  under $\kappa_{X,n}$ is equal to $f(c_1^\Ainf(E|_{X_{\mc O_{\mbb C_p}}}), \ldots c_r^\Ainf(E|_{X_{\mc O_{\mbb C_p}}}))$ and one then sees that 
	$$
	\int_X \big(f(c_1^\Ainf(E|_{X_{\mc O_{\mbb C_p}}}), \ldots, c_r^\Ainf(E|_{X_{\mc O_{\mbb C_p}}}))\bmod \ker \theta\big)= c(f,E(\mbb C))\cdot \int_X [X(\mbb C)]^\dR
	$$
and similarly for the crystalline case.	Recall that $\int_X [X(\mbb C)]^\dR\in \mc O_K^\times$ and $\int_X [X(\mbb C)]^\crys\in \mbb Z_p^\times$.

 Since $H^{2d}_{\Prism}(X_{\mathcal O_{\mathbb C_p}}/\Ainf)$ is $\mu$-torsion free we can apply \Cref{rem: polynomials in Chern classes also differ by units!}: namely the image of $$f(c_1^\Ainf(E|_{X_{\mc O_{\mbb C_p}}}), \ldots c_r^\Ainf(E|_{X_{\mc O_{\mbb C_p}}})) \bmod \ker \theta $$ is equal up to a $\mbb Z_p^\times$-multiple to $f(c_1^\dR(E), \ldots, c_r^\dR(E))$ (and similarly to $f(c_1^\crys(E), \ldots, c_r^\crys(E))$ modulo $W(\mf m_{\mbb C_p^\flat})$ in the crystalline case). Thus so are the corresponding integrals and the statement about the Chern numbers follows. 
\end{proof}
\end{prop}
\begin{rem}
	In a sequel we aim to show that in fact the equality holds on the nose:
	$$
	c(f,E)=c^\dR(f,E(\mbb C))= c^\crys(f,E|_{X_k}).
	$$
	This reduces to showing that $c_1^\dR$ and $c_1^\crys$ are equal to the corresponding reductions of $c_1^\Ainf$ and the compatibility of the volume forms as in \Cref{rem:compatibility of volume forms}.
\end{rem}
\begin{rem}
Using \Cref{prop: p-adic valuations of Chern numbers} one can deduce the analogue of \cite[Theorem 8.1]{FGV_symplectic} for the de Rham characteristic classes: namely, if $p$ divides the numerator of the $(d+1)$-st Bernoulli number $B_d$, then, given a projective smooth $\mc O_K$-scheme $X$ of relative dimension $d< p$ with a principally polarized abelian scheme $\mc A\ra X$ over it, the de Rham Chern number
$
c^\dR(\mr{ch}_d,H_{\mc A})
$ is divisible by $p$. Here $\mr{ch}_d$ is (the $d$-th graded component of) the Chern character and $H_{\mc A}$ is the Hodge bundle associated to $\mc A$; note that $\mr{ch}_d\in \mbb Z_p[c_1,\ldots, c_d]$ because $d<p$ and thus \Cref{prop: p-adic valuations of Chern numbers} indeed applies.
\end{rem}
\subsection{De Rham cohomology of conical resolutions}\label{sec: conical resolutions}
For the reminder and the definition of conical resolution see \Cref{ex: conical resolutions}.
In \cite[(generalized version of) Conjecture 5.2.3]{Kubrak_Travkin} Roman Travkin and the first author conjectured that for a conical resolution $\pi\colon X\ra Y$ over $\mc O_K$ one should have an inequality 
$$
\dim_{k} H^1_{\dR}([X/\mbb G_m]_{k}/k)\ge \dim_{\mbb F_p} H^1_{\mr{sing}}(X(\mbb C),\mbb F_p).
$$
In fact it is more natural to expect the following more general inequality:
\begin{conj} \label{conj: conical resolutions} Let $\pi\colon X\ra Y$ be a conical resolution, then for any $i\ge 0$
$$
\dim_{k} H^i_{\dR}([X/\mbb G_m]_{k}/k) \ge \dim_{\mbb F_p} H^i_{\mr{sing}}(X(\mbb C)/\mbb C^\times,\mbb F_p),
$$
where $X(\mbb C)/\mbb C^\times$ denotes the homotopy quotient.
\end{conj}

 However, we claim that for $i=1$ there is in fact no difference between the two. This follows from the following lemma:

\begin{lem}\label{lem: about S1-action} Let $Z$ be a connected CW-complex with an action of $S^1$ and assume that the set of the (non-homotopy) fixed points $Z^{S^1}$ is non-empty. Then
	$$
	H^1_{\mr{sing}}(Z/S^1,\mbb F_p)\simeq H^1_{\mr{sing}}(Z,\mbb F_p),
	$$
	where $Z/S^1$ is the homotopy quotient.
\end{lem}
\begin{proof}
	Consider the fiber sequence $Z\ra Z/S^1\xra{p} BS^1$. Recall that $BS^1$ has an explicit model as $S^\infty/S^1\simeq \mbb CP^\infty$ (where $S^\infty=\colim_n S^{2n-1}$ and $S^{2n-1} \subset \mbb C^n$ is the unit sphere endowed with the natural action of $S^1=\{z\in \mbb C^\times| |z|=1\}$). Since $\mbb CP^\infty$ is simply connected one has the Serre-Lerray spectral sequence 
	$$
	E_2^{p,q}=H^{p}_{\mr{sing}}(\mbb CP^\infty, H^q(Z,\mbb F_p)) \Rightarrow H^{p+q}_{\mr{sing}}(Z/S^1).
	$$
	Once again, by simply-connectedness, $E^{1,0}\simeq 0$, and so $H^1_{\mr{sing}}(Z/S^1,\mbb F_p)\simeq \Ker\ \! d_{2}^{0,1} \subset  E^{0,1}$. Thus it is enough to prove that $d_{2}^{0,1}\colon E^{0,1}\ra E^{2,0}$ is equal to 0. Note that since $Z$ is connected $H^0(Z,\mbb F_p)=\mbb F_p$ and  $E^{2,0}\simeq H^2(\mbb CP^\infty,\mbb F_p)\simeq \mbb F_p$. 
	
	Essentially by definition (of the differentials in the Lerray spectral sequence) and the Hurewicz theorem, $d_{2}^{0,1}$ has the following geometric description. Namely, consider the projective line $\mbb CP^1\subset \mbb CP^\infty$; its class $[\mbb CP^1]\in H_2(\mbb CP^\infty,\mbb Z)\simeq \pi_2(\mbb CP^\infty)\simeq \mbb Z$ is the generator. It also gives a generator of $H_2(\mbb CP^\infty,\mbb F_p)$ by reduction mod $p$. Pick a point $z\in Z/S^1$ and identify the homotopy fiber over $p(z)$ with $Z$. Let's consider a boundary map in the long exact sequence of homotopy groups of our fibration 
	$$
	\xymatrix{\ldots \ar[r] & \pi_2(\mbb CP^\infty, p(z))\ar[r]^(.57){\delta} &\pi_1(Z,z)\ar[r] & \ldots }.
	$$
	We have a pairing $\langle\cdot ,\cdot \rangle\colon  H^1_{\mr{sing}}(Z,\mbb F_p)\times \pi_1(Z,z)\ra \mbb F_p$ and the differential can be computed by the formula $d_{2}^{0,1}(x)=\langle x, \delta([\mbb CP^1])\rangle \in \mbb F_p$.
	
	Thus it is enough to show that if $Z^{S^1}\neq \varnothing$, then we have $\delta([\mbb CP^1])=0$. Pick $z\in Z/{S^1}$ to be the image of an $S^1$-fixed point $\widetilde{z}\in Z^{S^1}$ under $Z\ra Z/S^1$. Under the identification $\pi_2(BS^1,p(z))\simeq \pi_2(\Omega BS^1,p(z))\simeq \pi_1(S^1,e)$ (with $e\in S^1$ being the unit element) the boundary map comes from applying $\pi_1(\cdot)$ to the fiber sequence $\Omega BS^1 \ra Z \ra Z/S^1$, where $\Omega BS^1\simeq S^1$ is identified with the homotopy fiber of $Z \ra Z/S^1$ over $z\ra Z/S^1$. However, the map $S^1\ra Z$, obtained this way, quotients over the $S^1$-orbit corresponding to $z$, which by our assumption is a point. Thus this map is null-homotopic; consequently the boundary map $\delta$ is 0 and we are done.
\end{proof}

\begin{cor}
	Let $\pi\colon X\ra Y$ be a conical resolution (\Cref{def: conical resolution}). Then
	$$
	H^1_{\mr{sing}}(X(\mbb C)/\mbb C^\times,\mbb F_p)\simeq H^1_{\mr{sing}}(X(\mbb C),\mbb F_p).
	$$
\end{cor}
\begin{proof}
	Since $\mbb C^\times \sim S^1$ (homotopy equivalence), it is enough to show that $X(\mbb C)$ satisfies the conditions of \Cref{lem: about S1-action}. Note that the statement depends only only on the base change $\pi_{\mbb C}\coloneqq \pi\times_{\mc O_K}\mbb C$ and for brevity we will continue to denote $X_{\mbb C}$, $Y_{\mbb C}$ and $A_{\mbb C}$ by $X$, $Y$ and $A$.
	
	Note that $Y$ is connected. Indeed, any idempotent $e\in A$ satisfies $e^2=e$ and thus should lie in $A_0$; since $A_0\simeq \mbb C$ we get $e=1$. Since $X$ is birationally equivalent to $Y$, $X$ is also connected (as an algebraic variety) and consequently $X(\mbb C)$ is connected. It remains to see that $X(\mbb C)^{S^1}\neq 0$. Note that the point $y_0\in Y(\mbb C)$ corresponding to the ideal $A_{<0}\subset A$ is fixed under the action of $\mbb C^\times$ and that for any other point $y\in Y(\mbb C)$ the limit $\lim_{t\ra 0} t\circ y$ is equal to $y_0$. Picking $x\in X(\mbb C)$ such that $\pi(x)\neq y_0$ we can consider $x_0\coloneqq \lim_{t\ra 0}t\circ x$; note that this limit exists since $\pi$ is proper. Since for any $t_0\in S^1\subset \mbb C^\times$ we have $t_0\circ x_0 = \lim_{t\ra 0}(t_0 t)\circ x =  \lim_{t'\ra 0}t'\circ x =x_0$ we get that $x_0\in X(\mbb C)^{S^1}$ and we are done.
\end{proof}
\begin{rem}\label{rem:conical resolutions refined conjecture}
Note that by Artin's comparison we have
$$
H^i_{\sing}(X(\mbb C),\mbb F_p)\simeq H^i_{\et}(X_{\mbb C_p},\mbb F_p)
$$
and that from the \'etale comparison \Cref{etale_comparison} we get an  inequality similar to the one in \Cref{conj: conical resolutions} but with the adic \'etale cohomology in place of the algebraic one:
$$
\dim_{k} H^i_{\dR}([X/\mbb G_m]_{k}/k) \ge \dim_{\mbb F_p} H^i_{\et}(\widehat{[X/\mbb G_m]}_{\mbb C_p},\mbb F_p).
$$
We would also like to add that \Cref{conj: conical resolutions} would follow from the comparison of the two versions of \'etale cohomology for formally proper stacks (see \Cref{etale_conjecture}). Indeed \cite[Proposition 4.2.3]{HL_RelaxedProperness} show that $[X/\mbb G_m]$ is formally proper if $\pi\colon X\ra Y$ is projective and more generally one can use the semi-orthogonal decompositions of \cite{HL_GIT} as in \cite[Section 3.2]{KubrakPrikhodko_HdR} to reduce to the case of a quotient of a proper scheme by the trivial action of $\mbb G_m$ (which is covered for example by \cite[Corollary 4.3.2 and Proposition 4.3.4]{HL_RelaxedProperness}). 
\end{rem}

\begin{rem}
	Some conical resolutions (e.g. the Springer resolution $T^*(G/B) \ra \mc N\subset \mf g$ with the rescaling action along fibers of $T^*(G/B)\ra G/B$) satisfy the assumptions in \Cref{rem: comparison for quotients by conical action} and thus \Cref{conj: conical resolutions} definitely holds for them. However, even though it is always true that $X^0$ is proper, it is usually not true that $X$ is equal to $X^+$. For example one can take\footnote{In fact this example is a Springer resolution for $\mf g=\mf{sl}_2$, but we have modified the standard $\mbb G_m$-action so that it is now non-trivial on the central fiber.} $X$ to be the blow up at vertex of the quadratic affine cone $Y=\{(x,y,z) \ \! | \ \! xy=z^2\}\subset \mbb A^3$ and consider the $\mbb G_m$-action on both defined by $t^*(x,y,z)=(tx,t^3y,t^2 z)$. Then one can see that $X^0$ is given by 2 points and is disconnected; however the map $X^+\ra X^0$ always induces a bijection on the sets of connected components and thus (since $X$ itself is connected) we can't have $X=X^+$. The property $X=X^+$ is in fact equivalent to $\mbb G_m$-action being trivial on the central fiber $\pi^{-1}(y_0)$ (where $y_0=Y^{\mbb G_m}$).
\end{rem}

\appendix
\section{Geometric stacks and quasi-coherent sheaves}\label{section_geometric_stacks_and_QCoh}
In this section we review some theory of geometric stacks and quasi-coherent sheaves on them. Here we largely follow \cite{TV_HAGII} and \cite{GaitsRozI}.

\subsection{Artin stacks}\label{subsect: Artin stacks}
In this subsection we recall the notion of a geometric stack and their basic properties, see also \cite[Chapter 1.3]{TV_HAGII}, \cite[Chapter 2, Section 4]{GaitsRozI}, or \cite[Section 2]{PortaYu_GAGA}. First recall
\begin{defn}
Let $(\fcat C, \tau)$ be a pair consisting of an $(\infty, 1)$-category $\fcat C$ and a Grothendieck topology $\tau$. We will denote the category of $\tau$-sheaves (see \Cref{def_sheaves}) with values in the $\infty$-category of spaces $\Type$ by $\Shv_\tau(\fcat C)$. By \cite[Proposition 6.2.2.7]{Lur_HTT} the inclusion $\Shv_\tau(\fcat C) \inj \PShv(\fcat C)$ admits a left exact left adjoint functor, which we will usually denote by $L_\tau$. A sheaf $\mathcal F$ is called a  \emdef{$\tau$-hypercomplete sheaf} or \emdef{$\tau$-hypersheaf} if for any $\tau$-hypercover $U_\bullet \to U$ the induced map $\mathcal F(U) \to \Tot \mathcal F(U_\bullet)$ is an equivalence. 

Let $\pi_0 \colon \Shv_\tau(\fcat C) \to \Shv_\tau(\fcat C, \Set)$ be the natural functor. A map of sheaves $\phi\colon \mathcal F \to \mathcal G$ is called \emdef{surjective} if the induced map $\pi_0(\phi)$ is an epimorphism (in the $1$-categorical sense).
\end{defn}
\begin{defn}
A \emdef{geometric context} is a triple $(\fcat C, \tau, \mathrm P)$ consisting of a $(\infty, 1)$-category $\fcat C$ equipped with a Grothendieck topology $\tau$ and a class of morphisms $\mathrm P$ in $\fcat C$ satisfying the following assumptions:
\begin{enumerate}
\item Every representable presheaf is a hypercomplete $\tau$-sheaf.

\item The class of morphisms $\mathrm P$ is closed under equivalences, compositions and pullbacks.

\item $\tau$-coverings consist of morphisms in $\mathrm P$.

\item $\mathrm P$ is $\tau$-local on the source: if $f\colon X\to Y$ is a morphism in $\fcat C$ such that there exists a covering $\{U_i \to X\}$ such that all compositions $U_i \to X \xrightarrow{f} Y$ are in $\mathrm P$, then $f$ is also in $\mathrm P$.
\end{enumerate}
\end{defn}
\begin{defn}\label{def:geometric_stacks}
Let $\fcat C$ be a category category equipped with a Grothendieck topology $\tau$. A \emdef{prestack $\mstack X$} is just a functor $\mstack X\colon \fcat C^\op \to \Type$. A \emdef{stack} is a prestack which is a sheaf with respect to $\tau$. We will denote the category of prestacks and stacks by $\PStk_{\fcat C}$ and $\Stk_{\fcat C, \tau}$ respectively or just $\PStk$ and $\Stk$ if $\fcat C$ and $\tau$ are clear from the context.
\end{defn}
Note that the previous definition did not use the class of morphisms $\mathrm P$. The choice of $\mathrm P$ allows to define stacks which are geometric in the following sense: they can be presented as iterated quotients of representable stacks by equivalence relations lying in $\mathrm P$. More precisely:
\begin{defn}[Geometric stacks]\label{defn:non_der_Artin}
Let $(\fcat C, \tau, \mathrm P)$ be a geometric context. The category of \emdef{$n$-geometric stacks} is the full subcategory of the category of $(\fcat C, \tau)$-stacks defined inductively as follows:
\begin{itemize}
\item[$\circ$] A stack $\mstack X$ is \emph{$(-1)$-geometric stack} if it is equivalent to a representable sheaf.

\item[$\circ$] A morphism $f\colon \mstack X\to \mstack Y$ of stacks is \emph{$(-1)$-representable} if for any representable sheaf $S$ and a morphism $S\to \mstack Y$ the fibered product $S\times_{\mstack Y}\mstack X$ is representable. It is called a $\mathrm P$-morphism if the corresponding map $S\times_{\mstack Y}\mstack X\to S$ is in $\mathrm P$ for any $S\to \mstack Y$.

\item[$\star$] Assuming the notions of an $(n-1)$-geometric stack, an $(n-1)$-representable morphism and a $\mathrm P$-morphism are already given one can proceed further and define:
\begin{itemize}
\item[$\bullet$] A stack $\mstack X$ is called \emdef{$n$-geometric} if there exists a disjoint union of representable stacks $S = \coprod_i S_i$ and a surjective morphism of sheaves $S\to \mstack X$ that is $\mathrm P$-$(n-1)$-representable. Such a morphism is called a \emph{$\mathrm P$-$n$-atlas} of $\mstack X$.

\item[$\bullet$] A morphism $f\colon\mstack X\to \mstack Y$ of prestacks is \emdef{$n$-representable} if for any representable sheaf $S$ and a map $S\to \mstack Y$ the fibered product $S\times_{\mstack Y}\mstack X$ is $n$-geometric over $S$.

\item[$\bullet$] An $n$-representable morphism $f\colon \mstack X\to \mstack Y$ of stacks is called a \emdef{$\mathrm P$-morphism} if for any representable sheaf $S$ and $S\to \mstack Y$ there exists a $\mathrm P$-$n$-atlas $U\to S\times_{\mstack Y}\mstack X$ such that the composite projection $U\to S$ is a $\mathrm P$-morphism.
\end{itemize}

\item[$\bigstar$] A stack $\mstack X$ is called \emph{geometric} if it is $n$-geometric for some $n$. We say that it is a $\mathrm P$-stack if the base morphism $\mstack Y\to *$ is a $\mathrm P$-morphism.
\end{itemize}
\end{defn}
\begin{defn}\label{relatively_geometric}
A morphism of prestacks $\mstack X \to \mstack Y$ is called \emdef{relatively geometric} if the pullback $S\times_{\mstack Y} \mstack X$ is a geometric stack for every representable $S$ mapping to $\mstack Y$.
\end{defn}

Let $\CAlg_{R/}$ be the (ordinary) category of commutative $R$-algebras and let $\Aff_{/S}\coloneqq\CAlg^\op_{R/}$ be the category of affine schemes over $S \coloneqq \Spec R$. Let also  $\PStk_R\coloneqq  \PStk_{\Aff_{/S}}$ and $\Stk_R\coloneqq  \Stk_{\Aff_{/S}}$. Geometric stacks on $\Aff_R\coloneqq \Aff_{/S}$ with respect to some geometric context are more commonly called \emdef{Artin stacks}. In this work we will mainly be interested in (\'etale, smooth)-Artin stacks in the previous notation. Another natural choice is to consider stacks in flat topology. However, the following result of To\"en shows that by choosing (fppf, fppf)-context we in fact obtain the same category of geometric stacks:
\begin{thm}[{\cite[Th\'eor\`em 0.1]{Toen_AutoFlat}}]\label{thm:Toen_AutoFlat}
We have:
\begin{itemize}
\item Let $\mstack X$ be an (\'etale, smooth)-Artin stack. Then it satisfies fppf descent.

\item Let $\mstack X$ be a (fppf, fppf)-Artin stack. Then $\mstack X$ admits a smooth atlas.
\end{itemize}
In other words, the full subcategories $\Stk^\Art_{\et, \mathrm{smooth},R}$ and $\Stk^\Art_{\fppf, \fppf, R}$ of $\PStk_{R}$ coincide.
\end{thm}

 \noindent We will denote the category of Artin stacks in any of the two equivalent definitions by $\Stk^\Art_R$. In \Cref{adic_local_systems} we will also consider another example of geometric stacks in the rigid analytic context.

\smallskip 

Here are some examples of Artin stacks:
\begin{ex}[Schemes]
Let $X$ be an $R$-scheme. If the diagonal map $\Delta\colon X\to X\times_S X$ is affine (e.g. if $X$ is separated), $X$ is a $0$-Artin stack. In general, $X$ is still a 1-Artin stack.
\end{ex}
\begin{ex}[Quotient stacks]\label{ex:quotient_stacks}
Let $X$ be an $S$-scheme and let $G$ be a group $S$-scheme acting on $X$. One defines \emph{the stacky quotient $[X/G]$} to be the sheafification of the presheaf $A \mapsto X(A)/G(A)$, where the latter means the homotopy quotient (taken in $\Type$). In particular for $X = *$ being the base scheme we will denote $[*/G]$ by $BG$. Unwinding the definitions one finds that a map $Y \to BG$ classifies locally trivial $G$-torsors on $Y$. The quotient stack $[X/G]$ is $1$-Artin and if $G$ is smooth over $S$ the natural map $X \to [X/G]$ is a smooth atlas.
\end{ex}
\begin{ex}
Analogously, if $G$ is an abelian group scheme one can iterate the $B$-construction (by sheafifying $A\mapsto B^nG(A)\coloneqq K(G(A),n)$) to obtain the $n$-classifying stack $B^n G$ for any $n\ge 0$. Note that for a scheme $Y$ the set of (homotopy classes of) maps $[Y, B^n G]$ is naturally isomorphic to $H^n_{\et}(Y, G)$. The iterated classifying stack $B^n G$ is $n$-Artin and if $G$ is smooth the point $*=S$ is an atlas.
\end{ex}
\begin{ex}[Constant stacks]
Let $K$ be a homotopy type. We can define the associated \emdef{constant stack $\underline K$} as the sheafification of the presheaf $A \mapsto K$. E.g. $\underline{\mathbb S^1} \simeq B\underline{\mathbb Z}$, where $\underline{\mathbb Z}$ is a constant group scheme associated to $\mathbb Z$. If $\pi_{>n}(K, x) \simeq 0$ for any point $x\in K$, the stack $\underline K$ is $n$-Artin with an atlas $\underline{\pi_0(K)}$. On the other hand, since the third point of \Cref{Artin_basics} below for an $n$-Artin stack $\mstack X$ and any point $x\colon \Spec A \to \mstack X$ the $n$-fold self product $\Omega^n(\mstack X, s)$ of $s$ with itself over $\mstack X$ must be equivalent to coproduct of affines (in particular it is a set valued functor), we see that the condition $\pi_{>n}(K, x) \simeq 0$ is also necessary for $\underline K$ to be $n$-Artin.
\end{ex}
\begin{ex}[Mapping stacks]
The category of prestacks admits an inner hom, the so called \emdef{mapping prestack $\Map(-, -)$}, given by
$$\Map(\mstack X, \mstack Y)(R) \coloneqq \Hom_{\PStk_{/S}}(\Spec R \times \mstack X, \mstack Y).$$
Moreover, it is easy to see that if $\mstack Y$ is a stack, then so is $\Map(\mstack X, \mstack Y)$ for any prestack $\mstack X$. In a special case when $X$ be a scheme and $\mstack Y = BG$ the corresponding stacks is called \emdef{stack of $G$-bundles on $X$}. See \cite[Section 5]{HL_RelaxedProperness} for sufficient conditions for mapping stacks to be Artin.
\end{ex}

Having any property of a morphism in $\Aff_{R}$ that is local with respect to the smooth topology one can extend it to a property of a morphism of Artin stacks:
\begin{construction}[{See \cite[Section 1.3.6]{TV_HAGII}}]\label{properties_of_morphisms_Artin}
Let $\fcat P$ be a class of morphisms between affine schemes closed under equivalencies, compositions and pullbacks and satisfying the following properties:
\begin{itemize}
\item ($\fcat P$ is local on source) Let $f\colon X\to Y$ be a morphism of affine schemes with the property that for every surjection (in topology $\tau$ from \Cref{defn:non_der_Artin}) $\coprod_i U_i \surj X$ all composite maps $U_i \to X \stackrel{f}{\lra} Y$ are in $\fcat P$. Then $f$ is in $\fcat P$.

\item ($\fcat P$ is local on target) Let $f\colon X\to Y$ be a morphism of affine schemes such that for some cover $\coprod_i U_i \surj Y$ each pullback morphism $U_i \times_Y X \to U_i$ is in $\fcat P$. Then $f$ is in $\fcat P$.
\end{itemize}

Then we say that \emph{an $n$-representable morphism of stacks $f\colon \mstack X \to \mstack Y$ is in $\fcat P$} if for any affine $A$ mapping to $\mstack Y$ there exists an atlas $\coprod_i U_i$ of $A\times_{\mstack Y} \mstack X$ such that all composite maps $U_i \to A$ are in $\fcat P$.
\end{construction}
\begin{ex}\label{ex:lfp_lft_et_sm_fl}
This way one obtains notions of \emdef{locally finitely presentable}, \emdef{locally of finite type}, \emdef{\'etale}, \emdef{smooth}, and \emdef{flat} morphisms of Artin stacks.
\end{ex}
\begin{rem}
For general prestacks one can give the following definition in terms of the categories of quasi-coherent sheaves (see \Cref{def:QCoh} below): a morphsim of prestacks $f\colon \mstack X \to \mstack Y$ is called \emdef{flat} if the pullback functor $f^*\colon \QCoh(\mstack Y) \to \QCoh(\mstack X)$ is $t$-exact (equivalently right $t$-exact). The resulting notion of flat morphism coincides with the one given in the previous example when $f$ is relatively Artin.
\end{rem}
If the class of morphisms is local only on target one can still introduce the following notion:
\begin{construction}
Let $\fcat P$ be a class of morphisms of schemes local on target. Then the map of prestacks $f\colon \mstack X \to \mstack Y$ is called \emdef{schematic $\fcat P$} or just \emdef{$\fcat P$} if for any affine scheme mapping to $\mstack Y$ the pullback $S\times_{\mstack Y} \mstack X$ is a scheme and the natural projection $S\times_{\mstack Y} \mstack X \to S$ lies in $\fcat P$.
\end{construction}
\begin{ex}\label{local_on_target_examples}
This way we obtain a notions of \emdef{(schematic) closed embedding}, \emdef{affine}, \emdef{schematic}, \emdef{schematic projective}, and \emdef{schematic proper} morphisms of prestacks.
\end{ex}

\smallskip The following facts follow formally from the definitions and will be freely used in the rest of the paper without further references:
\begin{prop}[{\cite[Chapter 1.3]{TV_HAGII}}]\label{Artin_basics}
We have:
\begin{enumerate}[label=(\arabic*)]
\item The class of $n$-representable morphisms is closed under pullbacks and compositions.

\item If a geometric stack $\mstack X$ admits an $n$-representable atlas $U \surj \mstack X$, then the diagonal morphism $\mstack X\to \mstack X\times\mstack X$ is $(n-1)$-representable.

\item Any morphism from an affine scheme to an $n$-Artin stack is $(n-1)$-representable.

\item The full subcategory of $n$-Artin stacks is closed under finite limits.

\item Let $\fcat P$ be a class of morphism of affine schemes as in \Cref{properties_of_morphisms_Artin}. Then the class of $\fcat P$-morphism between Artin stacks is closed under equivalences, compositions and pullbacks and is local on target.
\end{enumerate}
\end{prop}

We will also need the following result:
\begin{prop}[{\cite[Proposition 1.2]{Toen_AutoFlat}}]
Let $f\colon \mstack X \to \mstack Y$ be an fppf cover. Then if $\mstack X$ is smooth then so is $\mstack Y$.
\end{prop}
\begin{cor}
Let $G$ be a flat group scheme acting on a smooth scheme $X$. Then the quotient stack $[X/G]$ (see \Cref{ex:quotient_stacks}) is smooth. In particular $BG$ is smooth for any flat group scheme $G$.
\end{cor}

\subsubsection{Finiteness conditions}\label{sec:stacks_finiteness conditions}
In this subsection we recall various types of finiteness conditions on morphisms of Artin stacks, which will be useful for this work.

\paragraph{Quasi-compact and quasi-separated morphisms.}
\begin{defn}\label{quasicompact_Artin}
An Artin stack $\mstack X$ is called \emdef{quasi-compact} if there exists a smooth atlas $U \surj \mstack X$ such that $U$ is an affine scheme (as opposed to an infinite disjoint union of affine schemes, which always exists by the definition of Artin stacks). A morphism $f\colon \mstack X \to \mstack Y$ is called \emdef{quasi-compact} if the pullback $\Spec R \times_{\mstack Y} \mstack X$ is a quasi-compact Artin $R$-stack for any $\Spec R$ mapping to $\mstack Y$.
\end{defn}

The notion of quasi-separated morphism is defined inductively.
\begin{defn}\label{quasiseparated_Artin}
We say that an $n$-Artin stack $\mstack X$ is \emdef{$0$-quasi-separated} if the diagonal morphism $\mstack X \to \mstack X \times \mstack X$ is quasi-compact. We say that an $n$-representable map is \emdef{$0$-quasi-separated} if its base change to an affine scheme is a $0$-quasi-separated $n$-Artin stack.

For $0<k\le n$ we say that an $n$-Artin stack is \emdef{$k$-quasi-separated} if the diagonal morphism $\mstack X \to \mstack X \times \mstack X$ is $(k-1)$-quasi-separated as $(n-1)$-representable map. We say that $n$-representable map is $k$-quasi-separated if its base change to an affine scheme is a $k$-quasi-separated $n$-Artin stack.

An $n$-Artin stack $\mstack X$ is called \emdef{quasi-separated} if it is $k$-quasi-separated for all $0\le k\le n$. An $n$-representable map of stack is called \emdef{quasi-separated} if its base change to an affine scheme is a quasi-separated $n$-Artin stack.
\end{defn}

\paragraph{Finite type and finitely presentable morphisms.}
\begin{defn}\label{def:finite_type}
A morphism $f\colon \mstack X \to \mstack Y$ of Artin stacks is called \emdef{finitely presentable} if it is locally finitely presentable (see \Cref{ex:lfp_lft_et_sm_fl}), quasi-compact and quasi-separated. A morphism $f$ is called \emdef{of finite type} if it is locally of finite type (see \Cref{ex:lfp_lft_et_sm_fl}) and quasi-compact.
\end{defn}

\paragraph{Separated and proper morphisms.}
Following \cite[Section 4]{PortaYu_GAGA} we define:
\begin{defn}\label{def:proper morphism}
A $0$-representable morphism $f\colon \mstack X \to \mstack Y$ is called \emdef{proper} if for any affine scheme $S$ mapping to $\mstack Y$, the pullback $\mstack X\times_{\mstack Y} S$ is a proper $S$-scheme (i.e. if $f$ is schematic proper). Next, assuming that the notion of a proper $(n-1)$-representable morphism is already defined, an $n$-representable morphism $f\colon \mstack X \to \mstack Y$ is called \emdef{proper} if
\begin{itemize}
\item $f$ is \emdef{separated}, i.e. the diagonal map $\mstack X \to \mstack X \times_{\mstack Y} \mstack X$ (which is $(n-1)$-representable) is proper.

\item For any affine scheme $S$ mapping to $\mstack Y$ the pullback $\mstack X_S \coloneqq \mstack X \times_{\mstack Y} S$ admits a surjective $S$-morphism $P \surj \mstack X_S$ such that $P$ is a proper $S$-scheme.
\end{itemize}
\end{defn}
\begin{rem}
Since the property of a morphism of schemes to be proper is flat local on the target, it is enough in the definition above to check the second condition only for some atlas of $\mstack Y$.
\end{rem}
\begin{rem}
A potentially more familiar definition of a (classical) proper algebraic stack $p:\mstack X\ra S$ is that $p$ should be separated, finite type and universally closed. We note that such stacks over $S$ are proper 1-Artin stacks in the definition above. Indeed, by \cite[Theorem 1.1]{Olson_ProperArtin} in this case there exists a proper surjective map $U\ra \mstack X$ from a proper scheme $U$.
\end{rem}

The following properties are proved by the usual diagram chase:
\begin{prop}
The classes of quasi-compact, quasi-separated, finitely presentable, finite type, separated, and proper morphisms are closed under base change, compositions and fibered products.
\end{prop}

\subsection{Quasi-coherent sheaves}\label{sec:Quasi-coherent sheaves}
To each commutative (classical) $R$-algebra $A$ we can associate the unbounded derived category $\DMod{A}$ of $A$-modules. Given a homomorphism $f\colon A\to B$ one has a natural continuous $R$-linear pullback functor $f^*\colon \DMod{A}\to \DMod{B}$ defined by the derived tensor product: $f^*(M)\coloneqq M\otimes_A B$. This way we obtain a functor 
$$\DMod{-}\colon \Aff_{R}^\op \to \Prs^{\mathrm L, \otimes}_R , \qquad \Spec A\mapsto \DMod{A}$$
from $\Aff_{R}^\op$ to the category of $R$-linear presentably symmetric monoidal $\infty$-categories and continuous functors between them. Thinking of $\DMod{A}$ as the category of quasi-coherent sheaves on $\Spec A$, given a prestack $\mstack Y$, one can define the category of quasi-coherent sheaves on $\mstack Y$ via the right Kan extension:
\begin{defn}\label{def:QCoh}
The category $\QCoh(\mstack Y)$ of quasi-coherent sheaves on $\mstack Y$ is defined as the limit 
$$\QCoh(\mstack Y) \coloneqq \lim_{(\Spec A\to \mstack Y) \in (\Aff_{R}/\mstack Y)^\op} \DMod{A}.$$
More precisely, one can extend the assignment $\mstack X \mapsto \QCoh(\mstack X)$ to a functor $\QCoh^*\colon \PStk_R^\op \to \PrL_R$ ($-^*$ staying for the pull-back functor) as the right Kan extension of the $A \mapsto \DMod{A}$ functor along the inclusion $\Aff_{R}\inj \PStk_R$.
\end{defn}

In particular, for a morphism of prestacks $f\colon \mstack X\ra \mstack Y$ we have a well-defined functor $f^*\colon \QCoh(\mstack Y)\ra \QCoh(\mstack X)$.
\begin{rem}
One can think of an object $\mathcal F\in \QCoh(\mstack Y)$ as a collection of complexes $\mathcal F_{|A}\coloneqq x^* \mc F \in \DMod{A}$ associated to each $A$-point $x\colon \Spec A \ra \mstack Y$ compatible under pullbacks but only up to coherent set of higher homotopies.
\end{rem}

The following properties follow more or less formally from the definition (see e.g. \cite[Chapter 3]{GaitsRozI} for more details):
\begin{prop}
We have:
\begin{enumerate}
\item $\QCoh^*$ is a functor from the category of prestacks to the category $\Prs_R^{\otimes, \mathrm L}$ of presentably symmetric monoidal $R$-linear categories and continuous symmetric monoidal $R$-linear functors.

\item By the adjoint functor theorem there is a functor $\QCoh_*\colon \PStk_R \to \PrR_R$ (agreeing with $\QCoh^*$ on objects) such that for each morphism $f\colon \mstack X \to \mstack Y$ the functor $f_*\colon \QCoh(\mstack X) \to \QCoh(\mstack Y)$ is the right adjoint of $f^*\colon \QCoh(\mstack Y) \to \QCoh(\mstack X)$.

\item Let $\mstack X_\bullet \colon I \to \PStk_R$ be a diagram of prestacks. Then the natural map
$$\QCoh(\colim_I \mstack X_i) \xymatrix{\ar[r] &} \lim_I \QCoh(\mstack X_i)$$
is an equivalence.
\end{enumerate}
\end{prop}

Here are few basic examples of $\QCoh$ on various stacks:
\begin{ex}[Schemes]
Let $X$ be an $R$-scheme. Then $\QCoh(X)$ is an $(\infty, 1)$-enhancement of the usual full subcategory $D_{\QCoh}(X)$ of the unbounded derived category of $\mathcal O_X$-modules spanned by complexes with quasi-coherent cohomology sheaves. Note that if $X$ is quasi-compact quasi-separated with affine diagonal, then the natural map $D(\QCoh^\heartsuit(X)) \to D_{\QCoh}(X)$ is an equivalence by \cite[Tag 08DB]{StacksProject}.
\end{ex}
\begin{ex}[Structure sheaf]
For any stack $\mstack Y$ there is the \emph{structure sheaf $\mathcal O_{\mstack Y} \in \QCoh(\mstack Y)$} defined by $(\mathcal O_{\mstack Y})_{|A}\coloneqq A$. By construction, $\mathcal O_{\mstack Y}$ is the monoidal unit of $\QCoh(\mstack Y)$ and is preserved under pullbacks.
\end{ex}
\begin{ex}[Affine morphism]\label{QCoh_affine_morphism}
Let $f\colon \mstack X \to \mstack S$ be an affine morphism of prestacks (see \Cref{local_on_target_examples}). Let $\mathcal A \coloneqq f_*(\mathcal O_{\mstack X})$. Since the functor $f_*$ is right lax monoidal, $f_*$ factors naturally through $K\colon \QCoh(\mstack X) \to \Mod_{\mathcal A}(\QCoh(\mstack S))$. We claim that $K$ is an equivalence. The statement is obvious for affine $\mstack S$, and the general case follows, since by the previous proposition both sides transform colimits to limits. 
\end{ex}
\begin{ex}[Constant stacks]
Let $K$ be a homotopy type. Then since $\QCoh$ transforms colimits of prestacks to limits of categories, we obtain $\QCoh(\underline K) \simeq \LocSys(K, \DMod{R}) \coloneqq \Fun(K, \DMod{R})$. Moreover, $R\Gamma(\underline K, \mathcal O_{\underline K}) \simeq C^*(K, R)$, the complex of $R$-cochains on $K$ with its natural $E_\infty$-algebra structure.
\end{ex}
\begin{ex}[Quotient stacks]\label{qcoh_of_quotinet_stacks}
Let $X$ be an $R$-scheme and let $G$ be a group scheme. Note that $[X/G]$ is equivalent to the geometric realization of the following simplicial scheme
$$\xymatrix{\ldots \ar@<+1.35ex>[r] \ar@<+.45ex>[r] \ar@<-.45ex>[r] \ar@<-1.35ex>[r] & G \times G \times X \ar@<+.9ex>[r] \ar[r] \ar@<-.9ex>[r] & G \times X \ar@<-.45ex>[r]_-{p} \ar@<.45ex>[r]^-{a} & X},$$
where the morphisms in the diagram above are induced by the multiplication on $G$, the action morphisms $a\colon G\times X \to X$ and projections. Since $\QCoh^*$ transforms co-limits of stacks into limits of categories, we obtain
$$\QCoh([X/G]) \simeq \Tot\left(\xymatrix{\QCoh(X) \ar@<-.45ex>[r]_-{p^*} \ar@<.45ex>[r]^-{a^*} & \QCoh(G\times X) \ar@<+.9ex>[r] \ar[r] \ar@<-.9ex>[r] & \QCoh(G\times G \times X) \ar@<+1.35ex>[r] \ar@<+.45ex>[r] \ar@<-.45ex>[r] \ar@<-1.35ex>[r] & \ldots}\right).$$
In particular, to give a quasi-coherent sheaf on $[X/G]$ one need to choose a quasi-coherent sheaf $\mathcal F$ on $X$, choose an equivalence $a^*\mathcal F \simeq p^*\mathcal F$, etc. Hence $\QCoh([X/G])$ is a model for the category of $G$-equivariant sheaves on $X$.
\end{ex}

We also have the following consequence/generalization of faithfully flat descent:
\begin{prop}[{\cite[Chapter 3, Corollary 1.3.7]{GaitsRozI}}]
Let $L_{\widehat\fpqc}\colon \PStk_R \to \Stk_R$ be the hyper-sheafification functor with respect to the fpqc topology. Then for any prestack $\mstack X$ the natural map $\QCoh(L_{\widehat\fpqc} \mstack X) \to \QCoh(\mstack X)$ is an equivalence.
\end{prop}
\begin{cor}
Let $q\colon \mstack U_\bullet \ra \mstack X$ be an fpqc hypercover. Then the natural map
$$\QCoh(\mstack X) \xymatrix{\ar[r] &} \Tot \QCoh(U_\bullet)$$
is an equivalence. In particular, the pull-back functor $q^*\colon \QCoh(\mstack X)\ra \QCoh(\mstack U_0)$ is conservative and for any $\mathcal F \in \QCoh(\mstack X)$ the natural map
$$R\Gamma(\mstack X, \mathcal F) \xymatrix{\ar[r] &} \Tot R\Gamma(\mstack U_\bullet, p_\bullet^* \mathcal F)$$
is an equivalence. 
\end{cor}

The category $\QCoh(\mstack X)$ admits a natural $t$-structure. Recall that by a $t$-structure on a stable $(\infty, 1)$-category $\fcat C$ one means a $t$-structure on the triangulated category $\mathrm h \fcat C$. Equivalently, by \cite[Proposition 1.2.1.5]{Lur_HA} to give a $t$-structure on $\fcat C$ is the same as to specify a full subcategory $\fcat C^{\le 0}$ closed under extensions such that the inclusion $\fcat C^{\le 0} \inj \fcat C$ admits a right adjoint functor $\tau^{\le 0}$. It follows:
\begin{lem}\label{lem: t-structure in the limit}
Let $\{\fcat C_i\}_{i\in I}$ be a diagram of stable categories equipped with $t$-structures and left (reps. right) $t$-exact functors. Then there is a natural $t$-structure on $\fcat C\coloneqq\lim_I \fcat C_i$ such that the projection functors $p_i \colon \fcat C \to \fcat C_i$ are left (resp. right) $t$-exact and $X \in \fcat C^{\ge 0}$ (resp. $X\in \fcat C^{\le 0}$) if and only if $p_i(X) \in \fcat C_i^{\ge 0}$ (resp. $p_i(X) \in \fcat C_i^{\le 0}$) for all $i \in I$.
\end{lem}

\begin{cor}
Let $\mstack X$ be a prestack. There exists a natural $t$-structure on $\QCoh(\mstack X)$ determined by the property that $\mathcal F \in \QCoh(\mstack X)^{\le 0}$ if and only if $\mathcal F_{|A} \in \DMod{A}^{\le 0}$ for all $\Spec A \to \mstack X$. Moreover, for any map $f\colon \mstack X \to \mstack Y$ of pre-stacks the pullback functor $f^*$ is left $t$-exact and the pushforward functor $f_*$ is right $t$-exact.
\end{cor}

If $\mstack X$ is an Artin stack, this $t$-structure enjoys some additional nice property. First recall
\begin{defn}\label{defn:t_complete}
A stable category $\fcat C$ with a $t$-structure is called \emdef{left $t$-complete} if the natural functor
$$\fcat C \xymatrix{\ar[r] &} \prolim \left(\xymatrix{\ldots  \ar[r]^{\tau^{\ge -1}} & \fcat C^{\ge -1} \ar[r]^{\tau^{\ge 0}} & \fcat C^{\ge 0}} \right) =: \widehat{\fcat C}$$
is an equivalence. A stable category $\fcat C$ with a $t$-structure is called \emph{right $t$-complete} if $\fcat C^{op}$ is left $t$-complete.
\end{defn}
\begin{rem}
The category $\widehat{\fcat C}$ can be identified with a full subcategory of $\Fun(\mathbb Z_{\ge 0}^\op, \fcat C)$ spanned by diagrams $X(-)\colon \mathbb Z_{\ge 0}^\op \to \fcat C$ such that $X(i) \in \fcat C^{\le -i}$ and for all $i\ge j$ the natural map $X(i) \to X(j)$ induces an equivalence on $\tau^{\ge -j}$-th truncations. Moreover, if $\fcat C$ admits countable sequential limits, the natural functor $\fcat C \to \widehat{\fcat C}$ admits a right adjoint $X(-) \mapsto \prolim X(i)$. It follows that if $\fcat C$ is left $t$-complete, then the unit of adjunction $X \to \prolim \tau^{\ge -i} X$ is an equivalence for all $X\in \fcat C$, i.e. Postnikov towers converge in $\fcat C$.
\end{rem}
\begin{rem}
By \cite[Proposition 1.2.1.19.]{Lur_HA} if $\fcat C$ admits countable infinite products such that $\fcat C^{\le 0}$ is stable under these, then $\fcat C$ is left $t$-complete if and only if $\bigcap_{i\ge 0} \fcat C^{\le -i} \simeq 0$. On the other hand, in \cite{Neeman_NonLeftComplete} Neeman showed that the unbounded derived category of the abelian category of finite-dimensional representations of the additive group $\mathbb G_a$ over a field of characteristic $p>0$ is not left $t$-complete, precisely by showing that $D^{<0}(\Rep_{\mathbb G_a})$ is not closed under infinite products.
\end{rem}
\begin{lem}\label{limits_of_t_structures}
Let $I\to \Cat_\infty$ be a diagram of stable categories equipped with $t$-structures. If each term of the diagram is left (resp. right) $t$-complete and all translation functors are left (resp. right) $t$-exact, then the limit category $\fcat C\coloneqq\lim_I \fcat C_i$ is also left (resp. right) $t$-complete.

\begin{proof}
We will only prove the statement for the left complete case, the right complete case follows by passing to the opposite categories. By construction of the limit $t$-structure on $\fcat C$ we have $\fcat C^{\ge -j} \areq \lim_I \fcat C_i^{\ge -j}$ for all $j \in \mathbb Z$. But since all $\fcat C_i$ are left $t$-complete, we have $\fcat C_i \areq \lim\limits_{\leftarrow j} \fcat C_i^{\ge -j}$. Since limits commute with each other, we deduce that $\fcat C \areq \lim\limits_{\leftarrow j} \fcat C^{\ge -j}$ as desired.
\end{proof}
\end{lem}
\begin{prop}\label{t_struct_on_QCoh_Artin_stacks}
Let $\mstack X$ be an Artin stack. Then the natural $t$-structure on $\mstack X$ is left and right complete (see \Cref{defn:t_complete}).

\begin{proof}
Right completeness holds for arbitrary prestack by \Cref{limits_of_t_structures}. So it is enough to prove that the natural map
$$\QCoh(\mstack X) \xymatrix{\ar[r] &} \prolim \QCoh(\mstack X)^{\ge -i}$$
is an equivalence. Let $\mstack X$ be $k$-Artin. We will prove the statement by induction on $k$. The base of induction $k=-1$ case reduces to a well known statement that Postnikov's towers converge in the derived category of modules over a ring. Let $U\surj \mstack X$ be a smooth atlas and denote by $U_\bullet$ the corresponding \v Cech nerve. By the inductive assumption the natural map
$$\QCoh(U_n) \xymatrix{\ar[r] &} \prolim \QCoh(U_n)^{\ge -i}$$
is an equivalence for all $n$. But by flat descent
$$\QCoh(\mstack X)^{\ge -i} \xymatrix{\ar[r]^-\sim &} \Tot \QCoh(U_\bullet)^{\ge -i}$$
(here we had used that the pullback functors along smooth morphisms are $t$-exact). We conclude, since limits commute with each other.
\end{proof}
\end{prop}

We will also need the following flat base change for Artin stacks:
\begin{prop}[Base change, {\cite[Corollary 1.3.17]{DrinfeldGaitsgory_FinitenessStacks}}] \label{QCoh_base_change}
Let
\begin{align}\label{qc_bc_eq1}
\xymatrix{
\mstack W \ar[r]^q\ar[d]^g & \mstack Y \ar[d]^f \\
\mstack Z \ar[r]^p & \mstack X
}
\end{align}
be a \emph{derived} fiber product of Artin stacks, where $f$ is a quasi-compact quasi-separated morphism (see Definitions \ref{quasicompact_Artin} and \ref{quasiseparated_Artin}) and $p$ is of finite $\Tor$-amplitude. Then the natural map
\begin{align*}
p^* \circ f_* \xymatrix{\ar[r] &} g_* \circ q^{*}
\end{align*}
is an equivalence.
\end{prop}

\begin{cor}\label{cor:Tor_indep_basechange}
In \Cref{QCoh_base_change} suppose that $\mstack Y$ and $\mstack Z$ are $\Tor$-independent over $\mstack X$, i.e. the derived fibered product $\mstack W$ coincides with the non-derived one $\mstack W'$ (e.g. if either $f$ or $p$ is flat). Then the base change transformation for $\mstack W'$ is also an equivalence.
\end{cor}

Using this we can refine \Cref{qcoh_of_quotinet_stacks} in the affine case:
\begin{prop}
Let $G = \Spec H$ be an affine flat group scheme over $R$ and let $X \coloneqq \Spec A$ be an affine $R$-scheme equipped with an action of $G$. Then there is a natural symmetric monoidal $t$-exact equivalence
$$\QCoh([X/G]) \simeq \Mod_A(\coMod_{H, \DMod{R}}),$$
where by $\coMod_{H, D(R)}$ we denoted the category of $H$-comodules in the derived category $\DMod{R}$.

\begin{proof}
Let us treat the case $[*/G] = BG$ first. Consider the fibered diagram
$$\xymatrix{
G \ar[r]^q\ar[d]^q & {*} \ar[d]^p \\
{*} \ar[r]^p & BG.
}$$
By flat descent and \cite[Proposition 4.7.5.1]{Lur_HA} the pullback functor $p^*\colon \QCoh(BG) \to \DMod{R}$ is comonadic. Moreover, by the \Cref{continuous_comonads_on_mod} below $\coMod_{p^*p_*}(\DMod{R}) \simeq \coMod_H(\DMod{R})$, where $H$ is an $R$-coalgebra $p^* p_*(R)$. But by the base change $H = p^*p_*(R) \simeq q_* q^* (R) \simeq \RG(G, \mathcal O_G)$. The general case now follows from \Cref{QCoh_affine_morphism}. 
\end{proof}
\end{prop}\begin{lem}\label{continuous_comonads_on_mod}
Let $R$ be a commutative ring and let $Q$ be a colimit preserving $R$-linear comonad on $\DMod{R}$. Then $C \coloneqq Q(R)$ has a natural $R$-coalgebra structure and there is a natural equivalence $\coMod_Q(\DMod{R}) \simeq \coMod_C(\DMod{R})$.

\begin{proof}
For any cocomplete $R$-linear category $\fcat D$ the evaluation at $R$ induces an equivalence $\Fun_R^{\mathrm L}(\DMod{R}, \fcat D) \areq \fcat D$. By applying this observation to $\fcat D = \DMod{R}$ we obtain a monoidal equivalence $\End_R^{\mathrm L}(\DMod{R}) \areq \DMod{R}$. It is left to note that a continuous $R$-linear comonad on $\DMod{R}$ is by definition a coalgebra object in $\End_R^{\mathrm L}(\DMod{R})$.
\end{proof}
\end{lem}

\begin{rem}
	One can think of $\coMod_{H, \DMod{R}}$ as a good candidate for the unbounded derived category $\Rep_G$ of $R$-linear representations of $G$. Putting $\Rep_G\coloneqq \coMod_{H, \DMod{R}}$ this claim is supported by the equivalence $\Rep_G\simeq \QCoh(BG)$ (which follows from the proposition in the case $X=\Spec R$).
\end{rem}
By combining this with the fact that for a quasi-compact quasi-separated $1$-Artin stack $\mstack X$ with affine diagonal the natural map $D^+(\QCoh^\heartsuit(\mstack X)) \to \QCoh(\mstack X)^+$ is an equivalence (see \cite[Chapter 3, Proposition 2.4.3]{GaitsRozI}), we deduce:
\begin{cor}\label{derived_vs_derived}
Let $G = \Spec H$ and $X = \Spec A$ be as in the previous proposition. Then there is a natural equivalence
$$\QCoh([X/G])^+ \simeq D^+\big(\Mod_A(\Rep^\heartsuit_G)\big),$$
where $\Rep_G^\heartsuit\simeq \coMod_{H, R}$ is the abelian category of $H$-comodules in $\Mod_R$.

\end{cor}

\subsubsection{Finiteness conditions}
Sometimes it is useful to put some size restrictions on quasi-coherent sheaves. Classically, there are two notions of this type: perfect complexes and coherent sheaves. In this subsection we will discuss their natural extensions to stacks.

Recall that the subcategory $\DMod{A}^\perf\subset \DMod{A}$ of perfect $A$-modules is defined as the minimal full subcategory containing $A$ and closed under finite (co)limits and retracts. It can be also described as the full subcategory of dualizable objects \cite[Proposition 7.2.4.4]{Lur_HA}.
\begin{defn}
A sheaf $\mathcal F$ on a prestack $\mstack Y$ is called \emdef{perfect} if for any map $x\colon \Spec A \to \mstack Y$ the restriction $x^*\mathcal F$ is a perfect $A$-module.  We will denote the full subcategory consisting of perfect sheaves by $\QCoh(\mstack Y)^\perf$. This is an $R$-linear stable subcategory of $\QCoh(\mstack Y)$ closed under finite (co)limits and retracts.
\end{defn}
Similarly to $\DMod{A}^\perf$, there is also a description of $\QCoh(\mstack Y)^\perf$ in terms of dualizability. We recall the following statement:
\begin{lem}[{\cite[Proposition 4.6.1.11]{Lur_HA}}]\label{limit_of_dualizables_is_dualizable}
Let $\{\fcat C_i\}_{i\in I}$ be a diagram of symmetric monoidal categories and symmetric monoidal functors. Then an object $X\in \lim_I \fcat C_i$ is dualizable if and only if the projections of $X$ to all $\fcat C_i$ are dualizable.
\end{lem}

Thus, we get that $\mathcal F\in \QCoh(\mstack Y)$ is perfect if and only if it is dualizable. The following corollary also follows:

\begin{cor}\label{cor: Perf for colim}
Let $\mstack X_\bullet \colon I \to \PStk_S$ be a diagram of pre-stacks over $S$. Then the natural map
$$\QCoh(\colim_{i\in I} \mstack X_i)^\perf \xymatrix{\ar[r] &} \lim_{i\in I} \QCoh(\mstack X_i)^\perf$$
is an equivalence.
\end{cor}

In particular, given an fpqc-surjection $\mstack U\ra \mstack X$ we have an equivalence $q^*\colon \QCoh(\mstack X)^\perf\xra{\sim} \Tot\QCoh(\mstack U_\bullet)^\perf$. Also $\mc F\in \QCoh(\mstack X)$ is perfect if and only if $q^*\mc F$ is.

\begin{ex}
Let $\mstack X \coloneqq [X/G]$. Then by the previous corollary a sheaf $\mathcal F \in \QCoh([X/G])$ is perfect if and only if the pullback along the quotient map $\pi\colon X\to [X/G]$ is perfect. Recall now that by \Cref{qcoh_of_quotinet_stacks} $\QCoh([X/G])$ is equivalent to the category of $G$-equivariant sheaves on $X$ and the pullback along $\pi$ correspond to forgetting the $G$-action. In particular, $\mathcal F\in \QCoh(BG)$ is perfect if and only if the underlying complex of the corresponding representation is perfect.
\end{ex}

\smallskip Another finiteness condition one can put on a sheaf is that it is coherent. For the rest of this section we assume that $R$ is a Noetherian ring.
\begin{defn}
Let $A$ be a Noetherian $R$-algebra. We will call a complex of modules $M\in \DMod{A}$ \emph{coherent} if all but finitely many of its cohomology $H^i(M)$ vanish and $H^i(M)$ is a finitely generated $A$-module for any $i\in \mathbb Z$. We will denote the full subcategory of coherent modules by $\Coh(A)\subset \DMod{A}$.
\end{defn}
It is easy to see that $\Coh(A)$ is closed under finite (co)limits, and retracts. Note that $\DMod{A}^\perf\subset \DMod A$ and if $A$ is regular then in fact $\Coh(A)\simeq \DMod{A}^\perf$ by Serre's homological criterion of regularity.
\begin{defn}\label{def:coh}
A sheaf $\mathcal F$ on a finitely presentable Artin $R$-stack $\mstack Y$ is called \emdef{coherent} if for all smooth maps $y\colon \Spec A \to \mstack Y$ the pullback $y^*\mathcal F$ is a coherent module over $A$. We will denote the full subcategory of $\QCoh(\mstack Y)$ consisting of coherent sheaves by $\Coh(\mstack Y)$.
\end{defn}
\begin{rem}\label{coherent is smooth local}
In fact a sheaf $\mathcal F\in \QCoh(X)$ is coherent if and only if for any smooth atlas of finite type $U\to \mstack X$ the restriction $\mathcal F_{|U}$ lies in $\Coh(U)$. This follows from the fact that a module $M\in \Mod_A$ is finitely generated if and only $M\otimes_A B$ is finitely generated for some faithfully flat finitely generated $A$-algebra $B$. In particular a sheaf $\mathcal F$ is coherent if and only if all of its cohomology sheaves $\mathcal H^i(\mathcal F)$ (with respect to the $t$-structure on $\QCoh(\mstack X)$) are coherent.
\end{rem}

If $R$ is regular, for smooth Artin stacks of finite type coherent means the same as perfect: 
\begin{prop}
Assume $R$ is regular and $\mstack X$ is a smooth Artin stack of finite type over $R$. Then
$$\Coh(\mstack X)\simeq\QCoh(\mstack X)^\perf.$$

\begin{proof}
Since $\mstack X$ is of finite type there is a smooth atlas $\Spec A\ra \mstack X$. Since $\mstack X$ is smooth, $A$ is smooth over $R$, thus is regular. The statement then follows from $\Coh(A)\simeq \DMod{A}^\perf$ and the fact that both categories are local with respect to smooth topology (\Cref{coherent is smooth local} and \Cref{cor: Perf for colim}).
\end{proof}
\end{prop}

However if these assumptions are not satisfied, the two categories are usually different. Coherent sheaves are better behaved with respect to pushforwards along proper morphism:
\begin{prop}\label{proper_push_of_coh_is_coh}
Let $f\colon \mstack X \to \mstack Y$ be a proper schematic map of Artin stacks of finite type. Then for any $\mathcal F \in \Coh(\mstack X)$ its direct image $f_*\mathcal F$ is coherent.

\begin{proof}
Let $U \to \mstack Y$ be a smooth atlas and let $g\colon V\coloneqq U\times_{\mstack Y} \mstack X \to \mstack X$ be the base change of $f$. Then by the flat base change $f_*(\mathcal F)_{|U} \simeq g_*(\mathcal F_{|V})$, which is coherent by the Grothendieck's theorem for proper morphisms between schemes.
\end{proof}
\end{prop}

\subsubsection{Derived tensor functors}
Recall that for a scheme $X$ one defines $\Omega^n_X \coloneqq \wedge^n \Omega_X^1$, where $\wedge^n\colon \QCoh(X)^\heartsuit \to \QCoh(X)^\heartsuit$ is the usual exterior power functor. In order to define the Hodge cohomology of stacks we need to extend this construction to more general geometric objects. Unfortunately, outside of characteristic $0$ case for a stack $\mstack X$ one can not recover $\wedge^n$ purely in terms of $\QCoh(\mstack X)$ considered as a symmetric monoidal category equipped with a $t$-structure.\footnote{It is a folklore result that the additional data necessary to recover a simplicial commutative ring $A_\bullet$ from a symmetric monoidal category of $A_\bullet$-modules is precisely the data of a compatible action of derived symmetric power functors on $\Mod_{A_\bullet}$. We plan to make this assertion precise in future work.} Instead, following \cite[Section 3]{MathewBrantner_LMonad} we will first define $\wedge^n_R$ as an endofunctor of $\DMod{R}$ for arbitrary commutative ring $R$, then show that these functors satisfy base-change, hence extend naturally to stacks. 

First recall that there is a natural extension of a functor defined on the $1$-category $\UMod{R}$ to the connective derived category $\DMod{R}^{\le 0}$:
\begin{defn}[Non-abelian derived functors]
Let $A$ be a commutative ring. Let us denote by $\Mod_A^{\free, \fg}$ a full subcategory of $\UMod{A}$ spanned by finite rank free $A$-modules. By \cite[Section 7.2.2]{Lur_HA} for any category $\fcat D$ admitting sifted colimits the restriction functor
$$\Fun_\Sigma(\DMod{A}^{\le 0}, \fcat D) \xymatrix{\ar[r] &} \Fun(\Mod_A^{\free,\fg}, \fcat D)$$
is an equivalence with the inverse given by the left Kan extension. Here $\Fun_\Sigma(\DMod{A}^{\le 0}, \fcat D)$ denotes the full subcategory of $\Fun(\DMod{A}^{\le 0}, \fcat D)$ spanned by sifted colimits preserving functors. For a functor $F\colon \UMod{A}^{\free,\fg} \to \fcat D$ we will denote its image in $\Fun_\Sigma(\DMod{A}^{\le 0}, \fcat D)$ by $\mathrm LF$ and call it the \emdef{non-abelian derived functor of $F$}.
\end{defn}

An extension to the full derived category $\DMod{R}$ is more involved (e.g. as explained in \cite[Section 3.3]{Kaledin_NonAbelianDerived}, one should probably impose some assumptions on a functor to be able to extend it to the full derived category) and was first written down probably in \cite[Chapitre I.4]{Illusie_Cotangent}. For technical reasons we prefer to use the language developed in \cite[Section 3]{MathewBrantner_LMonad}\footnote{See also \cite{BGMN_K_theory_and_poly}, where the authors with Clark Barwick and Saul Glasman use similar ideas to show that polynomial (in the sense of Goodwillie) functors act on the spectrum of algebraic $K$-theory of a small, idempotent-complete stable $\infty$-category.}, which we briefly review here. First there is the following universal property of $\DMod{A}^\perf$:
\begin{prop}[{\cite[Proposition 3.8]{MathewBrantner_LMonad}}]\label{sifted_preserving}
Let $A$ be a commutative ring. For any category $\fcat D$ admitting sifted colimits, the restriction functor
$$\Fun_\Sigma(\DMod{A}, \fcat D) \xymatrix{\ar[r] &} \Fun_\sigma(\DMod{A}^\perf, \fcat D)$$
is an equivalence, where $\Fun_\sigma(\DMod{A}^\perf, \fcat D)$ denotes the full subcategory of $\Fun(\DMod{A}^\perf, \fcat D)$ spanned by functors preserving geometric realization of $m$-skeletal simplicial objects for some $m\in \mathbb Z_{\ge 0}$. The inverse functor is given by the left Kan extension.
\end{prop}

To use this proposition it is enough to extend our non-Abelian derived functors (defined a priory only on $\DMod{A}^{\le 0}$) to $\DMod{A}^\perf$. To this end we will use a bit the machinery of Goodwillie's calculus of functors (see \cite{Goodwillie_Calculu3} and \cite[Section 6.1]{Lur_HA}). Recall the following definitions:
\begin{defn}
Let $T$ be a set. A \emdef{$T$-cube in a category $\fcat C$} is a functor $\mathrm P(T) \to \fcat C$, where $\mathrm P(T)$ denotes the poset of subsets of $T$. We will denote by $[n]$ the $(n+1)$-element set $\{0,1,\ldots, n\}$ and will usually call $[n]$-cubes just $n$-cubes. A cube $\mathcal X\colon \mathrm P([n]) \to \fcat C$ is called
\begin{itemize}
\item \emdef{strongly cocartesian} if it is left Kan extended from subsets of cardinality at most $1$.

\item \emdef{cartesian} if it is a limit diagram.
\end{itemize}
\end{defn}
\begin{defn}
Let $\fcat C$ be a category admitting finite colimits and let $\fcat D$ be a category admitting finite limits. A functor $F\colon \fcat C \to \fcat D$ is called \emdef{$n$-excisive} if for any strongly cocartesian $n$-cube $\mathcal X$ in $\fcat C$ the cube $F(\mathcal X)$ is cartesian, i.e. if the natural map
$$F(\mathcal X_{\eset}) \xymatrix{\ar[r] &} \lim_{S\in \mathrm P([n])\setminus \eset} F(\mathcal X_S)$$
is an equivalence.
\end{defn}
\begin{ex}
If $\fcat C$ admits the final object $*$, then the functor $F\colon \fcat C \to \fcat D$ is $0$-excisive if and only if it is constant with the value $F(*)$. If the category $\fcat D$ is stable (so that a square in $\fcat D$ is Cartesian if and only if it is coCartesian), then the functor $F\colon \fcat C \to \fcat D$ is $1$-excisive if and only if it preserves finite (co-)limits. 
\end{ex}
\begin{thm}[{\cite[Theorem 6.1.1.10]{Lur_HA}}]\label{n-excisive adjoints}
Let $\fcat C$ be a category admitting finite colimits and a final object and let $\fcat D$ be a category admitting finite limits and countable sequential colimits and such that those commute with each other. Let us denote by $\Exc^n(\fcat C, \fcat D)$ the full subcategory of $\Fun(\fcat C, \fcat D)$ spanned by $n$-excisive functors. Then the inclusion functor $i\colon \Exc^n(\fcat C, \fcat D) \inj \Fun(\fcat C, \fcat D)$ admits a left adjoint $P_n$, which moreover preserves finite limits.
\end{thm}
The relevance of this notion to our problem is explained by the following proposition:
\begin{prop}[{\cite[Theorem 3.35]{MathewBrantner_LMonad}}]\label{excisive_shifts}
The restriction functor
$$\Exc^n(\DMod{A}^{\perf, \le 0}, \fcat D) \xymatrix{\ar[r] &} \Exc^n(\DMod{A}^{\perf}, \fcat D)$$
is an equivalence.
\end{prop}
I.e. any $n$-excisive functor defined on $\DMod{A}^{\perf,\le 0}$ admits an essentially unique $n$-excisive extension to $\DMod{A}^{\perf}$. To be able to apply this proposition we need to know that $\wedge^n\colon \DMod{A}^{\perf, \le 0} \to \DMod{A}$ is $n$-excisive. To this end recall the following definition:
\begin{defn}
We say that a functor $F\colon \fcat C \to \fcat D$ \emdef{is of degree $0$} if it is constant. $F$ is called \emdef{degree $n$} if the difference functor $D_X F(Y) \coloneqq \fib\big(F(X\oplus Y) \to F(Y)\big)$ is a degree $n-1$ functor.
\end{defn}
\begin{ex}
Classically, $\Sym^n, \wedge^n$ and $\Gamma^n$ are degree $n$-functors.
\end{ex}
Then one has:
\begin{prop}[{\cite[Proposition 5.10]{JonsonMcCarthy_TaylorTowers}, \cite[Proposition 3.34]{MathewBrantner_LMonad}}]\label{deg_n_vs_exc}
Let $F\colon \Mod_A \to \fcat D$ be a degree $n$ functor. Then the derived functor $\mathrm LF\colon D^{\le 0}(\mathcal A) \to \fcat D$ is $n$-excisive.
\end{prop}

Finally, we have the following useful observation:
\begin{prop}\label{Goodwillie_very_useful}
Let $F\colon \fcat C \to \fcat D$ be a filtered colimit preserving functor and assume that $\fcat D$ is stable and cocomplete. Then for each $n \in \mathbb Z_{\ge 0}$ the functor $P_n F$ (see \ref{n-excisive adjoints}) preserves sifted colimits.

\begin{proof}
Consider the fiber sequence $D_n F \to P_n F \to P_{n-1} F$. By induction it is enough to prove that $D_n F$ preserves sifted colimits. By \cite[Theorem 6.1.4.7]{Lur_HA} there is a functor $\creff_n F\colon \fcat C^{\times n} \to \fcat D$ which is $1$-excisive separately in each argument and such that $D_n F(X) \simeq \big(\creff_n F(X,X,\ldots X)\big)_{h\Sigma_n}$. Since by assumption $F$ preserves filtered colimits, by construction the same holds for $\creff_n F$. It follows that $\creff_n F$ preserves all colimits separately in each argument. Since by definition of a sifted diagram $I$ the diagonal functor $I \to I^n$ is cofinal, the functor $X \mapsto \creff_n F(X,X, \ldots, X)$ preserves sifted colimits. Since colimits commute, the same holds for $D_n F$.
\end{proof}
\end{prop}

Combing all this one can perform the following construction:
\begin{construction}
Let $F\colon \UMod{A}^{\free, \fg} \to \fcat D$ be a degree $n$ functor. By \Cref{deg_n_vs_exc} the derived functor $\mathrm LF_{|\DMod{A}^{\perf, \le 0}} \colon \DMod{A}^{\perf, \le 0} \to \fcat D$ is $n$-excisive. By \Cref{excisive_shifts} it extends in a canonical way to an $n$-excisive functor $\mathrm LF^{\prime}\colon \DMod{A}^\perf \to \fcat D$. Moreover, by \Cref{Goodwillie_very_useful} $\mathrm LF^{\prime}$ preserves geometric realizations. Taking the left Kan extension, we obtain a unique (by \Cref{sifted_preserving}) sifted colimit preserving functor $\mathrm{RL}F\colon \DMod{A} \to \fcat D$. Note that since $\mathrm{RL}F$ preserves sifted colimits, its restriction to $\DMod{A}^{\le 0}$ is equivalent to $\mathrm LF$.
\end{construction}
\begin{notation}\label{derived_tensor_fns_aff}
For a commutative ring $A$ we will denote by $\Sym^n_A, \wedge^n_A, \Gamma^n_A \colon \DMod{A} \to \DMod{A}$ the functors obtained by the previous construction from the classical symmetric, exterior and divided power functors respectively. If the ring $A$ is clear from the context we will usually omit it from notations and write just $\Sym^n$, $\wedge^n$ and $\Gamma^n$.
\end{notation}
\begin{rem}
If $A$ is a $\mathbb Q$-algebra, then for a free discrete $A$-module $M$ the homotopy symmetric power $\Sym^n_{E_\infty}(M)\coloneqq (M^{\otimes n})_{h\Sigma_n}$ is equivalent to the discrete one. Since $\Sym^n_{E_\infty}$ is $n$-excisive and sifted colimit preserving, it follows that it coincides with the derived symmetric powers functor constructed above. A similar reasoning applies to derived exterior powers.

If $R$ is not a $\mathbb Q$-algebra, then $\Sym_{E_\infty}(M)$ is very different from the usual polynomial algebra on $M$ since its homotopy groups compute homology of symmetric groups with coefficients in $R$. Hence homotopy and derived symmetric powers do not coincide in this case.
\end{rem}

Next we explain how to extend these functors to the category of quasi-coherent sheaves on stacks.
\begin{prop}
Let $A\to B$ be a homomorphism of commutative rings and let $X\in\DMod{A}$ be an $A$-module. Then the natural map
$$(\wedge^n_A X)\otimes_A B \xymatrix{\ar[r] &} \wedge^n_B(X\otimes_A B)$$
is an equivalence.

\begin{proof}
First assume that $X$ is connective. Then $X \simeq |X_\bullet|$ for some simplicial diagram $X_\bullet$ of discrete free $A$-modules. Since both parts commute with sifted colimits, we reduced the statement to the case of free modules, which is well known. The general case follows from equivalences of \Cref{excisive_shifts} and \Cref{sifted_preserving}.
\end{proof}
\end{prop}

Passing to the right Kan extensions we obtain:
\begin{cor}\label{stacky_tensor_functos}
For each $n \in \mathbb Z_{\ge 0}$ there exist natural endotransformations $\Sym^n$, $\wedge^n$ and $\Gamma^n$ of the functor
$$\QCoh^*\colon \PStk^\op \xymatrix{\ar[r] &} {\Pr}^\Sigma$$
(where ${\Pr}^\Sigma$ denotes the $(\infty, 1)$-category of presentable categories and sifted colimit preserving functors) determined by the property that their restriction to $\Aff^\op$ is equivalent to the derived tensor functors from \Cref{derived_tensor_fns_aff}.
\end{cor}
Finally, we will need some standard results about derived tensor functor, which we recall here.
\begin{prop}[Decalage equivalence]
Let $\mstack X$ be a prestack and let $\mathcal F$ be a quasi-coherent sheaf on $\mstack X$. For each $n\in \mathbb Z_{\ge 0}$ there are canonical equivalences:
$$\Sym^n(\mathcal F[1]) \simeq \wedge^n(\mathcal F)[n] \qquad\text{and}\qquad \wedge^n(\mathcal F[1]) \simeq \Gamma^n(\mathcal F)[n].$$

\begin{proof}
The statement is local, so we can assume that $\mstack X \simeq \Spec A$ for some commutative ring $A$. Then the result for connective $\mathcal F$ is proved in \cite[Proposition 25.2.4.2]{Lur_SAG}. The general case follows from equivalences of \Cref{excisive_shifts} and \Cref{sifted_preserving}.
\end{proof}
\end{prop}

We also have the following finiteness result:
\begin{prop}\label{derived_tensor_finiteness}
Let $\mstack X$ be a prestack and let $\mathcal E$ be a perfect sheaf on $\mstack X$. Then for all $n\in \mathbb Z_{\ge 0}$ the quasi-coherent sheaves $\Sym^n(\mathcal E)$, $\wedge^n(\mathcal E)$ and $\Gamma^n(\mathcal E)$ are also perfect.

\begin{proof}
By the decalage equivalence it is enough to prove the proposition only for $\Sym^n$. Moreover, since the statement is local, we can assume that $\mstack X \simeq \Spec A$ for some commutative ring $A$. Let $E$ be perfect $k$-connective module for some $k\in \mathbb Z$. We will prove the statement by induction on $n$ and $k$. The base of induction $k=0$ was treated in \cite[Proposition 25.2.5.3]{Lur_SAG}. Since by construction $\Sym^n$ is $n$-excisive we have an equivalence
$$\Sym^n(E) \simeq \lim_{S\in \mathrm P([n])\setminus \eset} \Sym^n(S \star E),$$
where $S \star E \simeq (E[1])^{\oplus |S|-1}$. Since $E[1]$ is $(k+1)$-connective and since $\Sym^n_A$ of a direct sum can be expressed as a product of $\Sym^i$ for $i\le n$, we conclude by inductive assumption.
\end{proof}
\end{prop}

\subsection{Cotangent complex}\label{sec: cotangent complex}
In this section we will recall the definition and some basic properties of the cotangent complex $\mathbb L_{\mstack Y}$ of a prestack $\mstack Y$. To give a definition we first note that there exists a natural extension of $\mstack Y$ to a functor on simplicial $R$ algebras $\CAlg_{R/}^{\Delta^\op}$ (with its natural $(\infty, 1)$-structure) given by
\begin{align}\label{undervied_stacks_in_derived}
\mstack Y(A_\bullet) \coloneqq (\Lan_{i_\Delta} \mstack Y)(A_\bullet),
\end{align}
where $\Lan_{i_\Delta} \mstack Y$ is the left Kan extension of $\mstack Y$ along the embedding of commutative $R$-algebras into simplicial commutative $R$-algebras. 
\begin{rem}
The definition above is natural from the following perspective: if we denote by $i_{\mathrm{red}} \colon \CAlg^{\mathrm{red}}_{R/} \inj \CAlg_{R/}$ the natural embedding of the category of reduced $R$-algebras, then it is easy to see that for any reduced $R$-scheme $Y\in \Sch_{/S}$ and any (not necessary reduced) $R$-algebra $A$ we have
$$Y(A) \simeq (\Lan_{i_{\mathrm{red}}} Y_{|\CAlg_{/R}^{\mathrm{red}}})(A).$$
Now, one can think about simplicial algebra $A_\bullet$ as a derived nilpotent thickening of the usual algebra $\pi_0(A_\bullet)$ (some elements of $A_\bullet$ are so nilpotent, that they live in higher homotopy degrees). Hence \eqref{undervied_stacks_in_derived} generalizes the usual embedding of the full subcategory of reduced schemes into the category of all schemes.
\end{rem}
\begin{rem}
Outside of characteristic $0$ there are two different notions of cotangent complex: one controls deformations as a simplicial commutative algebra, the other deformations as $E_\infty$-algebras. The usual polynomial algebra $R[x]$ is not free over $R$  when considered as an $E_\infty$-algebra, thus its $E_\infty$-cotangent complex will in general be quite different from the usual module of K\"ahler differentials. For us this behavior is undesirable since we would at least like the cotangent complex of a smooth classical algebra $A$ to be given by the $A$-module of differential 1-forms $\Omega^1_{A/R}$. Essentially by construction one doesn't have this problem in the simplicial context, hence we stick to this case.
\end{rem}
Let $A$ be an $R$-algebra and let $M\in \DMod{A}^{\le 0}$ be a connective complex. Then one can consider the \emph{trivial square-zero extension $A\oplus M$ of $A$ by $M$} which is given as follows: by Dold-Kan correspondence we have a simplicial module $M_\bullet\in \UMod{A}^{\Delta^{op}}$ corresponding to $M$ and then can apply the usual square-zero extension degreewise. With this notations we can formulate the defining property of $\mathbb L_{-}$:
\begin{defn}
Let $f\colon \mstack X \to \mstack Y$ be a morphism of prestacks. The defining property of the \emph{relative cotangent complex $\mathbb L_f=\mathbb L_{\mstack X/\mstack Y} \in \QCoh(\mstack X)$} is that for any $R$ algebra $A$\footnote{One could also ask the analogues representability property for all simplicial $R$-algebras $A_\bullet$. For non-derived stacks there is no difference.}, a point $x \in \mstack X(A)$ and a module $M \in \DMod{A}^{\le 0}$ the space of morphisms $\Hom_{\DMod{A}}(x^* \mathbb L_{\mstack X/\mstack Y}, M)$ is equivalent to the space of lifts
$$\xymatrix{
\Spec A \ar[r]^-{x} \ar[d] & \mstack X \ar[d]^-{f}
\\
\Spec A\oplus M \ar@{-->}[ur] \ar[r]_-{f \circ x \circ q} & \mstack Y
}$$
where $q\colon \Spec A\oplus M \to \Spec A$ is the projection morphism. In the special case $\mstack Y = S$ we will call $\mathbb L_{\mstack X/S}$ just \emdef{cotangent complex of $\mstack X$} and will just write $\mathbb L_{\mstack Y}$ if the base $S$ is clear from the context.
\end{defn}
The cotangent complex $\mathbb L_{\mstack X}$ does not exist in general; however, it does if $\mstack X$ is Artin (see e.g. \cite[Theorem 1.4.3.2]{TV_HAGII} or \cite[Chapter 1, Proposition 7.4.2]{GaitsRozII}). We will now recall some basic properties of the cotangent complexes that are going to be relevant later. The following proposition follows formally from the defining property of $\mathbb L_{-}$:
\begin{prop}\label{cotangent:descent_and_kunneth}
In the above notation we have:
\begin{enumerate}
\item Let $f\colon \mstack X \to \mstack Y$ be a morphism of prestacks such that both $\mstack X$ and $\mstack Y$ admit cotangent complexes over $S$. Then the relative cotangent complex $\mathbb L_f$ exists and is canonically equivalent to the cofiber of the natural map $f^* \mathbb L_{\mstack Y} \to \mathbb L_{\mstack X}$.

\item (Base change for $\mathbb L_{-}$) Let
$$\xymatrix{
\mstack X \ar[r]^q\ar[d] & \mstack Z \ar[d] \\
\mstack Y \ar[r] & \mstack W
}$$
be a derived pullback square of prestacks such that $\mathbb L_{\mstack Z/\mstack W}$ and $\mathbb L_{\mstack X/\mstack Y}$ exist and. Then there exists a canonical equivalence $q^*\mathbb L_{\mstack Z/\mstack W} \simeq \mathbb L_{\mstack X/\mstack Y}$.

\item Let $\mstack X_\bullet \colon I \to \Stk_{/S}$ be a diagram of pre-stacks admitting cotangent complex and flat maps between them and let $\mstack X \coloneqq \lim_I \mstack X_i$. Then $\mathbb L_{\mstack X}$ exists and, if $p_i\colon \mstack X \to \mstack X_i$ is the natural projection, the induced map $\colim_I p_i^* \mathbb L_{\mstack X_i} \to \mathbb L_{\mstack X}$ is an equivalence.
\end{enumerate}
\end{prop}
\begin{rem}
As in \Cref{cor:Tor_indep_basechange} if $\mstack Y$ and $\mstack Z$ are $\Tor$-independent over $\mstack W$, we obtain the base-change for the relative cotangent complexes for the non-derived fibered product.
\end{rem}

It also satisfies the \'etale hyperdescent:
\begin{prop}\label{etale_hyperdescent_for_cotangent}
Let $p_\bullet\colon \mstack U_\bullet \to \mstack X$ be an \'etale hypercover of Artin stacks $\mstack X$. Then the natural map
$$\wedge^n \mathbb L_{\mstack X} \xymatrix{\ar[r] &} \Tot p_{\bullet *}(\wedge^n \mathbb L_{\mstack U_\bullet})$$
is an equivalence for all $n\in \mathbb Z_{\ge 0}$.

\begin{proof}
Since all maps $p_i\colon \mstack U_i \to \mstack X$ are \'etale we have canonical equivalences $p_i^* \wedge^n \mathbb L_{\mstack X} \areq \wedge^n \mathbb L_{\mstack U_i}$. It follows that
$$\Tot p_{\bullet *}(\wedge^n \mathbb L_{\mstack U_\bullet}) \simeq \Tot p_{\bullet *} p_\bullet^* \wedge^n \mathbb L_{\mstack X} \simeq \wedge^n \mathbb L_{\mstack X},$$
where the last equivalence is provided by the flat descent for quasi-coherent sheaves.
\end{proof}
\end{prop}

If $\mstack Y=Y$ is a classical smooth scheme then $\mathbb L_{\mstack Y/S}$ is given by the sheaf $\Omega^1_{Y/S}\in \QCoh(Y)^\heartsuit$ of $S$-relative differential 1-forms. Note that in this case $\Omega^1_{Y/S}$ is locally free and so is a perfect complex. Similar statement also holds for $n$-Artin stacks:
\begin{prop}\label{cotangent_of_smooth_stack}
Let $\mstack Y$ be a smooth $n$-Artin stack over $S$. Then $\mathbb L_{\mstack Y}$ is perfect of $\Tor$-amplitude $[0,n]$. In particular $\mathcal H^i(\mathbb L_{\mstack Y}) \simeq 0$ for $i\not \in [0; n]$.

\begin{proof}
Let $\mstack Y$ be $n$-Artin. We will prove the statement by induction on $n$. If $n=0$, (i.e. $\mstack Y$ is a smooth scheme) since Zariski locally every smooth scheme admits an \'etale map into affine space we can assume $\mstack Y \simeq \mathbb A^n_S$. In this case it is not hard to see from the universal property that $\mathbb L_{\mathbb A^n_S/S} \simeq \mathcal O_{\mathbb A^n_S}^{\oplus n}$, hence is flat (which is equivalent to be of $\Tor$-amplitude $[0;0]$) and dualizable.

Assume now the statement is proved for $n^\prime < n$. Let $p\colon U\surj \mstack Y$ be a smooth atlas. Since the functor $p^*$ is conservative, $t$-exact and detects dualizable objects, it is enough to prove the statement for $p^*\mathbb L_{\mstack Y}$ instead of $\mathbb L_{\mstack Y}$. We have a co-fiber sequence
$$p^*\mathbb L_{\mstack Y} \xymatrix{\ar[r] &} \mathbb L_U \xymatrix{\ar[r] &} \mathbb L_{p}$$
Since dualizable objects are closed under finite limit and since a fiber of $[0;n-1]$ $\Tor$-amplitude objects has $\Tor$-amplitude at most $[0;n]$, it is enough to prove $\mathbb L_p$ is perfect and has $\Tor$-amplitude $[0;n-1]$. For this end consider the fibered square
$$\xymatrix{
U\times_{\mstack Y} U \ar[r]^-q\ar[d]^q & U \ar[d]^p \\
U \ar[r]^p & \mstack Y
}$$
Again since $q$ is smooth and surjective is enough to prove $q^* \mathbb L_p$ has desired properties. But by base-change for the cotangent complex $q^* \mathbb L_p\simeq \mathbb L_q$. We conclude by induction, since the map $q$ is smooth and $(n-1)$-representable.
\end{proof}
\end{prop}

\begin{ex}\label{ex:contangent_for_BG}
Let $G$ be a flat group scheme over $R$ and consider the classifying stack $BG$ from \Cref{ex:quotient_stacks}. It is not hard to see that the diagonal map $BG \to BG\times BG$ classifies the action of $G\times G$ on $G$ given by $(g_1, g_2) \cdot h = g_1 h g_2^{-1}$. Consider the following pullback diagram
$$\xymatrix{
G_{\Ad_G} \ar[r]^p\ar[d]^p & BG \ar[d]^\Delta \\
\ar@/^0.5pc/[u]^e BG \ar[r]^-\Delta & BG\times BG,
}$$
where $G_{\Ad_G}$ stands for the quotient of $G$ by the adjoint action and $e$ is a section of $p$ induced by the unit morphism $* \to G$. It follows by base change  that $\mathbb L_{G_{\Ad_G} / BG} \simeq p^* \mathbb L_{BG/BG\times BG} \simeq p^*\mathbb L_{\mathbb BG}[1]$. On the other hand, by definition $e^*\mathbb L_{G_{\Ad_G} / BG} \simeq \mathfrak g^\vee$ with the adjoint action on the right hand side. It follows that $\mathfrak g^\vee \simeq e^*\mathbb L_{G_{\Ad_G} / BG} \simeq e^*p^*\mathbb L_{ BG} \simeq \mathbb L_{ BG}[1]$, hence $\mathbb L_{ BG} \simeq \mathfrak g^\vee[-1]$.
\end{ex}
\begin{ex}\label{ex: cotangent complex of X/G}
Let $X$ be an $R$-scheme and let $G$ be a flat group $R$-scheme acting on $X$. Note that the structure map $X \to *$ induces a map $p\colon X/G \to BG$ and that there is a fibered square
$$\xymatrix{
X \ar[r]^q\ar[d] & [X/G] \ar[d]^p \\
{*} \ar[r] & BG.
}$$
It follow, there is a fiber sequence $p^*\mathbb L_{BG} \to \mathbb L_{X/G} \to \mathbb L_p$. Moreover, by base change for cotangent complex, $q^* \mathbb L_p \simeq \mathbb L_X$, hence $\mathbb L_p$ is just $\mathbb L_X$ with its natural $G$-equivariant structure. One can check that the natural map $\mathbb L_X \to p^*\mathbb L_{BG}[1] \simeq \mathfrak g^\vee$ is just the natural coaction of $\mathfrak g^\vee$ on $\mathbb L_X$. In particular for a smooth $X$ the cotangent complex $\mathbb L_{X/G}$ is equivalent to the two-term complex $\Omega_X^1 \to \mathfrak g^\vee$ of $G$-equivariant sheaves on $X$.
\end{ex}

We will finish this section by recalling the following flat descent result for the cotangent complex. The proof is essentially due to Bhatt (see \cite[Corollary 2.7, Remark 2.8]{Bhatt_derivedDeRham} or \cite[Section 3]{BMS2}):
\begin{prop}[Flat descent for the cotangent complex, {\cite[Proposition 1.1.5]{KubrakPrikhodko_HdR}}] \label{flat_descent_for_cotangent_compl}
Let $p\colon \mstack U \to \mstack X$ be a surjective quasi-compact quasi-separated flat morphism between Artin stacks and denote by $p_\bullet \colon \mstack U_\bullet \to \mstack X$ the corresponding \v Cech nerve. Then the natural map
$$\wedge^n \mathbb L_{\mstack X} \xymatrix{\ar[r] &} \Tot p_{\bullet *}(\wedge^n \mathbb L_{\mstack U_\bullet})$$
is an equivalence for all $n\in \mathbb Z_{\ge 0}$.
\end{prop}

\numberwithin{thm}{section}
\section{Fontaine's infinitesimal period ring $\Ainf$}\label{sec: Ainf}
Here we recall the definition of Fontaine's ring $\Ainf\coloneqq \Ainf (\mc O_{\mbb C_p})$ and some of its properties.

Let $S$ be a $p$-adically complete ring. We let the \emdef{tilt $S^\flat\coloneqq (S/p)^\perf\coloneqq \lim_{x\mapsto x^p} S/p$} be the inverse limit perfection of $S/p$. It is a perfect ring of characteristic $p$. It is not hard to see that the natural map
$$\lim_{x\mapsto x^p}S \ra \lim_{x\mapsto x^p}S/p\simeq S^\flat$$
is an isomorphism of multiplicative monoids (\cite{BMS1}, Lemma 3.2(i)).
This way, any element $x\in S^\flat$ has a unique expression as a sequence $(x^{(0)},x^{(1)},x^{(2)},\ldots)$ such that $(x^{(i)})^p =x^{(i-1)}$ with $x^{(i)}\in S$. Note that in this description we have $(x^{(0)},x^{(1)},x^{(2)},\ldots)^{1/p}\coloneqq\phi^{-1}((x^{(0)},x^{(1)},x^{(2)},\ldots))=(x^{(1)},x^{(2)},x^{(3)},\ldots)$.
\begin{defn}
	Fontaine's ring $\Ainf(S)$ is defined as
	$$\Ainf(S)\coloneqq W(S^\flat),$$
	where $W(-)$ is the functor of $p$-typical Witt vectors.
\end{defn} 

Note that the absolute Frobenius $\phi\colon S^\flat \xra{\sim} S^\flat$ on $S^\flat$ induces an automorphism $\phi\colon \Ainf(S)\xra{\sim}\Ainf(S)$. Any element $x\in\Ainf(S)$ of Witt vectors has a unique $p$-adic expansion $x\in \sum_{i\ge 0} [x_i]\cdot p^i$ where $[x_i]$ are the Teichm\"uller lifts of some elements $x_i\in S^\flat$. In terms of this expression we have $\phi(x)=\sum_{i\ge 0} [x_i^p]\cdot p^i$. One can also define a map 
$$
\theta\colon \Ainf(S) \xymatrix{\ar[r] &} S,\quad\text{ given by }\quad x=\sum_{i\ge 0} [x_i]\cdot p^i \xymatrix{\ar@{|->}[r] &} \sum_{i\ge 0}(x_i)^{(0)}\cdot p^i \in S,
$$
which turns out to be a homomorphism of rings. 

We now specialize to the case $S=\mc O_{\mbb C_p}$, where $\mathcal{O}_{\mbb C_p}\subset \mbb C_p$ is the ring of integers of the $p$-adic completion of $ \overline{\mbb Q}_p$. For convenience in the paper we put $\Ainf\coloneqq\Ainf(\mc O_{\mbb C_p})$. In this case the $p$-th power map $\mc O_{\mbb C_p}\xra{x\mapsto x^p} \mc O_{\mbb C_p}$ is surjective, and thus so is $\theta\colon \Ainf \ra \mc O_{\mbb C_p}$. Moreover, $\mc O_{\mbb C_p}$ is an example of a \emdef{perfectoid} ring (see \cite[Definition 3.5]{BMS1}); in particular, the ideal $\ker(\theta)$ is principal.

\smallskip We now also discuss some natural elements of $\Ainf$, including an explicit generator $\xi$ for $\ker(\theta)$:
\begin{notation} \label{rem: xi}
	Fix a compatible choice $(1,\zeta_p, \zeta_{p^2},\ldots)$, $\zeta_{p^{n-1}}=\zeta_{p^n}^p\in \mc O_{\mbb C_p}$ of $p^n$-th roots of unity. It defines an element $\epsilon\coloneqq (1,\zeta_p, \zeta_{p^2},\ldots)\in \mc O_{\mbb C_p}^\flat$ with its canonical $p$-th root given by $\epsilon^{1/p}\coloneqq (\zeta_p, \zeta_{p^2},\zeta_{p^3},\ldots)$. The element 
	$$
	\xi\coloneqq 1+[\epsilon^{1/p}]+[\epsilon^{2/p}]+\cdots +[\epsilon^{p-1/p}],
	$$
	then satisfies $\theta(\xi)=1+\zeta_p+\zeta_p^2+\cdots +\zeta_p^{p-1}=0$, so $\xi\in \ker(\theta)$. By \cite[Example 3.16]{BMS1} $\xi$ is also distinguished (see \Cref{defn: distinguished}) and so it generates $\ker(\theta)$ (\cite[Lemma 3.10]{BMS1}). 
	
	We will also denote $\widetilde \xi \coloneqq \phi(\xi)= 1 +[\epsilon]+\cdots +[\epsilon]^{p-1}$. Note that $\widetilde\xi\equiv 1+1+\cdots +1\equiv p \mod (\xi)$.
	Let also $\mu\coloneqq[\epsilon]-1$. We have
	$$
	\frac{\phi(\mu)}{\mu} = \frac{[\epsilon]^p-1}{[\epsilon]-1}=1+[\epsilon]+[\epsilon]^2+\cdots +[\epsilon]^{p-1}=\widetilde \xi.
	$$
\end{notation}

For any element $\pi\in \mf m_{\mc O_{\mbb C_p}}$ and any compatible choice $\pi,\pi^{\frac{1}{p}},\pi^{\frac{1}{p^2}},\ldots$ of $p^n$-th roots of $\pi$, we get an element $\pi^\flat\coloneqq (\pi,\pi^{\frac{1}{p}},\pi^{\frac{1}{p^2}},\ldots)\in \mc O_{\mbb C_p}^\flat$. If we localize $\mc O_{\mbb C_p}^\flat$ at $\pi^\flat$, for any choice of $\pi\in \mf m_{\mc O_{\mbb C_p}}$, we obtain a field ${\mbb C_p^\flat}\coloneqq\mc O_{\mbb C_p^\flat}[\frac{1}{\pi^\flat}]$, which is the fraction field of $\mc O_{\mbb C_p}^\flat$. We have a natural flat map $\Ainf \ra W(\mbb C_p^\flat)$. 

\begin{lem}
	The map $\Ainf \ra W(\mbb C_p^\flat)$ identifies $W(\mbb C_p^\flat)$ with the $p$-adic completion of $\Ainf[\frac{1}{\xi}]$.
\end{lem}
\begin{proof}
	This is a particular case of \Cref{lem: p-completion of A[1/d]}
\end{proof}

\begin{rem}
	Note that the construction $S\mapsto \Ainf(S)$ is functorial. In particular, for a finite field extension $\mbb Q_p\subset K$ there is a natural $G_K$-action on $\Ainf$. It is continuous with respect to $(\xi,p)$-adic topology since the ideal $(\xi,p)$ is $G_K$-invariant.
\end{rem}

\addcontentsline{toc}{section}{References}
\printbibliography

\bigskip

\noindent Dmitry~Kubrak, {\sc Max Planck Institute for Mathematics, Bonn, Germany}
\href{mailto:dmkubrak@gmail.com}{dmkubrak@gmail.com}

\smallskip

\noindent 
Artem~Prikhodko, {\sc Department of Mathematics, National Research University Higher School of Economics, Moscow; Center for Advanced Studies, Skoltech, Moscow,}
\href{mailto:artem.n.prihodko@gmail.com}{artem.n.prihodko@gmail.com}
\end{document}